\documentclass[10pt]{amsart}
\usepackage{latexsym,amsmath,amssymb,graphicx,yhmath}
\usepackage{float}
\usepackage[bf, small]{caption}
\usepackage[all,web]{xy}
\usepackage{maplestd2e}
\usepackage{pdfpages}
\usepackage{color}
\usepackage[colorinlistoftodos]{todonotes}
\usepackage{hyperref}

\usepackage{xcolor}

\def\xyellowspace{%
  \sbox0{\colorbox{yellow}{\strut\ }}
  \dimen0=\wd0\relax
  \hskip0pt\cleaders\box0\hskip\dimen0\hskip0pt}

\begingroup
\catcode`\ =\active%
\gdef\makeyellowspace{\let \xyellowspace\catcode`\ =\active}%
\endgroup

\def\?#1{\colorbox{yellow}{\strut#1}}

\def\urlfont{\DeclareFontFamily{OT1}{cmtt}{\hyphenchar\font='057}
              \normalfont\ttfamily \hyphenpenalty=10000}

\DeclareFontFamily{OT1}{rsfs10}{}
\DeclareFontShape{OT1}{rsfs10}{m}{n}{ <-> rsfs10 }{}
\DeclareMathAlphabet{\mathscript}{OT1}{rsfs10}{m}{n}

\DeclareMathOperator{\Hom}{Hom}     
\DeclareMathOperator{\Tors}{Tors}    
\DeclareMathOperator{\Cl}{Cl}       
\DeclareMathOperator{\Div}{Div}     
\DeclareMathOperator{\Cox}{Cox}     
\DeclareMathOperator{\Eff}{Eff}     
\DeclareMathOperator{\conv}{Conv}   
\DeclareMathOperator{\lcm}{lcm}     

\makeatletter
\def\widebreve{\mathpalette\wide@breve}
\def\wide@breve#1#2{\sbox\z@{$#1#2$}%
     \mathop{\vbox{\m@th\ialign{##\crcr
\kern0.08em\brevefill#1{0.8\wd\z@}\crcr\noalign{\nointerlineskip}%
                    $\hss#1#2\hss$\crcr}}}\limits}
\def\brevefill#1#2{$\m@th\sbox\tw@{$#1($}%
  \hss\resizebox{#2}{\wd\tw@}{\rotatebox[origin=c]{90}{\upshape(}}\hss$}
\makeatletter

\title[Non-calibrated $f$-processes, $D$-equivalence and HMS]{Non-calibrated framed processes, derived equivalence and Homological Mirror Symmetry}

\author[M. Rossi]{Michele Rossi}

\date{\today}

\address{Dipartimento di Matematica e Applicazioni, Universit\`a di Milano-Bicocca,\newline
Ed.~U5-Ratio, via Roberto Cozzi, 55, 20125, Milano} \email{michele.rossi@unimib.it}

\thanks{The author was partially supported by the I.N.d.A.M. as a member of the G.N.S.A.G.A.\\
 Author's ORCID:0000-0001-6191-2087}

\def \a{\alpha }
\def \b{\beta }
\def \d{\delta }
\def \l{\lambda }
\def\mm{\boldsymbol{\mu}}
\def\ll{\boldsymbol{\lambda}}

\def \D{\Delta }
\def \De{\mathcal{D}}

\def \Si{\Sigma }
\def \g{\gamma}
\def \vf{\varphi}
\def \ve{\varepsilon}
\def \ét{\'{e}tale}
\def \aa{\mathbf{a}}
\def \bb{\mathbf{b}}
\def \cc{\mathbf{c}}

\def \q{\mathbf{q}}
\def \pp{\mathbf{p}}

\def \v{\mathbf{v}}

\def \m{\mathbf{m}}

\def \tt{\mathbf{t}}

\def \x{\mathbf{x}}
\def \y{\mathbf{y}}
\def \1{\mathbf{1}}
\def \0{\mathbf{0}}

\def\P{{\mathbb{P}}}

\def\p2{\mathbb{P}^2}
\def\p3{\mathbb{P}^3}
\def\p4{\mathbb{P}^4}

\def\cv#1{\wideparen{#1}}

\def\cO{\mathcal{O}}
\def\cY{\mathcal{Y}}
\def\cX{\mathcal{X}}

\def\SL{\operatorname{SL}}

\def\Z{\mathbb{Z}}

\def\C{\mathbb{C}}

\def\R{\mathbb{R}}

\def\Q{\mathbb{Q}}
\def\N{\mathbb{N}}

\def\T{\mathbb{T}}

\def\L{\Lambda}
\def\cD{\mathcal{D}}

\def\XX{\mathbb{X}}

\def\cI{\mathcal{I}}
\def\cM{\mathcal{M}}
\def\CV{\mathcal{V}}

\def\SF{\mathcal{SF}}
\def\PSF{\mathbb{P}\mathcal{SF}}

\def\Weil{\mathcal{W}_T}

\def\U1{\mathfrak{U}^{(1)}}

\theoremstyle{plain}
\newtheorem{theorem}{Theorem}[section]
\newtheorem{proposition}[theorem]{Proposition}

\newtheorem{thm-def}[theorem]{Theorem--Definition}
\newtheorem{corollary}[theorem]{Corollary}
\newtheorem{conjecture}[theorem]{Conjecture}

\newtheorem*{claim}{Claim}

\newtheorem*{a-proposition}{Proposition}

\theoremstyle{remark}
\newtheorem{remark}[theorem]{Remark}

\newtheorem{example}[theorem]{Example}

\theoremstyle{definition}
\newtheorem{definition}[theorem]{Definition}

\newtheorem*{step I}{Step I}
\newtheorem*{step II}{Step II}
\newtheorem*{step III}{Step III}
\newtheorem*{step IV}{Step IV}

\newtheorem*{acknowledgements}{Acknowledgements}

\newcommand{\oneline}{\vskip12pt}
\newcommand{\halfline}{\vskip6pt}
\newcommand{\cy}{Ca\-la\-bi-Yau }
\newcommand{\lt}{Lib\-go\-ber-Teit\-el\-baum }

\begin{document}
\pagestyle{empty}
\DefineParaStyle{Maple Heading 1}
\DefineParaStyle{Maple Text Output}
\DefineParaStyle{Maple Dash Item}
\DefineParaStyle{Maple Bullet Item}
\DefineParaStyle{Maple Normal}
\DefineParaStyle{Maple Heading 4}
\DefineParaStyle{Maple Heading 3}
\DefineParaStyle{Maple Heading 2}
\DefineParaStyle{Maple Warning}
\DefineParaStyle{Maple Title}
\DefineParaStyle{Maple Error}
\DefineCharStyle{Maple Hyperlink}
\DefineCharStyle{Maple 2D Math}
\DefineCharStyle{Maple Maple Input}
\DefineCharStyle{Maple 2D Output}
\DefineCharStyle{Maple 2D Input}

\begin{abstract} The present paper aims to discuss three kinds of problems:
\begin{enumerate}
  \item producing some ``mirror theorem'' for the recent mirror symmetric construction, called \emph{framed} duality ($f$-duality), described in \cite{R-fTV} and \cite{R-fpCI}: this is performed from the point of view proposed by Homological Mirror Symmetry (HMS), by studying \emph{derived equivalence} ($D$-equivalence) of multiple mirror models produced by means of a, so-called, \emph{uncalibrated $f$-process};
  \item proposing a general construction giving a big number of multiple mirror models to, in principle, any projective complete intersection of non-negative Kodaira dimension: these multiple mirrors turn out to be each other connected by means of uncalibrated $f$-processes and then, after (1), $D$-equivalent or $K$-equivalent, in the sense of Kawamata \cite{Kawamata};
  \item presenting a number of evidences for the Bondal-Orlov-Kawamata conjecture that $D$-equivalence is $K$-equivalence, and viceversa.
   \end{enumerate}
      \end{abstract}

\maketitle

\tableofcontents

\section*{Introduction}

The present paper aims to discuss three kinds of problems:
\begin{enumerate}
  \item producing some ``mirror theorem'' for the recent mirror symmetric construction, called \emph{framed} duality ($f$-duality), described in \cite{R-fTV} and \cite{R-fpCI}: this is performed from the point of view proposed by Homological Mirror Symmetry (HMS), introduced by Kontsevich \cite{Kontsevich}, by studying \emph{derived equivalence} ($D$-equivalence) of multiple mirror models produced by means of a, so-called, \emph{uncalibrated $f$-process};
  \item proposing a general construction giving a large number of multiple mirror  models to, in principle, any projective complete intersection of non-negative Kodaira dimension: these multiple mirrors turn out to be each other connected by means of uncalibrated $f$-processes and then, after (1), $D$-equivalent or $K$-equivalent, in the sense of Kawamata \cite{Kawamata};
  \item presenting a number of evidences for the Bondal-Orlov-Kawamata conjecture that $D$-equivalence is $K$-equivalence, and viceversa.
\end{enumerate}

More precisely, starting from (3), Bondal-Orlov and Kawamata formulated the following

\begin{conjecture}[Conj.~1.2 in \cite{Kawamata} and Conj.~4.4 in \cite{Bondal-Orlov}]\label{conj:Kawamata}
  Let $X$ and $Y$ be birationally equivalent smooth algebraic varieties. Then the following are equivalent.
  \begin{itemize}
    \item[$(D)$] There is an equivalence of triangulated categories between bounded derived categories of coherent sheaves
        \begin{equation*}
          \cD(X):=\cD^b(\rm{Coh}(X))\ \cong\ \cD^b(\rm{Coh}(Y))=:\cD(Y)
        \end{equation*}
    \item[$(K)$] There exists a smooth projective variety $Z$ and birational morphisms
    \begin{equation*}
      \xymatrix{&Z\ar[dl]_-f\ar[dr]^-g&\\
                X&&Y}
    \end{equation*}
    such that $f^*K_X\sim g^*K_Y$, where $\sim$ denotes linear equivalence of divisors in $Z$.
  \end{itemize}
\end{conjecture}

In particular, $X$ and $Y$ are called \emph{$D$-equivalent} if they satisfy $(D)$ and \emph{$K$-equivalent} if they satisfy $(K)$. This conjecture has been proven by Bridgeland when $\dim X=3=\dim Y$. \cite{Bridgeland}.
In \cite{Kawamata}, Kawamata discusses some evidences for this conjecture.

In this paper, a number of examples and (conjectural) evidences for this conjecture are presented via $f$-duality and HMS: namely they are given by Theorems~\ref{thm:Dequivalenza}, \ref{thm:Dequivalenza^A}, \ref{thm:Kequivalenza3}, \ref{thm:Kequivalenzad}, \ref{thm:Kequiv-generale}, \ref{thm:D}, \ref{thm:Dequivalenza^cy}, \ref{thm:Malter-type}, \ref{thm:Kequivalenza_cy}\,. Actually these results give evidences for a generalization of Conjecture~\ref{conj:Kawamata} to a mildly singular setup, namely

\begin{conjecture}\label{conj:BOKgen}
  Let $X$ and $Y$ be birationally equivalent normal $\Q$-Gorenstein projective varieties. Then the following are equivalent.
  \begin{itemize}
    \item[$(D')$] Their canonical covering stacks $\cX$ and $\cY$ (see Definition~\ref{def:K-stack}) admit equivalent derived categories of bounded complexes of coherent orbifold sheaves
  \begin{equation*}
    \cD^b(\cX)\cong\cD^b(\cY)
  \end{equation*}
    \item[$(K')$] The birational map $X\stackrel{\cong}{\dashrightarrow}Y$ is \emph{crepant} (see Definition~\ref{def:crepante}).
  \end{itemize}
\end{conjecture}

In particular, $X$ and $Y$ are called \emph{$D$-equivalent} if they satisfy $(D')$ and \emph{$K$-equivalent} if they satisfy $(K')$. Putting together results by Orlov \cite{Orlov} and Kawamata \cite{Kawamata}, one can also argue the following

\begin{conjecture}\label{conj:sing}
    Let $X$ and $Y$ be birationally equivalent normal $\Q$-Gorenstein projective varieties admitting \emph{sufficiently similar} singularities. Then $D$-equivalence for $X$ and $Y$ is equivalent to say that there exists an equivalence between the associated categories of singularities $$\cD_{sg}(X)\cong\cD_{sg}(Y)$$
\end{conjecture}
Recall that the category of singularites $\cD_{sg}(X)$ is, by definition, the Verdier quotient of $\cD^b(X)$ by the full subcategory $\text{Perf}(X)$ of perfect objects. What does mean ``sufficiently similar'' singularities in the previous statement will be specified in any particular considered setup: see Conjectures~\ref{conj:sings}, \ref{conj:sing'} and \ref{conj:D}.

For a further approach to an extension of Conjecture~\ref{conj:Kawamata} to a singular setup the reader can also consider the Kawamata logarithmic extension given by \cite[Conj.~2.2]{Kaw05} and, in the 3-dimensional case, Chen generalization of Bridgeland argument \cite{Chen}.

\subsection{Mirror theorems for framed duality}
All this evidence is obtained by performing mirror theorems of kind (1), by means of construction of kind (2) answering to the open problem \cite[8.1]{R-fTV}. In fact, in \cite{R-fTV} and \cite{R-fpCI} a generalization of Batyrev-Borisov duality was proposed, as a correspondence between \emph{framed toric varieties} (see \S~\ref{ssez:ftv}). For instance, one can think of a Fano toric variety as a complete toric variety framed by an anti-canonical divisor. The couple given by a complete toric variety and a strictly effective torus invariant Weil divisor $(X,D)$ is a framed toric variety (ftv). If the framing divisor admits a \emph{partition} $D=\sum_{k=1}^l D_k$ we have a partitioned ftv (see the definition in \S~\ref{ssez:ftv}). Then \emph{framed duality} ($f$-duality) is a correspondence between (partitioned) framed toric varieties restricting to give Batyrev-Borisov duality when $X$ is Gorenstein and $D=-K_X$. In general it is not an involutive process, that is, squaring $f$-duality does not give back the starting ftv. If this is the case, we have a, so called, \emph{calibrated} $f$-process, otherwise we have an non-calibrated (\emph{uncalibrated}) $f$-process (see \S~\ref{ssez:calibrato}). First of all, we verify consistency of calibration with HMS (see Theorem~\ref{thm:smallK&D-equivalenza}). Then a mirror theorem for $f$-mirror symmetry can be obtained by checking consistency with HMS of an uncalibrated $f$-process, that is:

\begin{theorem}[Mirror Theorem]\label{thm:mirror}
  Consider an uncalibrated $f$-process
  $$(X,D)\rightsquigarrow (X^\vee,D^\vee)\rightsquigarrow (X',D')$$
  Let $Y$ and $Y'$ be sufficiently general elements $Y\in|D|$ and $Y'\in|D'|$. Then $Y$ and $Y'$ are $D$-equivalent, up to some kind of resolution of singularities.
\end{theorem}

\subsection{Multiple mirrors}
This is where construction (2) comes into play. Mirror symmetry was born as a physical duality. But a mathematical translation of the word \emph{duality} has a stronger meaning including an involutive  behaviour which is not implied by the physical definition. Polar duality giving rise to Batyrev-Borisov mirror symmetry further strengthened the idea of mirror symmetry as an involutive correspondence making, in a sense, exceptional the occurrence of multiple mirrors phenomena. On the contrary, $f$-mirror symmetry shows that multiple mirrors can occur just by changing the framing inside the same linear equivalence class of divisors, so that one should think of mirror symmetry as a sort of web. The just mentioned method of producing multiple mirrors does not work for a \cy complete intersection, as there is a unique choice of framing inside the anti-canonical class. In this paper we will show how producing a big number of multiple mirrors, for nearly all projective complete  intersection of non-negative Kodaira dimension, hence for \cy complete intersections, too. All these multiple mirrors turn out to be birational equivalent and, moreover, $K$-equivalent and, in the 3-dimensional case, also $D$-equivalent, hence giving all the evidences for Conjecture~\ref{conj:BOKgen} mentioned above.
In the different examples and setup here considered, we will prove a contextualization of the following

\begin{theorem}[Metatheorem]\label{thm:metateorema}
    Given a projective complete intersection $Y$, let $Y^\vee_{BB}$ be the mirror model of $Y$ obtained by means of a calibrated $f$-process. Then it belongs to a suitable list, denoted $\cM$, of mirror models of $Y$, such that for every $Y^\vee\in \cM$
    \begin{itemize}
      \item[(i)] there exists an uncalibrated $f$-process connecting the three of them, that is,
      \begin{equation*}
        Y^\vee\stackrel{\text{$f$-dual}}{\rightsquigarrow}Y\stackrel{\text{$f$-dual}}{\rightsquigarrow}Y^\vee_{BB}
      \end{equation*}
      \item[(ii)] $Y^\vee$ and $Y^\vee_{BB}$ are $K$-equivalent,
      \item[(iii)] if $Y$ is a (Calabi-Yau) threefold then $Y^\vee$ and $Y^\vee_{BB}$ are also $D$-equivalent.
    \end{itemize}
\end{theorem}
The \cy hypothesis in brackets means that it can be dropped if a suitable common level of desingularization for both $Y^\vee$ and $Y^\vee_{BB}$ is chosen: see e.g. Theorem~\ref{thm:D} and Remark~\ref{rem:noCY}.

Let me say that, in constructing such a \emph{suitable list of mirror models}, I was stimulated by interesting papers by Favero and Kelly \cite{Favero-Kelly} and Malter \cite{Malter}. In particular, I am referring to Remarks~3.8 and 4.10 in Malter's paper. Namely, in 1993, Libgober and Teitelbaum \cite{LT} proposed a mirror model for the projective \cy threefold given as the complete intersection of two cubic hypersurfaces in $\P^5$, generalizing the Greene-Plesser approach \cite{GP} for the 5-tic threefold. Later, Batyrev and Borisov proposed their general mirror symmetric construction \cite{Batyrev94}, \cite{Borisov}, \cite{BB}. In spite of the fact that the Greene-Plesser (GP) mirror model and the Batyrev one coincide in the hypersurface case, for complete intersection this fact does no more hold, as Libgober-Teitelbaum (LT) and Batyrev-Borisov (BB) mirror models look to be quite different \cy manifolds. This fact proposes a first evident example of multiple mirrors, although the different behaviour of the hypersurface case with respect to the case of complete intersection sounds a bit unnatural. Actually, we will show that these two mirror models turn out to be linked by a crepant birational morphism (see \S~\ref{sssez:birational_33} and in particular Proposition~\ref{prop:birat3,3}). Then they are birational and, moreover, $K$-equivalent (see Theorem~\ref{thm:Kequivalenza3}) and $D$-equivalent (see Theorem~\ref{thm:Dequivalenza} and Corollary~\ref{cor:D-equiv_3,3}), so proving, in this setup, Theorem~\ref{thm:metateorema} for the particular list $\cM=\{Y^\vee_{LT},Y^\vee_{BB}\}$. Their $D$-equivalence has been also partially proved by Malter \cite[Thm.~2.23]{Malter}: see the following Theorem~\ref{thm:Malter 2.23} and considerations given in Remark~\ref{rem:noMalter}.

The particular Libgober-Teitelbaum construction generalizes in several directions allowing us
\begin{itemize}
  \item[(a)] to propose a (non-unique) generalized LT-mirror model associated to \emph{almost} every complete intersection of projective hypersurfaces, actually beyond the \cy constraint; that is a different mirror model with respect to the one obtained by a calibrated $f$-mirror process, as in \cite{R-fpCI}, and still denoted by BB-mirror model, although Batyrev-Borisov duality cannot be applied in such a generalized setup: this is the construction proposed in \S~\ref{ssez:mirror-costruzione} and in particular in \S~\ref{ssez:mirrors}, where the meaning of the word \emph{almost} is clarified by Assumptions (A), (B) and (C) in \S~\ref{sssez:LT-ipotesi}; in \S~\ref{ssez:LTmirrors} this construction is explicitly studied for the \cy threefold $Y_{2,2,3}$ given as the complete intersection of two hyperquadrics and a cubic hypersurface in $\P^6$;
  \item[(b)] to show the existence of an uncalibrated $f$-process connecting the LT-mirror with the $BB$-mirror: this is done in several setups with Proposition~\ref{prop:LTdual}, Remark~\ref{rem:f-picture}, Proposition~\ref{prop:LTdual_cubics}, Remark~\ref{rem:f-picture3}, Theorem~\ref{thm:CY_d,d} and \S~\ref{sssez:LT-ipotesi}, \ref{ssez:mirrors}, \S~\ref{ssez:BB_2,2,3} and \S~\ref{ssez:LTmirrors}, so proving item (i) of Theorem~\ref{thm:metateorema}, in every considered setup, for the list $\cM=\{Y^\vee_{LT},Y^\vee_{BB}\}$;
  \item[(c)] to show that the birational morphism linking the BB-mirror and the LT-mirror factorizes through a potentially big number of, so called, \emph{intermediate mirror models}, so interestingly enriching the multiple mirror picture attached to a projective complete intersection: just to give an idea of what means \emph{potentially big}, a complete intersection of $l$ hypersurfaces in $\P^n$ and satisfying assumptions (A), (B) and (C) in \S~\ref{sssez:LT-ipotesi}, admits at least $2^{(l-1)(n+1)}$ distinct multiple mirrors; for each intermediate mirror model there exists an uncalibrated $f$-process landing to the BB-mirror model; this is performed in several setups, namely \S~\ref{sssez:mirrorintermedi3,3} and Proposition~\ref{prop:mirrorintermedi3,3}, \S~\ref{sssez:mirrorintermedid,d} and Proposition~\ref{prop:mirrorintermedid,d}, \S~\ref{ssez:mirrors}(2), \S~\ref{sssez:mirrorintermedi223}; clearly these results give an extension of item (i) in Theorem~\ref{thm:metateorema} to the big list $\cM$ of intermediate mirrors;
  \item[(d)] all multiple mirrors, presented in the previous item (c), are $K$-equivalent and, in the 3-dimensional \cy cases, also $D$-equivalent, so proving the Mirror Theorem~\ref{thm:mirror} and items (ii) and (iii) in Theorem~\ref{thm:metateorema}, in the considered setups, for the list $\cM$ of all the intermediate mirrors; namely, for $K$-equivalence consider Theorems \ref{thm:Kequivalenza3}, \ref{thm:Kequivalenzad}, \ref{thm:Kequiv-generale} and \ref{thm:Kequivalenza_cy}, while for $D$-equivalence in dimension 3 consider Theorems~\ref{thm:Dequivalenza}, \ref{thm:Dequivalenza^A}, \ref{thm:D} and \ref{thm:Dequivalenza^cy}; finally evidences for Conjecture~\ref{conj:sing} are given in Theorem~\ref{thm:Malter-type}.
\end{itemize}

\subsection{Metatheorem vs Mirror Theorem} Item (iii) of Theorem~\ref{thm:metateorema} proves the mirror theorem~\ref{thm:mirror} in the particular case that the considered projective complete intersection $Y$ is a \cy threefold. Taking into account Conjectures~\ref{conj:Kawamata} and \ref{conj:BOKgen}, item (ii) of Theorem~\ref{thm:metateorema} gives a conjectural approach to the mirror theorem~\ref{thm:mirror}, beyond the constraint to be a \cy threefold. $K$-equivalence guaranteed by Theorem~\ref{thm:metateorema} is proved, in the several setups here proposed, by analyzing birational maps linking ambient toric varieties and then restricting those maps to embedded mirror models and their strict transforms. The \cy condition makes \emph{crepant} these restricted birational maps, then allowing us, in the 3-dimensional case, to directly apply Kawamata's arguments \cite{Kawamata} to definitively prove the conjectured $D$-equivalence. Often, the \cy condition can be bypassed by the choice of a suitable level of \emph{common partial resolution of singularities} for the involved multiple mirror models: this phenomenon is described in \S~\ref{ssez:KvsD} and in particular by Theorems~\ref{thm:Kequiv-generale} and \ref{thm:D} and Remark~\ref{rem:noCY}.

\oneline
The present paper is organized as follows. In \S~\ref{sez:preliminari} needed preliminaries and notation are quickly recalled and fixed. A first result concerning the HMS consistency of a calibrated $f$-process is here proved (see Theorem~\ref{thm:smallK&D-equivalenza}). \S~\ref{sez:Malter} is a warm up, devoted to describe the general strategy in the easiest and lowest dimensional cases of the elliptic curve $Y_{2,2}\subset\P^4$ and the \cy threefold $Y_{3,3}\subset\P^5$, also considered by Malter in \cite{Malter}. In \S~\ref{sez:d,d} a first level of generalization on the degree/dimension $d$ is introduced, by studying multiple mirrors of the \cy complete intersection $Y_{d,d}\subset\P^{2d-1}$. The general recipe is presented and described in \S~\ref{sez:LTgeneral} and checked in the particular cases of Kodaira positive dimensional complete intersections $Y_{2,2,3}\subset\P^5$ and $Y_{3,4,5}\subset\P^8$. The more general results and evidences on $K$-equivalence, $D$-equivalence and their interplays among the different mirror models are here finally discussed. The final \S~\ref{sez:2,2,3} is devoted to giving an application of general techniques and results previously analyzed in the particular case of the \cy threefold $Y_{2,2,3}\subset\P^6$\,. Appendices \ref{app:A}, \ref{app:B} and \ref{app:C} explicitly describe some characteristic data of multiple mirrors described \S~\ref{sez:Malter}.

 \begin{acknowledgements}
   It is a pleasure to thank T.~Kelly and his PhD student A.~Malter for interesting correspondence we had: with the latter after his paper \cite{Malter} appeared on the arXiv, and with the former in the occurrence of the online workshop \emph{Toric Fano Varieties and beyond}, December 4, 2020. It is the occasion to thank also the workshop organizers G.~Bini and D.~Iacono. In particular, the employment of non-calibrated $f$-processes, which are the main objects studied in the present paper, answers a stimulating question of G.~Bini, related with the open problem \cite[8.1]{R-fTV}.

 I also wish to thank Geometry staffs of both the Maths Departments of the University and the Polytechnic of Turin and in particular the participants in the study seminar on Derived Categories organized during the winter 2021/22: the present paper could never have appeared without that series of inspiring events.

 Many thanks also go to M.~Kalck, who pointed out some inaccuracies in an early version of this paper.

Many computations and proofs' prototypes have been partially performed by means of several Maple routines, mostly of them jointly written with L.~Terracini, and some of them based on the Maple package \texttt{Convex} \cite{Convex}.

Seven months later than the first version of the present paper appeared on the ArXiv, M.~Pochekay published the first version of \cite{Pochekay}, which is part of his PHD thesis, written under the supervision of his advisor D.~Eriksson. Results in \S~\ref{ssez:cubics} of the present paper should be compared with Pochekay's constructions and in particular with \S~5 in \cite{Pochekay}. It is my pleasure to thank both D.Eriksson and M.Pochekay for interesting discussions we had in occasion of the official defense of the Pochekay thesis. Many thanks for invitation and warm hospitality in G\"{o}teborg.
 \end{acknowledgements}

\section{Preliminaries and Notation}\label{sez:preliminari}

For foundational concepts on toric varieties and general notation we completely refer the reader to \S~1 and \S~2 in \cite{R-fTV} and references there cited. For the reader convenience we just recall the following few facts.

\subsection{Cox quotient presentation, fan matrix and weight matrix}\label{ssez:Cox} Consider a complete toric variety $X$ of dimension $n$ with Picard number $r$, and
\begin{equation}\label{div}
  \xymatrix{0\ar[r]&M\ar[r]^-{div}& \Weil(X)\ar[r]^-{cl}&\Cl(X)\ar[r]&0}
\end{equation}
be the Weil divisor exact sequence \cite[Thm.~4.1.3]{CLS}, where $\Weil(X)\cong\Z^{n+r}$ is the free group of torus invariant Weil divisors and $\Cl(X)$ the class group. A \emph{fan matrix} $V$ of $X$ is the transposed matrix of a representative matrix $V^T$ of the homomorphism $div$. A weight matrix $Q$ of $X$ is a \emph{Gale dual} matrix of $V$. In particular, given a representation $\Cl(X)\cong\Z^r\oplus\Tors(\Cl(X))$, being $\Tors(\Cl(X))$ the canonical torsion subgroup of $\Cl(X)$, the class homomorphism $cl$ is represented, for the torsion landing part, by a \emph{torsion matrix} $T$ and, for the free landing part, by a weight matrix $Q$ \cite[\S~3]{RT-QUOT},\cite{RT-Erratum}. Dualizing (\ref{div}) over $\C^*$ one obtains the multiplicative dual exact sequence
\begin{equation*}
  \xymatrix{1\ar[r]&\Hom(\Cl(X), \C^*)\ar[r]^-{\exp(cl^\vee)}&\Hom(\Weil(X),\C^*)\ar[r]&\Hom(M,\C^*)\ar[r]&1}
\end{equation*}
where $\Hom(M,\C^*)\cong N\otimes\C^*=\T^n$ is the torus acting on $X$, that is,
\begin{equation*}
  \T^n\cong X\setminus\bigcup_{i=1}^{n+r} D_i
\end{equation*}
being $D_i$ the torus invariant prime divisors of $X$ obtained as the closure of the torus orbit of special points of rays generated by columns of the fan matrix $V$ and freely generating $\Weil(X)$. Under the natural multiplicative action of $\Hom(\Weil(X),\C^*)\cong\T^{n+r}$ on $\C^{n+r}$, the monomorphism $\exp(cl^\vee)$ induces an action of the quasi-torus $\Hom(\Cl(X), \C^*)\cong\T^r\times\mm$ on $\C^{n+r}$ so that
\begin{equation*}
  X\cong (\C^{n+r}\setminus Z_X)/(\T^r\times\mm)
\end{equation*}
being $Z_X$ the unstable locus of the multiplicative action of
\begin{equation}\label{azione}
  \exp(cl^\vee)(\Hom(\Cl(X), \C^*)\cong\exp(Q)(\T^r)\times \exp(T)(\mm)
\end{equation}
We will say that \emph{the weight matrix $Q$ and the torsion matrix $T$ determine the quasi-torus action of $\T^r\times\mm$ on the characteristic space $\C^{n+r}\setminus Z_X$}, as their columns give exponents of the quasi-torus action (\ref{azione}).  \\
For further details the interested reader is referred to the original Cox's paper \cite{Cox}.

\subsection{Small $\Q$-factorial resolutions of a complete toric variety}\label{ssez:SF} Consider a complete toric variety $X$ with a corresponding fan matrix $V$. We will denote by
\begin{itemize}
  \item $\SF(V)$ the set of simplicial fans $\Si$ admitting as 1-skeleton $\Si(1)$ the set of rays generated by all the columns of $V$, that is,
      \begin{equation*}
        \Si(1)=\{\langle\v_i\rangle\,|\,\forall\,i=1,\ldots,n+r\quad\v_i\ \text{is the $i$-th column of}\ V\}
      \end{equation*}
  \item $\PSF(V)$ the subset of $\SF(V)$ giving rise to a \emph{projective} $\Q$-factorial toric variety.
\end{itemize}
Since $X$ is complete, by a suitable simplicial subdivision of its maximal cones one obtains a fan $\Si\in\SF(V)$ giving a refinement of the fan of $X$ obtained without adding any further ray. Calling $\widehat{X}(\Si)$ the $\Q$-factorial, complete, toric variety defined by $\Si$, there is an induced birational morphism
\begin{equation*}
  \xymatrix{\psi_\Si:\widehat{X}(\Si)\ar[r]&X}
\end{equation*}
which is a \emph{small}, partial, resolution of singularities of $X$.

\subsection{Framed and partitioned framed toric varieties}\label{ssez:ftv} Given a complete toric variety $X$ of dimension $n$ and Picard number $r$, let us call $D_1,\ldots,D_{n+r}$ the prime torus invariant divisors generating $\Weil(X)$, as above. Then, we introduced the following notions:
\begin{enumerate}
  \item \cite[Def. 2.1]{R-fTV} a \emph{framing} of $X$ is defined as a strictly effective divisor $D_\aa=\sum_{i=1}^{n+r}a_i D_i$, that is, $a_i>0$ for every $i$; the couple $(X,D_\a)$, often also denoted $(X,\aa)$, is called a \emph{framed toric variety} (ftv);
  \item \cite[Def. 6.1]{R-fTV}, \cite[Def. 1.4]{R-fpCI} A \emph{partition} of a given framing $D_\aa$ is the datum of a partition
  $$\exists\,l\in \N:\quad I_1\cup\cdots\cup I_l=\{1,\ldots,m\}\ ,\quad\forall\,i\ I_i\ne\emptyset\ ,\quad\forall\,i\neq j\quad I_i\cap I_j =\emptyset$$
   and divisors $D_{\aa_1},\ldots,D_{\aa_l}$ such that
  $$\forall\,k=1,\ldots,l\quad D_{\aa_k}:=\sum_{i\in I_k}a_iD_i$$
  Clearly $D_\aa=\sum_{k=1}^l D_{\aa_k}$, that is, $\aa=\sum_{k=1}^l \aa_k$\,. The toric variety $X$ endowed with a \emph{partitioned framing} $\aa=\sum_{k=1}^l\aa_k$ is called a \emph{partitioned ftv} and denoted by $(X,\aa=\sum_{k=1}^l\aa_k)$\,.
\end{enumerate}

\subsection{Calibrated and uncalibrated $f$-processes}\label{ssez:calibrato} Given a ftv $(X,\aa)$ there is a unique \emph{framed dual} ($f$-dual) ftv $(\XX_\aa,\bb)$ described by construction 2.1.1 in \cite{R-fTV}. Analogously, given a partitioned ftv $(X,\aa=\sum_{k=1}^l\aa_k)$ there is a unique $f$-dual partitioned ftv $(\cv{\XX}_\aa,\cv{\bb}=\sum_{k=1}^l\bb_k)$, assigned by algorithm 1.1.1 in \cite{R-fpCI}.

\begin{remark}
  In the following we will consider only partitioned ftv, then the partitioned $f$-dual will be denoted by $(\XX_\aa,\bb= \sum_{k=1}^l\bb_k)$, for ease, but one should always recall that $\cv{\XX}_\aa\neq\XX_\aa$ and $\cv{\bb}\neq\bb$: notice that the latter are framing on distinct toric varieties.
\end{remark}

By definition, we call a (partitioned) \emph{$f$-process} the double application of (partitioned) $f$-duality. This gives rise to a third (partitioned) ftv $(\XX_\bb, \cc= \sum_{k=1}^l\cc_k)$, that is
\begin{equation}\label{f-processo}
  \xymatrix{(X,\aa)\ar@{~>}[rr]^-{\text{$f$-process}}\ar@{~>}[dr]_-{\text{$f$-dual}}&&(\XX_\bb, \cc)\\
            &(\XX_\aa,\bb)\ar@{~>}[ur]_-{\text{$f$-dual}}}
\end{equation}

\begin{definition}[calibrated $f$-process, see Def.~2.14 in \cite{R-fTV} and Def.~1.7 in \cite{R-fpCI}]\label{def:calibrato}
   Given a (partitioned) $f$-process (\ref{f-processo}), let $V$ and $\L$ be fan matrices of $X$ and $\XX_\bb$, respectively. The $f$-process is called \emph{calibrated} if  there exist $\Xi\in\SF(V)$ and $\Xi'\in\SF(\L)$ such that
  $$\left(\widehat{X}(\Xi),\psi_\Xi^*D_\aa\right)\ {\cong}\ \left(\widehat{X}'(\Xi'),\psi_{\Xi'}^*D'_{\cc}\right)$$
  are isomorphic framed toric varieties, and
  $$\psi_\Xi:\widehat{X}(\Xi)\longrightarrow X\quad\text{and}\quad\psi_{\Xi'}:\widehat{X}'(\Xi')\longrightarrow \XX_\bb$$
  are small $\Q$-factorial resolutions of $\Xi$ and $\Xi'$, respectively.

  On the contrary, if the calibration condition is not satisfied, the $f$-process is referred to as \emph{uncalibrated}.
\end{definition}

\subsubsection{Calibration and HMS} The Definition~\ref{def:calibrato} of a calibrated $f$-process is consistent with HMS, at least conjecturally. The crucial point is the following

\begin{theorem}\label{thm:smallK&D-equivalenza}
  Let $X$ be a complete toric variety, $V$ a fan matrix of $X$ and $D_\aa=\sum_{k=1}^lD_{\aa_k}$ a partitioned framing on $X$.
  For any simplicial refinement $\Si,\Si'\in\SF(V)$ of the fan of $X$, let
  \begin{equation*}
    \psi_\Si:\widehat{X}(\Si)\longrightarrow X\quad\text{and}\quad\psi_{\Si'}:\widehat{X}'(\Si')\longrightarrow X
  \end{equation*}
  be the associated small $\Q$-factorial resolutions. Then:
  \begin{enumerate}
    \item $\widehat{X}$ and $\widehat{X}'$ are $K$-equivalent;
    \item $(\widehat{X},\psi_\Si^*D_\aa=\sum_{k=1}^l\psi_\Si^*D_{\aa_k})$ and $(\widehat{X}',\psi_{\Si'}^*D_\aa=\sum_{k=1}^l\psi_{\Si'}^*D_{\aa_k})$ are partitioned ftv and, for generic
        \begin{eqnarray*}
          \widehat{Y} &=& \bigcap_{k=1}^l \widehat{Y}_k\subset \widehat{X}\quad\text{with}\ \forall\,k\ \widehat{Y}_k\in|\psi_\Si^*D_{\aa_k}| \\
          \widehat{Y}' &=& \bigcap_{k=1}^l \widehat{Y}'_k\subset \widehat{X}'\quad\text{with}\ \forall\,k\ \widehat{Y}'_k\in|\psi_{\Si'}^*D_{\aa_k}|
        \end{eqnarray*}
       then $\widehat{Y}$ and $\widehat{Y}'$ are $K$-equivalent;
    \item if $\dim X=3$ then $\widehat{X}$ and $\widehat{X}'$ are $D$-equivalent;
    \item if $\dim\widehat{Y}=\dim\widehat{Y}'=3$ then $\widehat{Y}$ and $\widehat{Y}'$ are $D$-equivalent.
  \end{enumerate}
  \end{theorem}
  \begin{remark}
    Beyond the dimensional constraints in items (3) and (4) of the previous statement, by Conjecture~\ref{conj:BOKgen}, one should expect that $D$-equivalences $\widehat{X}\sim_\cD \widehat{X}'$ and $\widehat{Y}\sim_\cD\widehat{Y}'$ hold in general, as  a consequence of items (1) and (2).
  \end{remark}

  \begin{proof}[Proof of Thm.~\ref{thm:smallK&D-equivalenza}]
    The birational equivalence between $\widehat{X}$ and $\widehat{X}'$ given by
     \begin{equation*}
       \xymatrix{\widehat{X}(\Si)\ar@{-->}[rr]^\vf\ar[dr]_{\psi_\Si}&&\widehat{X}'(\Si')\ar[dl]^{\psi_{\Si'}}\\
                    &X&}
     \end{equation*}
     is an isomorphism in codimension 1, also called a \emph{small $\Q$-factorial modification} (s$\Q$m) between $\Q$-factorial projective toric varieties. Recalling the geometry of the secondary fan, $\vf$ is obtained by a finite (non unique) sequence of wall-crossings and so it is a finite sequence of flops:
 \begin{equation*}
   \exists\ s\in\N\ :\ \vf=\vf_1\circ\cdots\circ\vf_s
 \end{equation*}
 For any $i=1,\ldots,s$\,, $\vf_i$ either is the identity or replaces one facet $\tau_i$, between neighboring maximal cones of the fan $\Si$, with a different facet $\tau'_i$ between neighboring maximal cones of the fan $\Si'$. Then, there is a chain of commutative diagrams
  \begin{equation}\label{flop-catena}
  \resizebox{1\hsize}{!}{$
    \xymatrix{&\widehat{X}_s\ar[dl]^-{\text{\ blowups}\ \tau_s\cap\tau'_s}_-{f_s}\ar[dr]^-{g_s}&&\cdots\ar[dl]\ar[dr]&&\widehat{X}_1\ar[dl]^-{\text{\ blowups}\ \tau_1\cap\tau'_1}_-{f_1}\ar[dr]^-{g_1}&\\
    \widehat{X}(\Si)\ar[rr]^-{\vf_s}\ar[rd]_{\text{contract}\,\tau_s}&&
    \widehat{X}^{(s-1)}(\Si_{s-1})\ar[dl]^{\text{contract}\,\tau'_s}\ar[r]^-{\vf_{s-1}}
    \ar[dr]&\cdots\ar[r]^-{\vf_{2}}&
    \widehat{X}^{(1)}(\Si_{1})\ar[dl] \ar[rd]_{\text{contract}\,\tau_1} \ar[rr]^-{\vf_1}\ar[rd]&&\widehat{X}'(\Si')\ar[dl]^{\text{contract}\,\tau'_1}\\
             &X_s&&\cdots&&X_1&        }$}
  \end{equation}
  where $\widehat{X}_i=\widehat{X}^{(i)}(\Si_i)\times_{X_i}\widehat{X}^{(i-1)}(\Si_{i-1})$, for $i=1,\ldots,s$, and setting $\widehat{X}^{(0)}=\widehat{X}'$ and $\widehat{X}^{(s)}=\widehat{X}$\,. Moreover
  $$\Si=\Si_s\ ,\ \Si_{s-1}\ ,\ \ldots\ ,\ \Si_1\ ,\ \Si_0=\Si'$$
  is a sequence of intermediate simplicial fans corresponding, by Gale duality, to a chosen sequence of wall crossings, that is a chain of pairwise neighboring full-dimensional chambers, connecting the chambers determined by $\Si$ and $\Si'$ inside the secondary fan. Notice that, for any $i=1\ldots,s$, the blow up of $\tau_i\cap\tau'_{i}$, say $\theta_i:\widehat{X}_i\to X_i$, is a divisorial one, giving rise to a number of exceptional divisors, say $D_{N+1},\ldots,D_{N+h}$, being
  $$N=|\Si_{i-1}(1)|=|\Si_{i}(1)|$$
  the cardinality of associated 1-skeletons. Therefore, from (\ref{flop-catena}) one gets
  \begin{equation*}
    \forall\,i=1,\ldots,n\quad f_i^*K_{\widehat{X}^({i-1})}\sim_\Q-\sum_{i=1}^N D_i=K_{\widehat{X}_i}+\sum_{j=1}^h D_{N+j}\sim_\Q g_i^*K_{\widehat{X}^{(i)}}
  \end{equation*}
  meaning that, for any $i$, $\widehat{X}^{({i})}$ and $\widehat{X}^{({i-1})}$ are $K$-equivalent. Then $\widehat{X}$ and $\widehat{X}'$ are $K$-equivalent,
   hence proving item (1).

   The first part of item (2) follows by the fact that both $\psi_\Si$ and $\psi_{\Si'}$ are small resolutions (see Prop.~1.8 and Cor.~2.12 in \cite{R-fTV}). To show that $\widehat{Y}$ and $\widehat{Y}'$ are $K$-equivalent, restrict birational morphisms $f_i,\,g_i$ to the embedded complete intersections. More precisely, call $\psi_i:\widehat{X}^{(i)}\to X$ the small resolution determined by the simplicial subdivision $\Si_i$ of the fan of $X$.  In particular, $\psi_s=\psi_\Si$ and $\psi_0=\psi_{\Si'}$. Consider the strict transform $\widehat{Y}^{(i)}:=\psi_i^*(Y)$ of the complete intersection $Y\subset X$\,. Thinking of $D_j$ as the prime divisor given by the closure of the torus orbit of the special point of the ray $\rho_j$ and observing that $1$-skeletons of the fans of $\widehat{X}^{(i)},\, \widehat{X}^{(i-1)},\,X$ coincide, by abuse of notation, we will denote by the same letter $D_j$, the homologous prime divisor on the three toric varieties, that is $\psi_i^*D_j=D_j=\psi_{i-1}^*D_j$\,.
   For every $i$, let $r_i\in\N\setminus\{0\}$ the minimum positive integer such that $r_iK_{\widehat{X}^{(i)}}$ is Cartier. Set $r:=\lcm(r_i,r_{i-1})$ so that, by adjunction, for $l=i-1,i$
   \begin{equation*}
     \omega_{\widehat{Y}^{(l)}}^{\otimes r}\cong (\omega_{\widehat{X}^{(l)}}^{\otimes r}\otimes\cO_{\widehat{X}^{(l)}}(rD_\aa))|_{\widehat{Y}^{(l)}}
     \cong\cO_{\widehat{X}^{(l)}}\left(r\sum_{j=1}^N(a_j-1)D_j\right)|_{\widehat{Y}^{(l)}}
   \end{equation*}
   Therefore,  taking into account the strict transform $\widehat{Y}_i:=f_i^*(\widehat{Y}^{(i)})\cong g_i^*(\widehat{Y}^{(i-1)})$,
   \begin{eqnarray*}
     (f_i^*\omega_{\widehat{Y}^{(i)}})^{\otimes r}&\cong&\cO_{\widehat{X}_i}
     \left(r\sum_{j=1}^N(a_j-1)f^*_i D_j\right)|_{\widehat{Y}_i}\\
     &\cong& \cO_{\widehat{X}_i}
     \left(r\sum_{j=1}^N(a_j-1)g^*_i D_j\right)|_{\widehat{Y}_i}\cong (g^*_i\omega_{\widehat{Y}^{(i-1)}})^{\otimes r}
   \end{eqnarray*}
so that $f_i^*K_{\widehat{Y}^{(i)}}\sim_\Q g_i^*K_{\widehat{Y}^{(i-1)}}$, that is, $\widehat{Y}^{(i)}$ and $\widehat{Y}^{(i-1)}$ are $K$-equivalent for any $i$, by Definition~\ref{def:crepante}. Then $\widehat{Y}$ and $\widehat{Y}'$ are $K$-equivalent, definitively proving item (2).

  Finally, items (3) and (4) follow by applying Kawamata results \cite[Thm.~4.6]{Kawamata}, decomposing the birational equivalence $\vf$ into a sequence of flops (in the sense of \cite[Def.~4.5]{Kawamata}), and \cite[Thm.~6.5]{Kawamata}.
  \end{proof}

  \begin{remark}\label{rem:triangolo} If the reader is interested in getting the top variety of \emph{triangles} realizing $K$-equivalences stated in items (1) and (2) of previous Theorem~\ref{thm:smallK&D-equivalenza}, then he has to going on, by doing fibred products $\widehat{X}_i\times_{\widehat{X}^(i)} \widehat{X}_{i-1}$ until obtaining a top toric variety $\widetilde{X}$ and triangles
  \begin{equation*}
    \xymatrix{&\widetilde{X}\ar[dl]_f\ar[dr]^g&\\
            \widehat{X}&&\widehat{X}'}\quad\begin{array}{c}
                                             \text{restricting to} \\
                                             \text{embedded c.i.} \\
                                             \Longrightarrow
                                           \end{array}\quad
    \xymatrix{&\widetilde{Y}\ar[dl]_f\ar[dr]^g&\\
            \widehat{Y}&&\widehat{Y}'}
  \end{equation*}
  such that, by construction, $f^*K_{\widehat{X}}\sim_\Q g^*K_{\widehat{X}'}$ and $f^*K_{\widehat{Y}}\sim_\Q g^*K_{\widehat{Y}'}$\,.

  \end{remark}

\subsection{Framing polytopes}\label{ssez:politopi} Let $X$ be a complete toric variety of dimension $n$ and Picard number $r$ and let $V$ be a fan matrix of $X$. Let $D_\aa=\sum_{i=1}^{n+r} a_i D_i\in \Weil(X)$ be a framing of $X$. The associated polytope $\Delta_{D_\alpha}$ is also denoted by $\Delta_\alpha$ and is defined as
\begin{equation*}
  \D_\aa:=\{\m\in M\otimes\R\,|\, V^T\cdot \m\ge -\aa\}
\end{equation*}
where $\ge$ means that $V^T\cdot \m +\aa$ is a vector whose entries are non-negative. In general, $\D_\aa$ is a polytope whose vertices lives in $M\otimes\Q$.
The \emph{integer part} of $\D_\aa$ is by definition
\begin{equation*}
  [\D_\aa]:=\conv(M\cap\D_\aa)
\end{equation*}
Given a polytope $\D$ such that $\0$ is a relative interior point of $\D$, one can construct a complete toric variety $\XX$ \emph{spanned by $\D$} by considering the fan obtained by projecting from $\0$ every facet of $\D$ and than taking every subcone of these maximal projecting cones (for further details the interested reader is referred to \cite[Prop.~1.4]{R-fTV}).

\subsection{Vectors, subvectors, matrices, submatrices}\label{ssez:vettori} In the following, we adopt the following notation
\begin{enumerate}
  \item $\0_n:=\left(
                \begin{array}{ccc}
                  0 & \cdots & 0 \\
                \end{array}
              \right)\in \C^n$ and $\0_n^T$ will denote its transpose, that is, the zero column vector;
  \item $\1_n:=\left(
                \begin{array}{ccc}
                  1 & \cdots & 1 \\
                \end{array}
              \right)\in \C^n$ and $\1_n^T$ will denote its transpose;
  \item $\0_{m,n}$ is the zero matrix in $\C^{m,n}$;
  \item $\1_{m,n}$ is the matrix in $\C^{m,n}$ whose entries are all equal to 1;
  \item given a matrix $M=\left(
                            \begin{array}{ccc}
                              M^1 & \cdots & M^n \\
                            \end{array}
                          \right)\in \C^{m,n}$ and a sublist $A\subseteq\{1,\ldots,n\}$,
                          $M_A$ will denote the submatrix whose columns are indexed by $A$, that is
                          \begin{equation*}
                            M_A=(M^i\,|\,i\in A)
                          \end{equation*}
                          while $M^A$ will denote the complementary submatrix of $M_A$ in $M$.
\end{enumerate}

\section{Malter's examples and framed mirror symmetry}\label{sez:Malter}

As a warm up, let us start by discussing in detail examples treated by Malter in \cite{Malter}, that are the complete intersection of two quadrics in $\P^3$ and the complete intersection of two cubics in $\P^5$: they will be driving examples, in which introducing all main strategies and results.

\subsection{The complete intersections of two quadrics in $\P^3$}\label{ssez:quadrics}
We treat in detail this example, although main Theorems~\ref{thm:mirror} and \ref{thm:metateorema} become trivial in this case: but low dimension allows us to introducing main techniques and performing needed computations in an easier context. Moreover, this example has not been actually treated by Malter, but just extrapolated as an easier consequence (see Thm.~4.8 and Remark~4.9 in \cite{Malter}).

Let $Q_1,Q_2$ be two generic quadrics in $\P^3$. Calling $D_i=\CV(x_i),\ 1\le i\le 4,$ the four prime torus invariant divisors of $\P^3$, we can assume $Q_1\sim D_1+D_2$ and $Q_2\sim D_3+D_4$. The complete intersection
\begin{equation*}
  Y=Y_{2,2}=Q_1\cap Q_2 \subset \P^3
\end{equation*}
is a smooth elliptic curve corresponding to the choice of the nef partitioned framing of $\P^3$
\begin{equation*}
  \aa=(1,1,1,1)=\aa_1+\aa_2\ ,\quad \text{with}\quad \aa_1=(1,1,0,0)\,,\,\aa_2=(0,0,1,1)
\end{equation*}

\subsubsection{The Batyrev-Borisov mirror}
Applying Batyrev-Borisov duality is the same as applying framed duality to the partitioned framed toric variety (ftv) $(\P^3,\aa=\aa_1+\aa_2)$, so getting the following

\begin{proposition}\label{prop:BBdual}
  The $f$-dual partitioned ftv of $(\P^3,\aa=\aa_1+\aa_2)$ is the partitioned ftv $(\XX,\bb=\bb_1+\bb_2)$ where $\XX$ is the complete toric variety whose fan is spanned by the faces of the polytope
  \begin{equation}\label{Lambda}
    \D:=\conv\left( \begin {array}{cccccccc} 1&-1&-1&-1&2&0&0&0\\ -1&1&-1&-1&0&2&0&0\\ 0&0&2&0&-1&-1&1&-1\end {array} \right)
  \end{equation}
  and $\bb_1=(1,1,1,1,0,0,0,0)\,,\,\bb_2=(0,0,0,0,1,1,1,1)$. In particular, the $f$-mirror family of the family $\{Y\}_{Q_1,Q_2}$ is the Batyrev-Borisov mirror given by the family of elliptic curves $Y^\vee_{BB}=\CV(p_{1,\psi},p_{2,\psi})_{\psi}$ with
  \begin{eqnarray*}
    p_{1,\psi}&=& {x_{{1}}}^{2}{x_{{5}}}^{2}+{x_{{2}}}^{2}{x_{{6}}}^{2}+\psi x_{{1}}x_{{2}}x_{{3}}x_{{4}}\in \Cox(\XX)\\
    p_{2,\psi}&=& {x_{{3}}}^{2}{x_{{7}}}^{2}+{x_{{4}}}^{2}{x_{{8}}}^{2}+\psi x_{{5}}x_{{6}}x_{{7}}x_{{8}}\in \Cox(\XX)
  \end{eqnarray*}
  being $\Cox(\XX)=\C[x_1,\ldots,x_8]_{\Cl(\XX)}$.
\end{proposition}

\begin{proof}
  The proof follows Algorithm~1.1.1 in \cite{R-fpCI}. What follows is sketched in the upper part of Fig.~\ref{Fig1}. Recalling \S~\ref{ssez:politopi}, let $\D_{\aa_i}$ be the polytope associated with the divisor $D_{\aa_i}$, then
  \begin{equation*}
    \D_{\aa_1}=\conv\left( \begin {array}{cccc} 1&-1&-1&-1\\ -1&1&-1&-1\\ 0&0&2&0\end {array} \right)\ ,\ \D_{\aa_2}=\conv\left( \begin {array}{cccc} 2&0&0&0\\ 0&2&0&0\\ -1&-1&1&-1\end {array} \right)
  \end{equation*}
  giving $\D=\conv(\D_{\aa_1},\D_{\aa_2})$. Calling $\L$ the fan matrix of $\XX$, that is the matrix whose columns are given by vertices of $\D$, and $V$ the fan matrix of $\P^3$, that is
  \begin{equation*}
    V= \left( \begin {array}{cccc} 1&0&0&-1\\ 0&1&0&-1\\ 0&0&1&-1\end {array} \right)
  \end{equation*}
  then
  \begin{equation*}
    \L^T\cdot V=\left( \begin {array}{cccc} 1&-1&0&0\\ \noalign{\medskip}-1&1&0&0\\ \noalign{\medskip}-1&-1&2&0\\ \noalign{\medskip}-1&-1&0&2\\ \noalign{\medskip}2&0&-1&-1\\ \noalign{\medskip}0&2&-1&-1\\ \noalign{\medskip}0&0&1&-1\\ \noalign{\medskip}0&0&-1&1\end {array} \right)
  \end{equation*}
  so giving $\bb_1$ and $\bb_2$ as in the statement. Call $D_{\bb_i}$ the divisor of $\XX$ determined by $\bb_i$ and $\D_{\bb_i}$ the associated polytope, so getting
  \begin{equation*}
    \D_{\bb_1}=\conv\left( \begin {array}{ccc} 0&0&1\\ \noalign{\medskip}0&1&0\\ \noalign{\medskip}0&0&0\end {array} \right)\ ,\ \D_{\bb_2}=\conv\left( \begin {array}{ccc} -1&0&0\\ \noalign{\medskip}-1&0&0\\ \noalign{\medskip}-1&0&1\end {array} \right)
  \end{equation*}
  In particular, lattice points of the Newton polytope of $p_{i,\psi}$, that are exponents of momomials in $p_{i,\psi}$, are given by the columns of the matrix $M_i=\L^T\cdot\L_{\bb_i}+B_i$, where $\L_{\bb_i}$ is the matrix whose columns are given by lattice elements in $\D_{\bb_i}$, that is, vertices of $\D_{\bb_i}$, and
  $B_i=(\bb_i^T,\bb_i^T,\bb_i^T)$, so that
  \begin{equation*}
    M_1=\left( \begin {array}{ccc} 1&0&2\\ \noalign{\medskip}1&2&0\\ \noalign{\medskip}1&0&0\\ \noalign{\medskip}1&0&0\\ \noalign{\medskip}0&0&2\\ \noalign{\medskip}0&2&0\\ \noalign{\medskip}0&0&0\\ \noalign{\medskip}0&0&0\end {array} \right)\ ,\ M_2=\left( \begin {array}{ccc} 0&0&0\\ \noalign{\medskip}0&0&0\\ \noalign{\medskip}0&0&2\\ \noalign{\medskip}2&0&0\\ \noalign{\medskip}0&1&0\\ \noalign{\medskip}0&1&0\\ \noalign{\medskip}0&1&2\\ \noalign{\medskip}2&1&0\end {array} \right)
  \end{equation*}
  The proof that $Y^\vee_{BB}$ is an elliptic curve, that is a smooth complete intersection, is deferred 1  to Remark~\ref{rem:iso_quadric}.
\end{proof}

\subsubsection{The Libgober-Teitelbaum mirror} Consider the complete, $\Q$-factorial, toric variety $X=\P^3/\mu_4$, where $\mu_4$ is the group of 4-th roots of unity acting on $\P^3$ as follows
\begin{eqnarray*}
  &&\xymatrix{\ \mu_4\times \P^3\ \ar[rrrr]&&&&\ \P^4}\\
  &&\xymatrix{(\eta,[x_1:x_2:x_3:x_4])\ar@{|->}[r]&[x_1:\eta^2 x_2: \eta^3 x_3: \eta x_4]}
\end{eqnarray*}
A fan matrix of $X$ is given by
\begin{equation}\label{W}
  W=\left( \begin {array}{cccc} 2&0&-1&-1\\ \noalign{\medskip}0&2&-1&-1\\ \noalign{\medskip}-1&-1&2&0\end {array} \right)
\end{equation}
Then consider the partitioned framing $\cc=\cc_1+\cc_2$ with
\begin{equation*}
  \cc_1=(1,1,0,0)\ ,\ \cc_2=(0,0,1,1)
\end{equation*}
Notice that this gives a partition of the canonical divisor $K_X$ which is not a Cartier divisor. Then Batyrev-Borisov duality cannot be applied to $X$, as it is not Gorenstein. But $f$-duality can be applied to the partitioned ftv $(X,\cc=\cc_1+\cc_2)$.

\begin{proposition}\label{prop:LTdual}
  The $f$-dual partitioned ftv of $(X,\cc=\cc_1+\cc_2)$ is the nef partitioned ftv $(\P^3,\aa=\aa_1+\aa_2)$, with $\aa_1=(1,1,0,0)\,,\,\aa_2=(0,0,1,1)$. In particular, the family of elliptic curves given by complete intersections $\{Y\}_{Q_1,Q_2}$ in $\P^3$ is the $f$-mirror family of the family of elliptic curves  $Y^\vee_{LT}=\CV(q_{1,\psi},q_{2,\psi})$ with
  \begin{eqnarray*}
    q_{1,\psi}&=& {x_{{1}}}^{2}+{x_{{2}}}^{2}+\psi x_{{3}}x_{{4}}\in \Cox(X)\\
    q_{2,\psi}&=& \psi {x_{{1}}x_{{2}}+{x_{{3}}}^{2}+{x_{{4}}}^{2}}\in \Cox(X)
  \end{eqnarray*}
  being $\Cox(X)=\C[x_1,\ldots,x_4]_{\Cl(X)}$.
\end{proposition}

\begin{proof} What follows is sketched by the right part of Fig.~\ref{Fig1}. Recalling notation on polytopes resumed in \S~\ref{ssez:politopi}, polytopes associated with divisors $D_{\cc_1}, D_{\cc_2}$ of $X$ are given by
  \begin{equation*}
    \D_{\cc_1}=\left( \begin {array}{cccc} 1/2&-1/2&0&-1\\ \noalign{\medskip}-1/2&1/2&0&-1\\ \noalign{\medskip}0&0&1&-1\end {array} \right)\ ,\ \D_{\cc_2}= \left( \begin {array}{cccc} 1&0&1/2&-1/2\\ \noalign{\medskip}0&1&1/2&-1/2\\ \noalign{\medskip}0&0&1&-1\end {array} \right)
  \end{equation*}
  Then
  \begin{eqnarray*}
    \cv{\D}_\cc&=&\conv(\D_{\cc_1},\D_{\cc_2})=\left( \begin {array}{cccccccc} 1&0&1/2&-1/2&1/2&-1/2&0&-1\\ \noalign{\medskip}0&1&1/2&-1/2&-1/2&1/2&0&-1\\ \noalign{\medskip}0&0&1&-1&0&0&1&-1\end {array} \right)\\
    &\Longrightarrow& [\cv{\D}_\cc]=\D_V:=\conv\left( \begin {array}{cccc} 1&0&0&-1\\ 0&1&0&-1\\ 0&0&1&-1\end {array} \right)
  \end{eqnarray*}
  Since the toric variety spanned by the faces of $\D_V$ is clearly $\P^3$, the latter is the $f$-dual toric variety we are looking for. For the dual framing notice that
  \begin{equation*}
    V^T\cdot W= \left( \begin {array}{cccc} 2&0&-1&-1\\ \noalign{\medskip}0&2&-1&-1\\ \noalign{\medskip}-1&-1&2&0\\ \noalign{\medskip}-1&-1&0&2\end {array} \right)
  \end{equation*}
  so giving the nef partitioned framing $\aa=\aa_1+\aa_2$, with $\aa_i$ as given in the statement.
  In particular, lattice points of the Newton polytope of $q_{i,\psi}$, that are exponents of momomials in $q_{i,\psi}$, are given by the columns of the matrix $M^\vee_i=V^T\cdot\L_{\cc_i}+C_i$, where $\L_{\cc_i}$ is the matrix whose columns are given by lattice elements in $\D_{\cc_i}$ and
  $C_i=(\cc_i^T,\cc_i^T,\cc_i^T)$, so that, up to a column permutation,
\begin{equation*}
  M^\vee_1=\left(\begin {array}{ccc} 0&0&2\\ \noalign{\medskip}0&2&0\\ \noalign{\medskip}1&0&0\\ \noalign{\medskip}1&0&0\end {array}\right)\ ,\quad M^\vee_2=\left(\begin {array}{ccc} 0&1&0\\ \noalign{\medskip}0&1&0\\ \noalign{\medskip}0&0&2\\ \noalign{\medskip}2&0&0\end {array}\right)
\end{equation*}
  The proof that $Y^\vee_{LT}$ is an elliptic curve, that is a smooth complete intersection, is deferred to Remark~\ref{rem:iso_quadric}.
\end{proof}
\begin{figure}
\begin{center}
\includegraphics[width=12truecm]{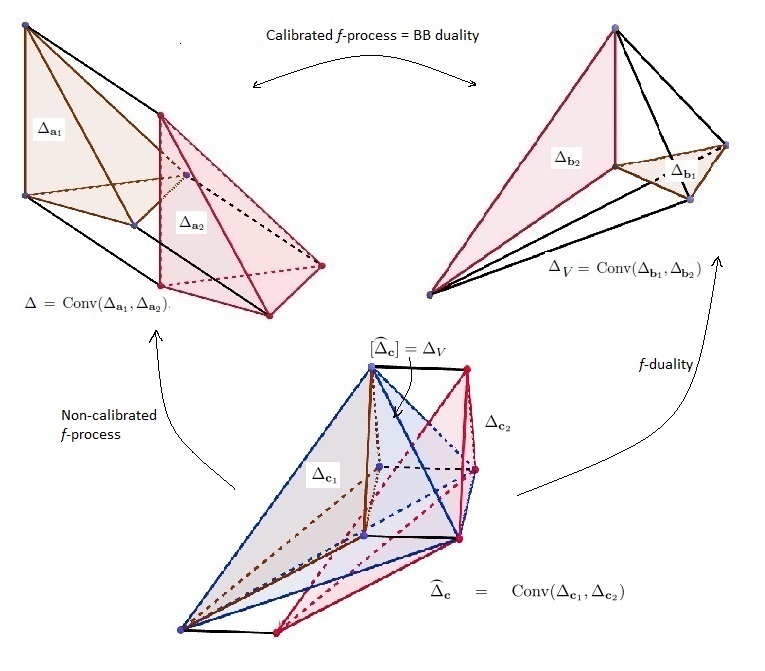}
\caption{\label{Fig1} Interplay of calibrated and non-calibrated framed processes connecting the LT-mirror and the BB-mirror.}
\end{center}
\end{figure}

\begin{remark}\label{rem:f-picture}
  Recalling Def.~1.7 in \cite{R-fpCI} of a calibrated partitioned $f$-process, the one given by
  \begin{equation*}
    (\P^3,\aa=\aa_1+\aa_2)\leftrightsquigarrow (\XX,\bb=\bb_1+\bb_2)
  \end{equation*}
  as described in Proposition~\ref{prop:BBdual} and sketched in the upper part of Fig.~\ref{Fig1}, is clearly calibrated because it is the Batyrev-Borisov duality. On the other hand the partitioned $f$-process
  \begin{equation*}
    (X,\cc=\cc_1+\cc_2)\rightsquigarrow (\P^3,\aa=\aa_1+\aa_2)\rightsquigarrow (\XX,\bb=\bb_1+\bb_2)
  \end{equation*}
  is a non-calibrated one, \emph{connecting the two multiple mirrors of the family $\{Y\}_{Q_1,Q_2}$ in $\P^3$}. In \cite[Thm.~4.8]{Malter}, Malter stated that these two mirrors are derived equivalent, in the sense that there exists an equivalence of triangulated categories
  $$\De ^b\left(Y_{LT}\right)\cong\De ^b\left(Y_{BB}\right)$$
  between their bounded derived categories of coherent sheaves.  These categories are expected to be equivalent to the Fukaya category of the complete intersection $Y_{2,2}\subset \P^3$, by the Homological Mirror Symmetry (HMS) conjecture. Putting all together we get an example in which \emph{framed mirror symmetry is consistent with the HMS conjecture}, that is, a proof of the mirror theorem~\ref{thm:mirror} when restricted to the two mirrors $Y^\vee_{BB}$ and $Y^\vee_{LT}$,
  although what will be observed in the next \S~\ref{sssez:birational_22} trivializes this example.
\end{remark}

\subsubsection{Birational link between BB-mirror and LT-mirror}\label{sssez:birational_22} Consider the two matrices $\L$ and $W$, as given by displays (\ref{Lambda}) and (\ref{W}), respectively, the former being obtained as the matrix whose columns are the vertices of the polytope $\D$. These are fan matrices of the ambient complete toric varieties $\XX$ and $X$, respectively. It is evident that the columns of $W$ are the four central ones of the matrix $\L$, so meaning that:
 \begin{enumerate}
   \item \emph{$\XX$ is the blow up in 4 distinct points of $X$, $\phi:\XX\to X$, the 4 points given by $[1:0:0:0],[0:1:0:0],[0:0:1:0],[0:0:0:1]$, in Cox's coordinates; exceptional divisors are the closure of torus orbits of  special points of rays generated the 4 remaining columns of $\L$.}
 \end{enumerate}
 The crucial fact is the following

 \begin{proposition}
   The BB-mirror $Y^\vee_{BB}$ is the strict transform of the LT-mirror $Y^\vee_{LT}$ under the blow up $\phi:\XX\to X$, that is, $$\phi^*Y^\vee_{LT}:=\phi^{-1}_*(Y^\vee_{LT})=\overline{\phi^{-1}(Y^\vee_{LT})}$$
 \end{proposition}

 \begin{proof} Recalling \S~\ref{ssez:Cox}, as Cox quotients one has
   \begin{equation*}
     \XX=(\C^8\setminus Z_\L)/[(\C^*)^5\times\mu_2]\ ,\ X=(\C^4\setminus Z_W)/(\C^*\times\mu_4)
   \end{equation*}
   where $Z_\L\subset\C^8$ and $Z_W=\{\0\}\subset\C^4$ are closed subsets determined by the fans of $\XX$ and $X$, respectively, and quotients are taken with respect to the following actions:
   \begin{equation*}
     \xymatrix{\a:[(\C^*)^5\times\mu_2]\times(\C^8\setminus Z_\L)\ar[r]&\C^8\setminus Z_\L}\quad\text{with}
   \end{equation*}
   \begin{equation*}
     \a(\ll,\ve,\x)=(\l_1x_1,\l_2x_2,\l_3x_3,\l_1\l_2\l_3\l_4^2x_4,\ve\l_2\l_3\l_4x_5,\l_1\l_3\l_4x_6,\ve\l_1\l_2\l_4^2\l_5x_7,\l_5x_8)
   \end{equation*}
   \begin{equation*}
     \xymatrix{\b:(\C^*\times\mu_4)\times(\C^4\setminus Z_W)\ar[r]&\C^4\setminus Z_W}\quad\text{with}
   \end{equation*}
   \begin{equation*}
     \b(\l,\eta,\y)=(\l y_1,\eta^2\l y_2,\eta^3\l y_3,\eta\l y_4)
   \end{equation*}
   In fact, actions $\a$ and $\b$ are determined by weight and torsion matrices
   \begin{eqnarray}\label{actions_22}
     \text{for action}\ \a  &:& Q_{\L}= \left( \begin {array}{cccccccc} 1&0&0&1&0&1&1&0\\
     \noalign{\medskip}0&1&0&1&1&0&1&0\\ \noalign{\medskip}0&0&1&1&1&1&0&0\\
     \noalign{\medskip}0&0&0&2&1&1&2&0\\ \noalign{\medskip}0&0&0&0&0&0&1&1\end {array} \right)\ ,\\ \nonumber
     &&T_\L=\left( \begin {array}{cccccccc} 0_2&0_2&0_2&0_2&1_2&0_2&1_2&0_2\end {array} \right) \\
     \nonumber
     \text{for action}\ \b &:& Q_W=\left( \begin {array}{cccc} 1&1&1&1\end {array} \right)\ ,\\
     \nonumber
     && T_W=\left( \begin {array}{cccc} 0_4&2_4&3_4&1_4\end{array} \right)
   \end{eqnarray}
   (columns of these matrices give exponents of quasi-tori actions $\a$ and $\b$).

   The birational morphism $\phi$ admits a global lifting $\widehat{\phi}$ induced by the projection on the central four coordinates, that is $\y=\widehat{\phi}(\x)=(x_3,x_4,x_5,x_6)$, such that the following diagram commutes
   \begin{equation*}
     \xymatrix{\C^8\setminus Z_\L\ar[r]^{\widehat{\phi}}\ar[d]_{/(\C^*)^3\times\mu_2}&\C^4\setminus Z_W\ar[d]^{/\C^*\times\mu_4}\\
                \XX\ar[r]^{\phi}&X}
   \end{equation*}
   Consider a generic fibre $Y^\vee_{BB}=\CV(p_{1,\psi},p_{2,\psi})$ in the BB-mirror family. Notice that one can assume that Cox coordinates of a point $p\in Y^\vee_{BB}$ necessarily have to satisfy the condition $x_1x_2x_7x_8\neq 0$: let us defer this check to the following Remark~\ref{rem:iso_quadric}, as in particular this means that $Y^\vee_{BB}$ does not meet any exceptional divisor of the blow up $\phi$. Then, one can set $x_1=x_2=x_7=x_8=1$ by choosing
   \begin{equation*}
     \l_1=x_1^{-1}\ ,\ \l_2=x_2^{-1}\ ,\  \l_5=x_8^{-1}\ ,\ \l_4^2=\pm x_1x_2x_8x_7^{-1}
   \end{equation*}
   Therefore $\phi(Y^\vee_{BB})=\CV(q_{1,\psi}(\y),q_{2,\psi}(\y))/(\C^*\times \mu_4)=Y^\vee_{LT}\subset X$\,.
 \end{proof}

 \begin{remark}\label{rem:iso_quadric}
   In particular, \emph{the birational morphism $\phi$ restricts to give an isomorphism over a generic fiber $Y^\vee_{BB}$ of the BB-family,} as $Y^\vee_{BB}$ does not meet any of the four exceptional divisors of the blow up $\phi$. In fact, fibers of both the BB and LT families are given by smooth complete intersections, hence elliptic curves by adjunction. In fact, critical points of
   \begin{equation*}
     p=(p_{1,\psi},p_{2,\psi}): \C^8\longrightarrow\C^2\quad,\quad q=(q_{1,\psi},q_{2,\psi}): \C^4\longrightarrow\C^2
   \end{equation*}
   belong to the unstable loci $Z_\L$ and $Z_W$, respectively. The latter is immediate as the only critical point of $q$ is $\0\in\C^4$ which is also the unique point of $Z_W$. For the former, notice that critical points of $p$ are given by
   \begin{eqnarray*}
     &\mathcal{C}=\{(0,0,x_{{3}},x_{{4}},x_{{5}},x_{{6}},x_{{7}},x_{{8}}) ,(0,x_{{2}},0,x_{{4}},x_{{5}},0,x_{{7}},x_{{8}}) , (0,x_{{2}},x_{{3}},0,x_{{5}},0,x_{{7}},x_{{8}}),&\\
     &     (x_{{1}},0,0,x_{{4}},0,x_{{6}},x_{{7}},x_{{8}}),( x_{{1}},0,x_{{3}},0,0,x_{{6}},x_{{7}},x_{{8}}),&\\
     &(x_{{1}},x_{{2}},0,0,0,0,x_{{7}},x_{{8}}), (x_{{1}},x_{{2}},x_{{3}},0,x_{{5}},0,0,x_{{8}})(x_{{1}},x_{{2}},0,x_{{4}},0,x_{{6}},x_{{7}},0), &\\ &(x_{{1}},x_{{2}},0,x_{{4}},x_{{5}},0,x_{{7}},0) , (x_{{1}},x_{{2}},x_{{3}},x_{{4}},x_{{5}},x_{{6}},0,0)\}&
     \end{eqnarray*}
     and the unstable locus $Z_\L$ is the closed subset determined by the irrelevant ideal
     \begin{eqnarray*}
       \mathfrak{I}= &(x_{{1}}x_{{5}}x_{{6}}x_{{7}}x_{{8}},x_{{2}}x_{{5}}x_{{6}}x_{{7}}x_{{8}},
       x_{{1}}x_{{3}}x_{{5}}x_{{7}},x_{{1}}x_{{4}}x_{{5}}x_{{8}},&\\
       &x_{{1}}x_{{2}}x_{{3}}x_{{4}}x_{{7}},x_{{2}}x_{{3}}x_{{6}}x_{{7}},
       x_{{1}}x_{{2}}x_{{3}}x_{{4}}x_{{8}},x_{{2}}x_{{4}}x_{{6}}x_{{8}})&
     \end{eqnarray*}
     One can easily check that $\mathcal{C}\subset\mathcal{V}(\mathfrak{I})=Z_\L$\,, so proving that $Y^\vee_{BB}$ and $Y^\vee_{LT}$ are quasi-smooth complete intersections. Moreover, there is the following inclusion of ideals
     \begin{equation*}
       \mathfrak{I}\subseteq (p_{1,\psi},p_{2,\psi},x_1x_2x_7x_8)
     \end{equation*}
     so proving that $Y^\vee_{BB}$ does not meet any of the four exceptional divisors of $\phi$.
     For smoothness, it is now enough checking that $Y^\vee_{LT}$ does not meet the ramification of the $\mu_4$-action, so giving the smoothness of both $Y^\vee_{LT}$ and $Y^\vee_{BB}=\phi^{-1}(Y^\vee_{LT})$. This is clear by recalling the torsion matrix $T_W$ in (\ref{actions_22}), so getting four points of ramification 4, given by the blown up points listed in \ref{sssez:birational_22}~(1), and the lines $y_1=y_2=0$ and $y_3=y_4=0$, composed by points of ramification 2. Notice that these two lines cannot meet $Y^\vee_{BB}$.

     In other terms, what is here observed is that
      \begin{itemize}
        \item \emph{the two mirrors models of Batyrev-Borisov and Libgober-Teitelbaum of the complete intersection $Y_{2,2}\subset\P^3$ are actually the same mirror.}
      \end{itemize}
      As a pleonastic conclusion, this fact obviously implies that $Y^\vee_{BB}$ and $Y^\vee_{LT}$ are $K$-equivalent, in the sense of Kawamata \cite{Kawamata}, and $D$-equivalent, as stated by Malter.
 \end{remark}

 \subsubsection{Intermediate mirrors of $Y_{2,2}\subset\P^3$}\label{sssez:mirrorintermedi} Taking into account what has been just observed in the previous Remark~\ref{rem:iso_quadric}, the following loses its meaning in the present case $d=2$. But it will be of considerable interest for $d\ge 3$. For this reason, I decided to firstly propose the following construction for $d=2$, where every check can be easily performed.

 Starting from the $BB$-mirror and considering the set of exceptional rays in the fan of $\XX$, with respect to the blow up $\phi:\XX\to X$, namely generated by columns $1,2,7,8$ in the fan matrix $\L$, one can produce a number of further mirror models of the projective complete intersection $Y_{2,2}$, in a sense \emph{intermediate} between the BB and the LT ones. More precisely, recalling notation introduced in \S~\ref{ssez:vettori}, the following result holds

 \begin{proposition}\label{prop:mirrorintermedi}
   For any subset $A\subset\{1,2,7,8\}$ the complete toric variety $\XX^A$, whose fan matrix is the submatrix $\L^A$ of $\L$, is the blow up, say $\phi^A:\XX^A\to X$, in $4-|A|$ points of $X$. Calling $\bb^A=\bb_1^A+\bb_2^A$ the partitioned framing of $\XX^A$ obtained by removing entries indexed by $A$ from the $BB$ partitioned framing $\bb=\bb_1+\bb_2$ of $\XX$, one obtains a partitioned ftv $(\XX^A,\bb^A=\bb_1^A+\bb_2^A)$ whose f-dual partitioned ftv is $(\P^3,\aa=\aa_1+\aa_2)$. In particular, the family of elliptic curves given by complete intersection $Y_{2,2}\subset\P^3$ is the f-mirror family of the family of elliptic curves $Y^\vee_A=\mathcal{V}(p_{1,\psi}^A,p_{2,\psi}^A)$ with $p_{i,\psi}^A\in\Cox(\XX^A)=\C[x_1,\ldots,x_{8-|A|}]_{\Cl(\XX^A)}$.
 \end{proposition}
The proof is analogous to the one proving Proposition~\ref{prop:LTdual}. Clearly
$$Y^\vee_\emptyset=Y^\vee_{BB}\quad\text{and}\quad Y^\vee_{\{1,2,7,8\}}=Y^\vee_{LT}$$
but the remaining $2^4-2=14$ cases give apparently distinct further mirror models of $Y_{2,2}$, all connected each other by means of non-calibrated $f$-processes. In particular, if $A$ is a proper subset of $\{1,2,7,8\}$ then $\XX^A$ is a non-Gorenstein $\Q$-Fano complete toric variety: therefore the Batyrev-Borisov duality does not apply to  $\XX^A$\,, so giving, for any proper $A$, a picture like that described in Remark~\ref{rem:f-picture} and in Fig.~\ref{Fig1}.

Actually, for $d=2$, $\phi^A$ restricted to $Y^\vee_A$ gives an isomorphism $Y^\vee_A\cong Y^\vee_{LT}$, hence, \emph{all these multiple mirrors are the same mirror model}: this fact can be proved exactly as for $A=\{1,2,7,8\}$ in Remark~\ref{rem:iso_quadric}, so trivializing the mirror theorem~\ref{thm:mirror}, when considered with respect to the two mirrors $Y^\vee_A$ and $Y^\vee_{LT}$, and Theorem~\ref{thm:metateorema} considered with respect to the list of mirrors given by $\cM=\{Y^\vee_A\,|\,A\subseteq\{1,2,7,8\}\}$\,.

Appendix \ref{app:A} is devoted to collect data characterizing the fourteen intermediate mirror models $(\XX^A,\aa^A=\aa_1^A+\aa_2^A)$, for $A\subset\{1,2,7,8\}$ proper subset, obtained by running suitable Maple routines partially written jointly with L.Terracini and sometimes involving routines from the package \texttt{Convex} by M.~Franz \cite{Convex}.

 \subsection{The complete intersection of two cubics in $\P^5$}\label{ssez:cubics}
Let us now consider the main example studied by Malter. Methods and strategy will be the same as for the quadric case, but data will be clearly more cumbersome and conclusions certainly more interesting.

Let $C_1,C_2$ be two generic cubic hypersurfaces in $\P^5$. Then we can assume $C_1\sim D_1+D_2+D_3$ and $C_2\sim D_4+D_5+D_6$. The complete intersection
\begin{equation*}
  Y=Y_{3,3}=C_1\cap C_2 \subset \P^5
\end{equation*}
is a smooth \cy threefold corresponding to the choice of the nef partitioned framing of $\P^5$
\begin{equation*}
  \aa=(1,1,1,1,1,1)=\aa_1+\aa_2\ ,\quad \text{with}\quad \aa_1=(1,1,1,0,0,0)\,,\,\aa_2=(0,0,0,1,1,1)
\end{equation*}

\subsubsection{The Batyrev-Borisov mirror}
Applying Batyrev-Borisov duality, that is framed duality to the partitioned framed toric variety (ftv) $(\P^5,\aa=\aa_1+\aa_2)$, we get

\begin{proposition}\label{prop:BBdualcubics}
  The $f$-dual partitioned ftv of $(\P^5,\aa=\aa_1+\aa_2)$ is the partitioned ftv $(\XX,\bb=\bb_1+\bb_2)$ where $\XX$ is the complete toric variety whose fan is spanned by the faces of the polytope
  \begin{equation}\label{Lambda_cubic}
    \D:=\conv\left(\begin {array}{cccccccccccc} 2&-1&-1&-1&-1&-1&3&0&0&0&0&0\\ \noalign{\medskip}-1&2&-1&-1&-1&-1&0&3&0&0&0&0\\ \noalign{\medskip}-1&-1&2&-1&-1&-1&0&0&3&0&0&0\\ \noalign{\medskip}0&0&0&3&0&0&-1&-1&-1&2&-1&-1\\ \noalign{\medskip}0&0&0&0&3&0&-1&-1&-1&-1&2&-1\end {array} \right)
  \end{equation}
  and $\bb_1=(\1_6,\0_6)\,,\,\bb_2=(\0_6,\1_6)$. In particular, the $f$-mirror family of the family $\{Y\}_{C_1,C_2}$ is the Batyrev-Borisov mirror family of \cy threefolds $\widehat{Y}^\vee_{BB}$ obtained as suitable resolution of singularities of the quasi-smooth complete intersection $Y^\vee_{BB}=\CV(p_{1,\psi},p_{2,\psi})$ with
  \begin{eqnarray*}
    p_{1,\psi}&=& {x_{{1}}}^{3}{x_{{7}}}^{3}+{x_{{2}}}^{3}{x_{{8}}}^{3}+
    {x_{{3}}}^{3}{x_{{9}}}^{3}+\psi x_{{1}}x_{{2}}x_{{3}}x_{{4}}x_{{5}}x_{{6}}\in \Cox(\XX)\\
    p_{2,\psi}&=& {x_{{4}}}^{3}{x_{{10}}}^{3}+{x_{{5}}}^{3}{x_{{11}}}^{3}+{x_{{6}}}^{3}{x_{{12}}}^{3}+
    \psi x_{{7}}x_{{8}}x_{{9}}x_{{10}}x_{{11}}x_{{12}}\in \Cox(\XX)
  \end{eqnarray*}
  being $\Cox(\XX)=\C[x_1,\ldots,x_{12}]_{\Cl(\XX)}$.
\end{proposition}

\begin{proof}
 Let $\D_{\aa_i}$ be the polytope associated with the divisor $D_{\aa_i}$, then
  \begin{eqnarray*}
    \D_{\aa_1}&=&\conv\left( \begin {array}{cccccc} 2&-1&-1&-1&-1&-1\\ \noalign{\medskip}-1&2&-1&-1&-1&-1\\ \noalign{\medskip}-1&-1&2&-1&-1&-1\\ \noalign{\medskip}0&0&0&3&0&0\\ \noalign{\medskip}0&0&0&0&3&0\end {array}  \right)\\
    \D_{\aa_2}&=&\conv\left( \begin {array}{cccccc} 3&0&0&0&0&0\\ \noalign{\medskip}0&3&0&0&0&0\\ \noalign{\medskip}0&0&3&0&0&0\\ \noalign{\medskip}-1&-1&-1&2&-1&-1\\ \noalign{\medskip}-1&-1&-1&-1&2&-1\end {array} \right)
  \end{eqnarray*}
  giving $\D=\conv(\D_{\aa_1},\D_{\aa_2})$. Calling $\L$ the fan matrix of $\XX$, that is the matrix whose columns are given by vertices of $\D$, and $V$ the fan matrix of $\P^5$, that is
  \begin{equation}\label{V5}
    V= \left( \begin {array}{cccccc} 1&0&0&0&0&-1\\ \noalign{\medskip}0&1&0&0&0&-1\\ \noalign{\medskip}0&0&1&0&0&-1\\ \noalign{\medskip}0&0&0&1&0&-1\\ \noalign{\medskip}0&0&0&0&1&-1\end {array}  \right)
  \end{equation}
  then
  \begin{equation*}
    \L^T\cdot V=\left( \begin {array}{cccccc} 2&-1&-1&0&0&0\\ \noalign{\medskip}-1&2&-1&0&0&0\\ \noalign{\medskip}-1&-1&2&0&0&0\\ \noalign{\medskip}-1&-1&-1&3&0&0\\ \noalign{\medskip}-1&-1&-1&0&3&0\\ \noalign{\medskip}-1&-1&-1&0&0&3\\ \noalign{\medskip}3&0&0&-1&-1&-1\\ \noalign{\medskip}0&3&0&-1&-1&-1\\ \noalign{\medskip}0&0&3&-1&-1&-1\\ \noalign{\medskip}0&0&0&2&-1&-1\\ \noalign{\medskip}0&0&0&-1&2&-1\\ \noalign{\medskip}0&0&0&-1&-1&2\end {array} \right)
  \end{equation*}
  so giving $\bb_1$ and $\bb_2$ as in the statement. Call $D_{\bb_i}$ the divisor of $\XX$ determined by $\bb_i$ and $\D_{\bb_i}$ the associated polytope, so getting
  \begin{equation*}
    \D_{\bb_1}=\conv\left( \begin {array}{cccc} 1&0&0&0\\ \noalign{\medskip}0&1&0&0\\ \noalign{\medskip}0&0&1&0\\ \noalign{\medskip}0&0&0&0\\ \noalign{\medskip}0&0&0&0\end {array} \right)\ ,\ \D_{\bb_2}=\conv\left( \begin {array}{cccc} 0&0&-1&0\\ \noalign{\medskip}0&0&-1&0\\ \noalign{\medskip}0&0&-1&0\\ \noalign{\medskip}1&0&-1&0\\ \noalign{\medskip}0&1&-1&0\end {array} \right)
  \end{equation*}
  In particular, lattice points of the Newton polytope of $p_{i,\psi}$, that are exponents of momomials in $p_{i,\psi}$, are given by the columns of the matrix $M_i=\L^T\cdot\L_{\bb_i}+B_i$, where $\L_{\bb_i}$ is the matrix whose columns are given by vertices of $\D_{\bb_i}$ and
  $B_i=(\bb_i^T,\cdots,\bb_i^T)$, so that
  \begin{equation*}
    M_1=\left( \begin {array}{cccc} 3&0&0&1\\ \noalign{\medskip}0&3&0&1\\ \noalign{\medskip}0&0&3&1\\ \noalign{\medskip}0&0&0&1\\ \noalign{\medskip}0&0&0&1\\ \noalign{\medskip}0&0&0&1\\ \noalign{\medskip}3&0&0&0\\ \noalign{\medskip}0&3&0&0\\ \noalign{\medskip}0&0&3&0\\ \noalign{\medskip}0&0&0&0\\ \noalign{\medskip}0&0&0&0\\ \noalign{\medskip}0&0&0&0\end {array} \right)\ ,\ M_2=\left( \begin {array}{cccc} 0&0&0&0\\ \noalign{\medskip}0&0&0&0\\ \noalign{\medskip}0&0&0&0\\ \noalign{\medskip}3&0&0&0\\ \noalign{\medskip}0&3&0&0\\ \noalign{\medskip}0&0&3&0\\ \noalign{\medskip}0&0&0&1\\ \noalign{\medskip}0&0&0&1\\ \noalign{\medskip}0&0&0&1\\ \noalign{\medskip}3&0&0&1\\ \noalign{\medskip}0&3&0&1\\ \noalign{\medskip}0&0&3&1\end {array} \right)
  \end{equation*}
  The check of quasi-smoothness of $Y^\vee_{BB}$ is deferred to \S~\ref{ssez:desing}.
\end{proof}

\subsubsection{The Libgober-Teitelbaum mirror} Consider the complete, $\Q$-factorial, toric variety $X=\P^5/(\mu_3^2\times\mu_9)$, where $\mu_k$ is the group of k-th roots of unity and $G_{81}= \mu_3^2\times\mu_9$ acts on $\P^5$ as follows
\begin{eqnarray}\label{Cox_LTcubic}
  &\xymatrix{G_{81}\times \P^5\ \ar[rrrrrrr]&&&&&&&\ \P^5}&\\
  \nonumber
  &\xymatrix{((\ve_1,\ve_2,\eta),[x_1:\cdots:x_6])\ar@{|->}[r]&[x_1:\ve_1\ve_2^2\eta^3 x_2: \ve_1^2\ve_2\eta^3 x_3: \ve_1^2\ve_2\eta^5 x_4:\ve_1\ve_2\eta^2 x_5: \ve_2\eta^2 x_6]}&
\end{eqnarray}
A fan matrix of $X$ is given by
\begin{equation}\label{W_cubic}
  W=\left( \begin {array}{cccccc} 3&0&0&-1&-1&-1\\ \noalign{\medskip}0&3&0&-1&-1&-1\\ \noalign{\medskip}0&0&3&-1&-1&-1\\ \noalign{\medskip}-1&-1&-1&3&0&0\\ \noalign{\medskip}-1&-1&-1&0&3&0\end {array} \right)
\end{equation}
Then consider the partitioned framing $\cc=\cc_1+\cc_2$ with
\begin{equation*}
  \cc_1=(1,1,1,0,0,0)\ ,\ \cc_2=(0,0,0,1,1,1)
\end{equation*}
Notice that this gives a partition of the canonical divisor $K_X$ which is not a Cartier divisor. Then Batyrev-Borisov duality cannot be applied to $X$, as it is not Gorenstein. But $f$-duality can be applied to the partitioned ftv $(X,\cc=\cc_1+\cc_2)$.

\begin{proposition}\label{prop:LTdual_cubics}
  The $f$-dual partitioned ftv of $(X,\cc=\cc_1+\cc_2)$ is the nef partitioned ftv $(\P^5,\aa=\aa_1+\aa_2)$, with $\aa_1=(1,1,1,0,0,0)\,,\,\aa_2=(0,0,0,1,1,1)$. In particular, the family of \cy threefolds given by complete intersections $\{Y\}_{C_1,C_2}$ in $\P^5$ is the $f$-mirror family of the family of \cy threefolds $\widehat{Y}^\vee_{LT}$ obtained as a suitable resolution of singularities of the quassi-smooth complete intersection $Y^\vee_{LT}=\CV(q_{1,\psi},q_{2,\psi})$ with
  \begin{eqnarray*}
    q_{1,\psi}&=& {x_{{1}}}^{3}+{x_{{2}}}^{3}+{x_{{3}}}^{3}+\psi x_{{4}}x_{{5}}x_{{6}}\in \Cox(X)\\
    q_{2,\psi}&=& \psi x_{{1}}x_{{2}}x_{{3}}+{x_{{4}}}^{3}+{x_{{5}}}^{3}+{x_{{6}}}^{3}\in \Cox(X)
  \end{eqnarray*}
  being $\Cox(X)=\C[x_1,\ldots,x_4]_{\Cl(X)}$.
\end{proposition}

\begin{proof}
  The polytopes associated with divisors $D_{\cc_1}, D_{\cc_2}$ of $X$ are given by
  \begin{eqnarray*}
    \D_{\cc_1}&=&\left( \begin {array}{cccccc} 2/3&-1/3&-1/3&0&0&-1\\ \noalign{\medskip}-1/3&2/3&-1/3&0&0&-1\\ \noalign{\medskip}-1/3&-1/3&2/3&0&0&-1\\ \noalign{\medskip}0&0&0&1&0&-1\\ \noalign{\medskip}0&0&0&0&1&-1\end {array} \right)\\
    \D_{\cc_2}&=& \left( \begin {array}{cccccc} 1&0&0&1/3&1/3&-2/3\\ \noalign{\medskip}0&1&0&1/3&1/3&-2/3\\ \noalign{\medskip}0&0&1&1/3&1/3&-2/3\\ \noalign{\medskip}0&0&0&1&0&-1\\ \noalign{\medskip}0&0&0&0&1&-1\end {array} \right)
  \end{eqnarray*}
  Then $\cv{\D}_\cc=\conv(\D_{\cc_1},\D_{\cc_2})$ gives
  \begin{equation*}
    [\cv{\D}_\cc]=\D_V:=\conv\left( \begin {array}{cccccc} 1&0&0&0&0&-1\\ \noalign{\medskip}0&1&0&0&0&-1\\ \noalign{\medskip}0&0&1&0&0&-1\\ \noalign{\medskip}0&0&0&1&0&-1\\ \noalign{\medskip}0&0&0&0&1&-1\end {array} \right)
  \end{equation*}
  and $\P^5$ turns out to be the $f$-dual toric variety of $X$.
   For the dual framing notice that
  \begin{equation*}
    V^T\cdot W= \left( \begin {array}{cccccc} 3&0&0&-1&-1&-1\\ \noalign{\medskip}0&3&0&-1&-1&-1\\ \noalign{\medskip}0&0&3&-1&-1&-1\\ \noalign{\medskip}-1&-1&-1&3&0&0\\ \noalign{\medskip}-1&-1&-1&0&3&0\\ \noalign{\medskip}-1&-1&-1&0&0&3\end {array} \right)
  \end{equation*}
  so giving the nef partitioned framing $\aa=\aa_1+\aa_2$, with $\aa_i$ as given in the statement.
  Finally, one has
  \begin{equation*}
    M^\vee_1=\left(\begin {array}{cccc} 3&0&0&0\\ \noalign{\medskip}0&3&0&0\\ \noalign{\medskip}0&0&3&0\\ \noalign{\medskip}0&0&0&1\\ \noalign{\medskip}0&0&0&1\\ \noalign{\medskip}0&0&0&1\end {array}\right)\ ,\quad M^\vee_2=\left(\begin {array}{cccc} 0&0&0&1\\ \noalign{\medskip}0&0&0&1\\ \noalign{\medskip}0&0&0&1\\ \noalign{\medskip}3&0&0&0\\ \noalign{\medskip}0&3&0&0\\ \noalign{\medskip}0&0&3&0\end {array}\right)
  \end{equation*}
  whose columns give monomial exponents of $q_{1,\psi},q_{2,\psi}$, respectively, as given in the statement.

    The check that $Y^\vee_{LT}$ is quasi-smooth is deferred to \S~\ref{ssez:desing}.
\end{proof}

\begin{remark}\label{rem:f-picture3}
  The partitioned $f$-process
  \begin{equation*}
    (\P^5,\aa=\aa_1+\aa_2)\leftrightsquigarrow (\XX,\bb=\bb_1+\bb_2)
  \end{equation*}
  as described in Proposition~\ref{prop:BBdualcubics}, is clearly calibrated because it is the Batyrev-Borisov duality. On the other hand the partitioned $f$-process
  \begin{equation*}
    (X,\cc=\cc_1+\cc_2)\rightsquigarrow (\P^5,\aa=\aa_1+\aa_2)\rightsquigarrow (\XX,\bb=\bb_1+\bb_2)
  \end{equation*}
  is a non-calibrated one, \emph{connecting the two multiple mirrors of the family $\{Y\}_{C_1,C_2}$ in $\P^5$}. Malter proved  these two mirrors to be derived equivalent, consistently with the prediction of HMS conjecture, as their equivalent bounded derived category of coherent sheaves is expected to be equivalent to the Fukaya category of the complete intersection $Y_{3,3}\subset \P^5$. Putting all together we get a further example in which \emph{framed mirror symmetry is consistent with the HMS conjecture}, that is, a proof of the mirror theorem~\ref{thm:mirror}, when considered with respect to the two mirrors $Y^\vee_{BB}$ and $Y^\vee_{LT}$\,.
\end{remark}

\subsubsection{Birational link between BB-mirror and LT-mirror}\label{sssez:birational_33} Consider the two matrices $\L$ and $W$, as given by displays (\ref{Lambda_cubic}) and (\ref{W_cubic}), respectively, the former being obtained as the matrix whose columns are the vertices of the polytope $\D$. These are fan matrices of the ambient complete toric varieties $\XX$ and $X$, respectively. It is evident that the columns of $W$ are the six central ones of the matrix $\L$, so meaning that:
 \begin{enumerate}
   \item \emph{$\XX$ is the blow up in 6 distinct points of $X$, $\phi:\XX\to X$, the 6 points given by
   \begin{eqnarray*}
     &[1:0:0:0:0:0],[0:1:0:0:0:0],[0:0:1:0:0:0],&  \\
      &[0:0:0:1:0:0],[0:0:0:0:1:0],[0:0:0:0:0:1]&
   \end{eqnarray*}
       in Cox's coordinates; the exceptional divisors are the closure of torus orbits of  special points of rays generated the 6 remaining columns of $\L$.}
 \end{enumerate}
 The crucial fact is the following

 \begin{proposition}\label{prop:birat3,3}
   The BB-mirror $Y^\vee_{BB}$ is the strict transform of the LT-mirror $Y^\vee_{LT}$ under the blow up $\phi:\XX\to X$, that is, $Y^\vee_{BB}=\phi^{-1}_*(Y^\vee_{LT})$.
 \end{proposition}

 \begin{proof}
   As Cox quotients one has
   \begin{equation*}
     \XX=(\C^{12}\setminus Z_\L)/[(\C^*)^7\times\mu_3\times\mu_9]\ ,\ X=(\C^6\setminus Z_W)/(\C^*\times\mu_3^2\times\mu_9)
   \end{equation*}
   where $Z_\L\subset\C^{12}$ and $Z_W=\{\0\}\subset\C^6$ are closed subsets determined by the fans of $\XX$ and $X$, respectively, and quotient are taken with respect to the following actions:
   \begin{eqnarray*}
     &\xymatrix{\a:[(\C^*)^7\times\mu_3\times \mu_9]\times(\C^{12}\setminus Z_\L)\ar[r]&\C^{12}\setminus Z_\L}&\\
     &\xymatrix{\b:(\C^*\times\mu_3^2\times\mu_9)\times(\C^6\setminus Z_W)\ar[r]&\C^6\setminus Z_W}&
   \end{eqnarray*}
   with
   \begin{eqnarray*}
     &\a_{|(\C^*)^7}(\ll,\x)=(\l_1x_1,\l_2x_2,\l_3x_3,\l_4x_4,\l_4\l_5x_5,\l_1^2\l_2^2\l_3^2\l_4\l_5^2\l_6^3x_6,
     \l_2\l_3\l_4\l_5\l_6x_7,&\\
     &\l_1\l_3\l_4\l_5\l_6x_8,\l_1\l_2\l_4\l_5\l_6x_9,\l_1^2\l_2^2\l_3^2\l_5^2\l_6^3\l_7x_{10},
     \l_1^2\l_2^2\l_3^2\l_5\l_6^3\l_7x_{11},\l_7x_{12})&
   \end{eqnarray*}
   \begin{equation*}
     \b_{|\C^*}(\l,\y)=(\l y_1,\l y_2,\l y_3,\l y_4,\l y_5,\l y_6)
   \end{equation*}
   as one can check by noticing that Gale dual weight matrices of $\L$ and $W$, respectively, are given by
   \begin{eqnarray*}
     \text{for action}\ \a  &:& Q_{\L}= \left( \begin {array}{cccccccccccc} 1&0&0&0&0&2&0&1&1&2&2&0\\ \noalign{\medskip}0&1&0&0&0&2&1&0&1&2&2&0\\ \noalign{\medskip}0&0&1&0&0&2&1&1&0&2&2&0\\ \noalign{\medskip}0&0&0&1&1&1&1&1&1&0&0&0\\ \noalign{\medskip}0&0&0&0&1&2&1&1&1&2&1&0\\ \noalign{\medskip}0&0&0&0&0&3&1&1&1&3&3&0\\ \noalign{\medskip}0&0&0&0&0&0&0&0&0&1&1&1\end {array} \right)\ ,\\
     \text{for action}\ \b &:& Q_W=\left( \begin {array}{cccccc} 1&1&1&1&1&1\end {array} \right)
   \end{eqnarray*}
   The birational morphism $\phi$ admits a global lifting $\widehat{\phi}$ induced by the projection on the central six coordinates, that is $\y=\widehat{\phi}(\x)=(x_4,\ldots,x_9)$, such that the following diagram commutes
   \begin{equation*}
     \xymatrix{\C^{12}\setminus Z_\L\ar[r]^{\widehat{\phi}}\ar[d]_{/(\C^*)^7\times\mu_3\times\mu_9}&\C^6\setminus Z_W\ar[d]^{/\C^*\times\mu_3^2\times\mu_9}\\
                \XX\ar[r]^{\phi}&X}
   \end{equation*}
   Consider a generic fibre $Y^\vee_{BB}=\CV(p_{1,\psi},p_{2,\psi})$ in the BB-mirror family.  Points in $Y^\vee_{BB}$ whose Cox coordinates satisfy the condition $x_1x_2x_3x_{10}x_{11}x_{12}\neq 0$ define a Zariski open subset $U\subset \C^{12}\setminus Z_\L$. This is the open subset of points in $Y^\vee_{BB}$ not belonging to any exceptional divisor of the blow up $\phi$. By the torus part of action $\a$,  for every point in $U$, one can set $x_1=x_2=x_3=x_{10}=x_{11}=x_{12}=1$ by choosing
   \begin{equation*}
     \l_1=x_1^{-1}\ ,\ \l_2=x_2^{-1}\ ,\ \l_3=x_3^{-1}\ ,\  \l_7=x_{12}^{-1}\ ,\ \l_5=x_{10}^{-1}x_{11}\ ,\ \l_6^3= x_1^2x_2^2x_3^2x_{10}x_{11}^{-2}x_{12}
   \end{equation*}
   Therefore $\phi(U)=\CV(q_{1,\psi}(\y),q_{2,\psi}(\y))/(\C^*\times \mu_3^2\times \mu_9)=Y^\vee_{LT}\subset X$ and $Y^\vee_{BB}$ is the Zariski closure of $U$ inside $\phi^{-1}(Y^\vee_{LT})$.
 \end{proof}

 \subsubsection{The desingularization problem}\label{ssez:desing}
  Fibers of both the BB and LT families are given by quasi-smooth complete intersections. In fact, critical points of
   \begin{equation*}
     p=(p_{1,\psi},p_{2,\psi}): \C^8\longrightarrow\C^2\quad,\quad q=(q_{1,\psi},q_{2,\psi}): \C^4\longrightarrow\C^2
   \end{equation*}
   belong to the unstable loci $Z_\L$ and $Z_W$, respectively: the check is completely analogous to that given in Remark~\ref{rem:iso_quadric}. Unfortunately, $Y^\vee_{LT}$ meets some ramification of the $(\mu_3^2\times\mu_9)$-action in $\b$, represented by the torsion matrix
   \begin{equation*}
     T_W= \left( \begin {array}{cccccc} 0_3&1_3&2_3&2_3&1_3&0_3\\ \noalign{\medskip}0_3&2_3&1_3&1_3&1_3&1_3\\ \noalign{\medskip}0_9&3_9&3_9&5_9&2_9&2_9\end {array} \right)
   \end{equation*}
   For instance, points in the codimension 2 linear subvariety $x_1=x_2=0$ have ramification of order at least 3, so cutting 9 ramification points of order 3 on the generic $Y^\vee_{LT}$. Same considerations hold for the codimension 3 linear subvariety $x_4=x_5=x_6=0$, giving rise to further 3 ramification points of order 3 on   the generic $Y^\vee_{LT}$. Anyway, this means that a possibly partial resolution of $X$ may induce a desingularization of $Y^\vee_{LT}$. Following \cite[Def.~4.1]{Kawamata} and \cite[Def.~1.3.3]{Lazarsfeld} let us define a crepant resolution of a $\Q$-Gorenstein variety as follows
    \begin{definition}[Crepant birational morphism and maps]\label{def:crepante}
      Let $V$ and $W$ be normal varieties whose canonical divisors are $\Q$-Cartier divisors. A birational morphism $f:V\to W$ is said to be \emph{crepant} if there exists $r\in\N$ such that $rf^*K_W$ and $rK_V$ are linearly equivalent Cartier divisors (denoted by $f^*K_W\sim_\Q K_V$). Moreover, a birational map $h:V\dashrightarrow W$ is called \emph{crepant} if there exists a smooth quasi-projective variety $Z$ with birational projective morphisms to $V$ and $W$, making commutative the following diagram
      \begin{equation*}
        \xymatrix{&Z\ar[ld]_-f\ar[dr]^-g&\\
                    V\ar@{-->}[rr]^-h&&W}
      \end{equation*}
      and such that $f^*K_V\sim_\Q g^*K_W$\,.
    \end{definition}
    Therefore, by recalling \cite[Prop.~2.2.12, Thm.~2.2.24]{Batyrev94}, one has the following

   \begin{proposition}\label{prop:resol_W}
     Since $X$ is a $\Q$-Fano toric variety, it admits \emph{maximally projective crepant partial} (MPCP) resolutions. These birational morphisms are defined by maximal triangularizations (in the sense of \cite[Def.~2.2.16]{Batyrev94}) of the polytope $\D_W=\conv(W)$ and are parameterized by choices of full dimensional chambers of the secondary fan of a complete toric variety admitting fan matrix given by
     \begin{equation*}
     \resizebox{1\hsize}{!}{$
       \widehat{W}= \left( \begin {array}{cccccccccccccccccccc} -1&-1&-1&-1&-1&-1&-1&-1&-1&-1&0&0&0&0&1&1&1&2&2&3\\ \noalign{\medskip}-1&-1&-1&-1&-1&-1&-1&-1&-1&-1&0&1&2&3&0&1&2&0&1&0\\ \noalign{\medskip}-1&-1&-1&-1&-1&-1&-1&-1&-1&-1&3&2&1&0&2&1&0&1&0&0\\ \noalign{\medskip}0&0&0&0&1&1&1&2&2&3&-1&-1&-1&-1&-1&-1&-1&-1&-1&-1\\ \noalign{\medskip}0&1&2&3&0&1&2&0&1&0&-1&-1&-1&-1&-1&-1&-1&-1&-1&-1\end {array} \right)$}
     \end{equation*}
     whose columns are given by all primitive lattice points contained in $\D_W\setminus\{\0\}$, being $\D_W=\conv(W)$\,.
      In other words, recalling notation introduced in \S~\ref{ssez:SF}, every MPCP-resolution is obtained by a simplicial fan $\Si\in\PSF(\widehat{W})$.
   \end{proposition}
   Unfortunately, since $X$ is not Gorenstein, calling $\phi_\Si:\widehat{X}(\Si)\to X$ an associated MPCP-resolution, we cannot guarantee that $\phi_{\Si}|_{\widehat{Y}^\vee_{LT}}:\widehat{Y}^\vee_{LT}\to Y^\vee_{LT}$ is a desingularization of $Y^\vee_{LT}$, being $\widehat{Y}^\vee_{LT}$ the strict transform of $Y^\vee_{LT}$ under $\phi$.

    On the other hand, $\XX$ is a Fano toric variety and \cite[Thm.~2.2.24]{Batyrev94} applies to give the following

    \begin{proposition}\label{prop:resol_L}
     Since $\XX$ is a Fano toric variety, it admits MPCP-resolutions. These birational morphisms are defined by maximal triangularizations  of the polytope $\D=\conv(\D_{\aa_1},\D_{\aa_2})$ and are parameterized by choices of full dimensional chambers of the secondary fan of a complete toric variety admitting fan matrix $\widehat{\L}$, that is, the $5\times 110$ matrix whose columns are given by all primitive lattice points contained in $\D\setminus\{\0\}$, and displayed in Appendix~\ref{app:B}.
      In other words, every MPCP-resolution is obtained by a simplicial fan $\Si\in\PSF(\widehat{\L})$.
   \end{proposition}

   Therefore, by \cite[Cor.~3.1.7]{Batyrev94}, calling $\psi_\Si:\widehat{\XX}(\Si)\to\XX$ an associated MPCP-resolution, the restriction $\psi_{\Si}|_{\widehat{Y}^\vee_{BB}}:\widehat{Y}^\vee_{BB}\to Y^\vee_{BB}$ is a desingularization of a generic $Y^\vee_{BB}$, being $\widehat{Y}^\vee_{BB}$ the strict transform of $Y^\vee_{BB}$ under $\psi$, as a generic complete intersection of hypersurfaces admitting at most isolated singularities.

   \begin{remark}
     The blow up $\phi:\XX\to X$ gives a minimal Gorenstein resolution of $X$, in the sense that any further partial resolution $Z\to X$ such that $Z$ admits at worst Gorenstein singularities necessarily factorizes through $\phi$.

     Moreover, for every refinement $\Si\in\PSF(\widehat{\L})$ of the fan of $\XX$, the composition $\phi\circ\psi_\Si:\widehat{\XX}(\Si)\to X$ gives rise to a non-crepant resolution of the $\Q$-Fano toric variety $X$, restricting to giving a crepant desingularization $(\phi\circ\psi_\Si)|_{\widehat{Y}^\vee_{BB}}:{\widehat{Y}^\vee_{BB}}\to Y^\vee_{LT}$ of the generic $Y^\vee_{LT}\subset X$, as ${\widehat{Y}^\vee_{BB}}$ is a smooth \cy threefold. In particular, the toric resolution of $X$ described by Malter in \cite[\S~4.1]{Malter} is of this kind, as follows from vertices $P_0,\ldots, P_{11}$ there listed. Therefore, given two generic desingularizations $\widehat{Y}^\vee_{BB}$ and $\widehat{Y}^\vee_{LT}$ of this kind, they turns out to be at worst \emph{isomorphic in codimension 1}. Then, \cite[Thm.~1.1]{Malter} shows that these two smooth \cy threefolds turn out to be derived equivalent.
   \end{remark}

Actually Malter's result is a deeper one:

\begin{theorem}[Thm. 2.23 in \cite{Malter}]\label{thm:Malter 2.23} Consider two simplicial fans $\Theta_1,\Theta_2\in \PSF(\L)$ giving rise to two small, $\Q$-factorial partial resolutions $\psi_i:\XX_i\to\XX$ of $\XX$, for $i=1,2$. Let $Y^\vee_1=(\psi_1)^{-1}_*(Y^\vee_{BB})$ and $Y^\vee_2=(\phi\circ\psi_2)^{-1}_*(Y^\vee_{LT})$ be the induced strict transforms of generic $Y^\vee_{BB}\subset\XX$ and $Y^\vee_{LT}\subset X$, respectively. Then the associated \emph{categories of singularities} $\cD_{sg}(Y^\vee_1)$ and $\cD_{sg}(Y^\vee_2)$ turn out to be equivalent.
\end{theorem}

Recall that the category of singularites $\cD_{sg}(Y^\vee_i)$ is, by definition, the Verdier quotient of $\cD^b(Y^\vee_i)$ by the full subcategory $\text{Perf}(Y^\vee_i)$ of perfect objects. Then, an argument by Favero and Kelly \cite[Prop.~4.7]{Favero-Kelly} gives the above stated derived equivalence of $\widehat{Y}^\vee_{BB}$ and $\widehat{Y}^\vee_{LT}$, as a consequence of Theorem~\ref{thm:Malter 2.23}, without explicitly passing through effective desingularizations.

\begin{remark}\label{rem:noMalter}
  The comparison of Propositions \ref{prop:resol_W} and \ref{prop:resol_L}  is naturally reflected in a question to which Malter's results fail to give an answer.
  \begin{itemize}
    \item[(i)] Consider fans $\Theta_1\in\PSF(\L)$ and $\Theta_2\in\PSF(W)$ giving rise to small, $\Q$-factorial, projective resolutions
        $$\psi_1:\XX_1\to\XX\quad\text{and}\quad \phi_2:X_2\to X$$
        of $\XX$ and $X$, respectively. Let $Y^\vee_1=(\psi_1)^{-1}_*(Y^\vee_{BB})$ and $Y_2=(\phi_2)^{-1}_*(Y^\vee_{LT})$ be the associated strict transforms of generic $Y^\vee_{BB}\subset\XX$ and $Y^\vee_{LT}\subset X$, respectively.\\
        Is there an equivalence of their categories of singularities $$\cD_{sg}(Y_1^\vee)\stackrel{?}{\cong}\cD_{sg}(Y^\vee_2)$$
  \end{itemize}
  By \cite{Favero-Kelly}, this question can be reformulated in milder and more esplicite terms as follows:
  \begin{itemize}
    \item[(ii)] assume there exists a refinement $\Si\in\PSF(\widehat{W})$ of the fan of $X$, such that the associated toric resolution $\phi_\Si:\widehat{X}(\Si)\to X$ induces a desingularization $$\widehat{Y}^\vee_{LT}=(\phi_\Si)^{-1}_*(Y^\vee_{LT})\to Y^\vee_{LT}$$
        and consider the desingularization
        $$\widehat{Y}^\vee_{BB}=(\psi_{\Si'})^{-1}_*(Y^\vee_{BB})\to Y^\vee_{BB}$$
        induced by the choice of a refinement $\Si'\in\PSF(\widehat{\L})$ of the fan of $\XX$, with  associated toric resolution $\psi_{\Si'}:\widehat{\XX}(\Si')\to \XX$.

        \noindent Are $\widehat{Y}^\vee_{BB}$ and $\widehat{Y}^\vee_{LT}$ derived equivalent?
  \end{itemize}
  By HMS, at least the answer to (ii) should be affermative.
Notice that the existence of a smooth strict transform $\widehat{Y}^\vee_{LT}\subset\widehat{X}(\Si)$ is enough likely, as the six points blown up by $\phi:\XX\to X$ do not belong to $Y^\vee_{LT}$. If this assumption is satisfied then a derived equivalence $\cD^b(\widehat{Y}^\vee_{BB})\cong \cD^b(\widehat{Y}^\vee_{LT})$ cannot be deduced from Theorem~\ref{thm:Malter 2.23}. In other words, by \cite{Favero-Kelly}, Theorem~\ref{thm:Malter 2.23} actually shows that:
 \begin{itemize}
   \item[(iii)] two different desingularizations of $Y^\vee_{BB}\subset\XX$, induced by MPCP-resolutions of $\XX$, are derived equivalent.
 \end{itemize}
 \end{remark}

Anyway, an affermative answer to (ii) can be obtained by means of results by Kawamata \cite{Kawamata}. Let us first recall the definition of \emph{canonical covering stack}:

\begin{definition}[see Def.\,6.1 in \cite{Kawamata}]\label{def:K-stack} Let $Y$ be a normal quasiprojective variety whose canonical divisor $K_Y$ is a $\Q$-Cartier divisor. Each point $x\in Y$ has an open neighborhood $U_x$ such that $m_xK_Y$ is a principal Cartier divisor on $U_x$, for a minimum postive integer $m_x$\,. The \emph{canonical covering} $\pi_x:\widetilde{U}_x\to U_x$ is a finite morphism of degree $m_x$ from a normal variety which is \'{e}tale in codimension 1 and such that $K_{U_x}$ is a Cartier divisor. The canonical coverings are \'{e}tale locally uniquely determined, thus we can define the \emph{canonical covering stack $\mathcal{Y}$} as the stack above $Y$ given by the collection of canonical coverings $\pi_x:\widetilde{U}_x\to U_x$\,.
\end{definition}

\begin{theorem}\label{thm:Dequivalenza}
  For any choice $\Si\in\PSF(\widehat{W}),\,\Si'\in\PSF(\widehat{\L})$ consider the induced (possibly partial) desingularizations
  $$\widehat{Y}^\vee_{LT}=(\phi_\Si)^{-1}_*(Y^\vee_{LT})\to Y^\vee_{LT}\quad\text{and}\quad\widehat{Y}^\vee_{BB}=(\psi_{\Si'})^{-1}_*(Y^\vee_{BB})\to Y^\vee_{BB}$$
  and their canonical covering stacks $\cY,\,\cY'$. Then there exists an equivalence of triangulated categories
  \begin{equation*}
    \cD^b(\cY)\cong\cD^b(\cY')
  \end{equation*}
  between their derived categories of bounded complexes of coherent orbifold sheaves.
\end{theorem}
\begin{corollary}\label{cor:D-equiv}
  If $\widehat{Y}^\vee_{LT}$ and $\widehat{Y}^\vee_{BB}$ are smooth, then there is an equivalence of triangulated categories $\cD^b(\widehat{Y}^\vee_{LT})\cong \cD^b(\widehat{Y}^\vee_{BB})$\,.
\end{corollary}

\begin{proof}[Proof of Thm.\,\ref{thm:Dequivalenza}] Recall we have crepant birational morphisms
\begin{equation*}
  \xymatrix{\widehat{Y}^\vee_{BB}\ar[dr]_-{\phi\circ\psi_{\Si'}}&&
  \widehat{Y}^\vee_{LT}\ar[dl]^-{\phi_\Si}\\
                &Y^\vee_{LT}&}
\end{equation*}
giving rise to a crepant birational map $f:\widehat{Y}^\vee_{BB}\dashrightarrow\widehat{Y}^\vee_{LT}$ between 3-dimensional projective varieties admitting at most canonical singularities, by \cite[Prop.~2.2.2,\,2.2.4]{Batyrev94}: the variety $Z$ in Definition~\ref{def:crepante} is given by $\widehat{Y}^\vee_{BB}\times_{Y^\vee_{LT}}\widehat{Y}^\vee_{LT}$. Then, by \cite[Thm.~4.6]{Kawamata}, $f$ decomposes into a sequence of flops (in the sense of \cite[Def.~4.5]{Kawamata}), and the statement follows immediately by applying \cite[Thm.~6.5]{Kawamata}.
\end{proof}
Notice that Theorem~\ref{thm:Dequivalenza} and Corollary~\ref{cor:D-equiv} represent the mirror theorem~\ref{thm:mirror} when considered with respect to the two mirrors $\widehat{Y}^\vee_{LT}$ and $\widehat{Y}^\vee_{BB}$\,.

\subsubsection{Intermediate mirrors of $Y_{3,3}\subset\P^5$}\label{sssez:mirrorintermedi3,3} Let us here taking into account the construction of intermediate mirror models for $d=3$, recalling \S~\ref{sssez:mirrorintermedi}.

 Starting from the $BB$-mirror and considering the set of exceptional rays in the fan of $\XX$, with respect to the blow up $\phi:\XX\to X$, namely generated by columns $1,2,3,10,11,12$ in the fan matrix $\L$, one can produce a number of further mirror models of the projective complete intersection $Y_{3,3}$, intermediate between the BB and the LT ones. Then, Proposition~\ref{prop:mirrorintermedi} admits the following analogue

 \begin{proposition}\label{prop:mirrorintermedi3,3}
   For any subset $A\subset\{1,2,3,10,11,12\}$ the complete toric variety $\XX^A$, whose fan matrix is the submatrix $\L^A$ of $\L$, is the blow up, say $\phi^A:\XX^A\to X$, in $6-|A|$ points of $X$. Calling $\bb^A=\bb_1^A+\bb_2^A$ the partitioned framing of $\XX^A$ obtained by removing entries indexed by $A$ from the $BB$ partitioned framing $\bb=\bb_1+\bb_2$ of $\XX$, one obtains a partitioned ftv $(\XX^A,\bb^A=\bb_1^A+\bb_2^A)$ whose f-dual partitioned ftv is $(\P^5,\aa=\aa_1+\aa_2)$. In particular, the family of \cy threefolds given by complete intersection $Y_{3,3}\subset\P^5$ is the f-mirror family of the family $Y^\vee_A=\mathcal{V}(p_{1,\psi}^A,p_{2,\psi}^A)$ with $p_{i,\psi}^A\in\Cox(\XX^A)=\C[x_1,\ldots,x_{12-|A|}]_{\Cl(\XX^A)}$.
 \end{proposition}
The proof is analogous to the one proving Proposition~\ref{prop:LTdual_cubics}. Clearly
$$Y^\vee_\emptyset=Y^\vee_{BB}\quad\text{and}\quad Y^\vee_{\{1,2,3,10,11,12\}}=Y^\vee_{LT}$$
but the remaining $2^6-2=62$ cases give distinct further mirror models of $Y_{3,3}$, all connected each other by means of non-calibrated $f$-processes. In particular, if $A$ is a proper subset of $\{1,2,3,10,11,12\}$ then $\XX^A$ is a non-Gorenstein $\Q$-Fano complete toric variety: therefore the Batyrev-Borisov duality does not apply to  $\XX^A$\,, so giving, for any proper $A$, a picture like that described in Remark~\ref{rem:f-picture} and in Fig.~\ref{Fig1}. In particular, Proposition~\ref{prop:mirrorintermedi3,3} give a proof of item (i) in Theorem~\ref{thm:metateorema}, for the complete intersection $Y_{3,3}$ and list of mirrors given by
\begin{equation*}
    \cM=\{Y^\vee_A\,|\,A\subseteq{\{1,2,3,10,11,12\}}\}
\end{equation*}

Appendix \ref{app:C} is devoted to collect data characterizing the 64 mirror models $(\XX^A,\aa^A=\aa_1^A+\aa_2^A)$, for $A\subseteq\{1,2,3,10,11,12\}$\,.

\subsubsection{$D$-equivalence and $K$-equivalence}
The proof of the following result goes exactly as the one proving the analogous (and included for $A=\emptyset$ and $A'=\{1,2,3,10,11,12\}$) Theorem~\ref{thm:Dequivalenza}.

\begin{theorem}\label{thm:Dequivalenza^A}
  Let $A,A'$ be two subsets of $\{1,2,3,10,11,12\}$ and let $\widehat{\L}^A$ be the matrix whose columns are given by all the primitive lattice points contained in $\conv(\L^A)\setminus\{\0\}$, and analogously for $\widehat{\L}^{A'}$.  For any choice $$\Si\in\PSF(\widehat{\L}^A)\ ,\quad\Si'\in\PSF(\widehat{\L}^{A'})$$
  consider the induced (possibly partial) desingularizations $$\widehat{Y}^\vee_{A}=(\psi_\Si)^{-1}_*(Y^\vee_{A})\to Y^\vee_{A}\ ,\quad\widehat{Y}^\vee_{A'}=(\psi_{\Si'})^{-1}_*(Y^\vee_{A'})\to Y^\vee_{A'}$$
  and their canonical covering stacks $\cY_A,\,\cY_{A'}$. Then there exists an equivalence of triangulated categories
  \begin{equation*}
    \cD^b(\cY_A)\cong\cD^b(\cY_{A'})
  \end{equation*}
  between their derived categories of bounded complexes of coherent orbifold sheaves.
\end{theorem}
\begin{corollary}\label{cor:D-equiv_3,3}
  If $\widehat{Y}^\vee_{A}$ and $\widehat{Y}^\vee_{A'}$ are smooth, then there is an equivalence of triangulated categories $\cD^b(\widehat{Y}^\vee_{A})\cong\cD^b(\widehat{Y}^\vee_{A'})$\,.
\end{corollary}

We can then prove the following

\begin{theorem}\label{thm:Kequivalenza3}
    Assume same hypothesis as in Theorem~\ref{thm:Dequivalenza^A}. Then $\widehat{Y}^\vee_A$ and $\widehat{Y}^\vee_{A'}$ are $K$-equivalent.
\end{theorem}

\begin{proof}
  Start with the LT-mirror model $Y^\vee_{LT}\subset X$ and the choice of subsets $A$ and $A'$ in $\{1,2,3,10,11,12\}$, and consider the blowups
  \begin{equation*}
    \xymatrix{\XX^A\ar[dr]_-{\phi^A}&&\XX^{A'}\ar[dl]^-{\phi^{A'}}\\
                    &X&}
  \end{equation*}
  Recalling the construction of matrices $\widehat{\L}^A$ and $\widehat{\L}^{A'}$ given in Theorem~\ref{thm:Dequivalenza^A}, the choice of $\Si\in\PSF(\widehat{\L}^A)$ and $\Si'\in\PSF(\widehat{\L}^{A'})$ gives two further birational morphisms
  \begin{equation*}
    \xymatrix{\psi_\Si:\widehat{\XX}^A(\Si)\ar[r]&\XX^A}\ ,\ \xymatrix{\psi_{\Si'}:\widehat{\XX}^{A'}(\Si')\ar[r]&\XX^{A'}}
  \end{equation*}
  which, composed with the previous ones, give a commutative diagram of birational maps between toric varieties descending to give an analogous diagram between embedded mirror models
  \begin{equation*}
    \xymatrix{\widehat{\XX}^A(\Si)\ar@{-->}[rr]^-{\vf^A_{A'}}_-{\cong}
    \ar[dr]_-{\psi_\Si\circ\phi^A}&&\widehat{\XX}^{A'}(\Si')
    \ar[dl]^-{\psi_{\Si'}\circ\phi^{A'}}\\
                    &X&}\quad \Longrightarrow\quad
                    \xymatrix{\widehat{Y}^\vee_A\ar@{-->}[rr]^-{\vf^A_{A'}}_-{\cong}\ar[dr]_-{\psi_\Si\circ\phi^A}&&\widehat{Y}^\vee_{A'}
    \ar[dl]^-{\psi_{\Si'}\circ\phi^{A'}}\\
                    &Y^\vee_{LT}&}
  \end{equation*}
 Consider the blow up $\phi^{A\cap A'}:\XX^{A\cap A'}\to X$\,. Clearly $\XX^{A\cap A'}=\XX^A\times_X \XX^{A'}$. By stellar subdivision of fans
 \begin{itemize}
   \item $\Si$ with respect to the new rays generated by columns of $\widehat{\L}^{A\cap A'}$ indexed by $A'\setminus A$
   \item $\Si'$ with respect to the new rays generated by columns of $\widehat{\L}^{A\cap A'}$ indexed by $A\setminus A'$
 \end{itemize}
 one obtains two fans $\widetilde{\Si}$ and $\widetilde{\Si}'$ in $\PSF(\widehat{\L}^{A\cap A'})$ (for the details apply \S~3.4.1 and Lemma 4 in \cite{RT-Pic}) with birational maps
 \begin{equation*}
   \xymatrix{\widetilde{\psi}:\widehat{\XX}^{A\cap A'}\left(\widetilde{\Si}\right)\ar[r]&\widehat{\XX}^A(\Si)}\quad,
   \quad\xymatrix{\widetilde{\psi}':\widehat{\XX}^{A\cap A'}\left(\widetilde{\Si}'\right)\ar[r]&\widehat{\XX}^{A'}(\Si')}
 \end{equation*}
 which are both divisorial blowups restricting to crepant birational morphisms between the embedded mirror models,
 so giving the following commutative diagram of birational maps and morphisms:
 \begin{equation}\label{bir.lift}
   \xymatrix{\widetilde{Y}^\vee\ar[d]_-{\widetilde{\psi}}\ar@{-->}[r]^-{\vf}_-{\cong}&
   \widetilde{Y'}^\vee\ar[d]^-{\widetilde{\psi}'}\\
   \widehat{Y}^\vee_A\ar@{-->}[r]^-{\vf^A_{A'}}_-{\cong}&\widehat{Y}^\vee_{A'}}
 \end{equation}
 The birational equivalence $\vf:\widehat{\XX}^{A\cap A'}(\widetilde{\Si})\dashrightarrow\widehat{\XX}^{A\cap A'}(\widetilde{\Si}')$ is a s$\Q$m between $\Q$-factorial projective toric varieties: recall the proof of Theorem~\ref{thm:smallK&D-equivalenza} and Remark~\ref{rem:triangolo} and go on in a similar way. Namely, $\vf$ is obtained by a finite (non unique) sequence of wall-crossings and so it is a finite sequence of flops:
 \begin{equation*}
   \exists\ s\in\N\ :\ \vf^A_{A'}=\vf_1\circ\cdots\circ\vf_s
 \end{equation*}
 For any $i=1,\ldots,s$\,, $\vf_i$ either is the identity or replaces one facet $\tau_i$, between neighboring maximal cones of the fan $\widetilde{\Si}$, with a different facet $\tau'_i$ between neighboring maximal cones of the fan $\widetilde{\Si}'$. Then, there is a chain of commutative diagrams
  \begin{equation*}
  \resizebox{1\hsize}{!}{$
    \xymatrix{&\XX_s\ar[dl]^-{\text{\ blowups of}\ \tau_s\cap\tau'_s}\ar[dr]&&\cdots\ar[dl]\ar[dr]&&\XX_1\ar[dl]^-{\text{\ blowups of}\ \tau_1\cap\tau'_1}\ar[dr]&\\
    \widehat{\XX}^{A\cap A'}(\widetilde{\Si})\ar[rr]^-{\vf_s}\ar[rd]_{\text{contract}\,\tau_s}&&
    \widehat{\XX}^{A\cap A'}(\widetilde{\Si}_{s-1})\ar[dl]^{\text{contract}\,\tau'_s}\ar[r]^-{\vf_{s-1}}
    \ar[dr]&\cdots\ar[r]^-{\vf_{2}}&
    \widehat{\XX}^{A\cap A'}(\widetilde{\Si}_{1})\ar[dl] \ar[rd]_{\text{contract}\,\tau_1} \ar[rr]^-{\vf_1}\ar[rd]&&\widehat{\XX}^{A\cap A'}(\widetilde{\Si}')\ar[dl]^{\text{contract}\,\tau'_1}\\
             &V_s&&\cdots&&V_1&        }$}
  \end{equation*}
  where $\XX_i=\widehat{\XX}^{A\cap A'}(\widetilde{\Si}_i)\times_{V_i}\widehat{\XX}^{A\cap A'}(\widetilde{\Si}_{i-1})$, for $i=1,\ldots,s$, and
  $$\widetilde{\Si}=\widetilde{\Si}_s\ ,\ \widetilde{\Si}_{s-1}\ ,\ \ldots\ ,\ \widetilde{\Si}_1\ ,\ \widetilde{\Si}_0=\widetilde{\Si}'$$
  is a sequence of intermediate simplicial fans.  By taking successive fibred products
   $$\XX_i\times_{\widehat{\XX}^{A\cap A'}(\widetilde{\Si}_{i})}\XX_{i-1}$$
   and going on in this way, one finally obtains a dominant toric variety $\widehat{\XX}$ with birational morphisms $f$ and $g$
  \begin{equation*}
    \xymatrix{&\widehat{\XX}\ar[dr]^g\ar[dl]_f&\\
               \widehat{\XX}^{A\cap A'}(\widetilde{\Si})\ar[d]_-{\widetilde{\psi}}\ar@{-->}[rr]^-{\vf}_-{\cong}&&
                \widehat{\XX}^{A\cap A'}(\widetilde{\Si}')\ar[d]_-{\widetilde{\psi}'}\\
               \widehat{\XX}^A(\Si)\ar@{-->}[rr]^-{\vf^A_{A'}}_-{\cong}
    &&\widehat{\XX}^{A'}(\Si') }
  \end{equation*}
  Let us now restrict this picture to embedded mirror models $\widehat{Y}^\vee_A$ and $\widehat{Y}^\vee_{A'}$ and consider the birational transform $(\widetilde{\psi}\circ f)^{-1}_*(\widehat{Y}^\vee_A)=Z=(\widetilde{\psi}'\circ g)^{-1}_*(\widehat{Y}^\vee_{A'})$\,.

   Then we get the diagram
     \begin{equation*}
    \xymatrix{&Z\ar[dl]_{\widetilde{\psi}\circ f}\ar[dr]^{\widetilde{\psi}'\circ g}&\\
                \widehat{Y}^\vee_A\ar@{-->}[rr]^-{\vf^A_{A'}}_-{\cong}
                &&\widehat{Y}^\vee_{A'}
                    }
  \end{equation*}
  with
  $$(\widetilde{\psi}\circ f)^*K_{\widehat{Y}^\vee_A}\sim_\Q  (\widetilde{\psi}'\circ g)^*K_{\widehat{Y}^\vee_{A'}}$$
\end{proof}

\begin{remark}
  Notice that previous Theorems~\ref{thm:Dequivalenza^A} and \ref{thm:Kequivalenza3} shows Theorem~\ref{thm:metateorema} for $\cM=\{Y^\vee_A\,|\,A\subseteq\{1,2,3,10,11,12\}\}$\,.
\end{remark}

Recalling Malter Theorem~\ref{thm:Malter 2.23} and question (i) in Remark~\ref{rem:noMalter}, the following statement implies a clear rewriting of Conjecture~\ref{conj:sing} in the present setup.

\begin{conjecture}\label{conj:sings}
Given $A,A'\subseteq\{1,2,3,10,11,12\}$ assume same hypothesis as in Theorem~\ref{thm:Dequivalenza^A}. Consider the two fans $\widetilde{\Si},\widetilde{\Si}'\in\PSF(\widehat{\L}^{A\cap A'})$ constructed by stellar subdivision of $\Si\in\PSF(\widehat{\L}^A)$ and $\Si'\in\PSF(\widehat{\L}^{A'})$, respectively, and the induced diagram of birational maps (\ref{bir.lift}). Call $\widetilde{\cY}$ and $\widetilde{\cY}'$ the canonical covering stacks of $\widetilde{Y}^\vee$ and $\widetilde{Y'}^\vee$, respectively. Then there exist equivalences of triangulated categories
   \begin{eqnarray*}
    \cD^b\left(\widetilde{\cY}\right)&\cong&\cD^b\left(\widetilde{\cY}'\right)\\
    \cD_{sg}\left(\widetilde{Y}^\vee\right)&\cong&\cD_{sg}\left(\widetilde{Y'}^\vee\right)
  \end{eqnarray*}
  In particular, if both $\widetilde{Y}^\vee$ and $\widetilde{Y'}^\vee$ are smooth, then there is an equivalence of triangulated categories $\cD^b(\widetilde{Y}^\vee)\cong\cD^b(\widetilde{Y'}^\vee)$\,.
\end{conjecture}
Notice that Theorem~\ref{thm:Malter 2.23} proves this conjecture when $A=A'=\emptyset$.

 \section{The complete intersection $Y_{d,d}\subset\P^{2d-1}$}\label{sez:d,d}
 After the warm up given by previous \S~\ref{ssez:quadrics} and \S~\ref{ssez:cubics}, we can now start with a first generalization on the degree/dimension $d$, to study the family of complete intersections of bi-degree $(d,d)$ in $\P^{2d-1}$, whose generic element gives a smooth \cy $(2d-3)$-fold $Y\subset\P^{2d-1}$. Unfortunately, as Malter observes in \cite[Rem.~4.9]{Malter}, for $d\ge 4$ the Libgober-Teitelbaum mirror $Y^\vee_{LT}=\mathcal{V}(q_{1,\psi},q_{2,\psi})$ is no more quasi-smooth. In fact,
 \begin{equation*}
   q_{1,\psi}=\sum_{i=1}^d x_i^d + \psi\prod_{j=1}^d x_{d+j}\quad,\quad q_{2,\psi}=\sum_{i=1}^d x_{d+i}^d + \psi\prod_{j=1}^d x_{j}
 \end{equation*}
 so giving that, for instance, $(0,\ldots,0,x_{d+3},\ldots,x_{2d})$ with $\sum_{k=d+3}^{2d} x_k^d=0$ is a singular point of $Y^\vee_{LT}$. For this reason, Malter's argument does no more work to prove a generalization on $d$ of Theorem~\ref{thm:Malter 2.23}. Moreover, both $Y$ and its mirrors are no more 3-dimensional, so that Kawamata results in \cite{Kawamata} do not more hold to prove the conjectured $D$-equivalence of $\widehat{Y}_{LT}$ and $\widehat{Y}_{BB}$.

  Nevertheless, the geometric argumentation described above for $2\le d \le 3$, still holds for $d\ge 4$ allowing us to conclude the following

 \begin{theorem}\label{thm:CY_d,d}
   For any positive integer $d\ge 2$, the family of \cy projective complete intersections $Y_{d,d}\subset\P^{2d-1}$ admits two mirror families given by
   \begin{itemize}
     \item the one obtained by Batyrev-Borisov duality and given by a suitable resolution $\widehat{Y}^\vee_{BB}$ of
         \begin{equation*}
           Y^\vee_{BB}=\mathcal{V}(p_{1,\psi},p_{2,\psi})\ \text{with}\ \left\{\begin{array}{cc}
                                                                        p_{1,\psi}= & \sum_{i=1}^d x_i^dx_{2d+i}^d +\psi\prod_{j=1}^{2d}x_j\\
                                                                        p_{2,\psi}= & \sum_{i=1}^d x_{d+i}^dx_{3d+i}^d +\psi\prod_{j=1}^{2d}x_{2d+j}
                                                                      \end{array}\right.
         \end{equation*}
         embedded in a $(2d-1)$-dimensional complete toric variety $\XX$ of Picard number $2d+1$,
     \item the one obtained by a generalized Libgober-Teitelbaum construction and given by a suitable resolution $\widehat{Y}^\vee_{LT}$ of
         \begin{equation*}
           Y^\vee_{LT}=\mathcal{V}(q_{1,\psi},q_{2,\psi})\ \text{with}\ \left\{\begin{array}{cc}
                                                                        q_{1,\psi}= & \sum_{i=1}^d x_i^d + \psi\prod_{j=1}^d x_{d+j}\\
                                                                        q_{2,\psi}= & \sum_{i=1}^d x_{d+i}^d + \psi\prod_{j=1}^d x_{j}
                                                                      \end{array}\right.
         \end{equation*}
         embedded in a suitable quotient $X=\P^{2d-1}/G$, by the action of a finite group $G\cong \Tors(\Cl(X))$.
   \end{itemize}
      They give two partitioned framed toric varieties, $(\XX,\bb=\bb_1+\bb_2)$ and $(X,\cc=\cc_1+\cc_2)$, respectively, with $\bb_1=(\1_{2d},\0_{2d}),\bb_2=(\0_{2d},\1_{2d}),\cc_1=(\1_d,\0_d),\cc_2=(\0_d,\1_d)$\,, connected by a non-calibrated $f$-process
      \begin{equation*}
        \xymatrix{&(\P^{2d-1},\aa=\aa_1+\aa_2)\ar@{<~>}[dr]^-{\hskip 2truecm\text{calibrated $f$-process $=$ BB-duality}}&\\
                    (X,\cc=\cc_1+\cc_2)\ar@{~>}[rr]_
                    -{\text{non-calibrated $f$-process}}\ar@{~>}[ru]^-{\text{$f$-duality}\quad}&&(\XX,\bb=\bb_1+\bb_2)}
      \end{equation*}
      where $\aa_1=(\1_d,\0_d)$ and $\aa_2=(\0_d,\1_d)$ (recall notation given in \S~\ref{ssez:vettori}).\\
       Moreover, $\XX$ turns out to be the blow up of $X$ in the $2d$ distinct points $$P_i=[0:\cdots:0:\overbrace{1}^i:0:\cdots:0]\ ,\quad 1\le i\le 2d$$
       say $\phi:\XX\to X$, and $Y^\vee_{BB}$ is the strict transform of $Y^\vee_{LT}$ under $\phi$.
 \end{theorem}
 \begin{proof}
   Ideas are the same as for the previous cases $d=2,3$. Start with polytopes associated with $D_{\aa_i}$, that are
    \begin{equation*}
      \D_{\aa_1}=\conv\left(
                        \begin{array}{ccccc}
                        &\vline&&&\\
                          d\,I_d-\1_{d,d}&\vline &  & -\1_{d,d}& \\
                          &\vline&&&\\
                          \hline
                          &\vline&&\vline&\\
                          \0_{d-1,d} & \vline & d\,I_{d-1} & \vline & \0_{d-1}^T \\
                          &\vline&&\vline&\\
                        \end{array}
                      \right)
    \end{equation*}
    \begin{equation*}
      \D_{\aa_2}=\conv\left(
                        \begin{array}{ccccc}
                        &\vline&&&\\
                          d\,I_d&\vline &  & \0_{d,d}& \\
                          &\vline&&&\\
                          \hline
                          &\vline&&\vline&\\
                          -\1_{d-1,d} & \vline & d\,I_{d-1}-\1_{d-1,d-1} & \vline & -\1_{d-1}^T \\
                          &\vline&&\vline&\\
                        \end{array}
                      \right)
    \end{equation*}
  giving $\D=\conv(\D_{\aa_1},\D_{\aa_2})$. Calling $\L$ the fan matrix of $\XX$, that is the matrix whose columns are given by vertices of $\D$, and $V$ the fan matrix of $\P^{2d-1}$,
  then
  \begin{equation*}
    \L^T\cdot V=\left(\begin{array}{ccc}
    &\vline&\\
       d\,I_d-\1_{d,d}&\vline &\0_{d,d} \\
       &\vline&\\
        \hline
        &\vline&\\
     -\1_{d,d}&\vline&d\,I_d\\
     &\vline&\\
        \hline
        &\vline&\\
        d\,I_d& \vline&-\1_{d,d}\\
        &\vline&\\
        \hline
        &\vline&\\
         \0_{d,d}&\vline& d\,I_d-\1_{d,d}\\
         &\vline&\\
        \end{array}
      \right)
  \end{equation*}
  so giving $\bb_1$ and $\bb_2$ as in the statement. Therefore
  \begin{equation*}
      \L_{\bb_1}=\left(\begin{array}{cccc}
      &&\vline&\\
           &I_d&\vline&\0^T_d  \\
           &&\vline&\\
           \hline
           &&\vline&\\
           &\0_{d-1,d}&\vline&\0^T_{d-1}\\
           &&\vline&
      \end{array}
      \right)\ ,\quad\L_{\bb_2}=\left(\begin{array}{cccccc}
      &&\vline&&\vline&\\
           &\0_{d,d-1}&\vline&-\1^T_d&\vline&\0^T_d  \\
           &&\vline&&\vline&\\
           \hline
           &&\vline&&\vline&\\
           &I_{d-1}&\vline&-\1^T_{d-1}&\vline&\0^T_{d-1}\\
           &&\vline&&\vline&
      \end{array}
      \right)
  \end{equation*}
  and calling $M_i=\L^T\cdot\L_{\bb_i}+B_i$, with $B_i=(\bb_i^T,\cdots,\bb_i^T)$, one obtains
  \begin{equation*}
    M_1=\left(\begin{array}{ccc}
    &\vline&\\
       d\,I_d&\vline &\1_d^T\\
       &\vline&\\
        \hline
        &\vline&\\
     \0_{d,d}&\vline&\1_d^T\\
     &\vline&\\
        \hline
        &\vline&\\
        d\,I_d& \vline&\0_{d}^T\\
        &\vline&\\
        \hline
        &\vline&\\
         \0_{d,d}&\vline& \0_{d}^T\\
         &\vline&\\
        \end{array}
      \right)\ ,\quad M_2=\left(\begin{array}{ccc}
      &\vline&\\
       \0_{d,d}&\vline& \0_{d}^T \\
       &\vline&\\
        \hline
        &\vline&\\
        d\,I_d& \vline&\0_{d}^T\\
        &\vline&\\
        \hline
        &\vline&\\
        \0_{d,d}&\vline&\1_d^T\\
        &\vline&\\
        \hline
        &\vline&\\
         d\,I_d&\vline &\1_d^T \\
         &\vline&\\
        \end{array}
      \right)
  \end{equation*}
so that $p_{1,\psi}$ and $p_{2,\psi}$
are as given in the statement.

Consider now the $2d$ central columns of $\L$, giving the fan matrix
\begin{equation*}
    W=\left(\begin{array}{ccccc}
    &&&\vline&\\
        &-\1_{d,d} && \vline &d\,I_d \\
        &&&\vline&\\
        \hline
        &\vline&&\vline&\\
       d\,I_{d-1}  & \vline&\0^T_{d-1}&\vline&-\1_{d-1,d}\\
       &\vline&&\vline&\\
    \end{array}\right)
\end{equation*}
of the complete and $\Q$-factorial toric variety $X=\P^{2d-1}/G$, being $G$ a finite group: this fact follows by observing that $Q=\left(\1_{2d}\right)$ is a Gale dual matrix of $W$, hence a weight matrix of $X$, meaning that $\P^{2d-1}$ is the universal covering of $X$ and $G=\Tors(\Cl(X))$ (see e.g. \cite[Thm.~2.2]{RT-QUOT}). The partitioned framing $\cc=\cc_1+\cc_2$ for $X$ gives the associated polytopes
\begin{eqnarray*}
  \D_{\cc_1}&=&\conv\left(
                          \begin{array}{ccccc}
                          &\vline&&\vline&\\
                           d\,I_d -(1/d)\1_{d,d} &\vline & \0_{d,d-1} & \vline & -\1_d^T \\
                           &\vline&&\vline&\\
                           \hline
                           &\vline&&\vline&\\
                            \0_{d-1,d} & \vline & I_{d-1} & \vline & -\1_{d-1}^T \\
                            &\vline&&\vline&\\
                          \end{array}
                        \right)\\
  \D_{\cc_2}&=&\conv\left(\begin{array}{ccccc}
  &\vline&&\vline&\\
                           I_d &\vline & (1/d)\1_{d,d-1} & \vline & -(d-1)/d\,\1_d^T \\
                           &\vline&&\vline&\\
                           \hline
                           &\vline&&\vline&\\
                            \0_{d,d} & \vline & I_{d-1} & \vline & -\1_{d-1}^T \\
                            &\vline&&\vline&\\
                          \end{array}
                        \right)
\end{eqnarray*}
Then $[\conv(\D_{\cc_1},\D_{\cc_2})]=\D_V$ where $V$ is the fan matrix of $\P^{2d-1}$ given by
\begin{equation*}
  V=\left(
      \begin{array}{ccc}
        I_{2d-1} & \vline & -\1_{2d-1}^T \\
      \end{array}
    \right)
\end{equation*}
For the dual framing notice that
  \begin{equation*}
    V^T\cdot W= \left( \begin {array}{ccc}
                            &\vline&\\
                        d\,I_d&\vline&-\1_{d,d}\\
                        &\vline&\\
                        \hline
                        &\vline&\\
                        -\1_{d,d}&\vline&d\,I_d\\
                        &\vline&\\
                        \end {array} \right)
  \end{equation*}
  so giving the nef partitioned framing $\aa=\aa_1+\aa_2$, with $\aa_i$ as given in the statement.
  Finally, one has
  \begin{equation*}
    M^\vee_1=\left(\begin{array}{ccc}
                        &\vline&\\
                     d\,I_d & \vline & \0_d^T \\
                     &\vline&\\
                     \hline
                     &\vline&\\
                     \0_{d,d} & \vline & \1_{d}^T\\
                     &\vline&\\
                   \end{array}\right)\ ,
    \quad M^\vee_2=\left(\begin {array}{ccc}
    &\vline&\\
    \0_{d,d}&\vline&\1_{d}^T \\
    &\vline&\\
     \hline
     &\vline&\\
     d\,I_d&\vline&\0_d^T\\
     &\vline&\\
      \end{array}
     \right)
  \end{equation*}
  whose columns give monomial exponents of $q_{1,\psi},q_{2,\psi}$, respectively, as given in the statement.

  For the last part of the statement, one observes that the fan matrix $W$ is composed by the $2d$ central columns of the fan matrix $\L$. Therefore one gets the divisorial blow up $\phi:\XX\to X$ by the same argument applied for $d=2,3$ whose geometric properties obviously extend to cases $d\ge 4$\,.
 \end{proof}

 \begin{remark}
   With the aid of some Maple subroutines, I could check quasi-smoothness of $Y^\vee_{BB}$ for $2\le d\le 5$, although $Y^\vee_{LT}$ stops to be quasi-smooth for $d\ge 3$. It looks quite likely that $Y^\vee_{BB}$ would be still quasi-smooth for any bigger degree $d$, although the cumbersome combinatorics of the critical locus $\mathcal{C}$ and the irrelevant ideal $\mathfrak{I}$, for general $d$, makes difficult to produce a rigorous proof of this fact.
 \end{remark}

 \subsection{Intermediate mirrors of $Y_{d,d}\subset\P^{2d-1}$}\label{sssez:mirrorintermedid,d} As done for $d=2,3$ in \S~\ref{sssez:mirrorintermedi} and \S~\ref{sssez:mirrorintermedi3,3}, respectively, it is possible produce a number of further mirror models of $Y_{d,d}\subset\P^{2d-1}$.

 Starting from the $BB$-mirror and considering the set of exceptional rays in the fan of $\XX$, with respect to the blow up $\phi:\XX\to X$, namely generated by columns $1,\ldots, d,3d+1,\ldots,4d$ in the fan matrix $\L$, one can produce a number of further mirror models of the projective complete intersection $Y_{d,d}$, intermediate between the BB and the LT ones. Then, Propositions~\ref{prop:mirrorintermedi} and \ref{prop:mirrorintermedi3,3} admit the following generalization.

\begin{proposition}\label{prop:mirrorintermedid,d}
   Call $\mathcal{I}_d=\{1,\ldots, d,3d+1,\ldots,4d\}$. For any subset $A\subset\cI_d$ the complete toric variety $\XX^A$, whose fan matrix is the submatrix $\L^A$ of $\L$, is the blow up, say $\phi^A:\XX^A\to X$, in $2d-|A|$ points of $X$. Calling $\bb^A=\bb_1^A+\bb_2^A$ the partitioned framing of $\XX^A$ obtained by removing entries indexed by $A$ from the $BB$ partitioned framing $\bb=\bb_1+\bb_2$ of $\XX$, one obtains a partitioned ftv $(\XX^A,\bb^A=\bb_1^A+\bb_2^A)$ whose f-dual partitioned ftv is $(\P^{2d-1},\aa=\aa_1+\aa_2)$. In particular, the family of \cy varieties given by complete intersection $Y_{d,d}\subset\P^{2d-1}$ is the f-mirror family of the family $Y^\vee_A=\mathcal{V}(p_{1,\psi}^A,p_{2,\psi}^A)$ with $p_{i,\psi}^A\in\Cox(\XX^A)=\C[x_1,\ldots,x_{4d-|A|}]_{\Cl(\XX^A)}$.
 \end{proposition}

 The proof is analogous to lower degree cases. Clearly
$Y^\vee_\emptyset=Y^\vee_{BB}$ and $Y^\vee_{\cI_d}=Y^\vee_{LT}$\,,
but the remaining $2^{2d}-2$ cases give distinct further mirror models of $Y_{d,d}$, everyone connected to the calibrated $f$-dual $Y^\vee_{BB}$ by means of a non-calibrated $f$-process. In particular, if $A$ is a proper subset of $\cI_d$ then $\XX^A$ is a non-Gorenstein $\Q$-Fano complete toric variety: therefore the Batyrev-Borisov duality does not apply to  $\XX^A$\,, so giving, for any proper subset $A$, a picture like that described in Remark~\ref{rem:f-picture} and in Fig.~\ref{Fig1}.

\subsection{$K$-equivalence vs $D$-equivalence} Unfortunately, Theorem~\ref{thm:Dequivalenza^A} cannot be generalized to degrees $d\ge 4$, as the needed Kawamata results hold in dimension 3. On the contrary, Theorem~\ref{thm:Kequivalenza3} generalizes immediately, with the same proof, so getting the following

\begin{theorem}\label{thm:Kequivalenzad}
  Let $A,A'$ be two subsets of $\cI_d$ and let $\widehat{\L}^A$ be the matrix whose columns are given by all the primitive lattice points contained in $\conv(\L^A)\setminus\{\0\}$, and analogously for $\widehat{\L}^{A'}$.  For any choice $$\Si\in\PSF(\widehat{\L}^A)\ ,\quad\Si'\in\PSF(\widehat{\L}^{A'})$$
  consider the induced (possibly partial) desingularizations $$\widehat{Y}^\vee_{A}=(\psi_\Si)^{-1}_*(Y^\vee_{A})\to Y^\vee_{A}\ ,\quad\widehat{Y}^\vee_{A'}=(\psi_{\Si'})^{-1}_*(Y^\vee_{A'})\to Y^\vee_{A'}$$
  Then $\widehat{Y}^\vee_{A}$ and $\widehat{Y}^\vee_{A'}$ are $K$-equivalent.
\end{theorem}

Taking into account Conjectures~\ref{conj:Kawamata}, \ref{conj:BOKgen}, one naturally gets the following

\begin{conjecture}\label{conj:K<=>D,d}
  Under the same hypotheses of the previous Theorem~\ref{thm:Kequivalenzad}, let $\cY_A$ and $\cY_{A'}$ be the canonical covering stacks of $\widehat{Y}^\vee_{A}$ and $\widehat{Y}^\vee_{A'}$, respectively. Then there exists an equivalence of triangulated categories
  \begin{equation*}
    \cD^b(\cY_A)\cong\cD^b(\cY_{A'})
  \end{equation*}
  In particular, if both $\widehat{Y}^\vee_{A}$ and $\widehat{Y}^\vee_{A'}$ are smooth, then there is an equivalence of triangulated categories $\cD^b(\widehat{Y}^\vee_{A})\cong\cD^b(\widehat{Y}^\vee_{A'})$\,.
\end{conjecture}

\begin{remark}
  Notice that Proposition~\ref{prop:mirrorintermedid,d} and Theorem~\ref{thm:Kequivalenzad} shows items (i) and (ii) in Theorem~\ref{thm:metateorema}, respectively, for $\cM=\{Y^\vee_A\,|\,A\subseteq\cI_d\}$\,. Moreover, item (iii) is expressed by the previous Conjecture~\ref{conj:K<=>D,d}\,. In particular, we are not able to prove any mirror theorem~\ref{thm:mirror} for $d\ge 4$, although arguments leading to state Conjecture~\ref{conj:K<=>D,d} show that the permanence of  $K$-equivalence between $\widehat{Y}^\vee_{A}$ and $\widehat{Y}^\vee_{A'}$ makes their $D$-equivalence quite likely.
\end{remark}

As above, the following statement implies a clear rewriting of Conjecture~\ref{conj:sing} in the present context.

\begin{conjecture}\label{conj:sing'}
    Given $A,A'$ be two subsets of $\cI_d$, in the same hypotheses of  Theorem~\ref{thm:Kequivalenzad}, consider the two fans $\widetilde{\Si},\widetilde{\Si}'\in\PSF(\widehat{\L}^{A\cap A'})$ constructed by stellar subdivision of $\Si\in\PSF(\widehat{\L}^A)$ and $\Si'\in\PSF(\widehat{\L}^{A'})$, respectively, and the induced diagram of birational maps (\ref{bir.lift}). Call $\widetilde{\cY}$ and $\widetilde{\cY}'$ the canonical covering stacks of $\widetilde{Y}^\vee$ and $\widetilde{Y'}^\vee$, respectively. Then there exist equivalences of triangulated categories
   \begin{eqnarray*}
    \cD^b\left(\widetilde{\cY}\right)&\cong&\cD^b\left(\widetilde{\cY}'\right)\\
    \cD_{sg}\left(\widetilde{Y}^\vee\right)&\cong&\cD_{sg}\left(\widetilde{Y'}^\vee\right)
  \end{eqnarray*}
  In particular, if both $\widetilde{Y}^\vee$ and $\widetilde{Y'}^\vee$ are smooth, then there is an equivalence of triangulated categories $\cD^b(\widetilde{Y}^\vee)\cong\cD^b(\widetilde{Y'}^\vee)$\,.
\end{conjecture}

\section{The construction of generalized LT-mirrors}\label{sez:LTgeneral}

The purpose of the present section is extending the generalization of \lt construction on the degree/dimension $d$ of codimension 2, projective, \cy complete intersections, given in the previous \S~\ref{sez:d,d}, as much as possible to most projective complete intersections of non-negative Kodaira dimension.

Consider the generic complete intersection
\begin{equation*}
  Y=Y_{d_1,\ldots,d_l}:=\bigcap_{k=1}^l Y_{d_k}\subset\P^n
\end{equation*}
 of $l\ge 1$ projective hypersurfaces of degree $d_1,\ldots,d_k$, respectively, such that $\sum_{k=1}^ld_k\ge n+1$, that is, $Y$ has nonnegative Kodaira dimension. Following notation introduced in \cite[\S1.3]{R-fpCI}, this means that one can consider the partitioned framing $D_\aa=\sum_k D_{\aa_k}$ of $\P^n$ such that:
 \begin{itemize}
   \item[i:] $D_{\aa_k}:=\sum_{i\in I_k}a_i D_i\in\Div_\T(\P^n)$ is an effective torus invariant divisor, given a partition $\{1,2,\ldots,n+1\}=\bigsqcup_{k=1}^l I_k$,
   \item[ii:] $Y_{d_k}$ is a sufficiently generic element of the linear system $|D_{\aa_k}|$, such that $Y$ is smooth.
 \end{itemize}
 Fix the following notation:
 \begin{eqnarray}\label{a,ak}
\nonumber
  \forall\,k=1,\ldots,l\quad \d_k&:=&d_k - m_k+1\geq 1\quad\text{with $m_k=|I_k|$}\\ \aa_k^T&:=&\underbrace{(0,\ldots,0,\overbrace{1,\ldots,1,\d_k}^{I_k},0,\ldots,0)}_{n+1}\\
\nonumber
  |\aa_k|&:=&\sum_{i=1}^{n+1}a_{ki}=m_k-1+\d_k=d_k\\
\nonumber
  |\aa|&:=&\sum_{i=1}^{n+1}a_i=\sum_{k=1}^l|\aa_k|=\sum_{k=1}^l d_k
\end{eqnarray}
Clearly, $(\P^n,\aa=\sum_{k=1}^l\aa_k)$ is a partitioned framed toric variety (ftv) \cite[Def.~6.1]{R-fTV}, \cite[Def.~1.4]{R-fpCI}, and the hypersurface $Y_{d_k}$ is linearly equivalent to $D_{\aa_k}$, which is a very ample divisor.

\subsection{Mirror models construction}\label{ssez:mirror-costruzione}

In \cite[\S1.3]{R-fpCI} we proved that $(\P^n,\aa=\sum_{k=1}^l\aa_k)$ admits a calibrated $f$-process and exhibited its $f$-mirror dual partner  $(\XX_\aa,\bb=\sum_{k=1}^l\bb_k)$.

\noindent The complete toric variety $\XX_\aa$ is determined by the complete fan $\Si_\aa$ spanned by the polytope $\D_\aa = \conv(\D_{\aa_1},\ldots,\D_{\aa_l})$, with
\begin{equation*}
\D_{\aa_k} = \conv\left(
                          \begin{array}{c}
                            d_k- a_{k,1} \\
                            -a_{k,2} \\
                            -a_{k,3}\\
                            \vdots \\
                            -a_{k,n} \\
                          \end{array}
                                        \begin{array}{c}
                                          -a_{k,1} \\
                                           d_k- a_{k,2}\\
                                          -a_{k,3}\\
                                          \vdots \\
                                          -a_{k,n} \\
                                        \end{array}
                                      \cdots
                                        \begin{array}{c}
                                          -a_{k,1} \\
                                          -a_{k,2}\\
                                          \vdots \\
                                          -a_{k,n-1} \\
                                          d_k- a_{k,n}\\
                                        \end{array}
                                        \begin{array}{c}
                                          -a_{k,1} \\
                                          -a_{k,2}\\
                                          \vdots \\
                                          -a_{k,n-1} \\
                                          -a_{k,n}\\
                                        \end{array}
                                      \right) \\
  \end{equation*}
  Recall the definition of a toric variety spanned by a polytope, given in \S~\ref{ssez:politopi}.

\begin{theorem}[Theorem 1.8 in \cite{R-fpCI}]\label{thm:BB-general}
   The partition ftv $(\P^n,\aa=\sum_{k=1}^l\aa_k)$ admits a calibrated $f$-process and its $f$-dual partner is given by the partitioned ftv $(\XX_\aa, \bb=\sum_{k=1}^l\bb_k)$ where $\XX_\aa$ is the complete toric variety spanned by the polytope $\D_\aa$ and $\bb_k$ is described by displays (13), (14), (15), (16) in the proof of \cite[Thm.~1.8]{R-fpCI}, for $k=1,\ldots,l$.

   In particular, the family of projective complete intersections $Y_{d_1,\ldots,d_l}\subset\P^n$ admits the  $f$-mirror family given by a suitable resolution $\widehat{Y}^\vee_{BB}$ of a complete intersection of $l$ hypersurfaces in the complete toric variety $\XX_\aa$, given in Cox coordinates by
         \begin{equation*}
           Y^\vee_{BB}=\mathcal{V}(p_{1,\psi},\ldots,p_{l,\psi})
         \end{equation*}
         for suitable $p_{1,\psi},\ldots,p_{l,\psi}\in\Cox(\XX_\aa)\cong\C[x_1,\ldots,x_{l(n+1)}]_{\Cl(\XX_\aa)}$\,.
 \end{theorem}

 \begin{remark}
   In the previous statement, we called $Y^\vee_{BB}$ the $f$-mirror family constructed via the calibrated $f$-process there described, just because it is the generalization of the Batyrev-Borisov mirror model obtained in the \cy case. Actually, when $Y_{d_1,\ldots,d_l}$ is a family of general type varieties, all the polytopes involved are no more reflexive, then Batyrev-Borisov duality no more applies.
 \end{remark}

 \subsubsection{The choice of LT-mirror}\label{sssez:LT-ipotesi}

Let us assume:
\begin{itemize}
  \item[(A)] $\forall\,k=1,\ldots,l\quad \text{either}\ m_k\ge 3\ \text{or}\ \aa_k=(1,1)$\,,
  \item[(B)] \emph{the fan matrix $\L_\aa$ admits a choice of $n+1=\sum_{k=1}^lm_k$ columns generating a submatrix $W$ turning out to be a fan matrix of a $\Q$-factorial, complete toric variety $X=\P(\aa)/G$, being $G$ a finite group and $\P(\aa)$ the weighted projective space whose weights are assigned by the original framing $\aa$ of $\P^n$.}
  \item[(C)] \emph{setting $\cc_k:=(\bb_k)_W$, for every $k=1,\dots,l$, the sublist of $\bb_k$ determined by entries corresponding to columns of $W$, the $f$-mirror model of the partitioned ftv $(X,\cc=\sum \cc_k)$ turns out to be precisely $(\P^n,\aa=\sum\aa_k)$, so giving rise to a non-calibrated $f$-process and to the usual picture }
      \begin{equation}\label{f-processi}
        \xymatrix{&(\P^{n},\aa=\aa_1+\aa_2)\ar@{<~>}[dr]^-{\hskip 2truecm\text{calibrated $f$-process}}&\\
                    (X,\cc=\cc_1+\cc_2)\ar@{~>}[rr]_
                    -{\text{non-calibrated $f$-process}}\ar@{~>}[ru]^-{\text{$f$-duality}\quad}&&(\XX_\aa,\bb=\bb_1+\bb_2)}
      \end{equation}
\end{itemize}

 \begin{remark}\label{rem:discussione}
 Assumption (A) is a technical one: in this case, the fan matrix $\L_\aa$ of $\XX_\aa$ admits $l(n+1)$ columns, that is, $\XX_\aa$ has Picard number $1+(l-1)(n+1)$, and it is the juxtaposition of matrices $\L_{\aa_k}$, whose columns are given by vertices of $\D_{\aa_k}$, which, in this case, are all primitive (see the proof of \cite[Thm.~1.18]{R-fpCI}). Moreover, for any $k=1,\ldots,l$, the non-trivial part of $\bb_k$ turns out to be a permutation of $[1,\d_k,\ldots,\d_k]$ (see display (16) in the above cited proof). The following Example~\ref{ex:2,2,3 in P5} studies a case in which assumption (A) is violated so that we no more can get multiple mirrors.

 Assumption (B) does not seem to be an effective assumption: in fact,  in any concrete example I carried out, I found it was always satisfied, for several different choices of $W$.
For instance, when $d=2$ a similar choice is given by $n+1$ ``central'' columns of $\L_\aa$. Notice that, being this choice non-unique, in general, we obtain \emph{multiple choices for the starting LT-mirror model}.

About the finite group $G$, in general, we can only say that there is a monomorphism
\begin{equation*}
\Tors(\Cl(X))\hookrightarrow G
\end{equation*}
where isomorphism is attained  if and only if the framing $\aa$ is, up to a permutation, a Gale dual matrix of the fan matrix $W$, that is
\begin{equation}\label{G=tors}
  \Tors(\Cl(X))\cong G\ \Longleftrightarrow\ W\cdot(\aa')^T=\0_n^T
\end{equation}
for some permutation $\aa'\sim\aa$.
This was the case in simpler cases with $l=2$, studied above. When $l\ge 3$ the situation seems to be more complicated and isomorphism (\ref{G=tors}) is never attained: see e.g. the following Example~\ref{ex:4,5,6 in P8}.

The remaining $(l-1)(n+1)$ columns of $\L_\aa$ generate as many rays in the fan of $\XX_\aa$, hence determining as many exceptional divisors of a naturally defined blow up $\phi:\XX_\aa\to X$\,.

Assumption (C) is not, in general, satisfied for any choice of the submatrix $W$ satisfying assumption (B). But, given assumption (A), I always could find a submatrix $W$ of $\L_\aa$ satisfying both assumptions (B) and (C), in any concrete example I carried out. The following Example~\ref{ex:4,5,6 in P8} is devoted to discussing  all these occurrences.
\end{remark}

 \subsubsection{Mirror models}\label{ssez:mirrors} Given assumptions (A), (B), and (C), one can then reformulate, in the present generalized setup, Theorem~\ref{thm:CY_d,d} and Proposition~\ref{prop:mirrorintermedid,d}, as follows.
  \begin{enumerate}
    \item \emph{The family of projective complete intersections $Y_{d_1,\ldots,d_l}\subset\P^n$ admits two mirror families given by
   \begin{itemize}
     \item the one obtained by a calibrated $f$-process and given by a suitable resolution $\widehat{Y}^\vee_{BB}$ of a complete intersection of $l$ hypersurfaces in the complete toric variety $\XX_\aa$, given in Cox coordinates by
         \begin{equation*}
           Y^\vee_{BB}=\mathcal{V}(p_{1,\psi},\ldots,p_{l,\psi})
         \end{equation*}
         being $p_{1,\psi},\ldots,p_{l,\psi}\in\Cox(\XX_\aa)\cong\C[x_1,\ldots,x_{l(n+1)}]_{\Cl(\XX_\aa)}$\,,
     \item the one obtained by a generalized Libgober-Teitelbaum construction and given by a suitable resolution $\widehat{Y}^\vee_{LT}$ of a complete intersection of $l$ hypersurfaces in the $\Q$-factorial complete toric variety $X=\P(\aa)/G$, given in Cox coordinates by
         \begin{equation*}
           Y^\vee_{LT}=\mathcal{V}(q_{1,\psi},\ldots,q_{l,\psi})
           \end{equation*}
         being $q_{1,\psi},\ldots,q_{l,\psi}\in\Cox(X)\cong\C[x_1,\ldots,x_{(n+1)}]_{\Cl(X)}$.
   \end{itemize}
      They give two partitioned framed toric varieties, $(\XX_\aa,\bb=\sum_{i=1}^l\bb_i)$ and $(X,\cc=\sum_{i=1}^l\cc_i)$, respectively, connected by a non-calibrated $f$-process and giving rise to a picture like (\ref{f-processi}).
             Moreover, $\XX_\aa$ turns out to be the blow up of $X$ in $(l-1)(n+1)$ distinct points,
       say $\phi:\XX_\aa\to X$, and $Y^\vee_{BB}$ is the strict transform of $Y^\vee_{LT}$ under $\phi$.}
    \item \emph{Let $\cI^W\subset\N$ be defined by integers labelling columns of $\L_\aa$ which are not columns of the submatrix $W$. For any subset $A\subset\cI^W$ the complete toric variety $\XX_\aa^A$, whose fan matrix is the submatrix $\L_\aa^A$ of $\L_\aa$, is the blow up, say $\phi^A:\XX_\aa^A\to X$, in $(l-1)(n+1)-|A|$ points of $X$. Calling $\bb^A=\sum_{i=1}^l\bb_i^A$ the partitioned framing of $\XX_\aa^A$ obtained by removing entries indexed by $A$ from the partitioned framing $\bb=\sum_i\bb_i$ of $\XX_\aa$, one obtains a partitioned ftv $(\XX_\aa^A,\bb^A=\sum_i\bb_i^A)$ whose f-dual partitioned ftv is $(\P^n,\aa=\sum_{i=1}^l\aa_i)$. In particular, the family of  projective complete intersection $Y_{d_1,\ldots,d_l}\subset\P^{n}$ is the f-mirror family of the family $Y^\vee_A=\mathcal{V}(p_{1,\psi}^A,\ldots,p_{l,\psi}^A)$ with $p_{i,\psi}^A\in\Cox(\XX_\aa^A)=\C[x_1,\ldots,x_{l(n+1)-|A|}]_{\Cl(\XX_\aa^A)}$. \\
        In particular, one has $Y^\vee_\emptyset=Y^\vee_{BB}$ and $Y^\vee_{\cI^W}=Y^\vee_{LT}$.}
  \end{enumerate}

  The following examples are given to propose evidences to the discussion of assumptions (A), (B) and (C), given in the previous \S~\ref{sssez:LT-ipotesi}.

  \begin{example}[$Y_{2,2,3}\subset\P^5$]\label{ex:2,2,3 in P5}
    The present example is aimed to give a motivation for assumption (A). In fact, the case of the complete intersection of two hyperquadrics and a cubic hypersurface in $\P^5$ corresponds to the partitioned ftv $(\P^5,\aa=\aa_1+\aa_2+\aa_3)$ with
    \begin{equation*}
      \aa_1=(1,1,0,0,0,0)\ ,\quad \aa_2=(0,0,1,1,0,0)\ ,\quad \aa_3=(0,0,0,0,1,2)
    \end{equation*}
    so that $m_3=2<3$ but $\aa_3\neq(1,1)$\,. In this case
    \begin{eqnarray*}
      \D_{\aa_3}&=&\conv\left( \begin {array}{cccccc} 3&0&0&0&0&0\\ \noalign{\medskip}0&3&0&0&0&0\\ \noalign{\medskip}0&0&3&0&0&0\\ \noalign{\medskip}0&0&0&3&0&0\\ \noalign{\medskip}-1&-1&-1&-1&2&-1\end {array} \right)\\
       \Longrightarrow\ \L_{\aa_3}&=&\left( \begin {array}{cccccc} 3&0&0&0&0&0\\ \noalign{\medskip}0&3&0&0&0&0\\ \noalign{\medskip}0&0&3&0&0&0\\ \noalign{\medskip}0&0&0&3&0&0\\ \noalign{\medskip}-1&-1&-1&-1&1&-1\end {array} \right)
    \end{eqnarray*}
    Moreover
    \begin{equation*}
    \resizebox{1\hsize}{!}{$
      \L_\aa=\left( \begin {array}{cccccccccccccccccc} 1&-1&-1&-1&-1&-1&2&0&0&0&0&0&3&0&0&0&0&0\\ \noalign{\medskip}-1&1&-1&-1&-1&-1&0&2&0&0&0&0&0&3&0&0&0&0\\ \noalign{\medskip}0&0&2&0&0&0&-1&-1&1&-1&-1&-1&0&0&3&0&0&0\\ \noalign{\medskip}0&0&0&2&0&0&-1&-1&-1&1&-1&-1&0&0&0&3&0&0\\ \noalign{\medskip}0&0&0&0&2&0&0&0&0&0&2&0&-1&-1&-1&-1&1&-1\end {array} \right)$}
    \end{equation*}
    and observe that there is a lot of sub-matrices $W$ of $\L_\aa$ satisfying assumption (B), but none of them satisfies assumption (C).

    For instance, choose
    \begin{equation*}
      W=\left( \begin {array}{cccccc} 1&-1&0&0&0&0\\ \noalign{\medskip}-1&-1&2&0&0&0\\ \noalign{\medskip}0&0&-1&-1&3&0\\ \noalign{\medskip}0&0&-1&-1&0&3\\ \noalign{\medskip}0&2&0&0&-1&-1\end {array} \right)
    \end{equation*}
    corresponding to columns of $\L_\aa$ indexed by $1, 5, 8, 12, 15, 16$. Then, calling $V$ the fan matrix (\ref{V5}) of $\P^5$, one has
    \begin{equation*}
      V^T\cdot W=\left( \begin {array}{cccccc} 1&-1&0&0&0&0\\ \noalign{\medskip}-1&-1&2&0&0&0\\ \noalign{\medskip}0&0&-1&-1&3&0\\ \noalign{\medskip}0&0&-1&-1&0&3\\ \noalign{\medskip}0&2&0&0&-1&-1\\ \noalign{\medskip}0&0&0&2&-2&-2\end {array} \right)
    \end{equation*}
    so that the partitioned framing of $X$ is given by $\cc=\cc_1+\cc_2+\cc_3$ with
    \begin{equation*}
     \cc_1=(1,1,0,0,0,0)\ ,\quad \cc_2=(0,0,1,1,0,0)\ ,\quad \cc_3=(0,0,0,0,2,2)
    \end{equation*}
    Then, the fan matrix defined by primitive reduction of vertices of the polytope $[\conv(\D_{\cc_1},\D_{\cc_2},\D_{\cc_3})]$ is
    \begin{equation*}
      \left( \begin {array}{cccccccccc} 0&1&-1&-1&0&0&1&2&0&-1\\ \noalign{\medskip}1&0&0&0&0&0&1&0&0&-1\\ \noalign{\medskip}0&0&0&0&0&1&0&0&0&-1\\ \noalign{\medskip}0&0&0&0&1&0&0&0&0&-1\\ \noalign{\medskip}0&0&0&-1&0&0&1&1&1&-1\end {array} \right)\neq V
    \end{equation*}
    and the complete intersection $Y_{2,2,3}\subset\P^5$ turns out to admit only the calibrated $f$-mirror model $Y^\vee_{BB}$.
  \end{example}

  \begin{example}[$Y_{4,5,6}\subset\P^8$]\label{ex:4,5,6 in P8} This is an example for which assumptions (A), (B) and (C) can be all satisfied but isomorphism in display (\ref{G=tors}) in  Remark~\ref{rem:discussione} is never attained.

  The intersection $Y_{4,5,6}\subset\P^8$  corresponds to the partitioned ftv $$(\P^8,\aa=\aa_1+\aa_2+\aa_3)$$
  with
    \begin{equation*}
      \aa_1=(1,1,2,0,0,0,0,0,0)\,,\ \aa_2=(0,0,0,1,1,3,0,0,0)\,,\ \aa_3=(0,0,0,0,0,0,1,1,4)
    \end{equation*}
    Threfore $\L_\aa$ is given by the following $8\times 27$ matrix
    \begin{equation*}
    \resizebox{1\hsize}{!}{$
      \left( \begin {array}{ccccccccccccccccccccccccccc} 3&-1&-1&-1&-1&-1&-1&-1&-1&5&0&0&0&0&0&0&0&0&6&0&0&0&0&0&0&0&0\\ \noalign{\medskip}-1&3&-1&-1&-1&-1&-1&-1&-1&0&5&0&0&0&0&0&0&0&0&6&0&0&0&0&0&0&0\\ \noalign{\medskip}-2&-2&2&-2&-2&-2&-2&-2&-2&0&0&5&0&0&0&0&0&0&0&0&6&0&0&0&0&0&0\\ \noalign{\medskip}0&0&0&4&0&0&0&0&0&-1&-1&-1&4&-1&-1&-1&-1&-1&0&0&0&6&0&0&0&0&0\\ \noalign{\medskip}0&0&0&0&4&0&0&0&0&-1&-1&-1&-1&4&-1&-1&-1&-1&0&0&0&0&6&0&0&0&0\\ \noalign{\medskip}0&0&0&0&0&4&0&0&0&-3&-3&-3&-3&-3&2&-3&-3&-3&0&0&0&0&0&6&0&0&0\\ \noalign{\medskip}0&0&0&0&0&0&4&0&0&0&0&0&0&0&0&5&0&0&-1&-1&-1&-1&-1&-1&5&-1&-1\\ \noalign{\medskip}0&0&0&0&0&0&0&4&0&0&0&0&0&0&0&0&5&0&-1&-1&-1&-1&-1&-1&-1&5&-1\end {array} \right) $}
    \end{equation*}
    Notice that there are lots of submatrices $W$ of $\L_\aa$ satisfying assumption (B), some of them also satisfying assumption (C), but none of them satisfying the condition $W\cdot(\aa')^T=\0_8^T$, for any permutation $\aa'\sim\aa$. In the following we will give three choices of the submatrix $W$ satisfying assumption (B): two of them satisfy also assumption (C) and so give two LT-mirror models. The third choice gives an example of a submatrix $W$ which does not satisfy assumption (C).

    As a first choice of $W$ consider the submatrix of $\L_\aa$ obtained by columns $7, 8, 9, 10$, $11, 12, 22, 23, 24$, that is
    \begin{equation*}
      W_1=\left( \begin {array}{ccccccccc} -1&-1&-1&5&0&0&0&0&0\\ \noalign{\medskip}-1&-1&-1&0&5&0&0&0&0\\ \noalign{\medskip}-2&-2&-2&0&0&5&0&0&0\\ \noalign{\medskip}0&0&0&-1&-1&-1&6&0&0\\ \noalign{\medskip}0&0&0&-1&-1&-1&0&6&0\\ \noalign{\medskip}0&0&0&-3&-3&-3&0&0&6\\ \noalign{\medskip}4&0&0&0&0&0&-1&-1&-1\\ \noalign{\medskip}0&4&0&0&0&0&-1&-1&-1\end {array} \right)
    \end{equation*}
    It is a fan matrix of
    $$X_1=\P(\q')/H_1\cong\P(\aa')/G_1\quad\text{with}\quad\begin{array}{c}
                                                \aa'=(1,1,4,1,1,2,1,1,3)\sim\aa \\
                                                \q'= (5,5,20,6,6,12,4,4,12)
                                              \end{array}$$
    and
    \begin{eqnarray*}
      H_1&\cong&\mu_2^2\times\mu_{60}^2\cong\Tors(\Cl(X_1))\\
      G_1&\cong&\mu_5^3\times\mu_6^3\times\mu_4^3\times H_1
    \end{eqnarray*}
    The action of $H_1$ on $\P(\q')$ is assigned by the torsion matrix
    \begin{equation*}
      T_{W_1}=\left( \begin {array}{ccccccccc} 1_2&0_2&1_2&0_2&0_2&0_2&1_2&1_2&0_2\\ \noalign{\medskip}1_2&0_2&1_2&0_2&0_2&0_2&1_2&0_2&1_2\\ \noalign{\medskip}0_{60}&15_{60}&30_{60}&45_{60}&45_{60}&6_{60}&46_{60}&46_{60}&28_{60}\\ \noalign{\medskip}15_{60}&0_{60}&15_{60}&30_{60}&54_{60}&12_{60}&6_{60}&16_{60}&38_{60}\end {array} \right)
    \end{equation*}
    (think its columns as exponents of corresponding primitive roots of the unity multiplying the same coordinate of $\P(\q')$)
    and that of $G_1$ is compatible with weight actions defining weighted projective spaces involved in the following diagram
     \begin{equation*}
       \xymatrix{&\P^8\ar[dr]\ar[dl]&\\
                \P(\q)\ar[dr]_-{/H_1}&&\P(\aa')\ar[ll]_{/(\mu_5^3\times\mu_6^3\times\mu_4^3)}
                \ar[ld]^{/G_1}\\
                &X_1&}
     \end{equation*}
     Evaluating $V^T\cdot W_1$, where $V=\left(
                                         \begin{array}{ccc}
                                           I_8 & \vline & -\1_8^T \\
                                         \end{array}
                                       \right)$ is the usual fan matrix of $\P^8$,
      the partitioned framing of $X_1$ turns out to be assigned by $\cc=\cc_1+\cc_2+\cc_3$ with
      \begin{equation*}
      \cc_1=(2,2,2,0,0,0,0,0,0)\,,\ \cc_2=(0,0,0,3,3,3,0,0,0)\,,\ \cc_3=(0,0,0,0,0,0,4,4,4)
    \end{equation*}
    Then one has
    \begin{equation*}
      \D_{\cc_1}=\conv\left( \begin {array}{ccccccccc} 2&0&0&1/2&1/2&1/2&1/2&1/2&-1/4\\ \noalign{\medskip}0&2&0&1/2&1/2&1/2&1/2&1/2&-1/4\\ \noalign{\medskip}0&0&1&1/2&1/2&1/2&1/2&1/2&-1/4\\ \noalign{\medskip}0&0&0&5/2&0&0&1/2&1/2&-1/4\\ \noalign{\medskip}0&0&0&0&5/2&0&1/2&1/2&-1/4\\ \noalign{\medskip}0&0&0&0&0&5/6&1/2&1/2&-1/4\\ \noalign{\medskip}0&0&0&0&0&0&3&0&-3/4\\ \noalign{\medskip}0&0&0&0&0&0&0&3&-3/4\end {array} \right)
    \end{equation*}
    \begin{equation*}
      \D_{\cc_2}=\conv\left(\begin {array}{ccccccccc} 9/5&-3/5&-3/5&0&0&0&0&0&-{\frac {9}{10}}\\ \noalign{\medskip}-3/5&9/5&-3/5&0&0&0&0&0&-{\frac {9}{10}}\\ \noalign{\medskip}-3/5&-3/5&3/5&0&0&0&0&0&-{\frac {9}{10}}\\ \noalign{\medskip}0&0&0&3&0&0&3/5&3/5&-3/10\\ \noalign{\medskip}0&0&0&0&3&0&3/5&3/5&-3/10\\ \noalign{\medskip}0&0&0&0&0&1&3/5&3/5&-3/10\\ \noalign{\medskip}0&0&0&0&0&0&{\frac {18}{5}}&0&-{\frac {9}{10}}\\ \noalign{\medskip}0&0&0&0&0&0&0&{\frac {18}{5}}&-{\frac {9}{10}}\end {array} \right)
    \end{equation*}
    \begin{equation*}
      \D_{\cc_3}=\conv\left(\begin {array}{ccccccccc} 2&-2/3&-2/3&0&0&0&0&0&-1\\ \noalign{\medskip}-2/3&2&-2/3&0&0&0&0&0&-1\\ \noalign{\medskip}-2/3&-2/3&2/3&0&0&0&0&0&-1\\ \noalign{\medskip}-2/3&-2/3&-2/3&8/3&-2/3&-2/3&0&0&-1\\ \noalign{\medskip}-2/3&-2/3&-2/3&-2/3&8/3&-2/3&0&0&-1\\ \noalign{\medskip}-2/3&-2/3&-2/3&-2/3&-2/3&4/9&0&0&-1\\ \noalign{\medskip}0&0&0&0&0&0&4&0&-1\\ \noalign{\medskip}0&0&0&0&0&0&0&4&-1\end {array} \right)
    \end{equation*}
    so that the fan matrix defined by primitive reduction of vertices of the polytope
    $$[\conv(\D_{\cc_1},\D_{\cc_3},\D_{\cc_3}]$$
    is, up to a columns permutation, the matrix $V$,
    so giving that $(\P^8,\aa_1+\aa_2+\aa_3)$ is an $f$-mirror of $(X_1,\cc_1+\cc_2+\cc_3)$.

    Furthermore one has all the intermediate mirrors determined by the choice of subsets $A\subset\cI^{W_1}$: we than obtain $2^{18}= 262\,144$ mirror models of $Y_{4,5,6}\subset\P^8$ (this time they will not be listed into some appendix!).

    In terms of defining polynomials, the $f$-calibrated mirror $Y^\vee_{BB}=Y^\vee_\emptyset$ is the complete intersection $\CV(p_{1,\psi},p_{2,\psi},p_{3,\psi})\subset\XX_\aa$ with
    \begin{eqnarray*}
      p_{1,\psi} &=& \,x_{{1}}{x_{{2}}}^{5}x_{{4}}x_{{5}}x_{{6}}x_{{7}}x_{{8}}x_{{9}}
\mbox{}{x_{{11}}}^{5}{x_{{20}}}^{6}+{x_{{3}}}^{3}{x_{{12}}}^{5}{x_{{21}}}^{6}
\mbox{}+{x_{{1}}}^{5}x_{{2}}x_{{4}}x_{{5}}x_{{6}}x_{{7}}x_{{8}}x_{{9}}\\
&&+\psi {x_{{1}}}^{2}{x_{{2}}}^{2}x_{{3}}{x_{{4}}}^{2}{x_{{5}}}^{2}{x_{{6}}}^{2}{x_{{7}}}^{2}
\mbox{}{x_{{8}}}^{2}{x_{{9}}}^{2}
\mbox{}{x_{{10}}}^{5}{x_{{19}}}^{6} \\
      p_{2,\psi} &=& {x_{{5}}}^{4}{x_{{10}}}^{2}{x_{{11}}}^{2}{x_{{12}}}^{2}{x_{{13}}}^{2}
\mbox{}{x_{{14}}}^{7}{x_{{16}}}^{2}{x_{{17}}}^{2}{x_{{18}}}^{2}{x_{{23}}}^{6}+{x_{{6}}}^{4}{x_{{15}}}^{3}{x_{{24}}}^{6}\\
\mbox{}&&+{x_{{4}}}^{4}{x_{{10}}}^{2}{x_{{11}}}^{2}{x_{{12}}}^{2}{x_{{13}}}^{7}{x_{{14}}}^{2}
\mbox{}{x_{{16}}}^{2}{x_{{17}}}^{2}{x_{{18}}}^{2}{x_{{22}}}^{6}+\psi {x_{{10}}}^{3}{x_{{11}}}^{3}{x_{{12}}}^{3}{x_{{13}}}^{3}{x_{{14}}}^{3}
\mbox{}x_{{15}}{x_{{16}}}^{3}{x_{{17}}}^{3}{x_{{18}}}^{3}  \\
      p_{3,\psi} &=& {x_{{7}}}^{4}{x_{{16}}}^{5}{x_{{19}}}^{3}{x_{{20}}}^{3}{x_{{21}}}^{3}
\mbox{}{x_{{22}}}^{3}{x_{{23}}}^{3}{x_{{24}}}^{3}{x_{{25}}}^{9}{x_{{26}}}^{3}
+{x_{{8}}}^{4}{x_{{17}}}^{5}{x_{{19}}}^{3}{x_{{20}}}^{3}{x_{{21}}}^{3}{x_{{22}}}^{3}
\mbox{}{x_{{23}}}^{3}{x_{{24}}}^{3}{x_{{25}}}^{3}{x_{{26}}}^{9}\\
\mbox{}&&+{x_{{9}}}^{4}{x_{{18}}}^{5}{x_{{27}}}^{3}+\psi{x_{{19}}}^{4}{x_{{20}}}^{4}{x_{{21}}}^{4}{x_{{22}}}^{4}{x_{{23}}}^{4}
\mbox{}{x_{{24}}}^{4}{x_{{25}}}^{4}{x_{{26}}}^{4}x_{{27}}
    \end{eqnarray*}
    which are polynomials in $\Cox(\XX_\aa)$ of degree
    \begin{eqnarray*}
      \deg(p_{1,\psi}) &=& [2, 2, 1, 2, 2, 2, 2, 2, 2, 0, 0, 0, 0, 0, 0, 0, 0, 0, 0] \\
      \deg(p_{2,\psi}) &=& [36, 36, 36, 48, 54, 46, 54, 54, 54, 18, 18, 18, 9, 18, 6, 18, 18, 18, 0] \\
      \deg(p_{3,\psi}) &=& [0,0,0,0,0,0,13,0,12,0,0,0,0,0,0,20,0,15,12]
    \end{eqnarray*}
    with respect to the grading defined by $\Cl(\XX_\aa)$\,.

    On the other hand, the LT-mirror $Y^\vee_{LT,1}=Y^\vee_{\cI^{W_1}}\subset X_1$ is the complete intersection $\CV(q'_{1,\psi},q'_{2,\psi},q'_{3,\psi})\subset X_1$ with
    \begin{eqnarray*}
      q'_{1,\psi} &=& x_{{1}}x_{{2}}x_{{3}}{x_{{4}}}^{5}+x_{{1}}x_{{2}}x_{{3}}{x_{{5}}}^{5}
\mbox{}+{x_{{6}}}^{5}+\psi{x_{{1}}}^{2}{x_{{2}}}^{2}{x_{{3}}}^{2}\\
      q'_{2,\psi} &=& {x_{{4}}}^{2}{x_{{5}}}^{2}{x_{{6}}}^{2}{x_{{7}}}^{6}
\mbox{}+{x_{{9}}}^{6}+{x_{{4}}}^{2}{x_{{5}}}^{2}{x_{{6}}}^{2}{x_{{8}}}^{6}+\psi{x_{{4}}}^{3}{x_{{5}}}^{3}{x_{{6}}}^{3} \\
      q'_{3,\psi} &=& {x_{{1}}}^{4}{x_{{7}}}^{3}{x_{{8}}}^{3}{x_{{9}}}^{3}+{x_{{2}}}^{4}{x_{{7}}}^{3}{x_{{8}}}^{3}{x_{{9}}}^{3}
\mbox{}+{x_{{3}}}^{4}+\psi{x_{{7}}}^{4}{x_{{8}}}^{4}{x_{{9}}}^{4}
    \end{eqnarray*}
    which are polynomials in $\Cox(X_1)$ of degree 60, 72 and 80, respectively, with respect to the grading defined by $\Cl(X_1)$.

    As a second choice of $W$ consider the submatrix of $\L_\aa$ obtained by columns $4, 5, 6, 16, 17, 18, 19, 20, 21$, that is
    \begin{equation*}
      W_2=\left( \begin {array}{ccccccccc} -1&-1&-1&0&0&0&6&0&0\\ \noalign{\medskip}-1&-1&-1&0&0&0&0&6&0\\ \noalign{\medskip}-2&-2&-2&0&0&0&0&0&6\\ \noalign{\medskip}4&0&0&-1&-1&-1&0&0&0\\ \noalign{\medskip}0&4&0&-1&-1&-1&0&0&0\\ \noalign{\medskip}0&0&4&-3&-3&-3&0&0&0\\ \noalign{\medskip}0&0&0&5&0&0&-1&-1&-1\\ \noalign{\medskip}0&0&0&0&5&0&-1&-1&-1\end {array} \right)
    \end{equation*}
    It is a fan matrix of
    $$X_2=\P(\q'')/H_2\cong\P(\aa'')/G_2 \quad\text{with}\quad\begin{array}{c}
                                                \aa''=(1,1,3,1,1,4,1,1,2)\sim\aa \\
                                                \q''= (6,6,18,4,4,16,5,5,10)
                                              \end{array}$$
                                              and
    \begin{eqnarray*}
      H_2&\cong&\mu_2^2\times\mu_{60}^2\cong\Tors(\Cl(X_2))\cong H_1\\
      G_2&\cong&\mu_6^3\times\mu_4^3\times\mu_5^3\times H_1\cong G_1
    \end{eqnarray*}
    The action of $H_2$ is assigned by the torsion matrix
    \begin{equation*}
      T_{W_2}=\left( \begin {array}{ccccccccc} 1_2&1_2&0_2&1_2&1_2&0_2&0_2&0_2&1_2\\ \noalign{\medskip}1_2&0_2&1_2&1_2&1_2&0_2&1_2&0_2&0_2\\ \noalign{\medskip}15_{60}&15_{60}&30_{60}&44_{60}&20_{60}&56_{60}&0_{60}&30_{60}&10_{60}\\ \noalign{\medskip}15_{60}&30_{60}&45_{60}&34_{60}&34_{60}&52_{60}&45_{60}&25_{60}&40_{60}\end {array} \right)
    \end{equation*}
    and that of $G_2$ is compatible with weight actions defining weighted projective spaces involved in the following diagram
     \begin{equation*}
       \xymatrix{&\P^8\ar[dr]\ar[dl]&\\
                \P(\q'')\ar[dr]_-{/H_2}&&\P(\aa'')\ar[ll]_{/(\mu_6^3\times\mu_4^3\times\mu_5^3)}
                \ar[ld]^{/G_2}\\
                &X_2&}
                \end{equation*}

    \begin{claim}
      $X_1$ and $X_2$ are non-isomorphic $\Q$-factorial, complete, toric varieties.
    \end{claim}
    In fact, $X_1\cong X_2$ if and only if there exist a uni-modular matrix $A\in \SL(8,\Z)$ and a permutation matrix $B\in\SL(9,\Z)$ such that
    $W_2=A\cdot W_1\cdot B$\,. By adding a bottom row $(0,\ldots,0,1)$ to $W_1,W_2$ and $A$ and a right column $(0,\ldots,0,1)^T$ to $A$, one has the matricial equation
    \begin{equation*}
       \left(
                      \begin{array}{ccc}
                         &  &  \\
                         & W_2 &  \\
                         &  &  \\
                         \hline
                        \0_8 & \vline & 1 \\
                      \end{array}
                    \right) = \left(
         \begin{array}{ccccc}
            &  &  & \vline& \\
            & A &  & \vline&\0_8^T \\
            &  &  & \vline & \\
            \hline
            & \0_8 &  & \vline& 1 \\
         \end{array}
       \right)\cdot \left(
                      \begin{array}{ccc}
                         &  &  \\
                         & W_1 &  \\
                         &  &  \\
                         \hline
                        \0_8 & \vline & 1 \\
                      \end{array}
                    \right)\cdot B
    \end{equation*}
   \begin{equation*}
     \Longrightarrow\quad \left|\det \left(
                      \begin{array}{ccc}
                         &  &  \\
                         & W_1 &  \\
                         &  &  \\
                         \hline
                        \0_8 & \vline & 1 \\
                      \end{array}
                    \right)\right| =\left|\det \left(
                      \begin{array}{ccc}
                         &  &  \\
                         & W_2 &  \\
                         &  &  \\
                         \hline
                        \0_8 & \vline & 1 \\
                      \end{array}
                    \right)\right|
   \end{equation*}
   giving an absurd as $172\,800\neq 144\,000$\,.

   Anyway notice that both $X_1$ and $X_2$ have anti-canonical class
   $$[-K_{X_i}]=[74]\in\Cl(X_i)\ ,\quad i=1,2$$
   because $|\q'|=74=|\q''|$\,.

   Evaluating $V^T\cdot W_2$, the partitioned framing of $X_2$ turns out to be the same of $X_1$ so that
    \begin{equation*}
      \D_{\cc_1}=\conv\left( \begin {array}{ccccccccc} 2&0&0&1/2&1/2&-1/4&0&0&0\\ \noalign{\medskip}0&2&0&1/2&1/2&-1/4&0&0&0\\ \noalign{\medskip}0&0&1&1/2&1/2&-1/4&0&0&0\\ \noalign{\medskip}0&0&0&0&0&-3/4&2&-1/2&-1/2\\ \noalign{\medskip}0&0&0&0&0&-3/4&-1/2&2&-1/2\\ \noalign{\medskip}0&0&0&0&0&-3/4&-1/2&-1/2&1/3\\ \noalign{\medskip}0&0&0&3&0&-3/4&0&0&0\\ \noalign{\medskip}0&0&0&0&3&-3/4&0&0&0\end {array} \right)
    \end{equation*}
    \begin{equation*}
      \D_{\cc_2}=\conv\left(\begin {array}{ccccccccc} {\frac {12}{5}}&0&0&3/5&3/5&-3/10&0&0&0\\ \noalign{\medskip}0&{\frac {12}{5}}&0&3/5&3/5&-3/10&0&0&0\\ \noalign{\medskip}0&0&6/5&3/5&3/5&-3/10&0&0&0\\ \noalign{\medskip}3/5&3/5&3/5&3/5&3/5&-3/10&3&0&0\\ \noalign{\medskip}3/5&3/5&3/5&3/5&3/5&-3/10&0&3&0\\ \noalign{\medskip}3/5&3/5&3/5&3/5&3/5&-3/10&0&0&1\\ \noalign{\medskip}0&0&0&{\frac {18}{5}}&0&-{\frac {9}{10}}&0&0&0\\ \noalign{\medskip}0&0&0&0&{\frac {18}{5}}&-{\frac {9}{10}}&0&0&0\end {array} \right)
    \end{equation*}
    \begin{equation*}
      \D_{\cc_3}=\conv\left(\begin {array}{ccccccccc} 2&-2/3&-2/3&0&0&-1&-2/3&-2/3&-2/3\\ \noalign{\medskip}-2/3&2&-2/3&0&0&-1&-2/3&-2/3&-2/3\\ \noalign{\medskip}-2/3&-2/3&2/3&0&0&-1&-2/3&-2/3&-2/3\\ \noalign{\medskip}0&0&0&0&0&-1&8/3&-2/3&-2/3\\ \noalign{\medskip}0&0&0&0&0&-1&-2/3&8/3&-2/3\\ \noalign{\medskip}0&0&0&0&0&-1&-2/3&-2/3&4/9\\ \noalign{\medskip}0&0&0&4&0&-1&0&0&0\\ \noalign{\medskip}0&0&0&0&4&-1&0&0&0\end {array} \right)
    \end{equation*}
    and the fan matrix defined by primitive reduction of vertices of the polytope
    $$[\conv(\D_{\cc_1},\D_{\cc_3},\D_{\cc_3}]$$
    is, up to a columns permutation, the matrix $V$,
    so giving that $(\P^8,\aa_1+\aa_2+\aa_3)$ is a $f$-mirror of $(X_2,\cc_1+\cc_2+\cc_3)$.

    In terms of defining polynomials, the LT-mirror $Y^\vee_{LT,2}=Y^\vee_{\cI^{W_2}}\subset X_2$ is the complete intersection $\CV(q''_{1,\psi},q''_{2,\psi},q''_{3,\psi})\subset X_2$ with
    \begin{eqnarray*}
      q''_{1,\psi} &=& x_{{1}}x_{{2}}x_{{3}}{x_{{7}}}^{6}+x_{{1}}x_{{2}}x_{{3}}{x_{{8}}}^{6}+{x_{{9}}}^{6}
\mbox{}+\psi{x_{{1}}}^{2}{x_{{2}}}^{2}{x_{{3}}}^{2}\\
      q''_{2,\psi} &=& {x_{{1}}}^{4}{x_{{4}}}^{2}{x_{{5}}}^{2}{x_{{6}}}^{2}+{x_{{2}}}^{4}{x_{{4}}}^{2}{x_{{5}}}^{2}{x_{{6}}}^{2}
      +{x_{{3}}}^{4}+\psi{x_{{4}}}^{3}{x_{{5}}}^{3}{x_{{6}}}^{3} \\
      q''_{3,\psi} &=& {x_{{4}}}^{5}{x_{{7}}}^{3}{x_{{8}}}^{3}{x_{{9}}}^{3}+{x_{{5}}}^{5}{x_{{7}}}^{3}{x_{{8}}}^{3}{x_{{9}}}^{3}+{x_{{6}}}^{5} +\psi{x_{{7}}}^{4}{x_{{8}}}^{4}{x_{{9}}}^{4}
    \end{eqnarray*}
    which turn out to be still polynomials of degree 60, 72 and 80, respectively, in $\Cox(X_2)$, with respect to the grading defined by $\Cl(X_2)$.

    Furthermore there are all the intermediate mirrors determined by the choice of subsets $A\subset\cI^{W_2}$: we than obtain further $2^{18}= 262\,144$ mirror models, that is, at least $2^{19}-1= 524\,287$ mirror models of $Y_{4,5,6}\subset\P^8$ by taking into account also the mirror models coming from the previous starting LT-mirror model and considering the repetition of $Y^\vee_{BB}=Y^\vee_\emptyset$.

    As a third choice of $W$ consider the submatrix of $\L_\aa$ obtained by columns $7, 8, 9, 12, 14, 15, 19, 20, 22$, that is
    \begin{equation*}
      W_3=\left( \begin {array}{ccccccccc} -1&-1&-1&0&0&0&6&0&0\\ \noalign{\medskip}-1&-1&-1&0&0&0&0&6&0\\ \noalign{\medskip}-2&-2&-2&5&0&0&0&0&0\\ \noalign{\medskip}0&0&0&-1&-1&-1&0&0&6\\ \noalign{\medskip}0&0&0&-1&4&-1&0&0&0\\ \noalign{\medskip}0&0&0&-3&-3&2&0&0&0\\ \noalign{\medskip}4&0&0&0&0&0&-1&-1&-1\\ \noalign{\medskip}0&4&0&0&0&0&-1&-1&-1\end {array} \right)
    \end{equation*}
    It is a fan matrix of
    $$X_3=\P(\q''')/H_3\cong\P(\aa''')/G_3 \quad\text{with}\quad\begin{array}{c}
                                                \aa'''=(1,1,4,1,1,3,1,1,2)\sim\aa \\
                                                \q'''= (5,5,20,12,12,36,5,5,10)
                                              \end{array}$$
                                              and
    \begin{eqnarray*}
      H_3&\cong&\mu_2\times\mu_{12}\times\mu_{120}\cong\Tors(\Cl(X_3))\\
      G_3&\cong&\mu_5^3\times\mu_12^3\times\mu_5^3\times H_3
    \end{eqnarray*}
    so that assumtion (B) turns out to be satisfied by this further choice.

    Evaluating $V^T\cdot W_3$, the partitioned framing of $X_3$ is now different, namely
    \begin{equation*}
      \cc_1=(2,2,2,0,0,0,0,0,0)\,,\ \cc_2=(0,0,0,3,3,1,0,0,0)\,,\ \cc_3=(0,0,0,0,0,0,4,4,4)
    \end{equation*}
    It suffices to observe that
    \begin{equation*}
      \D_{\cc_1}=\conv\left( \begin {array}{ccccccccc} 2&0&0&0&0&0&1/2&1/2&-1/4\\ \noalign{\medskip}0&2&0&0&0&0&1/2&1/2&-1/4\\ \noalign{\medskip}0&0&1&1&1&1&1/2&1/2&-1/4\\ \noalign{\medskip}0&0&0&1&0&0&1/2&1/2&-1/4\\ \noalign{\medskip}0&0&0&1&2&1&1/2&1/2&-1/4\\ \noalign{\medskip}0&0&0&1&1&4/3&1/2&1/2&-1/4\\ \noalign{\medskip}0&0&0&0&0&0&3&0&-3/4\\ \noalign{\medskip}0&0&0&0&0&0&0&3&-3/4\end {array} \right)
    \end{equation*}
    to conclude that fourth and fifth vertices of this polytope contribute two columns of the fan matrix coming from $[\D_{\cc_1},\D_{\cc_2},\D_{\cc_3}]$ by primitive reduction of vertices, which are not columns of $V$. Hence the $f$-mirror partner of $(X_3,\cc_1+\cc_2+\cc_3)$ cannot be $(\P^8,\aa_1+\aa_2+\aa_3)$,  that is, assumption (C) is not satisfied.
  \end{example}

  \begin{remark}
   In principle, the construction producing either $X_1$ or $X_2$ into the previous Example~\ref{ex:4,5,6 in P8}, can be repeated for every projective complete intersection $Y_{d_1,\ldots,d_l}\subset\P^n$ of nonnegative Kodaira dimension, satisfying assumptions (A), (B) and (C), so getting at least $2^{(l-1)(n+1)}$ mirror models, as stated in item (c) in the Introduction.
  \end{remark}

\subsection{$K$-equivalence vs $D$-equivalence}\label{ssez:KvsD} About $K$-equivalence,  Theorem~\ref{thm:Kequivalenzad} cannot be generalized in the same shape as, in the general case, it is no more possible to guarantee that the restriction to strict transforms $\widehat{Y}^\vee_{A}$ and $\widehat{Y}^\vee_{A'}$ of resolutions $\psi_\Si$ and $\psi_{\Si'}$, induced by the choice of fans $\Si\in\PSF(\widehat{\L}^A)$ and $\Si'\in\PSF(\widehat{\L}^{A'})$, are crepant morphisms: in general they give divisorial blowups of $Y^\vee_A$ and $Y^\vee_{A'}$, respectively. Then, we necessarily have to restrict our considerations to the resolution level given by fans in $\PSF(\widehat{\L}^{A\cap A'})$, to get the following

   \begin{theorem}[$K$-equivalence]\label{thm:Kequiv-generale}
  Let $A,A'$ be two subsets of $\cI^W$ and let $\widehat{\L}^{A\cap A'}$ be the matrix whose columns are given by all the primitive lattice points contained in $\conv(\L^{A\cap A'})\setminus\{\0\}$.  Consider any choice $$\Si\in\PSF(\widehat{\L}^{A\cap A'})\ ,\quad\Si'\in\PSF(\widehat{\L}^{A\cap A'})$$
  giving (possibly partial) desingularizations $$\widehat{Y}^\vee_{A}=(\psi_\Si)^{-1}_*(Y^\vee_{A})\to Y^\vee_{A}\ ,\quad\widehat{Y}^\vee_{A'}=(\psi_{\Si'})^{-1}_*(Y^\vee_{A'})\to Y^\vee_{A'}$$
  Then $\widehat{Y}^\vee_{A}$ and $\widehat{Y}^\vee_{A'}$  are $K$-equivalent.
   \end{theorem}

   \begin{proof}
     Starting with the LT-mirror model $Y^\vee_{LT}\subset X$ and the choice of subsets $A$ and $A'$ in $\cI^W$, consider the blowups
  \begin{equation*}
    \xymatrix{\XX^A\ar[dr]_-{\phi^A}&&\XX^{A'}\ar[dl]^-{\phi^{A'}}\\
                    &X&}
  \end{equation*}
  The choice of $\Si,\Si'\in\PSF(\widehat{\L}^{A\cap A'})$ gives two further birational morphisms
  \begin{equation*}
    \xymatrix{\psi_\Si:\widehat{\XX}^{A\cap A'}(\Si)\ar[r]&\XX^A}\ ,\ \xymatrix{\psi_{\Si'}:\widehat{\XX}^{A\cap A'}(\Si')\ar[r]&\XX^{A'}}
  \end{equation*}
  which, composed with the previous ones, give a commutative diagram of birational maps between toric varieties descending to give an analogous diagram between embedded mirror models
  \begin{equation*}
    \xymatrix{\widehat{\XX}^{A\cap A'}(\Si)\ar@{-->}[rr]^-{\vf}_-{\cong}
    \ar[dr]_-{\phi^A\circ\psi_\Si}&&\widehat{\XX}^{A\cap A'}(\Si')
    \ar[dl]^-{\phi^{A'}\circ\psi_{\Si'}}\\
                    &X&}\quad \Longrightarrow\quad
                    \xymatrix{\widehat{Y}^\vee_A\ar@{-->}[rr]^-{\vf}_-{\cong}\ar[dr]_-{\phi^A\circ\psi_\Si}&&\widehat{Y}^\vee_{A'}
    \ar[dl]^-{\phi^{A'}\circ\psi_{\Si'}}\\
                    &Y^\vee_{LT}&}
  \end{equation*}
  The proof now goes on by decomposing $\vf$ in a finite sequence of wall-crossings, as in the proof of Theorem~\ref{thm:Kequivalenza3}, so getting a dominant toric variety $\widehat{\XX}$ with birational morphisms $f$ and $g$
  \begin{equation*}
    \xymatrix{&\widehat{\XX}\ar[dr]^g\ar[dl]_f&\\
               \widehat{\XX}^{A\cap A'}(\Si)\ar@{-->}[rr]^-{\vf}_-{\cong}&&
                \widehat{\XX}^{A\cap A'}(\Si')}
  \end{equation*}
  Consider the strict transform $Z=f^{-1}_*(\widehat{Y}^\vee_A)= g^{-1}_*(\widehat{Y}^\vee_{A'})$ and restricted morphisms $f|_Z$ and $g|_Z$ so getting
     \begin{equation*}
    \xymatrix{&Z\ar[dl]_{f}\ar[dr]^{g}&\\
                \widehat{Y}^\vee_A\ar@{-->}[rr]^-{\vf}_-{\cong}
                &&\widehat{Y}^\vee_{A'}
                    }
  \end{equation*}
  Then, as explained in proving Theorem~\ref{thm:smallK&D-equivalenza} and recalling Remark~\ref{rem:triangolo},
  $$f^*K_{\widehat{Y}^\vee_A}\sim_\Q g^*K_{\widehat{Y}^\vee_{A'}}$$
   \end{proof}

   Same considerations driving to state Conjectures~\ref{conj:K<=>D,d} and \ref{conj:sing'}, give the following one, implying a clear restatement of Conjecture~\ref{conj:sing} in the present context.

   \begin{conjecture}[$D$-equivalence]\label{conj:D}
     Under the same hypotheses of previous Theo\-rem~\ref{thm:Kequiv-generale}, let $\cY_A$ and $\cY_{A'}$ be the canonical covering stacks of $\widehat{Y}^\vee_{A}$ and $\widehat{Y}^\vee_{A'}$, respectively. Then there exist equivalences of triangulated categories
  \begin{eqnarray*}
    \cD_{sg}(\widehat{Y}^\vee_A)&\cong&\cD_{sg}(\widehat{Y}^\vee_{A'})\\
    \cD^b(\cY_A)&\cong&\cD^b(\cY_{A'})
  \end{eqnarray*}
  In particular, if both $\widehat{Y}^\vee_{A}$ and $\widehat{Y}^\vee_{A'}$ are smooth, then there is an equivalence of triangulated categories $\cD^b(\widehat{Y}^\vee_{A})\cong\cD^b(\widehat{Y}^\vee_{A'})$\,.
   \end{conjecture}

   If $\dim Y=3$, that is $n=3+l$, then such a conjecture partially holds. Namely

   \begin{theorem}\label{thm:D}
     Under the same hypotheses of previous Conjecture~\ref{conj:D}, with the further assumption that $\dim Y=3$, then there exists an equivalence of triangulated categories
  \begin{equation*}
    \cD^b(\cY_A)\cong\cD^b(\cY_{A'})
  \end{equation*}
  In particular, if both $\widehat{Y}^\vee_{A}$ and $\widehat{Y}^\vee_{A'}$ are smooth, then there is an equivalence of triangulated categories $\cD^b(\widehat{Y}^\vee_{A})\cong\cD^b(\widehat{Y}^\vee_{A'})$\,.
   \end{theorem}

   \begin{proof}
     Resolutions $\psi_\Si:\widehat{\XX}_\aa^{A\cap A'}(\Si)\to \XX_\aa^A$ and $\psi_{\Si'}:\widehat{\XX}_\aa^{A\cap A'}(\Si')\to \XX_\aa^{A'}$ are obtained as compositions of blowups and small resolutions as follows
   \begin{equation*}
     \xymatrix{\widehat{\XX}_\aa^{A\cap A'}(\Si)\ar[dr]_-{\check{\psi}_\Si}^-{\text{small}}\ar[rr]^-{\psi_\Si}&&\XX_\aa^A\\
                        &\XX_\aa^{A\cap A'}\ar[ur]_-{\phi^{A\cap A'}_A}^-{\text{blowup}}&}\quad,\quad\xymatrix{\widehat{\XX}_\aa^{A\cap A'}(\Si')\ar[dr]_-{\check{\psi}_{\Si'}}^-{\text{small}}\ar[rr]^-{\psi_{\Si'}}
                        &&\XX_\aa^{A'}\\
                        &\XX_\aa^{A\cap A'}\ar[ur]_-{\phi^{A\cap A'}_{A'}}^-{\text{blowup}}&}
   \end{equation*}
   Then we have crepant birational morphisms
\begin{equation*}
  \xymatrix{\widehat{Y}^\vee_{A}\ar[dr]_-{\check{\psi}_\Si}&&
  \widehat{Y}^\vee_{A'}\ar[dl]^-{\check{\psi}_{\Si'}}\\
                &Y^\vee_{A\cap A'}&}
\end{equation*}
giving rise to a crepant birational map $f:\widehat{Y}^\vee_{A}\dashrightarrow\widehat{Y}^\vee_{A'}$ between 3-dimensional projective varieties admitting at most canonical singularities, by \cite[Prop.~2.2.2,\,2.2.4]{Batyrev94}: the variety $Z$ in Definition~\ref{def:crepante} is given by $\widehat{Y}^\vee_{A}\times_{Y^\vee_{A\cap A'}}\widehat{Y}^\vee_{A'}$. Then, by \cite[Thm.~4.6]{Kawamata}, $f$ decomposes into a sequence of flops  and the statement follows immediately by applying \cite[Thm.~6.5]{Kawamata}.
   \end{proof}

   \begin{remark}
     The conjectured equivalence $\cD_{sg}(\widehat{Y}^\vee_A)\cong\cD_{sg}(\widehat{Y}^\vee_{A'})$ could be proved by means of Malter's techniques when the ideal containment condition $\cI\subseteq\sqrt{\partial w, \mathcal{J}}$ is guaranteed: unfortunately this  cannot be done in general (see Remark~4.9 in \cite{Malter}).
   \end{remark}

   \begin{remark}\label{rem:noCY}
     The mirror models construction given in \S~\ref{ssez:mirrors} and Theorems~\ref{thm:Kequiv-generale} and \ref{thm:D} give a specialization of Theorem~\ref{thm:metateorema} to a certain level of resolution (determined by $A\cap A'$) of the involved mirror models. In particular, the \cy assumption in item (iii) can be dropped if a suitable level of resolution is considered. The latter gives, consequently, an analogous specialization of the mirror theorem~\ref{thm:mirror}.
   \end{remark}

   \section{Multiple mirrors of the \cy complete intersection $Y_{2,2,3}\subset\P^6$}\label{sez:2,2,3}
We are going to end up the present paper by giving an application to the family of projective \cy threefolds obtained as the complete intersection of two hyperquadrics and a cubic hypersurface in $\P^6$\,.

To describe this family we consider the partitioned ftv $(\P^6,\aa=\aa_1+\aa_2+\aa_3)$ with
\[
\aa_1=(1,1,0,0,0,0,0)\ ,\quad \aa_2=(0,0,1,1,0,0,0)\ ,\quad \aa_3=(0,0,0,0,1,1,1)
\]
Batyrev-Borisov duality, that is the associated calibrated $f$-process, gives a first mirror family by means of the $f$-dual partioned ftv $(\XX_\aa,\bb=\bb_1+\bb_2+\bb_3)$ where
\begin{itemize}
    \item by Theorem \ref{thm:BB-general}, $\XX_\aa$ is the complete toric variety spanned by the polytope $\conv(\D_{\aa_1},\D_{\aa_2},\D_{\aa_3})$  whose fan matrix is given by
    \begin{equation*}
    \resizebox{1\hsize}{!}{$
        \L_\aa=\left(\begin{array}{ccccccccccccccccccccc}
1 & -1 & -1 & -1 & -1 & -1 & -1 & 2 & 0 & 0 & 0 & 0 & 0 & 0 & 3 & 0 & 0 & 0 & 0 & 0 & 0
\\
 -1 & 1 & -1 & -1 & -1 & -1 & -1 & 0 & 2 & 0 & 0 & 0 & 0 & 0 & 0 & 3 & 0 & 0 & 0 & 0 & 0
\\
 0 & 0 & 2 & 0 & 0 & 0 & 0 & -1 & -1 & 1 & -1 & -1 & -1 & -1 & 0 & 0 & 3 & 0 & 0 & 0 & 0
\\
 0 & 0 & 0 & 2 & 0 & 0 & 0 & -1 & -1 & -1 & 1 & -1 & -1 & -1 & 0 & 0 & 0 & 3 & 0 & 0 & 0
\\
 0 & 0 & 0 & 0 & 2 & 0 & 0 & 0 & 0 & 0 & 0 & 2 & 0 & 0 & -1 & -1 & -1 & -1 & 2 & -1 & -1
\\
 0 & 0 & 0 & 0 & 0 & 2 & 0 & 0 & 0 & 0 & 0 & 0 & 2 & 0 & -1 & -1 & -1 & -1 & -1 & 2 & -1
\end{array}\right) $}
    \end{equation*}
    \item the framing $\bb=\sum\bb_k$ is given by
    \begin{equation*}
        \bb_1=(\1_7,\0_{14})\ ,\quad \bb_2=(\0_7,\1_7,\0_7)\ ,\quad \bb_3=(\0_{14},\1_7)
    \end{equation*}
\end{itemize}
\subsection{Batyrev-Borisov mirror family}\label{ssez:BB_2,2,3} The mirror family $Y^\vee_{BB}$ is obtained as the complete intersection of three hypersurfaces of $\XX_\aa$, assigned by the following polynomials of the Cox ring $\Cox(\XX_\aa)$
\begin{eqnarray*}
    p_{1,\psi}&=&x_{1}^{2} x_{8}^{2} x_{15}^{3}+x_{2}^{2} x_{9}^{2} x_{16}^{3}+\psi x_{1} x_{2} x_{3} x_{4} x_{5} x_{6} x_{7}\\
    p_{2,\psi}&=&x_{3}^{2} x_{10}^{2} x_{17}^{3}+x_{4}^{2} x_{11}^{2} x_{18}^{3}+\psi x_{8} x_{9} x_{10} x_{11} x_{12} x_{13} x_{14}\\
    p_{3,\psi}&=& x_{5}^{2} x_{12}^{2} x_{19}^{3}+x_{6}^{2} x_{13}^{2} x_{20}^{3}+x_{7}^{2} x_{14}^{2} x_{21}^{3}+\psi x_{15} x_{16} x_{17} x_{18} x_{19} x_{20} x_{21}
\end{eqnarray*}
of degree
\begin{eqnarray*}
    \deg( p_{1,\psi})&=& \left[4,4,4,4,4,2,3,2,6,0,0,0,0,0,0\right]\\
    \deg( p_{2,\psi})&=& \left[21,21,23,24,24,15,18,12,36,2,6,3,3,3,0\right]\\
    \deg( p_{3,\psi})&=& \left[26,26,26,26,28,16,26,16,50,0,2,2,6,6,3\right]
\end{eqnarray*}
with respect to the grading determined by $\Cl(\XX_\aa)$\,.
Notice that the generic element $Y^\vee_{BB}$ is quasi-smooth, as critical points of
$$\pp_\psi=(p_{1,\psi},p_{2,\psi},p_{3,\psi})$$
are all contained in the unstable locus $Z_{\L_\aa}$\,. Moreover
\begin{equation*}
    \sum_{i=1}^3\deg(p_{i,\psi})=[51, 51, 53, 54, 56, 33, 47, 30, 92, 2, 8, 5, 9, 9, 3]=[-K_{\XX_\aa}]\in\Cl(\XX_\aa)
\end{equation*}
so that, up to a desingularization, $Y^\vee_{BB}$ is a family of \cy threefolds.

\subsection{Libgober-Teitelbaum mirror families}\label{ssez:LTmirrors} There are 42 possible choices of submatrices $W$ of $\L_\aa$ satisfying assumptions (B) and (C), whose columns are indexed by the following lists
\begin{eqnarray*}
    &[3,4,12,13,15,16,21],[3,4,12,14,15,16,20],[3,4,13,14,15,16,19],[3,5,8,13,16,18,21],&\\
    &[3,5,8,14,16,18,20],[3,5,9,13,15,18,21],[3,5,9,14,15,18,20],[3,5,13,14,15,16,18],&\\
    &[3,6,8,12,16,18,21],[3,6,8,14,16,18,19],[3,6,9,12,15,18,21],[3,6,9,14,15,18,19],&\\
    &[3,6,12,14,15,16,18],[3,7,8,12,16,18,20],[3,7,8,13,16,18,19],[3,7,9,12,15,18,20],&\\
    &[3,7,9,13,15,18,19],[3,7,12,13,15,16,18],[4,5,8,13,16,17,21],[4,5,8,14,16,17,20],&\\
    &[4,5,9,13,15,17,21],[4,5,9,14,15,17,20],[4,5,13,14,15,16,17],[4,6,8,12,16,17,21],&\\
    &[4,6,8,14,16,17,19],[4,6,9,12,15,17,21],[4,6,9,14,15,17,19],[4,6,12,14,15,16,17],&\\
    &[4,7,8,12,16,17,20],[4,7,8,13,16,17,19],[4,7,9,12,15,17,20],[4,7,9,13,15,17,19],&\\
    &[4,7,12,13,15,16,17],[5,6,8,9,17,18,21],[5,6,8,14,16,17,18],[5,6,9,14,15,17,18],&\\
    &[5,7,8,9,17,18,20],[5,7,8,13,16,17,18],[5,7,9,13,15,17,18],[6,7,8,9,17,18,19],&\\
    &[6,7,8,12,16,17,18],[6,7,9,12,15,17,18]
\end{eqnarray*}
For every such list $L$ we get a LT-mirror model for $Y_{2,2,3}$ given by the choice of the fan matrix $W(L)=(\L_\aa)_L$. Since $\conv(W)$ is a simplex, there is a unique simplicial and complete fan admitting $W$ as a fan matrix, giving rise to the $\Q$-factorial, projective toric variety
\begin{equation*}
    X_L=\P(\q)/H\cong \P^6/G\quad\text{with}\quad\begin{array}{c}
         \q=(3,3,3,3,2,2,2)\\
     H\cong\Tors(\Cl(X_L)\\
     G\cong H\times \mu_3^4\times\mu_2^3
    \end{array}
\end{equation*}
For any $L$, evaluating $V^T\cdot W$ one gets $\cc_i=\aa_i$, for $i=1,2,3$\,,
so that the fan matrix defined by $[\conv(\D_{\cc_1},\D_{\cc_2},\D_{\cc_3})]$ is precisely $V$, that is, $(\P^6,\aa_1+\aa_2+\aa_3)$ is the non-calibrated partitioned $f$-dual of $(X_L,\cc_1+\cc_2+\cc_3)$.
By adding the bottom row $(\,\0_6\,|\,1\,)$ one gets that
\begin{equation*}
    \left\{\left|\det \left(
                      \begin{array}{ccc}
                         &  &  \\
                         & W(L) &  \\
                         &  &  \\
                         \hline
                        \0_6 & \vline & 1 \\
                      \end{array}
                    \right)\right|\ :\ \forall\,L\right\}=\{36,48\}
\end{equation*}
More precisely, 16 choices of $L$ give 48 and the remaining 24 choices give 36. Consequently, we can conclude that at least two of those LT-mirror models of $Y_{2,3,2}$ are distinguished\,.
More explicitly, let us consider the first and the fifth lists above and show they give different LT-mirror models. Then set
\begin{equation*}
    W_1:=(\L_\aa)_{[3,4,12,13,15,16,21]}= \left(\begin{array}{ccccccc}
-1 & -1 & 0 & 0 & 3 & 0 & 0
\\
 -1 & -1 & 0 & 0 & 0 & 3 & 0
\\
 2 & 0 & -1 & -1 & 0 & 0 & 0
\\
 0 & 2 & -1 & -1 & 0 & 0 & 0
\\
 0 & 0 & 2 & 0 & -1 & -1 & -1
\\
 0 & 0 & 0 & 2 & -1 & -1 & -1
\end{array}\right)
\end{equation*}
   \begin{equation*}
    W_2:=(\L_\aa)_{[3,5,8,14,16,18,20]}=\left(\begin{array}{ccccccc}
-1 & -1 & 2 & 0 & 0 & 0 & 0
\\
 -1 & -1 & 0 & 0 & 3 & 0 & 0
\\
 2 & 0 & -1 & -1 & 0 & 0 & 0
\\
 0 & 0 & -1 & -1 & 0 & 3 & 0
\\
 0 & 2 & 0 & 0 & -1 & -1 & -1
\\
 0 & 0 & 0 & 0 & -1 & -1 & 2
\end{array}\right)
\end{equation*}
Then, $W_1$ is a fan matrix of
\begin{equation*}
    X_1=\P(\q)/H_1\cong \P^6/G_1\quad\text{with}\quad\begin{array}{c}
         \q=(3,3,3,3,2,2,2)\\
     H_1\cong\mu_2\times\mu_{12}\cong \Tors(\Cl(X_1)\\
     G_1\cong H_1\times \mu_3^4\times\mu_2^3
    \end{array}
\end{equation*}
where the action of $H_1$ is described by the torsion matrix
\begin{equation*}
    T_{W_1}=\left(\begin{array}{ccccccc}
0_2 & 1_2 & 0_2 & 0_2 & 1_2 & 1_2 & 0_2
\\
 0_{12} & 6_{12} & 9_{12} & 3_{12} & 6_{12} & 2_{12} & 10_{12}
\end{array}\right)
\end{equation*}
and the one of $G_1$ is compatible with weight actions defining weighted projective spaces involved in the following diagram
     \begin{equation*}
       \xymatrix{&\P^6\ar[dd]^-{/G_1}\ar[ld]_-{/(\mu_3^4\times\mu_2^3)}\\
                \P(\q)\ar[dr]_-{/H_1}&\\
                &X_1}
                \end{equation*}
                Then, the LT-mirror partner $Y^\vee_{LT,1}\subset X_1$ is given by the quotient of the $H_1$-invariant weighted complete intersection defined in $\P(\q)$ by the following weighted homogeneous polynomials
\begin{eqnarray*}
    q'_{1,\psi}&=&x_{5}^{3}+x_{6}^{3}+\psi x_{1} x_{2}\\
    q'_{2,\psi}&=&x_{1}^{2}+x_{2}^{2}+\psi x_{3} x_{4}\\
    q'_{3,\psi}&=& x_{3}^{2}+x_{4}^{2}+x_{7}^{3}+\psi x_{5} x_{6} x_{7}
\end{eqnarray*}
Then the generic element of the family $Y^\vee_{LT,1}$ is quasi-smooth. Moreover
\begin{equation*}
    \deg(q'_{1,\psi})= \deg(q'_{2,\psi})= \deg((q'_{3.\psi})= 6
\end{equation*}
and $3\times 6=18=|\q|$, so proving that $Y^\vee_{LT,1}$ is, up to a desingularization, a \cy threefold.

For what concerns $W_2$, it is the fan matrix of
\begin{equation*}
    X_2=\P(\q)/H_2\cong \P^6/G_2\quad\text{with}\quad\begin{array}{c}
         \q=(3,3,3,3,2,2,2)\\
     H_2\cong\mu_{18}\cong \Tors(\Cl(X_2)\\
     G_2\cong H_2\times \mu_3^4\times\mu_2^3
    \end{array}
\end{equation*}
where the action of $H_2$ is described by the torsion matrix
\begin{equation*}
    T_{W_2}=\left(\begin{array}{ccccccc}
6_{18} & 12_{18} & 9_{18} & 3_{18} & 6_{18} & 10_{18} & 8_{18}
\end{array}\right)
\end{equation*}
and the one of $G_2$ is compatible with weight actions defining weighted projective spaces involved in the following diagram
     \begin{equation*}
       \xymatrix{&\P^6\ar[dd]^-{/G_2}\ar[ld]_-{/(\mu_3^4\times\mu_2^3)}\\
                \P(\q)\ar[dr]_-{/H_2}&\\
                &X_2}
                \end{equation*}
                Then, the LT-mirror partner $Y^\vee_{LT,2}\subset X_2$ is given by the quotient of the $H_2$-invariant weighted complete intersection defined in $\P(\q)$ by the following weighted homogeneous polynomials
\begin{eqnarray*}
    q''_{1,\psi}&=&x_{3}^{2}+x_{5}^{3}+\psi x_{1} x_{2}\\
    q''_{2,\psi}&=&x_{1}^{2}+x_{6}^{3}+\psi x_{3} x_{4}\\
    q''_{3,\psi}&=& x_{2}^{2}+x_{4}^{2}+x_{7}^{3}+\psi x_{5} x_{6} x_{7}
\end{eqnarray*}
As above, the generic element of the family $Y^\vee_{LT,2}$ is quasi-smooth. Moreover
\begin{equation*}
    \deg(q'_{1,\psi})= \deg(q'_{2,\psi})= \deg((q'_{3.\psi})= 6
\end{equation*}
so proving that $Y^\vee_{LT,2}$ is, up to a desingularization, a \cy threefold.

\subsubsection{Intermediate mirror models}\label{sssez:mirrorintermedi223} Furthermore, there are all the intermediate mirrors determined by the choice of subsets $A'\subset\cI^{W_1}$ and $A''\subset\cI^{W_2}$. We than obtain at least $2^{15}-1= 32\,767$ distinct (but birational) mirror models of $Y_{2,2,3}\subset\P^6$.

\subsection{$D$-equivalence} For every $A'\subset\cI^{W_1}$ and $A''\subset\cI^{W_2}$, the associated mirror models $Y^\vee_{A'}\subset \XX_\aa^{A'}$ and $Y^\vee_{A'}\subset \XX_\aa^{A'}$ are, up to a desingularization, \cy threefolds. Moreover they are all connected each other by means of crepant birational maps, as the two $LT$-mirrors $Y^\vee_{LT,1}$ and $Y^\vee_{LT,2}$ come endowed with embeddings $Y^\vee_{LT,i}\hookrightarrow X_i$ and blowups $\phi_i:\XX_\aa\to X_i$ giving rise to the following diagrams
\begin{equation}\label{diagrammi}
  \xymatrix{&\widehat{\XX}_\aa(\Si)\ar[d]^-{\psi_\Si}\ar[dl]_-f\ar[dr]^-g&\\
            \XX_\aa^{A'}\ar[d]_-{\phi_1^{A'}}&\ar[l]\ar[r]\XX_\aa\ar[dl]_-{\phi_1}
            \ar[dr]^-{\phi_2}&\XX_\aa^{A''}\ar[d]^-{\phi_2^{A''}}\\
            X_1\ar@{-->}[rr]^-\vf&&X_2}\xymatrix{&&\\
            &\Longrightarrow&\\
            &&} \xymatrix{&\widehat{Y}^\vee_{BB}\ar[dl]_-f\ar[dr]^-g&\\
            Y^\vee_{A'}\ar[d]_-{\phi_1^{A'}}\ar@{-->}[rr]&&Y^\vee_{A''}\ar[d]^-{\phi_2^{A''}}\\
            Y^\vee_{LT,1}\ar@{-->}[rr]&&Y^\vee_{LT,2}}
\end{equation}
where the choice of $\Si\in\PSF(\widehat{\L}_\aa)$ determines a suitable desingularization $\widehat{Y}^\vee_{BB}$\,, being $\widehat{\L}_\aa$ the fan matrix whose columns are given by all primitive lattice points contained in $\D_\aa\setminus\{\0\}$. The restriction of $f$ and $g$ to the embedded \cy $\widehat{Y}^\vee_{BB}$ are crepant birational morphisms, so that
\begin{equation*}
  f^*\cO_{Y^\vee_{A'}}\cong f^*K_{Y^\vee_{A'}}\sim K_{\widehat{Y}^\vee_{BB}}\cong\cO_{\widehat{Y}^\vee_{BB}}\sim g^*K_{Y^\vee_{A''}}\cong g^*\cO_{Y^\vee_{A''}}
\end{equation*}
Therefore, Kawamata results \cite[Thm.~4.6, Thm.~6.5]{Kawamata} hold, so that arguments proving Theorems~\ref{thm:Dequivalenza} and \ref{thm:Dequivalenza^A} applies exactly in the same way, to give a proof of the following

\begin{theorem}\label{thm:Dequivalenza^cy}
  Let $A',A''$ be two subsets of $\cI^{W_1}$ and $\cI^{W_2}$, respectively, and let $\widehat{\L}_\aa^{A'}$ be the matrix whose columns are given by all the primitive lattice points contained in $\conv(\L_\aa^{A'})\setminus\{\0\}$, and analogously for $\widehat{\L}_\aa^{A''}$.  For any choice $$\Si'\in\PSF(\widehat{\L}^{A'})\ ,\quad\Si''\in\PSF(\widehat{\L}^{A''})$$
  consider the induced (possibly partial) desingularizations $$\widehat{Y}^\vee_{A'}=(\psi_{\Si'})^{-1}_*(Y^\vee_{A'})\to Y^\vee_{A'}\ ,\quad\widehat{Y}^\vee_{A''}=(\psi_{\Si''})^{-1}_*(Y^\vee_{A''})\to Y^\vee_{A''}$$
  and their canonical covering stacks $\cY_{A'},\,\cY_{A''}$. Then there exists an equivalence of triangulated categories
  \begin{equation*}
    \cD^b(\cY_{A'})\cong\cD^b(\cY_{A''})
  \end{equation*}
  between their derived categories of bounded complexes of coherent orbifold sheaves.
\end{theorem}
\begin{corollary}\label{cor:D-equiv_cy}
  If $\widehat{Y}^\vee_{A'}$ and $\widehat{Y}^\vee_{A''}$ are smooth, then there is an equivalence of triangulated categories $\cD^b(\widehat{Y}^\vee_{A'})\cong\cD^b(\widehat{Y}^\vee_{A''})$\,.
\end{corollary}
Moreover, the following is an immediate application of a result by Herbst and Walcher.

\begin{theorem}\label{thm:Malter-type}
Let $A',A''$ be two subsets of $\cI^{W_1}$ and $\cI^{W_2}$, respectively.
  Choose fans $\Si',\Si''\in\PSF(\L_\aa)$ determining partial crepant resolutions
  $$\widehat{Y}^{\vee'}=(\psi_{\Si'})^{-1}_*(Y^\vee_{A'})\to Y^\vee_{A'}\ ,\quad\widehat{Y}^{\vee''}=(\psi_{\Si''})^{-1}_*(Y^\vee_{A''})\to Y^\vee_{A''}$$
  Then there exists an equivalence of triangulated categories
  \begin{equation*}
    \cD_{sg}({Y}^{\vee'})\cong\cD_{sg}({Y}^{\vee''})
    \end{equation*}
\end{theorem}

\begin{proof} Recalling picture (\ref{diagrammi}), the choice of the two fans $\Si',\Si''\in\PSF(\L_\aa)$ corresponds to choosing two chambers $\g'$ and $\g''$, respectively, of the secondary fan on the pseudo-effective cone $\overline{\Eff}(\XX_\aa)$. Then, \cite[Thm.~3]{HW} gives the stated equivalence  $\cD_{sg}({Y}^{\vee'})\cong\cD_{sg}({Y}^{\vee''})$\,.
\end{proof}

\begin{remark}
  In the present particular setup, Theorem~\ref{thm:Dequivalenza^cy} and Corollary~\ref{cor:D-equiv_cy} give a proof of items (i) and (iii) in Theorem~\ref{thm:metateorema} and of the mirror theorem~\ref{thm:mirror}, while Theorem~\ref{thm:Malter-type} proves Conjecture~\ref{conj:K<=>D,d}.
\end{remark}

\subsection{$K$-equivalence} Since mirror models involved are, up to a desingularization, \cy threefolds, we are in a position to propose the following result generalizing Theorem~\ref{thm:Kequivalenza3}.
\begin{theorem}\label{thm:Kequivalenza_cy}
  Assume same hypotheses as in Corollary~\ref{cor:D-equiv_cy}. Then $\widehat{Y}^\vee_{A'}$ and $\widehat{Y}^\vee_{A''}$ are $K$-equivalent.
\end{theorem}
\begin{proof}
  Consider the proof of Theorem~\ref{thm:Kequivalenza3} and replace the birational map $\vf^A_{A'}$ with the one, say $\widehat{\vf}$, obtained by the birational map $\vf$, in the right diagram of display (\ref{diagrammi}), composed with resolutions $\psi_{\Si'}$ and $\psi_{\Si''}$, that is
  \begin{equation*}
    \xymatrix{\widehat{\XX}_\aa^{A'}(\Si')\ar[d]_-{\psi_{\Si'}}\ar@{-->}[r]^-{\widehat{\vf}}&
                        \widehat{\XX}_\aa^{A''}(\Si'')\ar[d]^-{\psi_{\Si''}}\\
                \XX_\aa^{A'}\ar@{-->}[r]_-\vf&\XX_\aa^{A''}}
                \xymatrix{&&\\
            &\Longrightarrow&\\
            &&}
            \xymatrix{\widehat{Y}_{A'}^\vee\ar[d]_-{\psi_{\Si'}}\ar@{-->}[r]^-{\widehat{\vf}}&
                        \widehat{Y}_{A''}^\vee\ar[d]^-{\psi_{\Si''}}\\
                Y^\vee_{A'}\ar@{-->}[r]_-\vf&Y^\vee_{A''}}
  \end{equation*}
  We still have blowups
  \begin{equation*}
    \xymatrix{&\XX_\aa^{A'\cap A''}\ar[dl]\ar[dr]&\\
                \XX_\aa^{A'}\ar@{-->}[rr]&&\XX_\aa^{A''}}
  \end{equation*}
  allowing us to reproduce the same construction driving to a dominant toric variety $\widehat{\XX}_\aa$ with birational morphisms $f$ and $g$ such that the following diagram commutes
  \begin{equation*}
    \xymatrix{&\widehat{\XX}_\aa\ar[dr]^g\ar[dl]_f&\\
               \widehat{\XX}_\aa^{A'\cap A''}(\widetilde{\Si}')\ar[d]_-{\widetilde{\psi}'}\ar@{-->}[rr]&&
                \widehat{\XX}_\aa^{A'\cap A''}(\widetilde{\Si}'')\ar[d]_-{\widetilde{\psi}''}\\
               \widehat{\XX}_\aa^{A'}(\Si')\ar@{-->}[rr]^-{\widehat{\vf}}
    &&\widehat{\XX}_\aa^{A''}(\Si'') }
    \xymatrix{&&\\
            &\Longrightarrow&\\
            &&}
    \xymatrix{&Z\ar[ddl]_-{\widetilde{\psi}'\circ f}\ar[ddr]^-{\widetilde{\psi}''\circ g}&\\
                &&\\
                \widehat{Y}_{A'}^\vee\ar@{-->}[rr]^-{\widehat{\vf}}&
                        &\widehat{Y}_{A''}^\vee}
  \end{equation*}
  Then $K$-equivalence follows as in proving Theorem~\ref{thm:Kequivalenza3}.
\end{proof}

\begin{remark}
  The previous Theorem~\ref{thm:Kequivalenza_cy} proves item (ii) in Theorem~\ref{thm:metateorema}, in the present particular setup.
\end{remark}

   \newpage

\appendix
\section{Intermediate mirrors of $Y_{2,2}\subset\P^3$}\label{app:A}
In the present appendix we give data defining the fourteen intermediate mirrors described in \S~\ref{sssez:mirrorintermedi} and Proposition~\ref{prop:mirrorintermedi}.\halfline

\noindent\textbf{Legenda.} For any proper subset $A\subset\{1,2,7,8\}$,
\begin{itemize}
  \item \emph{ftv} gives the triad $[\L^A,\aa_1^A,\aa_2^A]$ characterizing the partitioned ftv $(\XX^A,\aa^A=\aa^A_1+\aa^A_2)$
  \item \emph{Polynomials} define the mirror $Y^\vee_A$ as a complete intersections in Cox coordinates
  \item \emph{IrrIdeals} is the irrelevant ideal in $\Cox(\XX^A)$ determined by the fan of $\XX^A$
\end{itemize}

\footnotesize
\begin{maplegroup}
\mapleresult
\begin{maplelatex}
\mapleinline{inert}{2d}{A = {1}, ftv = [Matrix(
\end{maplelatex}
\mapleresult
\begin{maplelatex}
\mapleinline{inert}{2d}{Polynomials = {x[1]^2*x[5]^2+x[1]*x[2]*x[3]+x[4]^2, x[2]^2*x[6]^2+x[3]^2*x[7]^2+x[4]*x[5]*x[6]*x[7]}}{\[\displaystyle {\it Polynomials}= \left\{ {x_{{1}}}^{2}{x_{{5}}}^{2}+x_{{1}}x_{{2}}x_{{3}}\\
\mbox{}+{x_{{4}}}^{2},{x_{{2}}}^{2}{x_{{6}}}^{2}+{x_{{3}}}^{2}{x_{{7}}}^{2}+x_{{4}}x_{{5}}x_{{6}}x_{{7}}\\
\mbox{} \right\} \]}
\end{maplelatex}
\mapleresult
\begin{maplelatex}
\mapleinline{inert}{2d}{IrrIdeal = (x[4]*x[5]*x[6]*x[7], x[1]*x[5]*x[6]*x[7], x[1]*x[2]*x[5]*x[6], x[1]*x[3]*x[5]*x[7], x[2]*x[4]*x[6], x[3]*x[4]*x[7], x[1]*x[2]*x[3]*x[6], x[1]*x[2]*x[3]*x[7])}{\[\displaystyle {\it IrrIdeal}={x_{{4}}x_{{5}}x_{{6}}x_{{7}},x_{{1}}x_{{5}}\\
\mbox{}x_{{6}}x_{{7}},x_{{1}}x_{{2}}x_{{5}}x_{{6}},x_{{1}}x_{{3}}x_{{5}}\\
\mbox{}x_{{7}},x_{{2}}x_{{4}}x_{{6}},x_{{3}}x_{{4}}x_{{7}},x_{{1}}x_{{2}}x_{{3}}\\
\mbox{}x_{{6}},x_{{1}}x_{{2}}x_{{3}}x_{{7}}}\]}
\end{maplelatex}
\mapleresult
\begin{maplelatex}
\mapleinline{inert}{2d}{A = {2}, ftv = [Matrix(
\end{maplelatex}
\mapleresult
\begin{maplelatex}
\mapleinline{inert}{2d}{Polynomials = {x[1]^2*x[4]^2+x[1]*x[2]*x[3]+x[5]^2, x[2]^2*x[6]^2+x[3]^2*x[7]^2+x[4]*x[5]*x[6]*x[7]}}{\[\displaystyle {\it Polynomials}= \left\{ {x_{{1}}}^{2}{x_{{4}}}^{2}+x_{{1}}x_{{2}}x_{{3}}\\
\mbox{}+{x_{{5}}}^{2},{x_{{2}}}^{2}{x_{{6}}}^{2}+{x_{{3}}}^{2}{x_{{7}}}^{2}+x_{{4}}x_{{5}}x_{{6}}x_{{7}}\\
\mbox{} \right\} \]}
\end{maplelatex}
\mapleresult
\begin{maplelatex}
\mapleinline{inert}{2d}{IrrIdeal = (x[1]*x[4]*x[6]*x[7], x[4]*x[5]*x[6]*x[7], x[1]*x[2]*x[4]*x[6], x[1]*x[3]*x[4]*x[7], x[1]*x[2]*x[3]*x[6], x[2]*x[5]*x[6], x[1]*x[2]*x[3]*x[7], x[3]*x[5]*x[7])}{\[\displaystyle {\it IrrIdeal}={x_{{1}}x_{{4}}x_{{6}}x_{{7}},x_{{4}}x_{{5}}\\
\mbox{}x_{{6}}x_{{7}},x_{{1}}x_{{2}}x_{{4}}x_{{6}},x_{{1}}x_{{3}}x_{{4}}\\
\mbox{}x_{{7}},x_{{1}}x_{{2}}x_{{3}}x_{{6}},x_{{2}}x_{{5}}x_{{6}},x_{{1}}x_{{2}}\\
\mbox{}x_{{3}}x_{{7}},x_{{3}}x_{{5}}x_{{7}}}\]}
\end{maplelatex}
\mapleresult
\begin{maplelatex}
\mapleinline{inert}{2d}{A = {7}, ftv = [Matrix(
\end{maplelatex}
\mapleresult
\begin{maplelatex}
\mapleinline{inert}{2d}{Polynomials = {x[4]^2*x[7]^2+x[5]*x[6]*x[7]+x[3]^2, x[1]^2*x[5]^2+x[1]*x[2]*x[3]*x[4]+x[2]^2*x[6]^2}}{\[\displaystyle {\it Polynomials}= \left\{ {x_{{4}}}^{2}{x_{{7}}}^{2}+x_{{5}}x_{{6}}x_{{7}}\\
\mbox{}+{x_{{3}}}^{2},{x_{{1}}}^{2}{x_{{5}}}^{2}+x_{{1}}x_{{2}}x_{{3}}x_{{4}}+{x_{{2}}}^{2}{x_{{6}}}^{2}\\
\mbox{} \right\} \]}
\end{maplelatex}
\mapleresult
\begin{maplelatex}
\mapleinline{inert}{2d}{IrrIdeal = (x[1]*x[5]*x[6]*x[7], x[2]*x[5]*x[6]*x[7], x[1]*x[4]*x[5]*x[7], x[1]*x[2]*x[4]*x[7], x[2]*x[4]*x[6]*x[7], x[1]*x[3]*x[5], x[1]*x[2]*x[3]*x[4], x[2]*x[3]*x[6])}{\[\displaystyle {\it IrrIdeal}={x_{{1}}x_{{5}}x_{{6}}x_{{7}},x_{{2}}x_{{5}}\\
\mbox{}x_{{6}}x_{{7}},x_{{1}}x_{{4}}x_{{5}}x_{{7}},x_{{1}}x_{{2}}x_{{4}}\\
\mbox{}x_{{7}},x_{{2}}x_{{4}}x_{{6}}x_{{7}},x_{{1}}x_{{3}}x_{{5}},x_{{1}}x_{{2}}\\
\mbox{}x_{{3}}x_{{4}},x_{{2}}x_{{3}}x_{{6}}}\]}
\end{maplelatex}
\mapleresult
\begin{maplelatex}
\mapleinline{inert}{2d}{A = {8}, ftv = [Matrix(
\end{maplelatex}
\mapleresult
\begin{maplelatex}
\mapleinline{inert}{2d}{Polynomials = {x[3]^2*x[7]^2+x[5]*x[6]*x[7]+x[4]^2, x[1]^2*x[5]^2+x[1]*x[2]*x[3]*x[4]+x[2]^2*x[6]^2}}{\[\displaystyle {\it Polynomials}= \left\{ {x_{{3}}}^{2}{x_{{7}}}^{2}+x_{{5}}x_{{6}}x_{{7}}\\
\mbox{}+{x_{{4}}}^{2},{x_{{1}}}^{2}{x_{{5}}}^{2}+x_{{1}}x_{{2}}x_{{3}}x_{{4}}+{x_{{2}}}^{2}{x_{{6}}}^{2}\\
\mbox{} \right\} \]}
\end{maplelatex}
\mapleresult
\begin{maplelatex}
\mapleinline{inert}{2d}{IrrIdeal = (x[1]*x[5]*x[6]*x[7], x[2]*x[5]*x[6]*x[7], x[1]*x[3]*x[5]*x[7], x[1]*x[2]*x[3]*x[7], x[2]*x[3]*x[6]*x[7], x[1]*x[4]*x[5], x[1]*x[2]*x[3]*x[4], x[2]*x[4]*x[6])}{\[\displaystyle {\it IrrIdeal}={x_{{1}}x_{{5}}x_{{6}}x_{{7}},x_{{2}}x_{{5}}\\
\mbox{}x_{{6}}x_{{7}},x_{{1}}x_{{3}}x_{{5}}x_{{7}},x_{{1}}x_{{2}}x_{{3}}\\
\mbox{}x_{{7}},x_{{2}}x_{{3}}x_{{6}}x_{{7}},x_{{1}}x_{{4}}x_{{5}},x_{{1}}x_{{2}}\\
\mbox{}x_{{3}}x_{{4}},x_{{2}}x_{{4}}x_{{6}}}\]}
\end{maplelatex}
\mapleresult
\begin{maplelatex}
\mapleinline{inert}{2d}{A = {1, 2}, ftv = [Matrix(
\end{maplelatex}
\mapleresult
\begin{maplelatex}
\mapleinline{inert}{2d}{Polynomials = {x[1]*x[2]+x[3]^2+x[4]^2, x[1]^2*x[5]^2+x[2]^2*x[6]^2+x[3]*x[4]*x[5]*x[6]}}{\[\displaystyle {\it Polynomials}= \left\{ x_{{1}}x_{{2}}+{x_{{3}}}^{2}+{x_{{4}}}^{2},{x_{{1}}}^{2}{x_{{5}}}^{2}\\
\mbox{}+{x_{{2}}}^{2}{x_{{6}}}^{2}+x_{{3}}x_{{4}}x_{{5}}x_{{6}} \right\} \]}
\end{maplelatex}
\mapleresult
\begin{maplelatex}
\mapleinline{inert}{2d}{IrrIdeal = (x[3]*x[5]*x[6], x[4]*x[5]*x[6], x[1]*x[3]*x[5], x[1]*x[4]*x[5], x[2]*x[3]*x[6], x[2]*x[4]*x[6], x[1]*x[2]*x[5], x[1]*x[2]*x[6])}{\[\displaystyle {\it IrrIdeal}={x_{{3}}x_{{5}}x_{{6}},x_{{4}}x_{{5}}x_{{6}}\\
\mbox{},x_{{1}}x_{{3}}x_{{5}},x_{{1}}x_{{4}}x_{{5}},x_{{2}}x_{{3}}x_{{6}},x_{{2}}\\
\mbox{}x_{{4}}x_{{6}},x_{{1}}x_{{2}}x_{{5}},x_{{1}}x_{{2}}x_{{6}}}\]}
\end{maplelatex}
\mapleresult
\begin{maplelatex}
\mapleinline{inert}{2d}{A = {1, 7}, ftv = [Matrix(
\end{maplelatex}
\mapleresult
\begin{maplelatex}
\mapleinline{inert}{2d}{Polynomials = {x[1]^2*x[5]^2+x[1]*x[2]*x[3]+x[4]^2, x[3]^2*x[6]^2+x[4]*x[5]*x[6]+x[2]^2}}{\[\displaystyle {\it Polynomials}= \left\{ {x_{{1}}}^{2}{x_{{5}}}^{2}+x_{{1}}x_{{2}}x_{{3}}\\
\mbox{}+{x_{{4}}}^{2},{x_{{3}}}^{2}{x_{{6}}}^{2}+x_{{4}}x_{{5}}x_{{6}}+{x_{{2}}}^{2} \right\} \]}
\end{maplelatex}
\mapleresult
\begin{maplelatex}
\mapleinline{inert}{2d}{IrrIdeal = (x[4]*x[5]*x[6], x[1]*x[5]*x[6], x[1]*x[2]*x[5], x[3]*x[4]*x[6], x[1]*x[3]*x[6], x[2]*x[4], x[1]*x[2]*x[3])}{\[\displaystyle {\it IrrIdeal}={x_{{4}}x_{{5}}x_{{6}},x_{{1}}x_{{5}}x_{{6}}\\
\mbox{},x_{{1}}x_{{2}}x_{{5}},x_{{3}}x_{{4}}x_{{6}},x_{{1}}x_{{3}}x_{{6}},x_{{2}}\\
\mbox{}x_{{4}},x_{{1}}x_{{2}}x_{{3}}}\]}
\end{maplelatex}
\mapleresult
\begin{maplelatex}
\mapleinline{inert}{2d}{A = {1, 8}, ftv = [Matrix(
\end{maplelatex}
\mapleresult
\begin{maplelatex}
\mapleinline{inert}{2d}{Polynomials = {x[1]^2*x[5]^2+x[1]*x[2]*x[3]+x[4]^2, x[2]^2*x[6]^2+x[4]*x[5]*x[6]+x[3]^2}}{\[\displaystyle {\it Polynomials}= \left\{ {x_{{1}}}^{2}{x_{{5}}}^{2}+x_{{1}}x_{{2}}x_{{3}}\\
\mbox{}+{x_{{4}}}^{2},{x_{{2}}}^{2}{x_{{6}}}^{2}+x_{{4}}x_{{5}}x_{{6}}+{x_{{3}}}^{2} \right\} \]}
\end{maplelatex}
\mapleresult
\begin{maplelatex}
\mapleinline{inert}{2d}{IrrIdeal = (x[4]*x[5]*x[6], x[1]*x[5]*x[6], x[2]*x[4]*x[6], x[1]*x[2]*x[6], x[1]*x[3]*x[5], x[3]*x[4], x[1]*x[2]*x[3])}{\[\displaystyle {\it IrrIdeal}={x_{{4}}x_{{5}}x_{{6}},x_{{1}}x_{{5}}x_{{6}}\\
\mbox{},x_{{2}}x_{{4}}x_{{6}},x_{{1}}x_{{2}}x_{{6}},x_{{1}}x_{{3}}x_{{5}},x_{{3}}\\
\mbox{}x_{{4}},x_{{1}}x_{{2}}x_{{3}}}\]}
\end{maplelatex}
\mapleresult
\begin{maplelatex}
\mapleinline{inert}{2d}{A = {2, 7}, ftv = [Matrix(
\end{maplelatex}
\mapleresult
\begin{maplelatex}
\mapleinline{inert}{2d}{Polynomials = {x[1]^2*x[4]^2+x[1]*x[2]*x[3]+x[5]^2, x[3]^2*x[6]^2+x[4]*x[5]*x[6]+x[2]^2}}{\[\displaystyle {\it Polynomials}= \left\{ {x_{{1}}}^{2}{x_{{4}}}^{2}+x_{{1}}x_{{2}}x_{{3}}\\
\mbox{}+{x_{{5}}}^{2},{x_{{3}}}^{2}{x_{{6}}}^{2}+x_{{4}}x_{{5}}x_{{6}}+{x_{{2}}}^{2} \right\} \]}
\end{maplelatex}
\mapleresult
\begin{maplelatex}
\mapleinline{inert}{2d}{IrrIdeal = (x[1]*x[4]*x[6], x[4]*x[5]*x[6], x[1]*x[2]*x[4], x[1]*x[3]*x[6], x[3]*x[5]*x[6], x[1]*x[2]*x[3], x[2]*x[5])}{\[\displaystyle {\it IrrIdeal}={x_{{1}}x_{{4}}x_{{6}},x_{{4}}x_{{5}}x_{{6}}\\
\mbox{},x_{{1}}x_{{2}}x_{{4}},x_{{1}}x_{{3}}x_{{6}},x_{{3}}x_{{5}}x_{{6}},x_{{1}}\\
\mbox{}x_{{2}}x_{{3}},x_{{2}}x_{{5}}}\]}
\end{maplelatex}
\mapleresult
\begin{maplelatex}
\mapleinline{inert}{2d}{A = {2, 8}, ftv = [Matrix(
\end{maplelatex}
\mapleresult
\begin{maplelatex}
\mapleinline{inert}{2d}{Polynomials = {x[1]^2*x[4]^2+x[1]*x[2]*x[3]+x[5]^2, x[2]^2*x[6]^2+x[4]*x[5]*x[6]+x[3]^2}}{\[\displaystyle {\it Polynomials}= \left\{ {x_{{1}}}^{2}{x_{{4}}}^{2}+x_{{1}}x_{{2}}x_{{3}}\\
\mbox{}+{x_{{5}}}^{2},{x_{{2}}}^{2}{x_{{6}}}^{2}+x_{{4}}x_{{5}}x_{{6}}+{x_{{3}}}^{2} \right\} \]}
\end{maplelatex}
\mapleresult
\begin{maplelatex}
\mapleinline{inert}{2d}{IrrIdeal = (x[1]*x[4]*x[6], x[4]*x[5]*x[6], x[1]*x[2]*x[6], x[2]*x[5]*x[6], x[1]*x[3]*x[4], x[1]*x[2]*x[3], x[3]*x[5])}{\[\displaystyle {\it IrrIdeal}={x_{{1}}x_{{4}}x_{{6}},x_{{4}}x_{{5}}x_{{6}}\\
\mbox{},x_{{1}}x_{{2}}x_{{6}},x_{{2}}x_{{5}}x_{{6}},x_{{1}}x_{{3}}x_{{4}},x_{{1}}\\
\mbox{}x_{{2}}x_{{3}},x_{{3}}x_{{5}}}\]}
\end{maplelatex}
\mapleresult
\begin{maplelatex}
\mapleinline{inert}{2d}{A = {7, 8}, ftv = [Matrix(
\end{maplelatex}
\mapleresult
\begin{maplelatex}
\mapleinline{inert}{2d}{Polynomials = {x[3]^2+x[4]^2+x[5]*x[6], x[1]^2*x[5]^2+x[1]*x[2]*x[3]*x[4]+x[2]^2*x[6]^2}}{\[\displaystyle {\it Polynomials}= \left\{ {x_{{3}}}^{2}+{x_{{4}}}^{2}+x_{{5}}x_{{6}},{x_{{1}}}^{2}{x_{{5}}}^{2}\\
\mbox{}+x_{{1}}x_{{2}}x_{{3}}x_{{4}}+{x_{{2}}}^{2}{x_{{6}}}^{2} \right\} \]}
\end{maplelatex}
\mapleresult
\begin{maplelatex}
\mapleinline{inert}{2d}{IrrIdeal = (x[1]*x[5]*x[6], x[2]*x[5]*x[6], x[1]*x[3]*x[5], x[1]*x[2]*x[3], x[2]*x[3]*x[6], x[1]*x[4]*x[5], x[1]*x[2]*x[4], x[2]*x[4]*x[6])}{\[\displaystyle {\it IrrIdeal}={x_{{1}}x_{{5}}x_{{6}},x_{{2}}x_{{5}}x_{{6}}\\
\mbox{},x_{{1}}x_{{3}}x_{{5}},x_{{1}}x_{{2}}x_{{3}},x_{{2}}x_{{3}}x_{{6}},x_{{1}}\\
\mbox{}x_{{4}}x_{{5}},x_{{1}}x_{{2}}x_{{4}},x_{{2}}x_{{4}}x_{{6}}}\]}
\end{maplelatex}
\mapleresult
\begin{maplelatex}
\mapleinline{inert}{2d}{A = {1, 2, 7}, ftv = [Matrix(
\end{maplelatex}
\mapleresult
\begin{maplelatex}
\mapleinline{inert}{2d}{Polynomials = {x[1]*x[2]+x[3]^2+x[4]^2, x[2]^2*x[5]^2+x[3]*x[4]*x[5]+x[1]^2}}{\[\displaystyle {\it Polynomials}= \left\{ x_{{1}}x_{{2}}+{x_{{3}}}^{2}+{x_{{4}}}^{2},{x_{{2}}}^{2}{x_{{5}}}^{2}\\
\mbox{}+x_{{3}}x_{{4}}x_{{5}}+{x_{{1}}}^{2} \right\} \]}
\end{maplelatex}
\mapleresult
\begin{maplelatex}
\mapleinline{inert}{2d}{IrrIdeal = (x[3]*x[5], x[4]*x[5], x[1]*x[3], x[1]*x[4], x[2]*x[5], x[1]*x[2])}{\[\displaystyle {\it IrrIdeal}={x_{{3}}x_{{5}},x_{{4}}x_{{5}},x_{{1}}x_{{3}},x_{{1}}\\
\mbox{}x_{{4}},x_{{2}}x_{{5}},x_{{1}}x_{{2}}}\]}
\end{maplelatex}
\mapleresult
\begin{maplelatex}
\mapleinline{inert}{2d}{A = {1, 2, 8}, ftv = [Matrix(
\end{maplelatex}
\mapleresult
\begin{maplelatex}
\mapleinline{inert}{2d}{Polynomials = {x[1]*x[2]+x[3]^2+x[4]^2, x[1]^2*x[5]^2+x[3]*x[4]*x[5]+x[2]^2}}{\[\displaystyle {\it Polynomials}= \left\{ x_{{1}}x_{{2}}+{x_{{3}}}^{2}+{x_{{4}}}^{2},{x_{{1}}}^{2}{x_{{5}}}^{2}\\
\mbox{}+x_{{3}}x_{{4}}x_{{5}}+{x_{{2}}}^{2} \right\} \]}
\end{maplelatex}
\mapleresult
\begin{maplelatex}
\mapleinline{inert}{2d}{IrrIdeal = (x[3]*x[5], x[4]*x[5], x[1]*x[5], x[2]*x[3], x[2]*x[4], x[1]*x[2])}{\[\displaystyle {\it IrrIdeal}={x_{{3}}x_{{5}},x_{{4}}x_{{5}},x_{{1}}x_{{5}},x_{{2}}\\
\mbox{}x_{{3}},x_{{2}}x_{{4}},x_{{1}}x_{{2}}}\]}
\end{maplelatex}
\mapleresult
\begin{maplelatex}
\mapleinline{inert}{2d}{A = {1, 7, 8}, ftv = [Matrix(
\end{maplelatex}
\mapleresult
\begin{maplelatex}
\mapleinline{inert}{2d}{Polynomials = {x[2]^2+x[3]^2+x[4]*x[5], x[1]^2*x[5]^2+x[1]*x[2]*x[3]+x[4]^2}}{\[\displaystyle {\it Polynomials}= \left\{ {x_{{2}}}^{2}+{x_{{3}}}^{2}+x_{{4}}x_{{5}},{x_{{1}}}^{2}{x_{{5}}}^{2}\\
\mbox{}+x_{{1}}x_{{2}}x_{{3}}+{x_{{4}}}^{2} \right\} \]}
\end{maplelatex}
\mapleresult
\begin{maplelatex}
\mapleinline{inert}{2d}{IrrIdeal = (x[4]*x[5], x[1]*x[5], x[2]*x[4], x[1]*x[2], x[3]*x[4], x[1]*x[3])}{\[\displaystyle {\it IrrIdeal}={x_{{4}}x_{{5}},x_{{1}}x_{{5}},x_{{2}}x_{{4}},x_{{1}}\\
\mbox{}x_{{2}},x_{{3}}x_{{4}},x_{{1}}x_{{3}}}\]}
\end{maplelatex}
\mapleresult
\begin{maplelatex}
\mapleinline{inert}{2d}{A = {2, 7, 8}, ftv = [Matrix(
\end{maplelatex}
\mapleresult
\begin{maplelatex}
\mapleinline{inert}{2d}{Polynomials = {x[2]^2+x[3]^2+x[4]*x[5], x[1]^2*x[4]^2+x[1]*x[2]*x[3]+x[5]^2}}{\[\displaystyle {\it Polynomials}= \left\{ {x_{{2}}}^{2}+{x_{{3}}}^{2}+x_{{4}}x_{{5}},{x_{{1}}}^{2}{x_{{4}}}^{2}\\
\mbox{}+x_{{1}}x_{{2}}x_{{3}}+{x_{{5}}}^{2} \right\} \]}
\end{maplelatex}
\mapleresult
\begin{maplelatex}
\mapleinline{inert}{2d}{IrrIdeal = (x[1]*x[4], x[4]*x[5], x[1]*x[2], x[2]*x[5], x[1]*x[3], x[3]*x[5])}{\[\displaystyle {\it IrrIdeal}={x_{{1}}x_{{4}},x_{{4}}x_{{5}},x_{{1}}x_{{2}},x_{{2}}\\
\mbox{}x_{{5}},x_{{1}}x_{{3}},x_{{3}}x_{{5}}}\]}
\end{maplelatex}
\end{maplegroup}
\begin{Maple Normal}{
\begin{Maple Normal}{
\mapleinline{inert}{2d}{}{\[\displaystyle \]}
}\end{Maple Normal}
}\end{Maple Normal}

\normalsize
\section{The $5\times 110$ fan matrix $\widehat{\L}$}\label{app:B}
\begin{equation*}
\resizebox{1\hsize}{!}{$
  \widehat{\L}=\left( \begin {array}{cccccccccccccccccccc} -1&-1&-1&-1&-1&-1&-1&-1&-1&-1&-1&-1&-1&-1&-1&-1&-1&-1&-1&-1\\ \noalign{\medskip}-1&-1&-1&-1&-1&-1&-1&-1&-1&-1&-1&-1&-1&-1&-1&-1&-1&-1&-1&-1\\ \noalign{\medskip}-1&-1&-1&-1&-1&-1&-1&-1&-1&-1&0&0&0&0&0&0&1&1&1&2\\ \noalign{\medskip}0&0&0&0&1&1&1&2&2&3&0&0&0&1&1&2&0&0&1&0\\ \noalign{\medskip}0&1&2&3&0&1&2&0&1&0&0&1&2&0&1&0&0&1&0&0\end {array}\right.$}
 \end{equation*}
 \begin{equation*}
   \resizebox{1\hsize}{!}{$
   \begin {array}{cccccccccccccccccccc} -1&-1&-1&-1&-1&-1&-1&-1&-1&-1&-1&-1&-1&-1&-1&0&0&0&0&0\\ \noalign{\medskip}0&0&0&0&0&0&0&0&0&0&1&1&1&1&2&-1&-1&-1&-1&-1\\ \noalign{\medskip}-1&-1&-1&-1&-1&-1&0&0&0&1&-1&-1&-1&0&-1&-1&-1&-1&-1&-1\\ \noalign{\medskip}0&0&0&1&1&2&0&0&1&0&0&0&1&0&0&0&0&0&1&1\\ \noalign{\medskip}0&1&2&0&1&0&0&1&0&0&0&1&0&0&0&0&1&2&0&1\end {array}$}
 \end{equation*}
 \begin{equation*}
 \resizebox{1\hsize}{!}{$
  \begin {array}{cccccccccccccccccccc} 0&0&0&0&0&0&0&0&0&0&0&0&0&0&0&0&0&0&0&0\\ \noalign{\medskip}-1&-1&-1&-1&-1&0&0&0&0&0&0&0&0&0&0&0&0&0&0&0\\ \noalign{\medskip}-1&0&0&0&1&-1&-1&-1&0&0&0&0&0&0&0&0&0&1&1&1\\ \noalign{\medskip}2&0&0&1&0&0&0&1&-1&-1&-1&-1&0&0&1&1&2&-1&-1&-1\\ \noalign{\medskip}0&0&1&0&0&0&1&0&-1&0&1&2&-1&1&-1&0&-1&-1&0&1\end {array}$}
  \end{equation*}
  \begin{equation*}
    \resizebox{1\hsize}{!}{$
    \begin {array}{cccccccccccccccccccc} 0&0&0&0&0&0&0&0&0&0&0&0&0&0&0&0&0&0&0&0\\ \noalign{\medskip}0&0&0&0&0&0&0&1&1&1&1&1&1&1&1&1&1&1&2&2\\ \noalign{\medskip}1&1&1&2&2&2&3&-1&0&0&0&0&0&0&1&1&1&2&0&0\\ \noalign{\medskip}0&0&1&-1&-1&0&-1&0&-1&-1&-1&0&0&1&-1&-1&0&-1&-1&-1\\ \noalign{\medskip}-1&0&-1&-1&0&-1&-1&0&-1&0&1&-1&0&-1&-1&0&-1&-1&-1&0\end {array}$}
  \end{equation*}
  \begin{equation*}
    \resizebox{1\hsize}{!}{$
    \begin {array}{cccccccccccccccccccc} 0&0&0&1&1&1&1&1&1&1&1&1&1&1&1&1&1&1&1&1\\ \noalign{\medskip}2&2&3&-1&-1&-1&-1&0&0&0&0&0&0&0&0&0&0&0&1&1\\ \noalign{\medskip}0&1&0&-1&-1&-1&0&-1&0&0&0&0&0&0&1&1&1&2&0&0\\ \noalign{\medskip}0&-1&-1&0&0&1&0&0&-1&-1&-1&0&0&1&-1&-1&0&-1&-1&-1\\ \noalign{\medskip}-1&-1&-1&0&1&0&0&0&-1&0&1&-1&0&-1&-1&0&-1&-1&-1&0\end {array} $}
  \end{equation*}
  \begin{equation*}
    \resizebox{.55\hsize}{!}{$
    \left. \begin {array}{cccccccccc} 1&1&1&2&2&2&2&2&2&3\\ \noalign{\medskip}1&1&2&-1&0&0&0&0&1&0\\ \noalign{\medskip}0&1&0&-1&0&0&0&1&0&0\\ \noalign{\medskip}0&-1&-1&0&-1&-1&0&-1&-1&-1\\ \noalign{\medskip}-1&-1&-1&0&-1&0&-1&-1&-1&-1\end {array} \right)$}
  \end{equation*}

  \newpage
  \section{Mirrors of $Y_{3,3}\subset\P^5$}\label{app:C}

  In the present appendix we give data defining the 64 mirrors described in \S~\ref{sssez:mirrorintermedi3,3} and Proposition~\ref{prop:mirrorintermedi3,3}. Notation is the same as in App.~\ref{app:A}.\halfline
{\tiny
  \begin{maplegroup}
\mapleresult
\begin{maplelatex}
\mapleinline{inert}{2d}{A = {}, ftv = [Matrix(
\mbox{}],[0,0,0,0,0,0,1,1,1,1,1,1]]\]}
\end{maplelatex}
\mapleresult
\begin{maplelatex}
\mapleinline{inert}{2d}{Polynomials = {x[1]^3*x[7]^3+x[1]*x[2]*x[3]*x[4]*x[5]*x[6]+x[2]^3*x[8]^3+x[3]^3*x[9]^3, x[4]^3*x[10]^3+x[5]^3*x[11]^3+x[6]^3*x[12]^3+x[7]*x[8]*x[9]*x[10]*x[11]*x[12]}}{\[\displaystyle {\it Polynomials}= \left\{ {x_{{1}}}^{3}{x_{{7}}}^{3}+x_{{1}}x_{{2}}x_{{3}}x_{{4}}x_{{5}}x_{{6}}\\
\mbox{}+{x_{{2}}}^{3}{x_{{8}}}^{3}+{x_{{3}}}^{3}{x_{{9}}}^{3},{x_{{4}}}^{3}{x_{{10}}}^{3}+{x_{{5}}}^{3}{x_{{11}}}^{3}\\
\mbox{}+{x_{{6}}}^{3}{x_{{12}}}^{3}+x_{{7}}x_{{8}}x_{{9}}x_{{10}}x_{{11}}x_{{12}}\\
\mbox{} \right\} \]}
\end{maplelatex}
\mapleresult
\begin{maplelatex}
\mapleinline{inert}{2d}{IrrIdeal = (x[1]*x[7]*x[8]*x[9]*x[10]*x[11]*x[12], x[2]*x[7]*x[8]*x[9]*x[10]*x[11]*x[12], x[3]*x[7]*x[8]*x[9]*x[10]*x[11]*x[12], x[1]*x[4]*x[7]*x[10], x[1]*x[5]*x[7]*x[11], x[1]*x[6]*x[7]*x[12], x[1]*x[2]*x[3]*x[4]*x[5]*x[6]*x[10])}{\[\displaystyle {\it IrrIdeal}={x_{{1}}x_{{7}}x_{{8}}x_{{9}}x_{{10}}\\
\mbox{}x_{{11}}x_{{12}},x_{{2}}x_{{7}}x_{{8}}x_{{9}}x_{{10}}x_{{11}}\\
\mbox{}x_{{12}},x_{{3}}x_{{7}}x_{{8}}x_{{9}}x_{{10}}x_{{11}}x_{{12}}\\
\mbox{},x_{{1}}x_{{4}}x_{{7}}x_{{10}},x_{{1}}x_{{5}}x_{{7}}x_{{11}},x_{{1}}\\
\mbox{}x_{{6}}x_{{7}}x_{{12}},x_{{1}}x_{{2}}x_{{3}}x_{{4}}x_{{5}}x_{{6}}\\
\mbox{}x_{{10}}}\]}
\end{maplelatex}
\mapleresult
\begin{maplelatex}
\mapleinline{inert}{2d}{x[2]*x[4]*x[8]*x[10], x[3]*x[4]*x[9]*x[10], x[1]*x[2]*x[3]*x[4]*x[5]*x[6]*x[11], x[2]*x[5]*x[8]*x[11], x[3]*x[5]*x[9]*x[11], x[1]*x[2]*x[3]*x[4]*x[5]*x[6]*x[12], x[2]*x[6]*x[8]*x[12], x[3]*x[6]*x[9]*x[12]}{\[\displaystyle x_{{2}}x_{{4}}x_{{8}}x_{{10}},\,x_{{3}}x_{{4}}x_{{9}}x_{{10}},\,x_{{1}}x_{{2}}x_{{3}}x_{{4}}x_{{5}}x_{{6}}x_{{11}},\,x_{{2}}x_{{5}}x_{{8}}x_{{11}},\,x_{{3}}x_{{5}}x_{{9}}x_{{11}},\,x_{{1}}x_{{2}}x_{{3}}x_{{4}}x_{{5}}x_{{6}}x_{{12}},\,x_{{2}}x_{{6}}x_{{8}}x_{{12}},\,x_{{3}}x_{{6}}x_{{9}}x_{{12}}\]}
\end{maplelatex}
\mapleresult
\begin{maplelatex}
\mapleinline{inert}{2d}{A = {1}, ftv = [Matrix(
\mbox{},0,0,0,0,1,1,1,1,1,1]]\]}
\end{maplelatex}
\mapleresult
\begin{maplelatex}
\mapleinline{inert}{2d}{Polynomials = {x[1]^3*x[7]^3+x[2]^3*x[8]^3+x[1]*x[2]*x[3]*x[4]*x[5]+x[6]^3, x[3]^3*x[9]^3+x[4]^3*x[10]^3+x[5]^3*x[11]^3+x[6]*x[7]*x[8]*x[9]*x[10]*x[11]}}{\[\displaystyle {\it Polynomials}= \left\{ {x_{{1}}}^{3}{x_{{7}}}^{3}+{x_{{2}}}^{3}{x_{{8}}}^{3}+x_{{1}}x_{{2}}x_{{3}}x_{{4}}x_{{5}}\\
\mbox{}+{x_{{6}}}^{3},{x_{{3}}}^{3}{x_{{9}}}^{3}+{x_{{4}}}^{3}{x_{{10}}}^{3}+{x_{{5}}}^{3}{x_{{11}}}^{3}+x_{{6}}x_{{7}}x_{{8}}x_{{9}}x_{{10}}x_{{11}}\\
\mbox{} \right\} \]}
\end{maplelatex}
\mapleresult
\begin{maplelatex}
\mapleinline{inert}{2d}{IrrIdeal = (x[6]*x[7]*x[8]*x[9]*x[10]*x[11], x[1]*x[7]*x[8]*x[9]*x[10]*x[11], x[2]*x[7]*x[8]*x[9]*x[10]*x[11], x[1]*x[3]*x[7]*x[9], x[1]*x[4]*x[7]*x[10], x[1]*x[5]*x[7]*x[11], x[3]*x[6]*x[9])}{\[\displaystyle {\it IrrIdeal}={x_{{6}}x_{{7}}x_{{8}}x_{{9}}x_{{10}}\\
\mbox{}x_{{11}},x_{{1}}x_{{7}}x_{{8}}x_{{9}}x_{{10}}x_{{11}},x_{{2}}x_{{7}}\\
\mbox{}x_{{8}}x_{{9}}x_{{10}}x_{{11}},x_{{1}}x_{{3}}x_{{7}}x_{{9}},x_{{1}}\\
\mbox{}x_{{4}}x_{{7}}x_{{10}},x_{{1}}x_{{5}}x_{{7}}x_{{11}},x_{{3}}x_{{6}}\\
\mbox{}x_{{9}}}\]}
\end{maplelatex}
\mapleresult
\begin{maplelatex}
\mapleinline{inert}{2d}{x[2]*x[3]*x[8]*x[9], x[4]*x[6]*x[10], x[2]*x[4]*x[8]*x[10], x[5]*x[6]*x[11], x[2]*x[5]*x[8]*x[11], x[1]*x[2]*x[3]*x[4]*x[5]*x[9], x[1]*x[2]*x[3]*x[4]*x[5]*x[10], x[1]*x[2]*x[3]*x[4]*x[5]*x[11]}{\[\displaystyle x_{{2}}x_{{3}}x_{{8}}x_{{9}},\,x_{{4}}x_{{6}}x_{{10}},\,x_{{2}}x_{{4}}x_{{8}}x_{{10}},\,x_{{5}}x_{{6}}x_{{11}},\,x_{{2}}x_{{5}}x_{{8}}x_{{11}},\,x_{{1}}x_{{2}}x_{{3}}x_{{4}}x_{{5}}x_{{9}},\,x_{{1}}x_{{2}}x_{{3}}x_{{4}}x_{{5}}x_{{10}},\,x_{{1}}x_{{2}}x_{{3}}x_{{4}}x_{{5}}x_{{11}}\]}
\end{maplelatex}
\mapleresult
\begin{maplelatex}
\mapleinline{inert}{2d}{A = {2}, ftv = [Matrix(
\mbox{},0,0,0,0,1,1,1,1,1,1]]\]}
\end{maplelatex}
\mapleresult
\begin{maplelatex}
\mapleinline{inert}{2d}{Polynomials = {x[1]^3*x[6]^3+x[2]^3*x[8]^3+x[1]*x[2]*x[3]*x[4]*x[5]+x[7]^3, x[3]^3*x[9]^3+x[4]^3*x[10]^3+x[5]^3*x[11]^3+x[6]*x[7]*x[8]*x[9]*x[10]*x[11]}}{\[\displaystyle {\it Polynomials}= \left\{ {x_{{1}}}^{3}{x_{{6}}}^{3}+{x_{{2}}}^{3}{x_{{8}}}^{3}+x_{{1}}x_{{2}}x_{{3}}x_{{4}}x_{{5}}\\
\mbox{}+{x_{{7}}}^{3},{x_{{3}}}^{3}{x_{{9}}}^{3}+{x_{{4}}}^{3}{x_{{10}}}^{3}+{x_{{5}}}^{3}{x_{{11}}}^{3}+x_{{6}}x_{{7}}x_{{8}}x_{{9}}x_{{10}}x_{{11}}\\
\mbox{} \right\} \]}
\end{maplelatex}
\mapleresult
\begin{maplelatex}
\mapleinline{inert}{2d}{IrrIdeal = (x[1]*x[6]*x[8]*x[9]*x[10]*x[11], x[6]*x[7]*x[8]*x[9]*x[10]*x[11], x[2]*x[6]*x[8]*x[9]*x[10]*x[11], x[1]*x[3]*x[6]*x[9], x[1]*x[4]*x[6]*x[10], x[1]*x[5]*x[6]*x[11], x[1]*x[2]*x[3]*x[4]*x[5]*x[9])}{\[\displaystyle {\it IrrIdeal}={x_{{1}}x_{{6}}x_{{8}}x_{{9}}x_{{10}}\\
\mbox{}x_{{11}},x_{{6}}x_{{7}}x_{{8}}x_{{9}}x_{{10}}x_{{11}},x_{{2}}x_{{6}}\\
\mbox{}x_{{8}}x_{{9}}x_{{10}}x_{{11}},x_{{1}}x_{{3}}x_{{6}}x_{{9}},x_{{1}}\\
\mbox{}x_{{4}}x_{{6}}x_{{10}},x_{{1}}x_{{5}}x_{{6}}x_{{11}},x_{{1}}x_{{2}}\\
\mbox{}x_{{3}}x_{{4}}x_{{5}}x_{{9}}}\]}
\end{maplelatex}
\mapleresult
\begin{maplelatex}
\mapleinline{inert}{2d}{x[3]*x[7]*x[9], x[2]*x[3]*x[8]*x[9], x[1]*x[2]*x[3]*x[4]*x[5]*x[10], x[4]*x[7]*x[10], x[2]*x[4]*x[8]*x[10], x[1]*x[2]*x[3]*x[4]*x[5]*x[11], x[5]*x[7]*x[11], x[2]*x[5]*x[8]*x[11]}{\[\displaystyle x_{{3}}x_{{7}}x_{{9}},\,x_{{2}}x_{{3}}x_{{8}}x_{{9}},\,x_{{1}}x_{{2}}x_{{3}}x_{{4}}x_{{5}}x_{{10}},\,x_{{4}}x_{{7}}x_{{10}},\,x_{{2}}x_{{4}}x_{{8}}x_{{10}},\,x_{{1}}x_{{2}}x_{{3}}x_{{4}}x_{{5}}x_{{11}},\,x_{{5}}x_{{7}}x_{{11}},\,x_{{2}}x_{{5}}x_{{8}}x_{{11}}\]}
\end{maplelatex}
\mapleresult
\begin{maplelatex}
\mapleinline{inert}{2d}{A = {3}, ftv = [Matrix(
\mbox{},0,0,0,0,1,1,1,1,1,1]]\]}
\end{maplelatex}
\mapleresult
\begin{maplelatex}
\mapleinline{inert}{2d}{Polynomials = {x[1]^3*x[6]^3+x[2]^3*x[7]^3+x[1]*x[2]*x[3]*x[4]*x[5]+x[8]^3, x[3]^3*x[9]^3+x[4]^3*x[10]^3+x[5]^3*x[11]^3+x[6]*x[7]*x[8]*x[9]*x[10]*x[11]}}{\[\displaystyle {\it Polynomials}= \left\{ {x_{{1}}}^{3}{x_{{6}}}^{3}+{x_{{2}}}^{3}{x_{{7}}}^{3}+x_{{1}}x_{{2}}x_{{3}}x_{{4}}x_{{5}}\\
\mbox{}+{x_{{8}}}^{3},{x_{{3}}}^{3}{x_{{9}}}^{3}+{x_{{4}}}^{3}{x_{{10}}}^{3}+{x_{{5}}}^{3}{x_{{11}}}^{3}+x_{{6}}x_{{7}}x_{{8}}x_{{9}}x_{{10}}x_{{11}}\\
\mbox{} \right\} \]}
\end{maplelatex}
\mapleresult
\begin{maplelatex}
\mapleinline{inert}{2d}{IrrIdeal = (x[1]*x[6]*x[7]*x[9]*x[10]*x[11], x[6]*x[7]*x[8]*x[9]*x[10]*x[11], x[2]*x[6]*x[7]*x[9]*x[10]*x[11], x[1]*x[3]*x[6]*x[9], x[1]*x[4]*x[6]*x[10], x[1]*x[5]*x[6]*x[11], x[1]*x[2]*x[3]*x[4]*x[5]*x[9])}{\[\displaystyle {\it IrrIdeal}={x_{{1}}x_{{6}}x_{{7}}x_{{9}}x_{{10}}\\
\mbox{}x_{{11}},x_{{6}}x_{{7}}x_{{8}}x_{{9}}x_{{10}}x_{{11}},x_{{2}}x_{{6}}\\
\mbox{}x_{{7}}x_{{9}}x_{{10}}x_{{11}},x_{{1}}x_{{3}}x_{{6}}x_{{9}},x_{{1}}\\
\mbox{}x_{{4}}x_{{6}}x_{{10}},x_{{1}}x_{{5}}x_{{6}}x_{{11}},x_{{1}}x_{{2}}\\
\mbox{}x_{{3}}x_{{4}}x_{{5}}x_{{9}}}\]}
\end{maplelatex}
\mapleresult
\begin{maplelatex}
\mapleinline{inert}{2d}{x[2]*x[3]*x[7]*x[9], x[3]*x[8]*x[9], x[1]*x[2]*x[3]*x[4]*x[5]*x[10], x[2]*x[4]*x[7]*x[10], x[4]*x[8]*x[10], x[1]*x[2]*x[3]*x[4]*x[5]*x[11], x[2]*x[5]*x[7]*x[11], x[5]*x[8]*x[11]}{\[\displaystyle x_{{2}}x_{{3}}x_{{7}}x_{{9}},\,x_{{3}}x_{{8}}x_{{9}},\,x_{{1}}x_{{2}}x_{{3}}x_{{4}}x_{{5}}x_{{10}},\,x_{{2}}x_{{4}}x_{{7}}x_{{10}},\,x_{{4}}x_{{8}}x_{{10}},\,x_{{1}}x_{{2}}x_{{3}}x_{{4}}x_{{5}}x_{{11}},\,x_{{2}}x_{{5}}x_{{7}}x_{{11}},\,x_{{5}}x_{{8}}x_{{11}}\]}
\end{maplelatex}
\mapleresult
\begin{maplelatex}
\mapleinline{inert}{2d}{A = {10}, ftv = [Matrix(
\mbox{},0,0,0,0,0,1,1,1,1,1]]\]}
\end{maplelatex}
\mapleresult
\begin{maplelatex}
\mapleinline{inert}{2d}{Polynomials = {x[5]^3*x[10]^3+x[6]^3*x[11]^3+x[7]*x[8]*x[9]*x[10]*x[11]+x[4]^3, x[1]^3*x[7]^3+x[1]*x[2]*x[3]*x[4]*x[5]*x[6]+x[2]^3*x[8]^3+x[3]^3*x[9]^3}}{\[\displaystyle {\it Polynomials}= \left\{ {x_{{5}}}^{3}{x_{{10}}}^{3}+{x_{{6}}}^{3}{x_{{11}}}^{3}+x_{{7}}x_{{8}}x_{{9}}x_{{10}}x_{{11}}\\
\mbox{}+{x_{{4}}}^{3},{x_{{1}}}^{3}{x_{{7}}}^{3}+x_{{1}}x_{{2}}x_{{3}}x_{{4}}x_{{5}}x_{{6}}\\
\mbox{}+{x_{{2}}}^{3}{x_{{8}}}^{3}+{x_{{3}}}^{3}{x_{{9}}}^{3} \right\} \]}
\end{maplelatex}
\mapleresult
\begin{maplelatex}
\mapleinline{inert}{2d}{IrrIdeal = (x[1]*x[7]*x[8]*x[9]*x[10]*x[11], x[2]*x[7]*x[8]*x[9]*x[10]*x[11], x[3]*x[7]*x[8]*x[9]*x[10]*x[11], x[1]*x[5]*x[7]*x[10], x[1]*x[6]*x[7]*x[11], x[1]*x[2]*x[3]*x[5]*x[6]*x[10], x[2]*x[5]*x[8]*x[10])}{\[\displaystyle {\it IrrIdeal}={x_{{1}}x_{{7}}x_{{8}}x_{{9}}x_{{10}}\\
\mbox{}x_{{11}},x_{{2}}x_{{7}}x_{{8}}x_{{9}}x_{{10}}x_{{11}},x_{{3}}x_{{7}}\\
\mbox{}x_{{8}}x_{{9}}x_{{10}}x_{{11}},x_{{1}}x_{{5}}x_{{7}}x_{{10}},x_{{1}}\\
\mbox{}x_{{6}}x_{{7}}x_{{11}},x_{{1}}x_{{2}}x_{{3}}x_{{5}}x_{{6}}x_{{10}}\\
\mbox{},x_{{2}}x_{{5}}x_{{8}}x_{{10}}}\]}
\end{maplelatex}
\mapleresult
\begin{maplelatex}
\mapleinline{inert}{2d}{x[3]*x[5]*x[9]*x[10], x[1]*x[2]*x[3]*x[5]*x[6]*x[11], x[2]*x[6]*x[8]*x[11], x[3]*x[6]*x[9]*x[11], x[1]*x[4]*x[7], x[1]*x[2]*x[3]*x[4]*x[5]*x[6], x[2]*x[4]*x[8], x[3]*x[4]*x[9]}{\[\displaystyle x_{{3}}x_{{5}}x_{{9}}x_{{10}},\,x_{{1}}x_{{2}}x_{{3}}x_{{5}}x_{{6}}x_{{11}},\,x_{{2}}x_{{6}}x_{{8}}x_{{11}},\,x_{{3}}x_{{6}}x_{{9}}x_{{11}},\,x_{{1}}x_{{4}}x_{{7}},\,x_{{1}}x_{{2}}x_{{3}}x_{{4}}x_{{5}}x_{{6}},\,x_{{2}}x_{{4}}x_{{8}},\,x_{{3}}x_{{4}}x_{{9}}\]}
\end{maplelatex}
\mapleresult
\begin{maplelatex}
\mapleinline{inert}{2d}{A = {11}, ftv = [Matrix(
\mbox{},0,0,0,0,0,1,1,1,1,1]]\]}
\end{maplelatex}
\mapleresult
\begin{maplelatex}
\mapleinline{inert}{2d}{Polynomials = {x[4]^3*x[10]^3+x[6]^3*x[11]^3+x[7]*x[8]*x[9]*x[10]*x[11]+x[5]^3, x[1]^3*x[7]^3+x[1]*x[2]*x[3]*x[4]*x[5]*x[6]+x[2]^3*x[8]^3+x[3]^3*x[9]^3}}{\[\displaystyle {\it Polynomials}= \left\{ {x_{{4}}}^{3}{x_{{10}}}^{3}+{x_{{6}}}^{3}{x_{{11}}}^{3}+x_{{7}}x_{{8}}x_{{9}}x_{{10}}x_{{11}}\\
\mbox{}+{x_{{5}}}^{3},{x_{{1}}}^{3}{x_{{7}}}^{3}+x_{{1}}x_{{2}}x_{{3}}x_{{4}}x_{{5}}x_{{6}}\\
\mbox{}+{x_{{2}}}^{3}{x_{{8}}}^{3}+{x_{{3}}}^{3}{x_{{9}}}^{3} \right\} \]}
\end{maplelatex}
\mapleresult
\begin{maplelatex}
\mapleinline{inert}{2d}{IrrIdeal = (x[1]*x[7]*x[8]*x[9]*x[10]*x[11], x[2]*x[7]*x[8]*x[9]*x[10]*x[11], x[3]*x[7]*x[8]*x[9]*x[10]*x[11], x[1]*x[4]*x[7]*x[10], x[1]*x[6]*x[7]*x[11], x[1]*x[2]*x[3]*x[4]*x[6]*x[10], x[2]*x[4]*x[8]*x[10])}{\[\displaystyle {\it IrrIdeal}={x_{{1}}x_{{7}}x_{{8}}x_{{9}}x_{{10}}\\
\mbox{}x_{{11}},x_{{2}}x_{{7}}x_{{8}}x_{{9}}x_{{10}}x_{{11}},x_{{3}}x_{{7}}\\
\mbox{}x_{{8}}x_{{9}}x_{{10}}x_{{11}},x_{{1}}x_{{4}}x_{{7}}x_{{10}},x_{{1}}\\
\mbox{}x_{{6}}x_{{7}}x_{{11}},x_{{1}}x_{{2}}x_{{3}}x_{{4}}x_{{6}}x_{{10}}\\
\mbox{},x_{{2}}x_{{4}}x_{{8}}x_{{10}}}\]}
\end{maplelatex}
\mapleresult
\begin{maplelatex}
\mapleinline{inert}{2d}{x[3]*x[4]*x[9]*x[10], x[1]*x[2]*x[3]*x[4]*x[6]*x[11], x[2]*x[6]*x[8]*x[11], x[3]*x[6]*x[9]*x[11], x[1]*x[5]*x[7], x[1]*x[2]*x[3]*x[4]*x[5]*x[6], x[2]*x[5]*x[8], x[3]*x[5]*x[9]}{\[\displaystyle x_{{3}}x_{{4}}x_{{9}}x_{{10}},\,x_{{1}}x_{{2}}x_{{3}}x_{{4}}x_{{6}}x_{{11}},\,x_{{2}}x_{{6}}x_{{8}}x_{{11}},\,x_{{3}}x_{{6}}x_{{9}}x_{{11}},\,x_{{1}}x_{{5}}x_{{7}},\,x_{{1}}x_{{2}}x_{{3}}x_{{4}}x_{{5}}x_{{6}},\,x_{{2}}x_{{5}}x_{{8}},\,x_{{3}}x_{{5}}x_{{9}}\]}
\end{maplelatex}
\mapleresult
\begin{maplelatex}
\mapleinline{inert}{2d}{A = {12}, ftv = [Matrix(
\mbox{},0,0,0,0,0,1,1,1,1,1]]\]}
\end{maplelatex}
\mapleresult
\begin{maplelatex}
\mapleinline{inert}{2d}{Polynomials = {x[4]^3*x[10]^3+x[5]^3*x[11]^3+x[7]*x[8]*x[9]*x[10]*x[11]+x[6]^3, x[1]^3*x[7]^3+x[1]*x[2]*x[3]*x[4]*x[5]*x[6]+x[2]^3*x[8]^3+x[3]^3*x[9]^3}}{\[\displaystyle {\it Polynomials}= \left\{ {x_{{4}}}^{3}{x_{{10}}}^{3}+{x_{{5}}}^{3}{x_{{11}}}^{3}+x_{{7}}x_{{8}}x_{{9}}x_{{10}}x_{{11}}\\
\mbox{}+{x_{{6}}}^{3},{x_{{1}}}^{3}{x_{{7}}}^{3}+x_{{1}}x_{{2}}x_{{3}}x_{{4}}x_{{5}}x_{{6}}\\
\mbox{}+{x_{{2}}}^{3}{x_{{8}}}^{3}+{x_{{3}}}^{3}{x_{{9}}}^{3} \right\} \]}
\end{maplelatex}
\mapleresult
\begin{maplelatex}
\mapleinline{inert}{2d}{IrrIdeal = (x[1]*x[7]*x[8]*x[9]*x[10]*x[11], x[2]*x[7]*x[8]*x[9]*x[10]*x[11], x[3]*x[7]*x[8]*x[9]*x[10]*x[11], x[1]*x[4]*x[7]*x[10], x[1]*x[5]*x[7]*x[11], x[1]*x[2]*x[3]*x[4]*x[5]*x[10], x[2]*x[4]*x[8]*x[10])}{\[\displaystyle {\it IrrIdeal}={x_{{1}}x_{{7}}x_{{8}}x_{{9}}x_{{10}}\\
\mbox{}x_{{11}},x_{{2}}x_{{7}}x_{{8}}x_{{9}}x_{{10}}x_{{11}},x_{{3}}x_{{7}}\\
\mbox{}x_{{8}}x_{{9}}x_{{10}}x_{{11}},x_{{1}}x_{{4}}x_{{7}}x_{{10}},x_{{1}}\\
\mbox{}x_{{5}}x_{{7}}x_{{11}},x_{{1}}x_{{2}}x_{{3}}x_{{4}}x_{{5}}x_{{10}}\\
\mbox{},x_{{2}}x_{{4}}x_{{8}}x_{{10}}}\]}
\end{maplelatex}
\mapleresult
\begin{maplelatex}
\mapleinline{inert}{2d}{x[3]*x[4]*x[9]*x[10], x[1]*x[2]*x[3]*x[4]*x[5]*x[11], x[2]*x[5]*x[8]*x[11], x[3]*x[5]*x[9]*x[11], x[1]*x[6]*x[7], x[1]*x[2]*x[3]*x[4]*x[5]*x[6], x[2]*x[6]*x[8], x[3]*x[6]*x[9]}{\[\displaystyle x_{{3}}x_{{4}}x_{{9}}x_{{10}},\,x_{{1}}x_{{2}}x_{{3}}x_{{4}}x_{{5}}x_{{11}},\,x_{{2}}x_{{5}}x_{{8}}x_{{11}},\,x_{{3}}x_{{5}}x_{{9}}x_{{11}},\,x_{{1}}x_{{6}}x_{{7}},\,x_{{1}}x_{{2}}x_{{3}}x_{{4}}x_{{5}}x_{{6}},\,x_{{2}}x_{{6}}x_{{8}},\,x_{{3}}x_{{6}}x_{{9}}\]}
\end{maplelatex}
\mapleresult
\begin{maplelatex}
\mapleinline{inert}{2d}{A = {1, 2}, ftv = [Matrix(
\mbox{},0,0,1,1,1,1,1,1]]\]}
\end{maplelatex}
\mapleresult
\begin{maplelatex}
\mapleinline{inert}{2d}{Polynomials = {x[1]^3*x[7]^3+x[1]*x[2]*x[3]*x[4]+x[5]^3+x[6]^3, x[2]^3*x[8]^3+x[3]^3*x[9]^3+x[4]^3*x[10]^3+x[5]*x[6]*x[7]*x[8]*x[9]*x[10]}}{\[\displaystyle {\it Polynomials}= \left\{ {x_{{1}}}^{3}{x_{{7}}}^{3}+x_{{1}}x_{{2}}x_{{3}}x_{{4}}\\
\mbox{}+{x_{{5}}}^{3}+{x_{{6}}}^{3},{x_{{2}}}^{3}{x_{{8}}}^{3}+{x_{{3}}}^{3}{x_{{9}}}^{3}+{x_{{4}}}^{3}{x_{{10}}}^{3}\\
\mbox{}+x_{{5}}x_{{6}}x_{{7}}x_{{8}}x_{{9}}x_{{10}} \right\} \]}
\end{maplelatex}
\mapleresult
\begin{maplelatex}
\mapleinline{inert}{2d}{IrrIdeal = (x[5]*x[7]*x[8]*x[9]*x[10], x[6]*x[7]*x[8]*x[9]*x[10], x[1]*x[7]*x[8]*x[9]*x[10], x[1]*x[2]*x[7]*x[8], x[1]*x[3]*x[7]*x[9], x[1]*x[4]*x[7]*x[10], x[2]*x[5]*x[8])}{\[\displaystyle {\it IrrIdeal}={x_{{5}}x_{{7}}x_{{8}}x_{{9}}x_{{10}},x_{{6}}\\
\mbox{}x_{{7}}x_{{8}}x_{{9}}x_{{10}},x_{{1}}x_{{7}}x_{{8}}x_{{9}}x_{{10}}\\
\mbox{},x_{{1}}x_{{2}}x_{{7}}x_{{8}},x_{{1}}x_{{3}}x_{{7}}x_{{9}},x_{{1}}\\
\mbox{}x_{{4}}x_{{7}}x_{{10}},x_{{2}}x_{{5}}x_{{8}}}\]}
\end{maplelatex}
\mapleresult
\begin{maplelatex}
\mapleinline{inert}{2d}{x[2]*x[6]*x[8], x[3]*x[5]*x[9], x[3]*x[6]*x[9], x[4]*x[5]*x[10], x[4]*x[6]*x[10], x[1]*x[2]*x[3]*x[4]*x[8], x[1]*x[2]*x[3]*x[4]*x[9], x[1]*x[2]*x[3]*x[4]*x[10]}{\[\displaystyle x_{{2}}x_{{6}}x_{{8}},\,x_{{3}}x_{{5}}x_{{9}},\,x_{{3}}x_{{6}}x_{{9}},\,x_{{4}}x_{{5}}x_{{10}},\,x_{{4}}x_{{6}}x_{{10}},\,x_{{1}}x_{{2}}x_{{3}}x_{{4}}x_{{8}},\,x_{{1}}x_{{2}}x_{{3}}x_{{4}}x_{{9}},\,x_{{1}}x_{{2}}x_{{3}}x_{{4}}x_{{10}}\]}
\end{maplelatex}
\mapleresult
\begin{maplelatex}
\mapleinline{inert}{2d}{A = {1, 3}, ftv = [Matrix(
\mbox{},0,0,1,1,1,1,1,1]]\]}
\end{maplelatex}
\mapleresult
\begin{maplelatex}
\mapleinline{inert}{2d}{Polynomials = {x[1]^3*x[6]^3+x[1]*x[2]*x[3]*x[4]+x[5]^3+x[7]^3, x[2]^3*x[8]^3+x[3]^3*x[9]^3+x[4]^3*x[10]^3+x[5]*x[6]*x[7]*x[8]*x[9]*x[10]}}{\[\displaystyle {\it Polynomials}= \left\{ {x_{{1}}}^{3}{x_{{6}}}^{3}+x_{{1}}x_{{2}}x_{{3}}x_{{4}}\\
\mbox{}+{x_{{5}}}^{3}+{x_{{7}}}^{3},{x_{{2}}}^{3}{x_{{8}}}^{3}+{x_{{3}}}^{3}{x_{{9}}}^{3}+{x_{{4}}}^{3}{x_{{10}}}^{3}\\
\mbox{}+x_{{5}}x_{{6}}x_{{7}}x_{{8}}x_{{9}}x_{{10}} \right\} \]}
\end{maplelatex}
\mapleresult
\begin{maplelatex}
\mapleinline{inert}{2d}{IrrIdeal = (x[5]*x[6]*x[8]*x[9]*x[10], x[6]*x[7]*x[8]*x[9]*x[10], x[1]*x[6]*x[8]*x[9]*x[10], x[1]*x[2]*x[6]*x[8], x[1]*x[3]*x[6]*x[9], x[1]*x[4]*x[6]*x[10], x[2]*x[5]*x[8])}{\[\displaystyle {\it IrrIdeal}={x_{{5}}x_{{6}}x_{{8}}x_{{9}}x_{{10}},x_{{6}}\\
\mbox{}x_{{7}}x_{{8}}x_{{9}}x_{{10}},x_{{1}}x_{{6}}x_{{8}}x_{{9}}x_{{10}}\\
\mbox{},x_{{1}}x_{{2}}x_{{6}}x_{{8}},x_{{1}}x_{{3}}x_{{6}}x_{{9}},x_{{1}}\\
\mbox{}x_{{4}}x_{{6}}x_{{10}},x_{{2}}x_{{5}}x_{{8}}}\]}
\end{maplelatex}
\mapleresult
\begin{maplelatex}
\mapleinline{inert}{2d}{x[2]*x[7]*x[8], x[3]*x[5]*x[9], x[3]*x[7]*x[9], x[4]*x[5]*x[10], x[4]*x[7]*x[10], x[1]*x[2]*x[3]*x[4]*x[8], x[1]*x[2]*x[3]*x[4]*x[9], x[1]*x[2]*x[3]*x[4]*x[10]}{\[\displaystyle x_{{2}}x_{{7}}x_{{8}},\,x_{{3}}x_{{5}}x_{{9}},\,x_{{3}}x_{{7}}x_{{9}},\,x_{{4}}x_{{5}}x_{{10}},\,x_{{4}}x_{{7}}x_{{10}},\,x_{{1}}x_{{2}}x_{{3}}x_{{4}}x_{{8}},\,x_{{1}}x_{{2}}x_{{3}}x_{{4}}x_{{9}},\,x_{{1}}x_{{2}}x_{{3}}x_{{4}}x_{{10}}\]}
\end{maplelatex}
\mapleresult
\begin{maplelatex}
\mapleinline{inert}{2d}{A = {1, 10}, ftv = [Matrix(
\mbox{},0,0,0,1,1,1,1,1]]\]}
\end{maplelatex}
\mapleresult
\begin{maplelatex}
\mapleinline{inert}{2d}{Polynomials = {x[1]^3*x[7]^3+x[2]^3*x[8]^3+x[1]*x[2]*x[3]*x[4]*x[5]+x[6]^3, x[4]^3*x[9]^3+x[5]^3*x[10]^3+x[6]*x[7]*x[8]*x[9]*x[10]+x[3]^3}}{\[\displaystyle {\it Polynomials}= \left\{ {x_{{1}}}^{3}{x_{{7}}}^{3}+{x_{{2}}}^{3}{x_{{8}}}^{3}+x_{{1}}x_{{2}}x_{{3}}x_{{4}}x_{{5}}\\
\mbox{}+{x_{{6}}}^{3},{x_{{4}}}^{3}{x_{{9}}}^{3}+{x_{{5}}}^{3}{x_{{10}}}^{3}+x_{{6}}x_{{7}}x_{{8}}x_{{9}}x_{{10}}\\
\mbox{}+{x_{{3}}}^{3} \right\} \]}
\end{maplelatex}
\mapleresult
\begin{maplelatex}
\mapleinline{inert}{2d}{IrrIdeal = (x[6]*x[7]*x[8]*x[9]*x[10], x[1]*x[7]*x[8]*x[9]*x[10], x[2]*x[7]*x[8]*x[9]*x[10], x[1]*x[4]*x[7]*x[9], x[1]*x[5]*x[7]*x[10], x[4]*x[6]*x[9], x[2]*x[4]*x[8]*x[9])}{\[\displaystyle {\it IrrIdeal}={x_{{6}}x_{{7}}x_{{8}}x_{{9}}x_{{10}},x_{{1}}\\
\mbox{}x_{{7}}x_{{8}}x_{{9}}x_{{10}},x_{{2}}x_{{7}}x_{{8}}x_{{9}}x_{{10}}\\
\mbox{},x_{{1}}x_{{4}}x_{{7}}x_{{9}},x_{{1}}x_{{5}}x_{{7}}x_{{10}},x_{{4}}\\
\mbox{}x_{{6}}x_{{9}},x_{{2}}x_{{4}}x_{{8}}x_{{9}}}\]}
\end{maplelatex}
\mapleresult
\begin{maplelatex}
\mapleinline{inert}{2d}{x[5]*x[6]*x[10], x[2]*x[5]*x[8]*x[10], x[1]*x[2]*x[4]*x[5]*x[9], x[1]*x[2]*x[4]*x[5]*x[10], x[1]*x[3]*x[7], x[3]*x[6], x[2]*x[3]*x[8], x[1]*x[2]*x[3]*x[4]*x[5]}{\[\displaystyle x_{{5}}x_{{6}}x_{{10}},\,x_{{2}}x_{{5}}x_{{8}}x_{{10}},\,x_{{1}}x_{{2}}x_{{4}}x_{{5}}x_{{9}},\,x_{{1}}x_{{2}}x_{{4}}x_{{5}}x_{{10}},\,x_{{1}}x_{{3}}x_{{7}},\,x_{{3}}x_{{6}},\,x_{{2}}x_{{3}}x_{{8}},\,x_{{1}}x_{{2}}x_{{3}}x_{{4}}x_{{5}}\]}
\end{maplelatex}
\mapleresult
\begin{maplelatex}
\mapleinline{inert}{2d}{A = {1, 11}, ftv = [Matrix(
\mbox{},0,0,0,1,1,1,1,1]]\]}
\end{maplelatex}
\mapleresult
\begin{maplelatex}
\mapleinline{inert}{2d}{Polynomials = {x[1]^3*x[7]^3+x[2]^3*x[8]^3+x[1]*x[2]*x[3]*x[4]*x[5]+x[6]^3, x[3]^3*x[9]^3+x[5]^3*x[10]^3+x[6]*x[7]*x[8]*x[9]*x[10]+x[4]^3}}{\[\displaystyle {\it Polynomials}= \left\{ {x_{{1}}}^{3}{x_{{7}}}^{3}+{x_{{2}}}^{3}{x_{{8}}}^{3}+x_{{1}}x_{{2}}x_{{3}}x_{{4}}x_{{5}}\\
\mbox{}+{x_{{6}}}^{3},{x_{{3}}}^{3}{x_{{9}}}^{3}+{x_{{5}}}^{3}{x_{{10}}}^{3}+x_{{6}}x_{{7}}x_{{8}}x_{{9}}x_{{10}}\\
\mbox{}+{x_{{4}}}^{3} \right\} \]}
\end{maplelatex}
\mapleresult
\begin{maplelatex}
\mapleinline{inert}{2d}{IrrIdeal = (x[6]*x[7]*x[8]*x[9]*x[10], x[1]*x[7]*x[8]*x[9]*x[10], x[2]*x[7]*x[8]*x[9]*x[10], x[1]*x[3]*x[7]*x[9], x[1]*x[5]*x[7]*x[10], x[3]*x[6]*x[9], x[2]*x[3]*x[8]*x[9])}{\[\displaystyle {\it IrrIdeal}={x_{{6}}x_{{7}}x_{{8}}x_{{9}}x_{{10}},x_{{1}}\\
\mbox{}x_{{7}}x_{{8}}x_{{9}}x_{{10}},x_{{2}}x_{{7}}x_{{8}}x_{{9}}x_{{10}}\\
\mbox{},x_{{1}}x_{{3}}x_{{7}}x_{{9}},x_{{1}}x_{{5}}x_{{7}}x_{{10}},x_{{3}}\\
\mbox{}x_{{6}}x_{{9}},x_{{2}}x_{{3}}x_{{8}}x_{{9}}}\]}
\end{maplelatex}
\mapleresult
\begin{maplelatex}
\mapleinline{inert}{2d}{x[5]*x[6]*x[10], x[2]*x[5]*x[8]*x[10], x[1]*x[2]*x[3]*x[5]*x[9], x[1]*x[2]*x[3]*x[5]*x[10], x[1]*x[4]*x[7], x[4]*x[6], x[2]*x[4]*x[8], x[1]*x[2]*x[3]*x[4]*x[5]}{\[\displaystyle x_{{5}}x_{{6}}x_{{10}},\,x_{{2}}x_{{5}}x_{{8}}x_{{10}},\,x_{{1}}x_{{2}}x_{{3}}x_{{5}}x_{{9}},\,x_{{1}}x_{{2}}x_{{3}}x_{{5}}x_{{10}},\,x_{{1}}x_{{4}}x_{{7}},\,x_{{4}}x_{{6}},\,x_{{2}}x_{{4}}x_{{8}},\,x_{{1}}x_{{2}}x_{{3}}x_{{4}}x_{{5}}\]}
\end{maplelatex}
\mapleresult
\begin{maplelatex}
\mapleinline{inert}{2d}{A = {1, 12}, ftv = [Matrix(
\mbox{},0,0,0,1,1,1,1,1]]\]}
\end{maplelatex}
\mapleresult
\begin{maplelatex}
\mapleinline{inert}{2d}{Polynomials = {x[1]^3*x[7]^3+x[2]^3*x[8]^3+x[1]*x[2]*x[3]*x[4]*x[5]+x[6]^3, x[3]^3*x[9]^3+x[4]^3*x[10]^3+x[6]*x[7]*x[8]*x[9]*x[10]+x[5]^3}}{\[\displaystyle {\it Polynomials}= \left\{ {x_{{1}}}^{3}{x_{{7}}}^{3}+{x_{{2}}}^{3}{x_{{8}}}^{3}+x_{{1}}x_{{2}}x_{{3}}x_{{4}}x_{{5}}\\
\mbox{}+{x_{{6}}}^{3},{x_{{3}}}^{3}{x_{{9}}}^{3}+{x_{{4}}}^{3}{x_{{10}}}^{3}+x_{{6}}x_{{7}}x_{{8}}x_{{9}}x_{{10}}\\
\mbox{}+{x_{{5}}}^{3} \right\} \]}
\end{maplelatex}
\mapleresult
\begin{maplelatex}
\mapleinline{inert}{2d}{IrrIdeal = (x[6]*x[7]*x[8]*x[9]*x[10], x[1]*x[7]*x[8]*x[9]*x[10], x[2]*x[7]*x[8]*x[9]*x[10], x[1]*x[3]*x[7]*x[9], x[1]*x[4]*x[7]*x[10], x[3]*x[6]*x[9], x[2]*x[3]*x[8]*x[9])}{\[\displaystyle {\it IrrIdeal}={x_{{6}}x_{{7}}x_{{8}}x_{{9}}x_{{10}},x_{{1}}\\
\mbox{}x_{{7}}x_{{8}}x_{{9}}x_{{10}},x_{{2}}x_{{7}}x_{{8}}x_{{9}}x_{{10}}\\
\mbox{},x_{{1}}x_{{3}}x_{{7}}x_{{9}},x_{{1}}x_{{4}}x_{{7}}x_{{10}},x_{{3}}\\
\mbox{}x_{{6}}x_{{9}},x_{{2}}x_{{3}}x_{{8}}x_{{9}}}\]}
\end{maplelatex}
\mapleresult
\begin{maplelatex}
\mapleinline{inert}{2d}{x[4]*x[6]*x[10], x[2]*x[4]*x[8]*x[10], x[1]*x[2]*x[3]*x[4]*x[9], x[1]*x[2]*x[3]*x[4]*x[10], x[1]*x[5]*x[7], x[5]*x[6], x[2]*x[5]*x[8], x[1]*x[2]*x[3]*x[4]*x[5]}{\[\displaystyle x_{{4}}x_{{6}}x_{{10}},\,x_{{2}}x_{{4}}x_{{8}}x_{{10}},\,x_{{1}}x_{{2}}x_{{3}}x_{{4}}x_{{9}},\,x_{{1}}x_{{2}}x_{{3}}x_{{4}}x_{{10}},\,x_{{1}}x_{{5}}x_{{7}},\,x_{{5}}x_{{6}},\,x_{{2}}x_{{5}}x_{{8}},\,x_{{1}}x_{{2}}x_{{3}}x_{{4}}x_{{5}}\]}
\end{maplelatex}
\mapleresult
\begin{maplelatex}
\mapleinline{inert}{2d}{A = {2, 3}, ftv = [Matrix(
\mbox{},0,0,1,1,1,1,1,1]]\]}
\end{maplelatex}
\mapleresult
\begin{maplelatex}
\mapleinline{inert}{2d}{Polynomials = {x[1]^3*x[5]^3+x[1]*x[2]*x[3]*x[4]+x[6]^3+x[7]^3, x[2]^3*x[8]^3+x[3]^3*x[9]^3+x[4]^3*x[10]^3+x[5]*x[6]*x[7]*x[8]*x[9]*x[10]}}{\[\displaystyle {\it Polynomials}= \left\{ {x_{{1}}}^{3}{x_{{5}}}^{3}+x_{{1}}x_{{2}}x_{{3}}x_{{4}}\\
\mbox{}+{x_{{6}}}^{3}+{x_{{7}}}^{3},{x_{{2}}}^{3}{x_{{8}}}^{3}+{x_{{3}}}^{3}{x_{{9}}}^{3}+{x_{{4}}}^{3}{x_{{10}}}^{3}\\
\mbox{}+x_{{5}}x_{{6}}x_{{7}}x_{{8}}x_{{9}}x_{{10}} \right\} \]}
\end{maplelatex}
\mapleresult
\begin{maplelatex}
\mapleinline{inert}{2d}{IrrIdeal = (x[1]*x[5]*x[8]*x[9]*x[10], x[5]*x[6]*x[8]*x[9]*x[10], x[5]*x[7]*x[8]*x[9]*x[10], x[1]*x[2]*x[5]*x[8], x[1]*x[3]*x[5]*x[9], x[1]*x[4]*x[5]*x[10], x[1]*x[2]*x[3]*x[4]*x[8])}{\[\displaystyle {\it IrrIdeal}={x_{{1}}x_{{5}}x_{{8}}x_{{9}}x_{{10}},x_{{5}}\\
\mbox{}x_{{6}}x_{{8}}x_{{9}}x_{{10}},x_{{5}}x_{{7}}x_{{8}}x_{{9}}x_{{10}}\\
\mbox{},x_{{1}}x_{{2}}x_{{5}}x_{{8}},x_{{1}}x_{{3}}x_{{5}}x_{{9}},x_{{1}}\\
\mbox{}x_{{4}}x_{{5}}x_{{10}},x_{{1}}x_{{2}}x_{{3}}x_{{4}}x_{{8}}}\]}
\end{maplelatex}
\mapleresult
\begin{maplelatex}
\mapleinline{inert}{2d}{x[2]*x[6]*x[8], x[2]*x[7]*x[8], x[1]*x[2]*x[3]*x[4]*x[9], x[3]*x[6]*x[9], x[3]*x[7]*x[9], x[1]*x[2]*x[3]*x[4]*x[10], x[4]*x[6]*x[10], x[4]*x[7]*x[10]}{\[\displaystyle x_{{2}}x_{{6}}x_{{8}},\,x_{{2}}x_{{7}}x_{{8}},\,x_{{1}}x_{{2}}x_{{3}}x_{{4}}x_{{9}},\,x_{{3}}x_{{6}}x_{{9}},\,x_{{3}}x_{{7}}x_{{9}},\,x_{{1}}x_{{2}}x_{{3}}x_{{4}}x_{{10}},\,x_{{4}}x_{{6}}x_{{10}},\,x_{{4}}x_{{7}}x_{{10}}\]}
\end{maplelatex}
\mapleresult
\begin{maplelatex}
\mapleinline{inert}{2d}{A = {2, 10}, ftv = [Matrix(
\mbox{},0,0,0,1,1,1,1,1]]\]}
\end{maplelatex}
\mapleresult
\begin{maplelatex}
\mapleinline{inert}{2d}{Polynomials = {x[1]^3*x[6]^3+x[2]^3*x[8]^3+x[1]*x[2]*x[3]*x[4]*x[5]+x[7]^3, x[4]^3*x[9]^3+x[5]^3*x[10]^3+x[6]*x[7]*x[8]*x[9]*x[10]+x[3]^3}}{\[\displaystyle {\it Polynomials}= \left\{ {x_{{1}}}^{3}{x_{{6}}}^{3}+{x_{{2}}}^{3}{x_{{8}}}^{3}+x_{{1}}x_{{2}}x_{{3}}x_{{4}}x_{{5}}\\
\mbox{}+{x_{{7}}}^{3},{x_{{4}}}^{3}{x_{{9}}}^{3}+{x_{{5}}}^{3}{x_{{10}}}^{3}+x_{{6}}x_{{7}}x_{{8}}x_{{9}}x_{{10}}\\
\mbox{}+{x_{{3}}}^{3} \right\} \]}
\end{maplelatex}
\mapleresult
\begin{maplelatex}
\mapleinline{inert}{2d}{IrrIdeal = (x[1]*x[6]*x[8]*x[9]*x[10], x[6]*x[7]*x[8]*x[9]*x[10], x[2]*x[6]*x[8]*x[9]*x[10], x[1]*x[4]*x[6]*x[9], x[1]*x[5]*x[6]*x[10], x[1]*x[2]*x[4]*x[5]*x[9], x[4]*x[7]*x[9])}{\[\displaystyle {\it IrrIdeal}={x_{{1}}x_{{6}}x_{{8}}x_{{9}}x_{{10}},x_{{6}}\\
\mbox{}x_{{7}}x_{{8}}x_{{9}}x_{{10}},x_{{2}}x_{{6}}x_{{8}}x_{{9}}x_{{10}}\\
\mbox{},x_{{1}}x_{{4}}x_{{6}}x_{{9}},x_{{1}}x_{{5}}x_{{6}}x_{{10}},x_{{1}}\\
\mbox{}x_{{2}}x_{{4}}x_{{5}}x_{{9}},x_{{4}}x_{{7}}x_{{9}}}\]}
\end{maplelatex}
\mapleresult
\begin{maplelatex}
\mapleinline{inert}{2d}{x[2]*x[4]*x[8]*x[9], x[1]*x[2]*x[4]*x[5]*x[10], x[5]*x[7]*x[10], x[2]*x[5]*x[8]*x[10], x[1]*x[3]*x[6], x[1]*x[2]*x[3]*x[4]*x[5], x[3]*x[7], x[2]*x[3]*x[8]}{\[\displaystyle x_{{2}}x_{{4}}x_{{8}}x_{{9}},\,x_{{1}}x_{{2}}x_{{4}}x_{{5}}x_{{10}},\,x_{{5}}x_{{7}}x_{{10}},\,x_{{2}}x_{{5}}x_{{8}}x_{{10}},\,x_{{1}}x_{{3}}x_{{6}},\,x_{{1}}x_{{2}}x_{{3}}x_{{4}}x_{{5}},\,x_{{3}}x_{{7}},\,x_{{2}}x_{{3}}x_{{8}}\]}
\end{maplelatex}
\mapleresult
\begin{maplelatex}
\mapleinline{inert}{2d}{A = {2, 11}, ftv = [Matrix(
\mbox{},0,0,0,1,1,1,1,1]]\]}
\end{maplelatex}
\mapleresult
\begin{maplelatex}
\mapleinline{inert}{2d}{Polynomials = {x[1]^3*x[6]^3+x[2]^3*x[8]^3+x[1]*x[2]*x[3]*x[4]*x[5]+x[7]^3, x[3]^3*x[9]^3+x[5]^3*x[10]^3+x[6]*x[7]*x[8]*x[9]*x[10]+x[4]^3}}{\[\displaystyle {\it Polynomials}= \left\{ {x_{{1}}}^{3}{x_{{6}}}^{3}+{x_{{2}}}^{3}{x_{{8}}}^{3}+x_{{1}}x_{{2}}x_{{3}}x_{{4}}x_{{5}}\\
\mbox{}+{x_{{7}}}^{3},{x_{{3}}}^{3}{x_{{9}}}^{3}+{x_{{5}}}^{3}{x_{{10}}}^{3}+x_{{6}}x_{{7}}x_{{8}}x_{{9}}x_{{10}}\\
\mbox{}+{x_{{4}}}^{3} \right\} \]}
\end{maplelatex}
\mapleresult
\begin{maplelatex}
\mapleinline{inert}{2d}{IrrIdeal = (x[1]*x[6]*x[8]*x[9]*x[10], x[6]*x[7]*x[8]*x[9]*x[10], x[2]*x[6]*x[8]*x[9]*x[10], x[1]*x[3]*x[6]*x[9], x[1]*x[5]*x[6]*x[10], x[1]*x[2]*x[3]*x[5]*x[9], x[3]*x[7]*x[9])}{\[\displaystyle {\it IrrIdeal}={x_{{1}}x_{{6}}x_{{8}}x_{{9}}x_{{10}},x_{{6}}\\
\mbox{}x_{{7}}x_{{8}}x_{{9}}x_{{10}},x_{{2}}x_{{6}}x_{{8}}x_{{9}}x_{{10}}\\
\mbox{},x_{{1}}x_{{3}}x_{{6}}x_{{9}},x_{{1}}x_{{5}}x_{{6}}x_{{10}},x_{{1}}\\
\mbox{}x_{{2}}x_{{3}}x_{{5}}x_{{9}},x_{{3}}x_{{7}}x_{{9}}}\]}
\end{maplelatex}
\mapleresult
\begin{maplelatex}
\mapleinline{inert}{2d}{x[2]*x[3]*x[8]*x[9], x[1]*x[2]*x[3]*x[5]*x[10], x[5]*x[7]*x[10], x[2]*x[5]*x[8]*x[10], x[1]*x[4]*x[6], x[1]*x[2]*x[3]*x[4]*x[5], x[4]*x[7], x[2]*x[4]*x[8]}{\[\displaystyle x_{{2}}x_{{3}}x_{{8}}x_{{9}},\,x_{{1}}x_{{2}}x_{{3}}x_{{5}}x_{{10}},\,x_{{5}}x_{{7}}x_{{10}},\,x_{{2}}x_{{5}}x_{{8}}x_{{10}},\,x_{{1}}x_{{4}}x_{{6}},\,x_{{1}}x_{{2}}x_{{3}}x_{{4}}x_{{5}},\,x_{{4}}x_{{7}},\,x_{{2}}x_{{4}}x_{{8}}\]}
\end{maplelatex}
\mapleresult
\begin{maplelatex}
\mapleinline{inert}{2d}{A = {2, 12}, ftv = [Matrix(
\mbox{},0,0,0,1,1,1,1,1]]\]}
\end{maplelatex}
\mapleresult
\begin{maplelatex}
\mapleinline{inert}{2d}{Polynomials = {x[1]^3*x[6]^3+x[2]^3*x[8]^3+x[1]*x[2]*x[3]*x[4]*x[5]+x[7]^3, x[3]^3*x[9]^3+x[4]^3*x[10]^3+x[6]*x[7]*x[8]*x[9]*x[10]+x[5]^3}}{\[\displaystyle {\it Polynomials}= \left\{ {x_{{1}}}^{3}{x_{{6}}}^{3}+{x_{{2}}}^{3}{x_{{8}}}^{3}+x_{{1}}x_{{2}}x_{{3}}x_{{4}}x_{{5}}\\
\mbox{}+{x_{{7}}}^{3},{x_{{3}}}^{3}{x_{{9}}}^{3}+{x_{{4}}}^{3}{x_{{10}}}^{3}+x_{{6}}x_{{7}}x_{{8}}x_{{9}}x_{{10}}\\
\mbox{}+{x_{{5}}}^{3} \right\} \]}
\end{maplelatex}
\mapleresult
\begin{maplelatex}
\mapleinline{inert}{2d}{IrrIdeal = (x[1]*x[6]*x[8]*x[9]*x[10], x[6]*x[7]*x[8]*x[9]*x[10], x[2]*x[6]*x[8]*x[9]*x[10], x[1]*x[3]*x[6]*x[9], x[1]*x[4]*x[6]*x[10], x[1]*x[2]*x[3]*x[4]*x[9], x[3]*x[7]*x[9])}{\[\displaystyle {\it IrrIdeal}={x_{{1}}x_{{6}}x_{{8}}x_{{9}}x_{{10}},x_{{6}}\\
\mbox{}x_{{7}}x_{{8}}x_{{9}}x_{{10}},x_{{2}}x_{{6}}x_{{8}}x_{{9}}x_{{10}}\\
\mbox{},x_{{1}}x_{{3}}x_{{6}}x_{{9}},x_{{1}}x_{{4}}x_{{6}}x_{{10}},x_{{1}}\\
\mbox{}x_{{2}}x_{{3}}x_{{4}}x_{{9}},x_{{3}}x_{{7}}x_{{9}}}\]}
\end{maplelatex}
\mapleresult
\begin{maplelatex}
\mapleinline{inert}{2d}{x[2]*x[3]*x[8]*x[9], x[1]*x[2]*x[3]*x[4]*x[10], x[4]*x[7]*x[10], x[2]*x[4]*x[8]*x[10], x[1]*x[5]*x[6], x[1]*x[2]*x[3]*x[4]*x[5], x[5]*x[7], x[2]*x[5]*x[8]}{\[\displaystyle x_{{2}}x_{{3}}x_{{8}}x_{{9}},\,x_{{1}}x_{{2}}x_{{3}}x_{{4}}x_{{10}},\,x_{{4}}x_{{7}}x_{{10}},\,x_{{2}}x_{{4}}x_{{8}}x_{{10}},\,x_{{1}}x_{{5}}x_{{6}},\,x_{{1}}x_{{2}}x_{{3}}x_{{4}}x_{{5}},\,x_{{5}}x_{{7}},\,x_{{2}}x_{{5}}x_{{8}}\]}
\end{maplelatex}
\mapleresult
\begin{maplelatex}
\mapleinline{inert}{2d}{A = {3, 10}, ftv = [Matrix(
\mbox{},0,0,0,1,1,1,1,1]]\]}
\end{maplelatex}
\mapleresult
\begin{maplelatex}
\mapleinline{inert}{2d}{Polynomials = {x[1]^3*x[6]^3+x[2]^3*x[7]^3+x[1]*x[2]*x[3]*x[4]*x[5]+x[8]^3, x[4]^3*x[9]^3+x[5]^3*x[10]^3+x[6]*x[7]*x[8]*x[9]*x[10]+x[3]^3}}{\[\displaystyle {\it Polynomials}= \left\{ {x_{{1}}}^{3}{x_{{6}}}^{3}+{x_{{2}}}^{3}{x_{{7}}}^{3}+x_{{1}}x_{{2}}x_{{3}}x_{{4}}x_{{5}}\\
\mbox{}+{x_{{8}}}^{3},{x_{{4}}}^{3}{x_{{9}}}^{3}+{x_{{5}}}^{3}{x_{{10}}}^{3}+x_{{6}}x_{{7}}x_{{8}}x_{{9}}x_{{10}}\\
\mbox{}+{x_{{3}}}^{3} \right\} \]}
\end{maplelatex}
\mapleresult
\begin{maplelatex}
\mapleinline{inert}{2d}{IrrIdeal = (x[1]*x[6]*x[7]*x[9]*x[10], x[6]*x[7]*x[8]*x[9]*x[10], x[2]*x[6]*x[7]*x[9]*x[10], x[1]*x[4]*x[6]*x[9], x[1]*x[5]*x[6]*x[10], x[1]*x[2]*x[4]*x[5]*x[9], x[2]*x[4]*x[7]*x[9])}{\[\displaystyle {\it IrrIdeal}={x_{{1}}x_{{6}}x_{{7}}x_{{9}}x_{{10}},x_{{6}}\\
\mbox{}x_{{7}}x_{{8}}x_{{9}}x_{{10}},x_{{2}}x_{{6}}x_{{7}}x_{{9}}x_{{10}}\\
\mbox{},x_{{1}}x_{{4}}x_{{6}}x_{{9}},x_{{1}}x_{{5}}x_{{6}}x_{{10}},x_{{1}}\\
\mbox{}x_{{2}}x_{{4}}x_{{5}}x_{{9}},x_{{2}}x_{{4}}x_{{7}}x_{{9}}}\]}
\end{maplelatex}
\mapleresult
\begin{maplelatex}
\mapleinline{inert}{2d}{x[4]*x[8]*x[9], x[1]*x[2]*x[4]*x[5]*x[10], x[2]*x[5]*x[7]*x[10], x[5]*x[8]*x[10], x[1]*x[3]*x[6], x[1]*x[2]*x[3]*x[4]*x[5], x[2]*x[3]*x[7], x[3]*x[8]}{\[\displaystyle x_{{4}}x_{{8}}x_{{9}},\,x_{{1}}x_{{2}}x_{{4}}x_{{5}}x_{{10}},\,x_{{2}}x_{{5}}x_{{7}}x_{{10}},\,x_{{5}}x_{{8}}x_{{10}},\,x_{{1}}x_{{3}}x_{{6}},\,x_{{1}}x_{{2}}x_{{3}}x_{{4}}x_{{5}},\,x_{{2}}x_{{3}}x_{{7}},\,x_{{3}}x_{{8}}\]}
\end{maplelatex}
\mapleresult
\begin{maplelatex}
\mapleinline{inert}{2d}{A = {3, 11}, ftv = [Matrix(
\mbox{},0,0,0,1,1,1,1,1]]\]}
\end{maplelatex}
\mapleresult
\begin{maplelatex}
\mapleinline{inert}{2d}{Polynomials = {x[1]^3*x[6]^3+x[2]^3*x[7]^3+x[1]*x[2]*x[3]*x[4]*x[5]+x[8]^3, x[3]^3*x[9]^3+x[5]^3*x[10]^3+x[6]*x[7]*x[8]*x[9]*x[10]+x[4]^3}}{\[\displaystyle {\it Polynomials}= \left\{ {x_{{1}}}^{3}{x_{{6}}}^{3}+{x_{{2}}}^{3}{x_{{7}}}^{3}+x_{{1}}x_{{2}}x_{{3}}x_{{4}}x_{{5}}\\
\mbox{}+{x_{{8}}}^{3},{x_{{3}}}^{3}{x_{{9}}}^{3}+{x_{{5}}}^{3}{x_{{10}}}^{3}+x_{{6}}x_{{7}}x_{{8}}x_{{9}}x_{{10}}\\
\mbox{}+{x_{{4}}}^{3} \right\} \]}
\end{maplelatex}
\mapleresult
\begin{maplelatex}
\mapleinline{inert}{2d}{IrrIdeal = (x[1]*x[6]*x[7]*x[9]*x[10], x[6]*x[7]*x[8]*x[9]*x[10], x[2]*x[6]*x[7]*x[9]*x[10], x[1]*x[3]*x[6]*x[9], x[1]*x[5]*x[6]*x[10], x[1]*x[2]*x[3]*x[5]*x[9], x[2]*x[3]*x[7]*x[9])}{\[\displaystyle {\it IrrIdeal}={x_{{1}}x_{{6}}x_{{7}}x_{{9}}x_{{10}},x_{{6}}\\
\mbox{}x_{{7}}x_{{8}}x_{{9}}x_{{10}},x_{{2}}x_{{6}}x_{{7}}x_{{9}}x_{{10}}\\
\mbox{},x_{{1}}x_{{3}}x_{{6}}x_{{9}},x_{{1}}x_{{5}}x_{{6}}x_{{10}},x_{{1}}\\
\mbox{}x_{{2}}x_{{3}}x_{{5}}x_{{9}},x_{{2}}x_{{3}}x_{{7}}x_{{9}}}\]}
\end{maplelatex}
\mapleresult
\begin{maplelatex}
\mapleinline{inert}{2d}{x[3]*x[8]*x[9], x[1]*x[2]*x[3]*x[5]*x[10], x[2]*x[5]*x[7]*x[10], x[5]*x[8]*x[10], x[1]*x[4]*x[6], x[1]*x[2]*x[3]*x[4]*x[5], x[2]*x[4]*x[7], x[4]*x[8]}{\[\displaystyle x_{{3}}x_{{8}}x_{{9}},\,x_{{1}}x_{{2}}x_{{3}}x_{{5}}x_{{10}},\,x_{{2}}x_{{5}}x_{{7}}x_{{10}},\,x_{{5}}x_{{8}}x_{{10}},\,x_{{1}}x_{{4}}x_{{6}},\,x_{{1}}x_{{2}}x_{{3}}x_{{4}}x_{{5}},\,x_{{2}}x_{{4}}x_{{7}},\,x_{{4}}x_{{8}}\]}
\end{maplelatex}
\mapleresult
\begin{maplelatex}
\mapleinline{inert}{2d}{A = {3, 12}, ftv = [Matrix(
\mbox{},0,0,0,1,1,1,1,1]]\]}
\end{maplelatex}
\mapleresult
\begin{maplelatex}
\mapleinline{inert}{2d}{Polynomials = {x[1]^3*x[6]^3+x[2]^3*x[7]^3+x[1]*x[2]*x[3]*x[4]*x[5]+x[8]^3, x[3]^3*x[9]^3+x[4]^3*x[10]^3+x[6]*x[7]*x[8]*x[9]*x[10]+x[5]^3}}{\[\displaystyle {\it Polynomials}= \left\{ {x_{{1}}}^{3}{x_{{6}}}^{3}+{x_{{2}}}^{3}{x_{{7}}}^{3}+x_{{1}}x_{{2}}x_{{3}}x_{{4}}x_{{5}}\\
\mbox{}+{x_{{8}}}^{3},{x_{{3}}}^{3}{x_{{9}}}^{3}+{x_{{4}}}^{3}{x_{{10}}}^{3}+x_{{6}}x_{{7}}x_{{8}}x_{{9}}x_{{10}}\\
\mbox{}+{x_{{5}}}^{3} \right\} \]}
\end{maplelatex}
\mapleresult
\begin{maplelatex}
\mapleinline{inert}{2d}{IrrIdeal = (x[1]*x[6]*x[7]*x[9]*x[10], x[6]*x[7]*x[8]*x[9]*x[10], x[2]*x[6]*x[7]*x[9]*x[10], x[1]*x[3]*x[6]*x[9], x[1]*x[4]*x[6]*x[10], x[1]*x[2]*x[3]*x[4]*x[9], x[2]*x[3]*x[7]*x[9])}{\[\displaystyle {\it IrrIdeal}={x_{{1}}x_{{6}}x_{{7}}x_{{9}}x_{{10}},x_{{6}}\\
\mbox{}x_{{7}}x_{{8}}x_{{9}}x_{{10}},x_{{2}}x_{{6}}x_{{7}}x_{{9}}x_{{10}}\\
\mbox{},x_{{1}}x_{{3}}x_{{6}}x_{{9}},x_{{1}}x_{{4}}x_{{6}}x_{{10}},x_{{1}}\\
\mbox{}x_{{2}}x_{{3}}x_{{4}}x_{{9}},x_{{2}}x_{{3}}x_{{7}}x_{{9}}}\]}
\end{maplelatex}
\mapleresult
\begin{maplelatex}
\mapleinline{inert}{2d}{x[3]*x[8]*x[9], x[1]*x[2]*x[3]*x[4]*x[10], x[2]*x[4]*x[7]*x[10], x[4]*x[8]*x[10], x[1]*x[5]*x[6], x[1]*x[2]*x[3]*x[4]*x[5], x[2]*x[5]*x[7], x[5]*x[8]}{\[\displaystyle x_{{3}}x_{{8}}x_{{9}},\,x_{{1}}x_{{2}}x_{{3}}x_{{4}}x_{{10}},\,x_{{2}}x_{{4}}x_{{7}}x_{{10}},\,x_{{4}}x_{{8}}x_{{10}},\,x_{{1}}x_{{5}}x_{{6}},\,x_{{1}}x_{{2}}x_{{3}}x_{{4}}x_{{5}},\,x_{{2}}x_{{5}}x_{{7}},\,x_{{5}}x_{{8}}\]}
\end{maplelatex}
\mapleresult
\begin{maplelatex}
\mapleinline{inert}{2d}{A = {10, 11}, ftv = [Matrix(
\mbox{},0,0,0,0,0,1,1,1,1]]\]}
\end{maplelatex}
\mapleresult
\begin{maplelatex}
\mapleinline{inert}{2d}{Polynomials = {x[6]^3*x[10]^3+x[7]*x[8]*x[9]*x[10]+x[4]^3+x[5]^3, x[1]^3*x[7]^3+x[1]*x[2]*x[3]*x[4]*x[5]*x[6]+x[2]^3*x[8]^3+x[3]^3*x[9]^3}}{\[\displaystyle {\it Polynomials}= \left\{ {x_{{6}}}^{3}{x_{{10}}}^{3}+x_{{7}}x_{{8}}x_{{9}}x_{{10}}\\
\mbox{}+{x_{{4}}}^{3}+{x_{{5}}}^{3},{x_{{1}}}^{3}{x_{{7}}}^{3}+x_{{1}}x_{{2}}x_{{3}}x_{{4}}x_{{5}}x_{{6}}\\
\mbox{}+{x_{{2}}}^{3}{x_{{8}}}^{3}+{x_{{3}}}^{3}{x_{{9}}}^{3} \right\} \]}
\end{maplelatex}
\mapleresult
\begin{maplelatex}
\mapleinline{inert}{2d}{IrrIdeal = (x[1]*x[7]*x[8]*x[9]*x[10], x[2]*x[7]*x[8]*x[9]*x[10], x[3]*x[7]*x[8]*x[9]*x[10], x[1]*x[6]*x[7]*x[10], x[1]*x[2]*x[3]*x[6]*x[10], x[2]*x[6]*x[8]*x[10], x[3]*x[6]*x[9]*x[10])}{\[\displaystyle {\it IrrIdeal}={x_{{1}}x_{{7}}x_{{8}}x_{{9}}x_{{10}},x_{{2}}\\
\mbox{}x_{{7}}x_{{8}}x_{{9}}x_{{10}},x_{{3}}x_{{7}}x_{{8}}x_{{9}}x_{{10}}\\
\mbox{},x_{{1}}x_{{6}}x_{{7}}x_{{10}},x_{{1}}x_{{2}}x_{{3}}x_{{6}}x_{{10}}\\
\mbox{},x_{{2}}x_{{6}}x_{{8}}x_{{10}},x_{{3}}x_{{6}}x_{{9}}x_{{10}}}\]}
\end{maplelatex}
\mapleresult
\begin{maplelatex}
\mapleinline{inert}{2d}{x[1]*x[4]*x[7], x[1]*x[2]*x[3]*x[4]*x[6], x[2]*x[4]*x[8], x[3]*x[4]*x[9], x[1]*x[5]*x[7], x[1]*x[2]*x[3]*x[5]*x[6], x[2]*x[5]*x[8], x[3]*x[5]*x[9]}{\[\displaystyle x_{{1}}x_{{4}}x_{{7}},\,x_{{1}}x_{{2}}x_{{3}}x_{{4}}x_{{6}},\,x_{{2}}x_{{4}}x_{{8}},\,x_{{3}}x_{{4}}x_{{9}},\,x_{{1}}x_{{5}}x_{{7}},\,x_{{1}}x_{{2}}x_{{3}}x_{{5}}x_{{6}},\,x_{{2}}x_{{5}}x_{{8}},\,x_{{3}}x_{{5}}x_{{9}}\]}
\end{maplelatex}
\mapleresult
\begin{maplelatex}
\mapleinline{inert}{2d}{A = {10, 12}, ftv = [Matrix(
\mbox{},0,0,0,0,0,1,1,1,1]]\]}
\end{maplelatex}
\mapleresult
\begin{maplelatex}
\mapleinline{inert}{2d}{Polynomials = {x[5]^3*x[10]^3+x[7]*x[8]*x[9]*x[10]+x[4]^3+x[6]^3, x[1]^3*x[7]^3+x[1]*x[2]*x[3]*x[4]*x[5]*x[6]+x[2]^3*x[8]^3+x[3]^3*x[9]^3}}{\[\displaystyle {\it Polynomials}= \left\{ {x_{{5}}}^{3}{x_{{10}}}^{3}+x_{{7}}x_{{8}}x_{{9}}x_{{10}}\\
\mbox{}+{x_{{4}}}^{3}+{x_{{6}}}^{3},{x_{{1}}}^{3}{x_{{7}}}^{3}+x_{{1}}x_{{2}}x_{{3}}x_{{4}}x_{{5}}x_{{6}}\\
\mbox{}+{x_{{2}}}^{3}{x_{{8}}}^{3}+{x_{{3}}}^{3}{x_{{9}}}^{3} \right\} \]}
\end{maplelatex}
\mapleresult
\begin{maplelatex}
\mapleinline{inert}{2d}{IrrIdeal = (x[1]*x[7]*x[8]*x[9]*x[10], x[2]*x[7]*x[8]*x[9]*x[10], x[3]*x[7]*x[8]*x[9]*x[10], x[1]*x[5]*x[7]*x[10], x[1]*x[2]*x[3]*x[5]*x[10], x[2]*x[5]*x[8]*x[10], x[3]*x[5]*x[9]*x[10])}{\[\displaystyle {\it IrrIdeal}={x_{{1}}x_{{7}}x_{{8}}x_{{9}}x_{{10}},x_{{2}}\\
\mbox{}x_{{7}}x_{{8}}x_{{9}}x_{{10}},x_{{3}}x_{{7}}x_{{8}}x_{{9}}x_{{10}}\\
\mbox{},x_{{1}}x_{{5}}x_{{7}}x_{{10}},x_{{1}}x_{{2}}x_{{3}}x_{{5}}x_{{10}}\\
\mbox{},x_{{2}}x_{{5}}x_{{8}}x_{{10}},x_{{3}}x_{{5}}x_{{9}}x_{{10}}}\]}
\end{maplelatex}
\mapleresult
\begin{maplelatex}
\mapleinline{inert}{2d}{x[1]*x[4]*x[7], x[1]*x[2]*x[3]*x[4]*x[5], x[2]*x[4]*x[8], x[3]*x[4]*x[9], x[1]*x[6]*x[7], x[1]*x[2]*x[3]*x[5]*x[6], x[2]*x[6]*x[8], x[3]*x[6]*x[9]}{\[\displaystyle x_{{1}}x_{{4}}x_{{7}},\,x_{{1}}x_{{2}}x_{{3}}x_{{4}}x_{{5}},\,x_{{2}}x_{{4}}x_{{8}},\,x_{{3}}x_{{4}}x_{{9}},\,x_{{1}}x_{{6}}x_{{7}},\,x_{{1}}x_{{2}}x_{{3}}x_{{5}}x_{{6}},\,x_{{2}}x_{{6}}x_{{8}},\,x_{{3}}x_{{6}}x_{{9}}\]}
\end{maplelatex}
\mapleresult
\begin{maplelatex}
\mapleinline{inert}{2d}{A = {11, 12}, ftv = [Matrix(
\mbox{},0,0,0,0,0,1,1,1,1]]\]}
\end{maplelatex}
\mapleresult
\begin{maplelatex}
\mapleinline{inert}{2d}{Polynomials = {x[4]^3*x[10]^3+x[7]*x[8]*x[9]*x[10]+x[5]^3+x[6]^3, x[1]^3*x[7]^3+x[1]*x[2]*x[3]*x[4]*x[5]*x[6]+x[2]^3*x[8]^3+x[3]^3*x[9]^3}}{\[\displaystyle {\it Polynomials}= \left\{ {x_{{4}}}^{3}{x_{{10}}}^{3}+x_{{7}}x_{{8}}x_{{9}}x_{{10}}\\
\mbox{}+{x_{{5}}}^{3}+{x_{{6}}}^{3},{x_{{1}}}^{3}{x_{{7}}}^{3}+x_{{1}}x_{{2}}x_{{3}}x_{{4}}x_{{5}}x_{{6}}\\
\mbox{}+{x_{{2}}}^{3}{x_{{8}}}^{3}+{x_{{3}}}^{3}{x_{{9}}}^{3} \right\} \]}
\end{maplelatex}
\mapleresult
\begin{maplelatex}
\mapleinline{inert}{2d}{IrrIdeal = (x[1]*x[7]*x[8]*x[9]*x[10], x[2]*x[7]*x[8]*x[9]*x[10], x[3]*x[7]*x[8]*x[9]*x[10], x[1]*x[4]*x[7]*x[10], x[1]*x[2]*x[3]*x[4]*x[10], x[2]*x[4]*x[8]*x[10], x[3]*x[4]*x[9]*x[10])}{\[\displaystyle {\it IrrIdeal}={x_{{1}}x_{{7}}x_{{8}}x_{{9}}x_{{10}},x_{{2}}\\
\mbox{}x_{{7}}x_{{8}}x_{{9}}x_{{10}},x_{{3}}x_{{7}}x_{{8}}x_{{9}}x_{{10}}\\
\mbox{},x_{{1}}x_{{4}}x_{{7}}x_{{10}},x_{{1}}x_{{2}}x_{{3}}x_{{4}}x_{{10}}\\
\mbox{},x_{{2}}x_{{4}}x_{{8}}x_{{10}},x_{{3}}x_{{4}}x_{{9}}x_{{10}}}\]}
\end{maplelatex}
\mapleresult
\begin{maplelatex}
\mapleinline{inert}{2d}{x[1]*x[5]*x[7], x[1]*x[2]*x[3]*x[4]*x[5], x[2]*x[5]*x[8], x[3]*x[5]*x[9], x[1]*x[6]*x[7], x[1]*x[2]*x[3]*x[4]*x[6], x[2]*x[6]*x[8], x[3]*x[6]*x[9]}{\[\displaystyle x_{{1}}x_{{5}}x_{{7}},\,x_{{1}}x_{{2}}x_{{3}}x_{{4}}x_{{5}},\,x_{{2}}x_{{5}}x_{{8}},\,x_{{3}}x_{{5}}x_{{9}},\,x_{{1}}x_{{6}}x_{{7}},\,x_{{1}}x_{{2}}x_{{3}}x_{{4}}x_{{6}},\,x_{{2}}x_{{6}}x_{{8}},\,x_{{3}}x_{{6}}x_{{9}}\]}
\end{maplelatex}
\mapleresult
\begin{maplelatex}
\mapleinline{inert}{2d}{A = {1, 2, 3}, ftv = [Matrix(
\mbox{},1,1,1,1,1]]\]}
\end{maplelatex}
\mapleresult
\begin{maplelatex}
\mapleinline{inert}{2d}{Polynomials = {x[1]*x[2]*x[3]+x[4]^3+x[5]^3+x[6]^3, x[1]^3*x[7]^3+x[2]^3*x[8]^3+x[3]^3*x[9]^3+x[4]*x[5]*x[6]*x[7]*x[8]*x[9]}}{\[\displaystyle {\it Polynomials}= \left\{ x_{{1}}x_{{2}}x_{{3}}+{x_{{4}}}^{3}+{x_{{5}}}^{3}+{x_{{6}}}^{3}\\
\mbox{},{x_{{1}}}^{3}{x_{{7}}}^{3}+{x_{{2}}}^{3}{x_{{8}}}^{3}+{x_{{3}}}^{3}{x_{{9}}}^{3}+x_{{4}}x_{{5}}x_{{6}}x_{{7}}x_{{8}}x_{{9}}\\
\mbox{} \right\} \]}
\end{maplelatex}
\mapleresult
\begin{maplelatex}
\mapleinline{inert}{2d}{IrrIdeal = (x[4]*x[7]*x[8]*x[9], x[5]*x[7]*x[8]*x[9], x[6]*x[7]*x[8]*x[9], x[1]*x[4]*x[7], x[1]*x[5]*x[7], x[1]*x[6]*x[7], x[2]*x[4]*x[8])}{\[\displaystyle {\it IrrIdeal}={x_{{4}}x_{{7}}x_{{8}}x_{{9}},x_{{5}}x_{{7}}\\
\mbox{}x_{{8}}x_{{9}},x_{{6}}x_{{7}}x_{{8}}x_{{9}},x_{{1}}x_{{4}}x_{{7}},x_{{1}}\\
\mbox{}x_{{5}}x_{{7}},x_{{1}}x_{{6}}x_{{7}},x_{{2}}x_{{4}}x_{{8}}}\]}
\end{maplelatex}
\mapleresult
\begin{maplelatex}
\mapleinline{inert}{2d}{x[2]*x[5]*x[8], x[2]*x[6]*x[8], x[3]*x[4]*x[9], x[3]*x[5]*x[9], x[3]*x[6]*x[9], x[1]*x[2]*x[3]*x[7], x[1]*x[2]*x[3]*x[8], x[1]*x[2]*x[3]*x[9]}{\[\displaystyle x_{{2}}x_{{5}}x_{{8}},\,x_{{2}}x_{{6}}x_{{8}},\,x_{{3}}x_{{4}}x_{{9}},\,x_{{3}}x_{{5}}x_{{9}},\,x_{{3}}x_{{6}}x_{{9}},\,x_{{1}}x_{{2}}x_{{3}}x_{{7}},\,x_{{1}}x_{{2}}x_{{3}}x_{{8}},\,x_{{1}}x_{{2}}x_{{3}}x_{{9}}\]}
\end{maplelatex}
\mapleresult
\begin{maplelatex}
\mapleinline{inert}{2d}{A = {1, 2, 10}, ftv = [Matrix(
\mbox{},1,1,1,1,1]]\]}
\end{maplelatex}
\mapleresult
\begin{maplelatex}
\mapleinline{inert}{2d}{Polynomials = {x[1]^3*x[7]^3+x[1]*x[2]*x[3]*x[4]+x[5]^3+x[6]^3, x[3]^3*x[8]^3+x[4]^3*x[9]^3+x[5]*x[6]*x[7]*x[8]*x[9]+x[2]^3}}{\[\displaystyle {\it Polynomials}= \left\{ {x_{{1}}}^{3}{x_{{7}}}^{3}+x_{{1}}x_{{2}}x_{{3}}x_{{4}}\\
\mbox{}+{x_{{5}}}^{3}+{x_{{6}}}^{3},{x_{{3}}}^{3}{x_{{8}}}^{3}+{x_{{4}}}^{3}{x_{{9}}}^{3}+x_{{5}}x_{{6}}x_{{7}}x_{{8}}x_{{9}}\\
\mbox{}+{x_{{2}}}^{3} \right\} \]}
\end{maplelatex}
\mapleresult
\begin{maplelatex}
\mapleinline{inert}{2d}{IrrIdeal = (x[5]*x[7]*x[8]*x[9], x[6]*x[7]*x[8]*x[9], x[1]*x[7]*x[8]*x[9], x[1]*x[3]*x[7]*x[8], x[1]*x[4]*x[7]*x[9], x[1]*x[2]*x[7], x[3]*x[5]*x[8])}{\[\displaystyle {\it IrrIdeal}={x_{{5}}x_{{7}}x_{{8}}x_{{9}},x_{{6}}x_{{7}}\\
\mbox{}x_{{8}}x_{{9}},x_{{1}}x_{{7}}x_{{8}}x_{{9}},x_{{1}}x_{{3}}x_{{7}}\\
\mbox{}x_{{8}},x_{{1}}x_{{4}}x_{{7}}x_{{9}},x_{{1}}x_{{2}}x_{{7}},x_{{3}}x_{{5}}\\
\mbox{}x_{{8}}}\]}
\end{maplelatex}
\mapleresult
\begin{maplelatex}
\mapleinline{inert}{2d}{x[3]*x[6]*x[8], x[4]*x[5]*x[9], x[4]*x[6]*x[9], x[1]*x[3]*x[4]*x[8], x[1]*x[3]*x[4]*x[9], x[2]*x[5], x[2]*x[6], x[1]*x[2]*x[3]*x[4]}{\[\displaystyle x_{{3}}x_{{6}}x_{{8}},\,x_{{4}}x_{{5}}x_{{9}},\,x_{{4}}x_{{6}}x_{{9}},\,x_{{1}}x_{{3}}x_{{4}}x_{{8}},\,x_{{1}}x_{{3}}x_{{4}}x_{{9}},\,x_{{2}}x_{{5}},\,x_{{2}}x_{{6}},\,x_{{1}}x_{{2}}x_{{3}}x_{{4}}\]}
\end{maplelatex}
\mapleresult
\begin{maplelatex}
\mapleinline{inert}{2d}{A = {1, 2, 11}, ftv = [Matrix(
\mbox{},1,1,1,1,1]]\]}
\end{maplelatex}
\mapleresult
\begin{maplelatex}
\mapleinline{inert}{2d}{Polynomials = {x[1]^3*x[7]^3+x[1]*x[2]*x[3]*x[4]+x[5]^3+x[6]^3, x[2]^3*x[8]^3+x[4]^3*x[9]^3+x[5]*x[6]*x[7]*x[8]*x[9]+x[3]^3}}{\[\displaystyle {\it Polynomials}= \left\{ {x_{{1}}}^{3}{x_{{7}}}^{3}+x_{{1}}x_{{2}}x_{{3}}x_{{4}}\\
\mbox{}+{x_{{5}}}^{3}+{x_{{6}}}^{3},{x_{{2}}}^{3}{x_{{8}}}^{3}+{x_{{4}}}^{3}{x_{{9}}}^{3}+x_{{5}}x_{{6}}x_{{7}}x_{{8}}x_{{9}}\\
\mbox{}+{x_{{3}}}^{3} \right\} \]}
\end{maplelatex}
\mapleresult
\begin{maplelatex}
\mapleinline{inert}{2d}{IrrIdeal = (x[5]*x[7]*x[8]*x[9], x[6]*x[7]*x[8]*x[9], x[1]*x[7]*x[8]*x[9], x[1]*x[2]*x[7]*x[8], x[1]*x[4]*x[7]*x[9], x[2]*x[5]*x[8], x[2]*x[6]*x[8])}{\[\displaystyle {\it IrrIdeal}={x_{{5}}x_{{7}}x_{{8}}x_{{9}},x_{{6}}x_{{7}}\\
\mbox{}x_{{8}}x_{{9}},x_{{1}}x_{{7}}x_{{8}}x_{{9}},x_{{1}}x_{{2}}x_{{7}}\\
\mbox{}x_{{8}},x_{{1}}x_{{4}}x_{{7}}x_{{9}},x_{{2}}x_{{5}}x_{{8}},x_{{2}}x_{{6}}\\
\mbox{}x_{{8}}}\]}
\end{maplelatex}
\mapleresult
\begin{maplelatex}
\mapleinline{inert}{2d}{x[4]*x[5]*x[9], x[4]*x[6]*x[9], x[1]*x[2]*x[4]*x[8], x[1]*x[2]*x[4]*x[9], x[1]*x[3]*x[7], x[3]*x[5], x[3]*x[6], x[1]*x[2]*x[3]*x[4]}{\[\displaystyle x_{{4}}x_{{5}}x_{{9}},\,x_{{4}}x_{{6}}x_{{9}},\,x_{{1}}x_{{2}}x_{{4}}x_{{8}},\,x_{{1}}x_{{2}}x_{{4}}x_{{9}},\,x_{{1}}x_{{3}}x_{{7}},\,x_{{3}}x_{{5}},\,x_{{3}}x_{{6}},\,x_{{1}}x_{{2}}x_{{3}}x_{{4}}\]}
\end{maplelatex}
\mapleresult
\begin{maplelatex}
\mapleinline{inert}{2d}{A = {1, 2, 12}, ftv = [Matrix(
\mbox{},1,1,1,1]]\]}
\end{maplelatex}
\mapleresult
\begin{maplelatex}
\mapleinline{inert}{2d}{Polynomials = {x[1]^3*x[7]^3+x[1]*x[2]*x[3]*x[4]+x[5]^3+x[6]^3, x[2]^3*x[8]^3+x[3]^3*x[9]^3+x[5]*x[6]*x[7]*x[8]*x[9]+x[4]^3}}{\[\displaystyle {\it Polynomials}= \left\{ {x_{{1}}}^{3}{x_{{7}}}^{3}+x_{{1}}x_{{2}}x_{{3}}x_{{4}}\\
\mbox{}+{x_{{5}}}^{3}+{x_{{6}}}^{3},{x_{{2}}}^{3}{x_{{8}}}^{3}+{x_{{3}}}^{3}{x_{{9}}}^{3}+x_{{5}}x_{{6}}x_{{7}}x_{{8}}x_{{9}}\\
\mbox{}+{x_{{4}}}^{3} \right\} \]}
\end{maplelatex}
\mapleresult
\begin{maplelatex}
\mapleinline{inert}{2d}{IrrIdeal = (x[5]*x[7]*x[8]*x[9], x[6]*x[7]*x[8]*x[9], x[1]*x[7]*x[8]*x[9], x[1]*x[2]*x[7]*x[8], x[1]*x[3]*x[7]*x[9], x[2]*x[5]*x[8], x[2]*x[6]*x[8])}{\[\displaystyle {\it IrrIdeal}={x_{{5}}x_{{7}}x_{{8}}x_{{9}},x_{{6}}x_{{7}}\\
\mbox{}x_{{8}}x_{{9}},x_{{1}}x_{{7}}x_{{8}}x_{{9}},x_{{1}}x_{{2}}x_{{7}}\\
\mbox{}x_{{8}},x_{{1}}x_{{3}}x_{{7}}x_{{9}},x_{{2}}x_{{5}}x_{{8}},x_{{2}}x_{{6}}\\
\mbox{}x_{{8}}}\]}
\end{maplelatex}
\mapleresult
\begin{maplelatex}
\mapleinline{inert}{2d}{x[3]*x[5]*x[9], x[3]*x[6]*x[9], x[1]*x[2]*x[3]*x[8], x[1]*x[2]*x[3]*x[9], x[1]*x[4]*x[7], x[4]*x[5], x[4]*x[6], x[1]*x[2]*x[3]*x[4]}{\[\displaystyle x_{{3}}x_{{5}}x_{{9}},\,x_{{3}}x_{{6}}x_{{9}},\,x_{{1}}x_{{2}}x_{{3}}x_{{8}},\,x_{{1}}x_{{2}}x_{{3}}x_{{9}},\,x_{{1}}x_{{4}}x_{{7}},\,x_{{4}}x_{{5}},\,x_{{4}}x_{{6}},\,x_{{1}}x_{{2}}x_{{3}}x_{{4}}\]}
\end{maplelatex}
\mapleresult
\begin{maplelatex}
\mapleinline{inert}{2d}{A = {1, 3, 10}, ftv = [Matrix(
\mbox{},1,1,1,1,1]]\]}
\end{maplelatex}
\mapleresult
\begin{maplelatex}
\mapleinline{inert}{2d}{Polynomials = {x[1]^3*x[6]^3+x[1]*x[2]*x[3]*x[4]+x[5]^3+x[7]^3, x[3]^3*x[8]^3+x[4]^3*x[9]^3+x[5]*x[6]*x[7]*x[8]*x[9]+x[2]^3}}{\[\displaystyle {\it Polynomials}= \left\{ {x_{{1}}}^{3}{x_{{6}}}^{3}+x_{{1}}x_{{2}}x_{{3}}x_{{4}}\\
\mbox{}+{x_{{5}}}^{3}+{x_{{7}}}^{3},{x_{{3}}}^{3}{x_{{8}}}^{3}+{x_{{4}}}^{3}{x_{{9}}}^{3}+x_{{5}}x_{{6}}x_{{7}}x_{{8}}x_{{9}}\\
\mbox{}+{x_{{2}}}^{3} \right\} \]}
\end{maplelatex}
\mapleresult
\begin{maplelatex}
\mapleinline{inert}{2d}{IrrIdeal = (x[5]*x[6]*x[8]*x[9], x[6]*x[7]*x[8]*x[9], x[1]*x[6]*x[8]*x[9], x[1]*x[3]*x[6]*x[8], x[1]*x[4]*x[6]*x[9], x[1]*x[2]*x[6], x[3]*x[5]*x[8])}{\[\displaystyle {\it IrrIdeal}={x_{{5}}x_{{6}}x_{{8}}x_{{9}},x_{{6}}x_{{7}}\\
\mbox{}x_{{8}}x_{{9}},x_{{1}}x_{{6}}x_{{8}}x_{{9}},x_{{1}}x_{{3}}x_{{6}}\\
\mbox{}x_{{8}},x_{{1}}x_{{4}}x_{{6}}x_{{9}},x_{{1}}x_{{2}}x_{{6}},x_{{3}}x_{{5}}\\
\mbox{}x_{{8}}}\]}
\end{maplelatex}
\mapleresult
\begin{maplelatex}
\mapleinline{inert}{2d}{x[3]*x[7]*x[8], x[4]*x[5]*x[9], x[4]*x[7]*x[9], x[1]*x[3]*x[4]*x[8], x[1]*x[3]*x[4]*x[9], x[2]*x[5], x[2]*x[7], x[1]*x[2]*x[3]*x[4]}{\[\displaystyle x_{{3}}x_{{7}}x_{{8}},\,x_{{4}}x_{{5}}x_{{9}},\,x_{{4}}x_{{7}}x_{{9}},\,x_{{1}}x_{{3}}x_{{4}}x_{{8}},\,x_{{1}}x_{{3}}x_{{4}}x_{{9}},\,x_{{2}}x_{{5}},\,x_{{2}}x_{{7}},\,x_{{1}}x_{{2}}x_{{3}}x_{{4}}\]}
\end{maplelatex}
\mapleresult
\begin{maplelatex}
\mapleinline{inert}{2d}{A = {1, 3, 11}, ftv = [Matrix(
\mbox{},1,1,1,1,1]]\]}
\end{maplelatex}
\mapleresult
\begin{maplelatex}
\mapleinline{inert}{2d}{Polynomials = {x[1]^3*x[6]^3+x[1]*x[2]*x[3]*x[4]+x[5]^3+x[7]^3, x[2]^3*x[8]^3+x[4]^3*x[9]^3+x[5]*x[6]*x[7]*x[8]*x[9]+x[3]^3}}{\[\displaystyle {\it Polynomials}= \left\{ {x_{{1}}}^{3}{x_{{6}}}^{3}+x_{{1}}x_{{2}}x_{{3}}x_{{4}}\\
\mbox{}+{x_{{5}}}^{3}+{x_{{7}}}^{3},{x_{{2}}}^{3}{x_{{8}}}^{3}+{x_{{4}}}^{3}{x_{{9}}}^{3}+x_{{5}}x_{{6}}x_{{7}}x_{{8}}x_{{9}}\\
\mbox{}+{x_{{3}}}^{3} \right\} \]}
\end{maplelatex}
\mapleresult
\begin{maplelatex}
\mapleinline{inert}{2d}{IrrIdeal = (x[5]*x[6]*x[8]*x[9], x[6]*x[7]*x[8]*x[9], x[1]*x[6]*x[8]*x[9], x[1]*x[2]*x[6]*x[8], x[1]*x[4]*x[6]*x[9], x[2]*x[5]*x[8], x[2]*x[7]*x[8])}{\[\displaystyle {\it IrrIdeal}={x_{{5}}x_{{6}}x_{{8}}x_{{9}},x_{{6}}x_{{7}}\\
\mbox{}x_{{8}}x_{{9}},x_{{1}}x_{{6}}x_{{8}}x_{{9}},x_{{1}}x_{{2}}x_{{6}}\\
\mbox{}x_{{8}},x_{{1}}x_{{4}}x_{{6}}x_{{9}},x_{{2}}x_{{5}}x_{{8}},x_{{2}}x_{{7}}\\
\mbox{}x_{{8}}}\]}
\end{maplelatex}
\mapleresult
\begin{maplelatex}
\mapleinline{inert}{2d}{x[4]*x[5]*x[9], x[4]*x[7]*x[9], x[1]*x[2]*x[4]*x[8], x[1]*x[2]*x[4]*x[9], x[1]*x[3]*x[6], x[3]*x[5], x[3]*x[7], x[1]*x[2]*x[3]*x[4]}{\[\displaystyle x_{{4}}x_{{5}}x_{{9}},\,x_{{4}}x_{{7}}x_{{9}},\,x_{{1}}x_{{2}}x_{{4}}x_{{8}},\,x_{{1}}x_{{2}}x_{{4}}x_{{9}},\,x_{{1}}x_{{3}}x_{{6}},\,x_{{3}}x_{{5}},\,x_{{3}}x_{{7}},\,x_{{1}}x_{{2}}x_{{3}}x_{{4}}\]}
\end{maplelatex}
\mapleresult
\begin{maplelatex}
\mapleinline{inert}{2d}{A = {1, 3, 12}, ftv = [Matrix(
\mbox{},1,1,1,1]]\]}
\end{maplelatex}
\mapleresult
\begin{maplelatex}
\mapleinline{inert}{2d}{Polynomials = {x[1]^3*x[6]^3+x[1]*x[2]*x[3]*x[4]+x[5]^3+x[7]^3, x[2]^3*x[8]^3+x[3]^3*x[9]^3+x[5]*x[6]*x[7]*x[8]*x[9]+x[4]^3}}{\[\displaystyle {\it Polynomials}= \left\{ {x_{{1}}}^{3}{x_{{6}}}^{3}+x_{{1}}x_{{2}}x_{{3}}x_{{4}}\\
\mbox{}+{x_{{5}}}^{3}+{x_{{7}}}^{3},{x_{{2}}}^{3}{x_{{8}}}^{3}+{x_{{3}}}^{3}{x_{{9}}}^{3}+x_{{5}}x_{{6}}x_{{7}}x_{{8}}x_{{9}}\\
\mbox{}+{x_{{4}}}^{3} \right\} \]}
\end{maplelatex}
\mapleresult
\begin{maplelatex}
\mapleinline{inert}{2d}{IrrIdeal = (x[5]*x[6]*x[8]*x[9], x[6]*x[7]*x[8]*x[9], x[1]*x[6]*x[8]*x[9], x[1]*x[2]*x[6]*x[8], x[1]*x[3]*x[6]*x[9], x[2]*x[5]*x[8], x[2]*x[7]*x[8])}{\[\displaystyle {\it IrrIdeal}={x_{{5}}x_{{6}}x_{{8}}x_{{9}},x_{{6}}x_{{7}}\\
\mbox{}x_{{8}}x_{{9}},x_{{1}}x_{{6}}x_{{8}}x_{{9}},x_{{1}}x_{{2}}x_{{6}}\\
\mbox{}x_{{8}},x_{{1}}x_{{3}}x_{{6}}x_{{9}},x_{{2}}x_{{5}}x_{{8}},x_{{2}}x_{{7}}\\
\mbox{}x_{{8}}}\]}
\end{maplelatex}
\mapleresult
\begin{maplelatex}
\mapleinline{inert}{2d}{x[3]*x[5]*x[9], x[3]*x[7]*x[9], x[1]*x[2]*x[3]*x[8], x[1]*x[2]*x[3]*x[9], x[1]*x[4]*x[6], x[4]*x[5], x[4]*x[7], x[1]*x[2]*x[3]*x[4]}{\[\displaystyle x_{{3}}x_{{5}}x_{{9}},\,x_{{3}}x_{{7}}x_{{9}},\,x_{{1}}x_{{2}}x_{{3}}x_{{8}},\,x_{{1}}x_{{2}}x_{{3}}x_{{9}},\,x_{{1}}x_{{4}}x_{{6}},\,x_{{4}}x_{{5}},\,x_{{4}}x_{{7}},\,x_{{1}}x_{{2}}x_{{3}}x_{{4}}\]}
\end{maplelatex}
\mapleresult
\begin{maplelatex}
\mapleinline{inert}{2d}{A = {1, 10, 11}, ftv = [Matrix(
\mbox{},0,0,1,1,1,1]]\]}
\end{maplelatex}
\mapleresult
\begin{maplelatex}
\mapleinline{inert}{2d}{Polynomials = {x[5]^3*x[9]^3+x[6]*x[7]*x[8]*x[9]+x[3]^3+x[4]^3, x[1]^3*x[7]^3+x[2]^3*x[8]^3+x[1]*x[2]*x[3]*x[4]*x[5]+x[6]^3}}{\[\displaystyle {\it Polynomials}= \left\{ {x_{{5}}}^{3}{x_{{9}}}^{3}+x_{{6}}x_{{7}}x_{{8}}x_{{9}}\\
\mbox{}+{x_{{3}}}^{3}+{x_{{4}}}^{3},{x_{{1}}}^{3}{x_{{7}}}^{3}+{x_{{2}}}^{3}{x_{{8}}}^{3}+x_{{1}}x_{{2}}x_{{3}}x_{{4}}x_{{5}}\\
\mbox{}+{x_{{6}}}^{3} \right\} \]}
\end{maplelatex}
\mapleresult
\begin{maplelatex}
\mapleinline{inert}{2d}{IrrIdeal = (x[6]*x[7]*x[8]*x[9], x[1]*x[7]*x[8]*x[9], x[2]*x[7]*x[8]*x[9], x[1]*x[5]*x[7]*x[9], x[5]*x[6]*x[9], x[2]*x[5]*x[8]*x[9], x[1]*x[2]*x[5]*x[9])}{\[\displaystyle {\it IrrIdeal}={x_{{6}}x_{{7}}x_{{8}}x_{{9}},x_{{1}}x_{{7}}\\
\mbox{}x_{{8}}x_{{9}},x_{{2}}x_{{7}}x_{{8}}x_{{9}},x_{{1}}x_{{5}}x_{{7}}\\
\mbox{}x_{{9}},x_{{5}}x_{{6}}x_{{9}},x_{{2}}x_{{5}}x_{{8}}x_{{9}},x_{{1}}x_{{2}}\\
\mbox{}x_{{5}}x_{{9}}}\]}
\end{maplelatex}
\mapleresult
\begin{maplelatex}
\mapleinline{inert}{2d}{x[1]*x[3]*x[7], x[3]*x[6], x[2]*x[3]*x[8], x[1]*x[2]*x[3]*x[5], x[1]*x[4]*x[7], x[4]*x[6], x[2]*x[4]*x[8], x[1]*x[2]*x[4]*x[5]}{\[\displaystyle x_{{1}}x_{{3}}x_{{7}},\,x_{{3}}x_{{6}},\,x_{{2}}x_{{3}}x_{{8}},\,x_{{1}}x_{{2}}x_{{3}}x_{{5}},\,x_{{1}}x_{{4}}x_{{7}},\,x_{{4}}x_{{6}},\,x_{{2}}x_{{4}}x_{{8}},\,x_{{1}}x_{{2}}x_{{4}}x_{{5}}\]}
\end{maplelatex}
\mapleresult
\begin{maplelatex}
\mapleinline{inert}{2d}{A = {1, 10, 12}, ftv = [Matrix(
\mbox{},0,0,1,1,1,1]]\]}
\end{maplelatex}
\mapleresult
\begin{maplelatex}
\mapleinline{inert}{2d}{Polynomials = {x[4]^3*x[9]^3+x[6]*x[7]*x[8]*x[9]+x[3]^3+x[5]^3, x[1]^3*x[7]^3+x[2]^3*x[8]^3+x[1]*x[2]*x[3]*x[4]*x[5]+x[6]^3}}{\[\displaystyle {\it Polynomials}= \left\{ {x_{{4}}}^{3}{x_{{9}}}^{3}+x_{{6}}x_{{7}}x_{{8}}x_{{9}}\\
\mbox{}+{x_{{3}}}^{3}+{x_{{5}}}^{3},{x_{{1}}}^{3}{x_{{7}}}^{3}+{x_{{2}}}^{3}{x_{{8}}}^{3}+x_{{1}}x_{{2}}x_{{3}}x_{{4}}x_{{5}}\\
\mbox{}+{x_{{6}}}^{3} \right\} \]}
\end{maplelatex}
\mapleresult
\begin{maplelatex}
\mapleinline{inert}{2d}{IrrIdeal = (x[6]*x[7]*x[8]*x[9], x[1]*x[7]*x[8]*x[9], x[2]*x[7]*x[8]*x[9], x[1]*x[4]*x[7]*x[9], x[4]*x[6]*x[9], x[2]*x[4]*x[8]*x[9], x[1]*x[2]*x[4]*x[9])}{\[\displaystyle {\it IrrIdeal}={x_{{6}}x_{{7}}x_{{8}}x_{{9}},x_{{1}}x_{{7}}\\
\mbox{}x_{{8}}x_{{9}},x_{{2}}x_{{7}}x_{{8}}x_{{9}},x_{{1}}x_{{4}}x_{{7}}\\
\mbox{}x_{{9}},x_{{4}}x_{{6}}x_{{9}},x_{{2}}x_{{4}}x_{{8}}x_{{9}},x_{{1}}x_{{2}}\\
\mbox{}x_{{4}}x_{{9}}}\]}
\end{maplelatex}
\mapleresult
\begin{maplelatex}
\mapleinline{inert}{2d}{x[1]*x[3]*x[7], x[3]*x[6], x[2]*x[3]*x[8], x[1]*x[2]*x[3]*x[4], x[1]*x[5]*x[7], x[5]*x[6], x[2]*x[5]*x[8], x[1]*x[2]*x[4]*x[5]}{\[\displaystyle x_{{1}}x_{{3}}x_{{7}},\,x_{{3}}x_{{6}},\,x_{{2}}x_{{3}}x_{{8}},\,x_{{1}}x_{{2}}x_{{3}}x_{{4}},\,x_{{1}}x_{{5}}x_{{7}},\,x_{{5}}x_{{6}},\,x_{{2}}x_{{5}}x_{{8}},\,x_{{1}}x_{{2}}x_{{4}}x_{{5}}\]}
\end{maplelatex}
\mapleresult
\begin{maplelatex}
\mapleinline{inert}{2d}{A = {1, 11, 12}, ftv = [Matrix(
\mbox{},0,0,1,1,1,1]]\]}
\end{maplelatex}
\mapleresult
\begin{maplelatex}
\mapleinline{inert}{2d}{Polynomials = {x[3]^3*x[9]^3+x[6]*x[7]*x[8]*x[9]+x[4]^3+x[5]^3, x[1]^3*x[7]^3+x[2]^3*x[8]^3+x[1]*x[2]*x[3]*x[4]*x[5]+x[6]^3}}{\[\displaystyle {\it Polynomials}= \left\{ {x_{{3}}}^{3}{x_{{9}}}^{3}+x_{{6}}x_{{7}}x_{{8}}x_{{9}}\\
\mbox{}+{x_{{4}}}^{3}+{x_{{5}}}^{3},{x_{{1}}}^{3}{x_{{7}}}^{3}+{x_{{2}}}^{3}{x_{{8}}}^{3}+x_{{1}}x_{{2}}x_{{3}}x_{{4}}x_{{5}}\\
\mbox{}+{x_{{6}}}^{3} \right\} \]}
\end{maplelatex}
\mapleresult
\begin{maplelatex}
\mapleinline{inert}{2d}{IrrIdeal = (x[6]*x[7]*x[8]*x[9], x[1]*x[7]*x[8]*x[9], x[2]*x[7]*x[8]*x[9], x[1]*x[3]*x[7]*x[9], x[3]*x[6]*x[9], x[2]*x[3]*x[8]*x[9], x[1]*x[2]*x[3]*x[9])}{\[\displaystyle {\it IrrIdeal}={x_{{6}}x_{{7}}x_{{8}}x_{{9}},x_{{1}}x_{{7}}\\
\mbox{}x_{{8}}x_{{9}},x_{{2}}x_{{7}}x_{{8}}x_{{9}},x_{{1}}x_{{3}}x_{{7}}\\
\mbox{}x_{{9}},x_{{3}}x_{{6}}x_{{9}},x_{{2}}x_{{3}}x_{{8}}x_{{9}},x_{{1}}x_{{2}}\\
\mbox{}x_{{3}}x_{{9}}}\]}
\end{maplelatex}
\mapleresult
\begin{maplelatex}
\mapleinline{inert}{2d}{x[1]*x[4]*x[7], x[4]*x[6], x[2]*x[4]*x[8], x[1]*x[2]*x[3]*x[4], x[1]*x[5]*x[7], x[5]*x[6], x[2]*x[5]*x[8], x[1]*x[2]*x[3]*x[5]}{\[\displaystyle x_{{1}}x_{{4}}x_{{7}},\,x_{{4}}x_{{6}},\,x_{{2}}x_{{4}}x_{{8}},\,x_{{1}}x_{{2}}x_{{3}}x_{{4}},\,x_{{1}}x_{{5}}x_{{7}},\,x_{{5}}x_{{6}},\,x_{{2}}x_{{5}}x_{{8}},\,x_{{1}}x_{{2}}x_{{3}}x_{{5}}\]}
\end{maplelatex}
\mapleresult
\begin{maplelatex}
\mapleinline{inert}{2d}{A = {2, 3, 10}, ftv = [Matrix(
\mbox{},1,1,1,1,1]]\]}
\end{maplelatex}
\mapleresult
\begin{maplelatex}
\mapleinline{inert}{2d}{Polynomials = {x[1]^3*x[5]^3+x[1]*x[2]*x[3]*x[4]+x[6]^3+x[7]^3, x[3]^3*x[8]^3+x[4]^3*x[9]^3+x[5]*x[6]*x[7]*x[8]*x[9]+x[2]^3}}{\[\displaystyle {\it Polynomials}= \left\{ {x_{{1}}}^{3}{x_{{5}}}^{3}+x_{{1}}x_{{2}}x_{{3}}x_{{4}}\\
\mbox{}+{x_{{6}}}^{3}+{x_{{7}}}^{3},{x_{{3}}}^{3}{x_{{8}}}^{3}+{x_{{4}}}^{3}{x_{{9}}}^{3}+x_{{5}}x_{{6}}x_{{7}}x_{{8}}x_{{9}}\\
\mbox{}+{x_{{2}}}^{3} \right\} \]}
\end{maplelatex}
\mapleresult
\begin{maplelatex}
\mapleinline{inert}{2d}{IrrIdeal = (x[1]*x[5]*x[8]*x[9], x[5]*x[6]*x[8]*x[9], x[5]*x[7]*x[8]*x[9], x[1]*x[3]*x[5]*x[8], x[1]*x[4]*x[5]*x[9], x[1]*x[2]*x[5], x[1]*x[3]*x[4]*x[8])}{\[\displaystyle {\it IrrIdeal}={x_{{1}}x_{{5}}x_{{8}}x_{{9}},x_{{5}}x_{{6}}\\
\mbox{}x_{{8}}x_{{9}},x_{{5}}x_{{7}}x_{{8}}x_{{9}},x_{{1}}x_{{3}}x_{{5}}\\
\mbox{}x_{{8}},x_{{1}}x_{{4}}x_{{5}}x_{{9}},x_{{1}}x_{{2}}x_{{5}},x_{{1}}x_{{3}}\\
\mbox{}x_{{4}}x_{{8}}}\]}
\end{maplelatex}
\mapleresult
\begin{maplelatex}
\mapleinline{inert}{2d}{x[3]*x[6]*x[8], x[3]*x[7]*x[8], x[1]*x[3]*x[4]*x[9], x[4]*x[6]*x[9], x[4]*x[7]*x[9], x[1]*x[2]*x[3]*x[4], x[2]*x[6], x[2]*x[7]}{\[\displaystyle x_{{3}}x_{{6}}x_{{8}},\,x_{{3}}x_{{7}}x_{{8}},\,x_{{1}}x_{{3}}x_{{4}}x_{{9}},\,x_{{4}}x_{{6}}x_{{9}},\,x_{{4}}x_{{7}}x_{{9}},\,x_{{1}}x_{{2}}x_{{3}}x_{{4}},\,x_{{2}}x_{{6}},\,x_{{2}}x_{{7}}\]}
\end{maplelatex}
\mapleresult
\begin{maplelatex}
\mapleinline{inert}{2d}{A = {2, 3, 11}, ftv = [Matrix(
\mbox{},1,1,1,1,1]]\]}
\end{maplelatex}
\mapleresult
\begin{maplelatex}
\mapleinline{inert}{2d}{Polynomials = {x[1]^3*x[5]^3+x[1]*x[2]*x[3]*x[4]+x[6]^3+x[7]^3, x[2]^3*x[8]^3+x[4]^3*x[9]^3+x[5]*x[6]*x[7]*x[8]*x[9]+x[3]^3}}{\[\displaystyle {\it Polynomials}= \left\{ {x_{{1}}}^{3}{x_{{5}}}^{3}+x_{{1}}x_{{2}}x_{{3}}x_{{4}}\\
\mbox{}+{x_{{6}}}^{3}+{x_{{7}}}^{3},{x_{{2}}}^{3}{x_{{8}}}^{3}+{x_{{4}}}^{3}{x_{{9}}}^{3}+x_{{5}}x_{{6}}x_{{7}}x_{{8}}x_{{9}}\\
\mbox{}+{x_{{3}}}^{3} \right\} \]}
\end{maplelatex}
\mapleresult
\begin{maplelatex}
\mapleinline{inert}{2d}{IrrIdeal = (x[1]*x[5]*x[8]*x[9], x[5]*x[6]*x[8]*x[9], x[5]*x[7]*x[8]*x[9], x[1]*x[2]*x[5]*x[8], x[1]*x[4]*x[5]*x[9], x[1]*x[2]*x[4]*x[8], x[2]*x[6]*x[8])}{\[\displaystyle {\it IrrIdeal}={x_{{1}}x_{{5}}x_{{8}}x_{{9}},x_{{5}}x_{{6}}\\
\mbox{}x_{{8}}x_{{9}},x_{{5}}x_{{7}}x_{{8}}x_{{9}},x_{{1}}x_{{2}}x_{{5}}\\
\mbox{}x_{{8}},x_{{1}}x_{{4}}x_{{5}}x_{{9}},x_{{1}}x_{{2}}x_{{4}}x_{{8}},x_{{2}}\\
\mbox{}x_{{6}}x_{{8}}}\]}
\end{maplelatex}
\mapleresult
\begin{maplelatex}
\mapleinline{inert}{2d}{x[2]*x[7]*x[8], x[1]*x[2]*x[4]*x[9], x[4]*x[6]*x[9], x[4]*x[7]*x[9], x[1]*x[3]*x[5], x[1]*x[2]*x[3]*x[4], x[3]*x[6], x[3]*x[7]}{\[\displaystyle x_{{2}}x_{{7}}x_{{8}},\,x_{{1}}x_{{2}}x_{{4}}x_{{9}},\,x_{{4}}x_{{6}}x_{{9}},\,x_{{4}}x_{{7}}x_{{9}},\,x_{{1}}x_{{3}}x_{{5}},\,x_{{1}}x_{{2}}x_{{3}}x_{{4}},\,x_{{3}}x_{{6}},\,x_{{3}}x_{{7}}\]}
\end{maplelatex}
\mapleresult
\begin{maplelatex}
\mapleinline{inert}{2d}{A = {2, 3, 12}, ftv = [Matrix(
\mbox{},1,1,1,1]]\]}
\end{maplelatex}
\mapleresult
\begin{maplelatex}
\mapleinline{inert}{2d}{Polynomials = {x[1]^3*x[5]^3+x[1]*x[2]*x[3]*x[4]+x[6]^3+x[7]^3, x[2]^3*x[8]^3+x[3]^3*x[9]^3+x[5]*x[6]*x[7]*x[8]*x[9]+x[4]^3}}{\[\displaystyle {\it Polynomials}= \left\{ {x_{{1}}}^{3}{x_{{5}}}^{3}+x_{{1}}x_{{2}}x_{{3}}x_{{4}}\\
\mbox{}+{x_{{6}}}^{3}+{x_{{7}}}^{3},{x_{{2}}}^{3}{x_{{8}}}^{3}+{x_{{3}}}^{3}{x_{{9}}}^{3}+x_{{5}}x_{{6}}x_{{7}}x_{{8}}x_{{9}}\\
\mbox{}+{x_{{4}}}^{3} \right\} \]}
\end{maplelatex}
\mapleresult
\begin{maplelatex}
\mapleinline{inert}{2d}{IrrIdeal = (x[1]*x[5]*x[8]*x[9], x[5]*x[6]*x[8]*x[9], x[5]*x[7]*x[8]*x[9], x[1]*x[2]*x[5]*x[8], x[1]*x[3]*x[5]*x[9], x[1]*x[2]*x[3]*x[8], x[2]*x[6]*x[8])}{\[\displaystyle {\it IrrIdeal}={x_{{1}}x_{{5}}x_{{8}}x_{{9}},x_{{5}}x_{{6}}\\
\mbox{}x_{{8}}x_{{9}},x_{{5}}x_{{7}}x_{{8}}x_{{9}},x_{{1}}x_{{2}}x_{{5}}\\
\mbox{}x_{{8}},x_{{1}}x_{{3}}x_{{5}}x_{{9}},x_{{1}}x_{{2}}x_{{3}}x_{{8}},x_{{2}}\\
\mbox{}x_{{6}}x_{{8}}}\]}
\end{maplelatex}
\mapleresult
\begin{maplelatex}
\mapleinline{inert}{2d}{x[2]*x[7]*x[8], x[1]*x[2]*x[3]*x[9], x[3]*x[6]*x[9], x[3]*x[7]*x[9], x[1]*x[4]*x[5], x[1]*x[2]*x[3]*x[4], x[4]*x[6], x[4]*x[7]}{\[\displaystyle x_{{2}}x_{{7}}x_{{8}},\,x_{{1}}x_{{2}}x_{{3}}x_{{9}},\,x_{{3}}x_{{6}}x_{{9}},\,x_{{3}}x_{{7}}x_{{9}},\,x_{{1}}x_{{4}}x_{{5}},\,x_{{1}}x_{{2}}x_{{3}}x_{{4}},\,x_{{4}}x_{{6}},\,x_{{4}}x_{{7}}\]}
\end{maplelatex}
\mapleresult
\begin{maplelatex}
\mapleinline{inert}{2d}{A = {2, 10, 11}, ftv = [Matrix(
\mbox{},0,0,1,1,1,1]]\]}
\end{maplelatex}
\mapleresult
\begin{maplelatex}
\mapleinline{inert}{2d}{Polynomials = {x[5]^3*x[9]^3+x[6]*x[7]*x[8]*x[9]+x[3]^3+x[4]^3, x[1]^3*x[6]^3+x[2]^3*x[8]^3+x[1]*x[2]*x[3]*x[4]*x[5]+x[7]^3}}{\[\displaystyle {\it Polynomials}= \left\{ {x_{{5}}}^{3}{x_{{9}}}^{3}+x_{{6}}x_{{7}}x_{{8}}x_{{9}}\\
\mbox{}+{x_{{3}}}^{3}+{x_{{4}}}^{3},{x_{{1}}}^{3}{x_{{6}}}^{3}+{x_{{2}}}^{3}{x_{{8}}}^{3}+x_{{1}}x_{{2}}x_{{3}}x_{{4}}x_{{5}}\\
\mbox{}+{x_{{7}}}^{3} \right\} \]}
\end{maplelatex}
\mapleresult
\begin{maplelatex}
\mapleinline{inert}{2d}{IrrIdeal = (x[1]*x[6]*x[8]*x[9], x[6]*x[7]*x[8]*x[9], x[2]*x[6]*x[8]*x[9], x[1]*x[5]*x[6]*x[9], x[1]*x[2]*x[5]*x[9], x[5]*x[7]*x[9], x[2]*x[5]*x[8]*x[9])}{\[\displaystyle {\it IrrIdeal}={x_{{1}}x_{{6}}x_{{8}}x_{{9}},x_{{6}}x_{{7}}\\
\mbox{}x_{{8}}x_{{9}},x_{{2}}x_{{6}}x_{{8}}x_{{9}},x_{{1}}x_{{5}}x_{{6}}\\
\mbox{}x_{{9}},x_{{1}}x_{{2}}x_{{5}}x_{{9}},x_{{5}}x_{{7}}x_{{9}},x_{{2}}x_{{5}}\\
\mbox{}x_{{8}}x_{{9}}}\]}
\end{maplelatex}
\mapleresult
\begin{maplelatex}
\mapleinline{inert}{2d}{x[1]*x[3]*x[6], x[1]*x[2]*x[3]*x[5], x[3]*x[7], x[2]*x[3]*x[8], x[1]*x[4]*x[6], x[1]*x[2]*x[4]*x[5], x[4]*x[7], x[2]*x[4]*x[8]}{\[\displaystyle x_{{1}}x_{{3}}x_{{6}},\,x_{{1}}x_{{2}}x_{{3}}x_{{5}},\,x_{{3}}x_{{7}},\,x_{{2}}x_{{3}}x_{{8}},\,x_{{1}}x_{{4}}x_{{6}},\,x_{{1}}x_{{2}}x_{{4}}x_{{5}},\,x_{{4}}x_{{7}},\,x_{{2}}x_{{4}}x_{{8}}\]}
\end{maplelatex}
\mapleresult
\begin{maplelatex}
\mapleinline{inert}{2d}{A = {2, 10, 12}, ftv = [Matrix(
\mbox{},0,0,1,1,1,1]]\]}
\end{maplelatex}
\mapleresult
\begin{maplelatex}
\mapleinline{inert}{2d}{Polynomials = {x[4]^3*x[9]^3+x[6]*x[7]*x[8]*x[9]+x[3]^3+x[5]^3, x[1]^3*x[6]^3+x[2]^3*x[8]^3+x[1]*x[2]*x[3]*x[4]*x[5]+x[7]^3}}{\[\displaystyle {\it Polynomials}= \left\{ {x_{{4}}}^{3}{x_{{9}}}^{3}+x_{{6}}x_{{7}}x_{{8}}x_{{9}}\\
\mbox{}+{x_{{3}}}^{3}+{x_{{5}}}^{3},{x_{{1}}}^{3}{x_{{6}}}^{3}+{x_{{2}}}^{3}{x_{{8}}}^{3}+x_{{1}}x_{{2}}x_{{3}}x_{{4}}x_{{5}}\\
\mbox{}+{x_{{7}}}^{3} \right\} \]}
\end{maplelatex}
\mapleresult
\begin{maplelatex}
\mapleinline{inert}{2d}{IrrIdeal = (x[1]*x[6]*x[8]*x[9], x[6]*x[7]*x[8]*x[9], x[2]*x[6]*x[8]*x[9], x[1]*x[4]*x[6]*x[9], x[1]*x[2]*x[4]*x[9], x[4]*x[7]*x[9], x[2]*x[4]*x[8]*x[9])}{\[\displaystyle {\it IrrIdeal}={x_{{1}}x_{{6}}x_{{8}}x_{{9}},x_{{6}}x_{{7}}\\
\mbox{}x_{{8}}x_{{9}},x_{{2}}x_{{6}}x_{{8}}x_{{9}},x_{{1}}x_{{4}}x_{{6}}\\
\mbox{}x_{{9}},x_{{1}}x_{{2}}x_{{4}}x_{{9}},x_{{4}}x_{{7}}x_{{9}},x_{{2}}x_{{4}}\\
\mbox{}x_{{8}}x_{{9}}}\]}
\end{maplelatex}
\mapleresult
\begin{maplelatex}
\mapleinline{inert}{2d}{x[1]*x[3]*x[6], x[1]*x[2]*x[3]*x[4], x[3]*x[7], x[2]*x[3]*x[8], x[1]*x[5]*x[6], x[1]*x[2]*x[4]*x[5], x[5]*x[7], x[2]*x[5]*x[8]}{\[\displaystyle x_{{1}}x_{{3}}x_{{6}},\,x_{{1}}x_{{2}}x_{{3}}x_{{4}},\,x_{{3}}x_{{7}},\,x_{{2}}x_{{3}}x_{{8}},\,x_{{1}}x_{{5}}x_{{6}},\,x_{{1}}x_{{2}}x_{{4}}x_{{5}},\,x_{{5}}x_{{7}},\,x_{{2}}x_{{5}}x_{{8}}\]}
\end{maplelatex}
\mapleresult
\begin{maplelatex}
\mapleinline{inert}{2d}{A = {2, 11, 12}, ftv = [Matrix(
\mbox{},0,0,1,1,1,1]]\]}
\end{maplelatex}
\mapleresult
\begin{maplelatex}
\mapleinline{inert}{2d}{Polynomials = {x[3]^3*x[9]^3+x[6]*x[7]*x[8]*x[9]+x[4]^3+x[5]^3, x[1]^3*x[6]^3+x[2]^3*x[8]^3+x[1]*x[2]*x[3]*x[4]*x[5]+x[7]^3}}{\[\displaystyle {\it Polynomials}= \left\{ {x_{{3}}}^{3}{x_{{9}}}^{3}+x_{{6}}x_{{7}}x_{{8}}x_{{9}}\\
\mbox{}+{x_{{4}}}^{3}+{x_{{5}}}^{3},{x_{{1}}}^{3}{x_{{6}}}^{3}+{x_{{2}}}^{3}{x_{{8}}}^{3}+x_{{1}}x_{{2}}x_{{3}}x_{{4}}x_{{5}}\\
\mbox{}+{x_{{7}}}^{3} \right\} \]}
\end{maplelatex}
\mapleresult
\begin{maplelatex}
\mapleinline{inert}{2d}{IrrIdeal = (x[1]*x[6]*x[8]*x[9], x[6]*x[7]*x[8]*x[9], x[2]*x[6]*x[8]*x[9], x[1]*x[3]*x[6]*x[9], x[1]*x[2]*x[3]*x[9], x[3]*x[7]*x[9], x[2]*x[3]*x[8]*x[9])}{\[\displaystyle {\it IrrIdeal}={x_{{1}}x_{{6}}x_{{8}}x_{{9}},x_{{6}}x_{{7}}\\
\mbox{}x_{{8}}x_{{9}},x_{{2}}x_{{6}}x_{{8}}x_{{9}},x_{{1}}x_{{3}}x_{{6}}\\
\mbox{}x_{{9}},x_{{1}}x_{{2}}x_{{3}}x_{{9}},x_{{3}}x_{{7}}x_{{9}},x_{{2}}x_{{3}}\\
\mbox{}x_{{8}}x_{{9}}}\]}
\end{maplelatex}
\mapleresult
\begin{maplelatex}
\mapleinline{inert}{2d}{x[1]*x[4]*x[6], x[1]*x[2]*x[3]*x[4], x[4]*x[7], x[2]*x[4]*x[8], x[1]*x[5]*x[6], x[1]*x[2]*x[3]*x[5], x[5]*x[7], x[2]*x[5]*x[8]}{\[\displaystyle x_{{1}}x_{{4}}x_{{6}},\,x_{{1}}x_{{2}}x_{{3}}x_{{4}},\,x_{{4}}x_{{7}},\,x_{{2}}x_{{4}}x_{{8}},\,x_{{1}}x_{{5}}x_{{6}},\,x_{{1}}x_{{2}}x_{{3}}x_{{5}},\,x_{{5}}x_{{7}},\,x_{{2}}x_{{5}}x_{{8}}\]}
\end{maplelatex}
\mapleresult
\begin{maplelatex}
\mapleinline{inert}{2d}{A = {3, 10, 11}, ftv = [Matrix(
\mbox{},0,0,1,1,1,1]]\]}
\end{maplelatex}
\mapleresult
\begin{maplelatex}
\mapleinline{inert}{2d}{Polynomials = {x[5]^3*x[9]^3+x[6]*x[7]*x[8]*x[9]+x[3]^3+x[4]^3, x[1]^3*x[6]^3+x[2]^3*x[7]^3+x[1]*x[2]*x[3]*x[4]*x[5]+x[8]^3}}{\[\displaystyle {\it Polynomials}= \left\{ {x_{{5}}}^{3}{x_{{9}}}^{3}+x_{{6}}x_{{7}}x_{{8}}x_{{9}}\\
\mbox{}+{x_{{3}}}^{3}+{x_{{4}}}^{3},{x_{{1}}}^{3}{x_{{6}}}^{3}+{x_{{2}}}^{3}{x_{{7}}}^{3}+x_{{1}}x_{{2}}x_{{3}}x_{{4}}x_{{5}}\\
\mbox{}+{x_{{8}}}^{3} \right\} \]}
\end{maplelatex}
\mapleresult
\begin{maplelatex}
\mapleinline{inert}{2d}{IrrIdeal = (x[1]*x[6]*x[7]*x[9], x[6]*x[7]*x[8]*x[9], x[2]*x[6]*x[7]*x[9], x[1]*x[5]*x[6]*x[9], x[1]*x[2]*x[5]*x[9], x[2]*x[5]*x[7]*x[9], x[5]*x[8]*x[9])}{\[\displaystyle {\it IrrIdeal}={x_{{1}}x_{{6}}x_{{7}}x_{{9}},x_{{6}}x_{{7}}\\
\mbox{}x_{{8}}x_{{9}},x_{{2}}x_{{6}}x_{{7}}x_{{9}},x_{{1}}x_{{5}}x_{{6}}\\
\mbox{}x_{{9}},x_{{1}}x_{{2}}x_{{5}}x_{{9}},x_{{2}}x_{{5}}x_{{7}}x_{{9}},x_{{5}}\\
\mbox{}x_{{8}}x_{{9}}}\]}
\end{maplelatex}
\mapleresult
\begin{maplelatex}
\mapleinline{inert}{2d}{x[1]*x[3]*x[6], x[1]*x[2]*x[3]*x[5], x[2]*x[3]*x[7], x[3]*x[8], x[1]*x[4]*x[6], x[1]*x[2]*x[4]*x[5], x[2]*x[4]*x[7], x[4]*x[8]}{\[\displaystyle x_{{1}}x_{{3}}x_{{6}},\,x_{{1}}x_{{2}}x_{{3}}x_{{5}},\,x_{{2}}x_{{3}}x_{{7}},\,x_{{3}}x_{{8}},\,x_{{1}}x_{{4}}x_{{6}},\,x_{{1}}x_{{2}}x_{{4}}x_{{5}},\,x_{{2}}x_{{4}}x_{{7}},\,x_{{4}}x_{{8}}\]}
\end{maplelatex}
\mapleresult
\begin{maplelatex}
\mapleinline{inert}{2d}{A = {3, 10, 12}, ftv = [Matrix(
\mbox{},0,0,1,1,1,1]]\]}
\end{maplelatex}
\mapleresult
\begin{maplelatex}
\mapleinline{inert}{2d}{Polynomials = {x[4]^3*x[9]^3+x[6]*x[7]*x[8]*x[9]+x[3]^3+x[5]^3, x[1]^3*x[6]^3+x[2]^3*x[7]^3+x[1]*x[2]*x[3]*x[4]*x[5]+x[8]^3}}{\[\displaystyle {\it Polynomials}= \left\{ {x_{{4}}}^{3}{x_{{9}}}^{3}+x_{{6}}x_{{7}}x_{{8}}x_{{9}}\\
\mbox{}+{x_{{3}}}^{3}+{x_{{5}}}^{3},{x_{{1}}}^{3}{x_{{6}}}^{3}+{x_{{2}}}^{3}{x_{{7}}}^{3}+x_{{1}}x_{{2}}x_{{3}}x_{{4}}x_{{5}}\\
\mbox{}+{x_{{8}}}^{3} \right\} \]}
\end{maplelatex}
\mapleresult
\begin{maplelatex}
\mapleinline{inert}{2d}{IrrIdeal = (x[1]*x[6]*x[7]*x[9], x[6]*x[7]*x[8]*x[9], x[2]*x[6]*x[7]*x[9], x[1]*x[4]*x[6]*x[9], x[1]*x[2]*x[4]*x[9], x[2]*x[4]*x[7]*x[9], x[4]*x[8]*x[9])}{\[\displaystyle {\it IrrIdeal}={x_{{1}}x_{{6}}x_{{7}}x_{{9}},x_{{6}}x_{{7}}\\
\mbox{}x_{{8}}x_{{9}},x_{{2}}x_{{6}}x_{{7}}x_{{9}},x_{{1}}x_{{4}}x_{{6}}\\
\mbox{}x_{{9}},x_{{1}}x_{{2}}x_{{4}}x_{{9}},x_{{2}}x_{{4}}x_{{7}}x_{{9}},x_{{4}}\\
\mbox{}x_{{8}}x_{{9}}}\]}
\end{maplelatex}
\mapleresult
\begin{maplelatex}
\mapleinline{inert}{2d}{x[1]*x[3]*x[6], x[1]*x[2]*x[3]*x[4], x[2]*x[3]*x[7], x[3]*x[8], x[1]*x[5]*x[6], x[1]*x[2]*x[4]*x[5], x[2]*x[5]*x[7], x[5]*x[8]}{\[\displaystyle x_{{1}}x_{{3}}x_{{6}},\,x_{{1}}x_{{2}}x_{{3}}x_{{4}},\,x_{{2}}x_{{3}}x_{{7}},\,x_{{3}}x_{{8}},\,x_{{1}}x_{{5}}x_{{6}},\,x_{{1}}x_{{2}}x_{{4}}x_{{5}},\,x_{{2}}x_{{5}}x_{{7}},\,x_{{5}}x_{{8}}\]}
\end{maplelatex}
\mapleresult
\begin{maplelatex}
\mapleinline{inert}{2d}{A = {3, 11, 12}, ftv = [Matrix(
\mbox{},0,0,1,1,1,1]]\]}
\end{maplelatex}
\mapleresult
\begin{maplelatex}
\mapleinline{inert}{2d}{Polynomials = {x[3]^3*x[9]^3+x[6]*x[7]*x[8]*x[9]+x[4]^3+x[5]^3, x[1]^3*x[6]^3+x[2]^3*x[7]^3+x[1]*x[2]*x[3]*x[4]*x[5]+x[8]^3}}{\[\displaystyle {\it Polynomials}= \left\{ {x_{{3}}}^{3}{x_{{9}}}^{3}+x_{{6}}x_{{7}}x_{{8}}x_{{9}}\\
\mbox{}+{x_{{4}}}^{3}+{x_{{5}}}^{3},{x_{{1}}}^{3}{x_{{6}}}^{3}+{x_{{2}}}^{3}{x_{{7}}}^{3}+x_{{1}}x_{{2}}x_{{3}}x_{{4}}x_{{5}}\\
\mbox{}+{x_{{8}}}^{3} \right\} \]}
\end{maplelatex}
\mapleresult
\begin{maplelatex}
\mapleinline{inert}{2d}{IrrIdeal = (x[1]*x[6]*x[7]*x[9], x[6]*x[7]*x[8]*x[9], x[2]*x[6]*x[7]*x[9], x[1]*x[3]*x[6]*x[9], x[1]*x[2]*x[3]*x[9], x[2]*x[3]*x[7]*x[9], x[3]*x[8]*x[9])}{\[\displaystyle {\it IrrIdeal}={x_{{1}}x_{{6}}x_{{7}}x_{{9}},x_{{6}}x_{{7}}\\
\mbox{}x_{{8}}x_{{9}},x_{{2}}x_{{6}}x_{{7}}x_{{9}},x_{{1}}x_{{3}}x_{{6}}\\
\mbox{}x_{{9}},x_{{1}}x_{{2}}x_{{3}}x_{{9}},x_{{2}}x_{{3}}x_{{7}}x_{{9}},x_{{3}}\\
\mbox{}x_{{8}}x_{{9}}}\]}
\end{maplelatex}
\mapleresult
\begin{maplelatex}
\mapleinline{inert}{2d}{x[1]*x[4]*x[6], x[1]*x[2]*x[3]*x[4], x[2]*x[4]*x[7], x[4]*x[8], x[1]*x[5]*x[6], x[1]*x[2]*x[3]*x[5], x[2]*x[5]*x[7], x[5]*x[8]}{\[\displaystyle x_{{1}}x_{{4}}x_{{6}},\,x_{{1}}x_{{2}}x_{{3}}x_{{4}},\,x_{{2}}x_{{4}}x_{{7}},\,x_{{4}}x_{{8}},\,x_{{1}}x_{{5}}x_{{6}},\,x_{{1}}x_{{2}}x_{{3}}x_{{5}},\,x_{{2}}x_{{5}}x_{{7}},\,x_{{5}}x_{{8}}\]}
\end{maplelatex}
\mapleresult
\begin{maplelatex}
\mapleinline{inert}{2d}{A = {10, 11, 12}, ftv = [Matrix(
\mbox{},0,0,0,0,0,1,1,1]]\]}
\end{maplelatex}
\mapleresult
\begin{maplelatex}
\mapleinline{inert}{2d}{Polynomials = {x[4]^3+x[5]^3+x[6]^3+x[7]*x[8]*x[9], x[1]^3*x[7]^3+x[1]*x[2]*x[3]*x[4]*x[5]*x[6]+x[2]^3*x[8]^3+x[3]^3*x[9]^3}}{\[\displaystyle {\it Polynomials}= \left\{ {x_{{4}}}^{3}+{x_{{5}}}^{3}+{x_{{6}}}^{3}+x_{{7}}x_{{8}}x_{{9}}\\
\mbox{},{x_{{1}}}^{3}{x_{{7}}}^{3}+x_{{1}}x_{{2}}x_{{3}}x_{{4}}x_{{5}}x_{{6}}+{x_{{2}}}^{3}{x_{{8}}}^{3}\\
\mbox{}+{x_{{3}}}^{3}{x_{{9}}}^{3} \right\} \]}
\end{maplelatex}
\mapleresult
\begin{maplelatex}
\mapleinline{inert}{2d}{IrrIdeal = (x[1]*x[7]*x[8]*x[9], x[2]*x[7]*x[8]*x[9], x[3]*x[7]*x[8]*x[9], x[1]*x[4]*x[7], x[1]*x[2]*x[3]*x[4], x[2]*x[4]*x[8], x[3]*x[4]*x[9])}{\[\displaystyle {\it IrrIdeal}={x_{{1}}x_{{7}}x_{{8}}x_{{9}},x_{{2}}x_{{7}}\\
\mbox{}x_{{8}}x_{{9}},x_{{3}}x_{{7}}x_{{8}}x_{{9}},x_{{1}}x_{{4}}x_{{7}},x_{{1}}\\
\mbox{}x_{{2}}x_{{3}}x_{{4}},x_{{2}}x_{{4}}x_{{8}},x_{{3}}x_{{4}}x_{{9}}}\]}
\end{maplelatex}
\mapleresult
\begin{maplelatex}
\mapleinline{inert}{2d}{x[1]*x[5]*x[7], x[1]*x[2]*x[3]*x[5], x[2]*x[5]*x[8], x[3]*x[5]*x[9], x[1]*x[6]*x[7], x[1]*x[2]*x[3]*x[6], x[2]*x[6]*x[8], x[3]*x[6]*x[9]}{\[\displaystyle x_{{1}}x_{{5}}x_{{7}},\,x_{{1}}x_{{2}}x_{{3}}x_{{5}},\,x_{{2}}x_{{5}}x_{{8}},\,x_{{3}}x_{{5}}x_{{9}},\,x_{{1}}x_{{6}}x_{{7}},\,x_{{1}}x_{{2}}x_{{3}}x_{{6}},\,x_{{2}}x_{{6}}x_{{8}},\,x_{{3}}x_{{6}}x_{{9}}\]}
\end{maplelatex}
\mapleresult
\begin{maplelatex}
\mapleinline{inert}{2d}{A = {1, 2, 3, 10}, ftv = [Matrix(
\mbox{},1,1,1]]\]}
\end{maplelatex}
\mapleresult
\begin{maplelatex}
\mapleinline{inert}{2d}{Polynomials = {x[1]*x[2]*x[3]+x[4]^3+x[5]^3+x[6]^3, x[2]^3*x[7]^3+x[3]^3*x[8]^3+x[4]*x[5]*x[6]*x[7]*x[8]+x[1]^3}}{\[\displaystyle {\it Polynomials}= \left\{ x_{{1}}x_{{2}}x_{{3}}+{x_{{4}}}^{3}+{x_{{5}}}^{3}+{x_{{6}}}^{3}\\
\mbox{},{x_{{2}}}^{3}{x_{{7}}}^{3}+{x_{{3}}}^{3}{x_{{8}}}^{3}+x_{{4}}x_{{5}}x_{{6}}x_{{7}}x_{{8}}\\
\mbox{}+{x_{{1}}}^{3} \right\} \]}
\end{maplelatex}
\mapleresult
\begin{maplelatex}
\mapleinline{inert}{2d}{IrrIdeal = (x[4]*x[7]*x[8], x[5]*x[7]*x[8], x[6]*x[7]*x[8], x[2]*x[4]*x[7], x[2]*x[5]*x[7], x[2]*x[6]*x[7], x[3]*x[4]*x[8], x[3]*x[5]*x[8], x[3]*x[6]*x[8], x[1]*x[4], x[1]*x[5], x[1]*x[6], x[2]*x[3]*x[7], x[2]*x[3]*x[8], x[1]*x[2]*x[3])}{\[\displaystyle {\it IrrIdeal}={x_{{4}}x_{{7}}x_{{8}},x_{{5}}x_{{7}}x_{{8}}\\
\mbox{},x_{{6}}x_{{7}}x_{{8}},x_{{2}}x_{{4}}x_{{7}},x_{{2}}x_{{5}}x_{{7}},x_{{2}}\\
\mbox{}x_{{6}}x_{{7}},x_{{3}}x_{{4}}x_{{8}},x_{{3}}x_{{5}}x_{{8}},x_{{3}}x_{{6}}\\
\mbox{}x_{{8}},x_{{1}}x_{{4}},x_{{1}}x_{{5}},x_{{1}}x_{{6}},x_{{2}}x_{{3}}x_{{7}}\\
\mbox{},x_{{2}}x_{{3}}x_{{8}},x_{{1}}x_{{2}}x_{{3}}}\]}
\end{maplelatex}
\mapleresult
\begin{maplelatex}
\mapleinline{inert}{2d}{A = {1, 2, 3, 11}, ftv = [Matrix(
\mbox{},1,1,1]]\]}
\end{maplelatex}
\mapleresult
\begin{maplelatex}
\mapleinline{inert}{2d}{Polynomials = {x[1]*x[2]*x[3]+x[4]^3+x[5]^3+x[6]^3, x[1]^3*x[7]^3+x[3]^3*x[8]^3+x[4]*x[5]*x[6]*x[7]*x[8]+x[2]^3}}{\[\displaystyle {\it Polynomials}= \left\{ x_{{1}}x_{{2}}x_{{3}}+{x_{{4}}}^{3}+{x_{{5}}}^{3}+{x_{{6}}}^{3}\\
\mbox{},{x_{{1}}}^{3}{x_{{7}}}^{3}+{x_{{3}}}^{3}{x_{{8}}}^{3}+x_{{4}}x_{{5}}x_{{6}}x_{{7}}x_{{8}}\\
\mbox{}+{x_{{2}}}^{3} \right\} \]}
\end{maplelatex}
\mapleresult
\begin{maplelatex}
\mapleinline{inert}{2d}{IrrIdeal = (x[4]*x[7]*x[8], x[5]*x[7]*x[8], x[6]*x[7]*x[8], x[1]*x[4]*x[7], x[1]*x[5]*x[7], x[1]*x[6]*x[7], x[3]*x[4]*x[8], x[3]*x[5]*x[8], x[3]*x[6]*x[8], x[1]*x[3]*x[7], x[1]*x[3]*x[8], x[2]*x[4], x[2]*x[5], x[2]*x[6], x[1]*x[2]*x[3])}{\[\displaystyle {\it IrrIdeal}={x_{{4}}x_{{7}}x_{{8}},x_{{5}}x_{{7}}x_{{8}}\\
\mbox{},x_{{6}}x_{{7}}x_{{8}},x_{{1}}x_{{4}}x_{{7}},x_{{1}}x_{{5}}x_{{7}},x_{{1}}\\
\mbox{}x_{{6}}x_{{7}},x_{{3}}x_{{4}}x_{{8}},x_{{3}}x_{{5}}x_{{8}},x_{{3}}x_{{6}}\\
\mbox{}x_{{8}},x_{{1}}x_{{3}}x_{{7}},x_{{1}}x_{{3}}x_{{8}},x_{{2}}x_{{4}},x_{{2}}\\
\mbox{}x_{{5}},x_{{2}}x_{{6}},x_{{1}}x_{{2}}x_{{3}}}\]}
\end{maplelatex}
\mapleresult
\begin{maplelatex}
\mapleinline{inert}{2d}{A = {1, 2, 3, 12}, ftv = [Matrix(
\mbox{},1,1]]\]}
\end{maplelatex}
\mapleresult
\begin{maplelatex}
\mapleinline{inert}{2d}{Polynomials = {x[1]*x[2]*x[3]+x[4]^3+x[5]^3+x[6]^3, x[1]^3*x[7]^3+x[2]^3*x[8]^3+x[4]*x[5]*x[6]*x[7]*x[8]+x[3]^3}}{\[\displaystyle {\it Polynomials}= \left\{ x_{{1}}x_{{2}}x_{{3}}+{x_{{4}}}^{3}+{x_{{5}}}^{3}+{x_{{6}}}^{3}\\
\mbox{},{x_{{1}}}^{3}{x_{{7}}}^{3}+{x_{{2}}}^{3}{x_{{8}}}^{3}+x_{{4}}x_{{5}}x_{{6}}x_{{7}}x_{{8}}\\
\mbox{}+{x_{{3}}}^{3} \right\} \]}
\end{maplelatex}
\mapleresult
\begin{maplelatex}
\mapleinline{inert}{2d}{IrrIdeal = (x[4]*x[7]*x[8], x[5]*x[7]*x[8], x[6]*x[7]*x[8], x[1]*x[4]*x[7], x[1]*x[5]*x[7], x[1]*x[6]*x[7], x[2]*x[4]*x[8], x[2]*x[5]*x[8], x[2]*x[6]*x[8], x[1]*x[2]*x[7], x[1]*x[2]*x[8], x[3]*x[4], x[3]*x[5], x[3]*x[6], x[1]*x[2]*x[3])}{\[\displaystyle {\it IrrIdeal}={x_{{4}}x_{{7}}x_{{8}},x_{{5}}x_{{7}}x_{{8}}\\
\mbox{},x_{{6}}x_{{7}}x_{{8}},x_{{1}}x_{{4}}x_{{7}},x_{{1}}x_{{5}}x_{{7}},x_{{1}}\\
\mbox{}x_{{6}}x_{{7}},x_{{2}}x_{{4}}x_{{8}},x_{{2}}x_{{5}}x_{{8}},x_{{2}}x_{{6}}\\
\mbox{}x_{{8}},x_{{1}}x_{{2}}x_{{7}},x_{{1}}x_{{2}}x_{{8}},x_{{3}}x_{{4}},x_{{3}}\\
\mbox{}x_{{5}},x_{{3}}x_{{6}},x_{{1}}x_{{2}}x_{{3}}}\]}
\end{maplelatex}
\mapleresult
\begin{maplelatex}
\mapleinline{inert}{2d}{A = {1, 2, 10, 11}, ftv = [Matrix(
\mbox{},1,1]]\]}
\end{maplelatex}
\mapleresult
\begin{maplelatex}
\mapleinline{inert}{2d}{Polynomials = {x[1]^3*x[7]^3+x[1]*x[2]*x[3]*x[4]+x[5]^3+x[6]^3, x[4]^3*x[8]^3+x[5]*x[6]*x[7]*x[8]+x[2]^3+x[3]^3}}{\[\displaystyle {\it Polynomials}= \left\{ {x_{{1}}}^{3}{x_{{7}}}^{3}+x_{{1}}x_{{2}}x_{{3}}x_{{4}}\\
\mbox{}+{x_{{5}}}^{3}+{x_{{6}}}^{3},{x_{{4}}}^{3}{x_{{8}}}^{3}+x_{{5}}x_{{6}}x_{{7}}x_{{8}}+{x_{{2}}}^{3}\\
\mbox{}+{x_{{3}}}^{3} \right\} \]}
\end{maplelatex}
\mapleresult
\begin{maplelatex}
\mapleinline{inert}{2d}{IrrIdeal = (x[5]*x[7]*x[8], x[6]*x[7]*x[8], x[1]*x[7]*x[8], x[1]*x[2]*x[7], x[1]*x[3]*x[7], x[4]*x[5]*x[8], x[4]*x[6]*x[8], x[1]*x[4]*x[8], x[2]*x[5], x[2]*x[6], x[1]*x[2]*x[4], x[3]*x[5], x[3]*x[6], x[1]*x[3]*x[4])}{\[\displaystyle {\it IrrIdeal}={x_{{5}}x_{{7}}x_{{8}},x_{{6}}x_{{7}}x_{{8}}\\
\mbox{},x_{{1}}x_{{7}}x_{{8}},x_{{1}}x_{{2}}x_{{7}},x_{{1}}x_{{3}}x_{{7}},x_{{4}}\\
\mbox{}x_{{5}}x_{{8}},x_{{4}}x_{{6}}x_{{8}},x_{{1}}x_{{4}}x_{{8}},x_{{2}}x_{{5}}\\
\mbox{},x_{{2}}x_{{6}},x_{{1}}x_{{2}}x_{{4}},x_{{3}}x_{{5}},x_{{3}}x_{{6}},x_{{1}}\\
\mbox{}x_{{3}}x_{{4}}}\]}
\end{maplelatex}
\mapleresult
\begin{maplelatex}
\mapleinline{inert}{2d}{A = {1, 2, 10, 12}, ftv = [Matrix(
\mbox{},1,1]]\]}
\end{maplelatex}
\mapleresult
\begin{maplelatex}
\mapleinline{inert}{2d}{Polynomials = {x[1]^3*x[7]^3+x[1]*x[2]*x[3]*x[4]+x[5]^3+x[6]^3, x[3]^3*x[8]^3+x[5]*x[6]*x[7]*x[8]+x[2]^3+x[4]^3}}{\[\displaystyle {\it Polynomials}= \left\{ {x_{{1}}}^{3}{x_{{7}}}^{3}+x_{{1}}x_{{2}}x_{{3}}x_{{4}}\\
\mbox{}+{x_{{5}}}^{3}+{x_{{6}}}^{3},{x_{{3}}}^{3}{x_{{8}}}^{3}+x_{{5}}x_{{6}}x_{{7}}x_{{8}}+{x_{{2}}}^{3}\\
\mbox{}+{x_{{4}}}^{3} \right\} \]}
\end{maplelatex}
\mapleresult
\begin{maplelatex}
\mapleinline{inert}{2d}{IrrIdeal = (x[5]*x[7]*x[8], x[6]*x[7]*x[8], x[1]*x[7]*x[8], x[1]*x[2]*x[7], x[3]*x[5]*x[8], x[3]*x[6]*x[8], x[1]*x[3]*x[8], x[2]*x[5], x[2]*x[6], x[1]*x[2]*x[3], x[1]*x[4]*x[7], x[4]*x[5], x[4]*x[6], x[1]*x[3]*x[4])}{\[\displaystyle {\it IrrIdeal}={x_{{5}}x_{{7}}x_{{8}},x_{{6}}x_{{7}}x_{{8}}\\
\mbox{},x_{{1}}x_{{7}}x_{{8}},x_{{1}}x_{{2}}x_{{7}},x_{{3}}x_{{5}}x_{{8}},x_{{3}}\\
\mbox{}x_{{6}}x_{{8}},x_{{1}}x_{{3}}x_{{8}},x_{{2}}x_{{5}},x_{{2}}x_{{6}},x_{{1}}\\
\mbox{}x_{{2}}x_{{3}},x_{{1}}x_{{4}}x_{{7}},x_{{4}}x_{{5}},x_{{4}}x_{{6}},x_{{1}}\\
\mbox{}x_{{3}}x_{{4}}}\]}
\end{maplelatex}
\mapleresult
\begin{maplelatex}
\mapleinline{inert}{2d}{A = {1, 2, 11, 12}, ftv = [Matrix(
\mbox{},1,1]]\]}
\end{maplelatex}
\mapleresult
\begin{maplelatex}
\mapleinline{inert}{2d}{Polynomials = {x[1]^3*x[7]^3+x[1]*x[2]*x[3]*x[4]+x[5]^3+x[6]^3, x[2]^3*x[8]^3+x[5]*x[6]*x[7]*x[8]+x[3]^3+x[4]^3}}{\[\displaystyle {\it Polynomials}= \left\{ {x_{{1}}}^{3}{x_{{7}}}^{3}+x_{{1}}x_{{2}}x_{{3}}x_{{4}}\\
\mbox{}+{x_{{5}}}^{3}+{x_{{6}}}^{3},{x_{{2}}}^{3}{x_{{8}}}^{3}+x_{{5}}x_{{6}}x_{{7}}x_{{8}}+{x_{{3}}}^{3}\\
\mbox{}+{x_{{4}}}^{3} \right\} \]}
\end{maplelatex}
\mapleresult
\begin{maplelatex}
\mapleinline{inert}{2d}{IrrIdeal = (x[5]*x[7]*x[8], x[6]*x[7]*x[8], x[1]*x[7]*x[8], x[2]*x[5]*x[8], x[2]*x[6]*x[8], x[1]*x[2]*x[8], x[1]*x[3]*x[7], x[3]*x[5], x[3]*x[6], x[1]*x[2]*x[3], x[1]*x[4]*x[7], x[4]*x[5], x[4]*x[6], x[1]*x[2]*x[4])}{\[\displaystyle {\it IrrIdeal}={x_{{5}}x_{{7}}x_{{8}},x_{{6}}x_{{7}}x_{{8}}\\
\mbox{},x_{{1}}x_{{7}}x_{{8}},x_{{2}}x_{{5}}x_{{8}},x_{{2}}x_{{6}}x_{{8}},x_{{1}}\\
\mbox{}x_{{2}}x_{{8}},x_{{1}}x_{{3}}x_{{7}},x_{{3}}x_{{5}},x_{{3}}x_{{6}},x_{{1}}\\
\mbox{}x_{{2}}x_{{3}},x_{{1}}x_{{4}}x_{{7}},x_{{4}}x_{{5}},x_{{4}}x_{{6}},x_{{1}}\\
\mbox{}x_{{2}}x_{{4}}}\]}
\end{maplelatex}
\mapleresult
\begin{maplelatex}
\mapleinline{inert}{2d}{A = {1, 3, 10, 11}, ftv = [Matrix(
\mbox{},1,1]]\]}
\end{maplelatex}
\mapleresult
\begin{maplelatex}
\mapleinline{inert}{2d}{Polynomials = {x[1]^3*x[6]^3+x[1]*x[2]*x[3]*x[4]+x[5]^3+x[7]^3, x[4]^3*x[8]^3+x[5]*x[6]*x[7]*x[8]+x[2]^3+x[3]^3}}{\[\displaystyle {\it Polynomials}= \left\{ {x_{{1}}}^{3}{x_{{6}}}^{3}+x_{{1}}x_{{2}}x_{{3}}x_{{4}}\\
\mbox{}+{x_{{5}}}^{3}+{x_{{7}}}^{3},{x_{{4}}}^{3}{x_{{8}}}^{3}+x_{{5}}x_{{6}}x_{{7}}x_{{8}}+{x_{{2}}}^{3}\\
\mbox{}+{x_{{3}}}^{3} \right\} \]}
\end{maplelatex}
\mapleresult
\begin{maplelatex}
\mapleinline{inert}{2d}{IrrIdeal = (x[5]*x[6]*x[8], x[6]*x[7]*x[8], x[1]*x[6]*x[8], x[1]*x[2]*x[6], x[1]*x[3]*x[6], x[4]*x[5]*x[8], x[4]*x[7]*x[8], x[1]*x[4]*x[8], x[2]*x[5], x[2]*x[7], x[1]*x[2]*x[4], x[3]*x[5], x[3]*x[7], x[1]*x[3]*x[4])}{\[\displaystyle {\it IrrIdeal}={x_{{5}}x_{{6}}x_{{8}},x_{{6}}x_{{7}}x_{{8}}\\
\mbox{},x_{{1}}x_{{6}}x_{{8}},x_{{1}}x_{{2}}x_{{6}},x_{{1}}x_{{3}}x_{{6}},x_{{4}}\\
\mbox{}x_{{5}}x_{{8}},x_{{4}}x_{{7}}x_{{8}},x_{{1}}x_{{4}}x_{{8}},x_{{2}}x_{{5}}\\
\mbox{},x_{{2}}x_{{7}},x_{{1}}x_{{2}}x_{{4}},x_{{3}}x_{{5}},x_{{3}}x_{{7}},x_{{1}}\\
\mbox{}x_{{3}}x_{{4}}}\]}
\end{maplelatex}
\mapleresult
\begin{maplelatex}
\mapleinline{inert}{2d}{A = {1, 3, 10, 12}, ftv = [Matrix(
\mbox{},1,1]]\]}
\end{maplelatex}
\mapleresult
\begin{maplelatex}
\mapleinline{inert}{2d}{Polynomials = {x[1]^3*x[6]^3+x[1]*x[2]*x[3]*x[4]+x[5]^3+x[7]^3, x[3]^3*x[8]^3+x[5]*x[6]*x[7]*x[8]+x[2]^3+x[4]^3}}{\[\displaystyle {\it Polynomials}= \left\{ {x_{{1}}}^{3}{x_{{6}}}^{3}+x_{{1}}x_{{2}}x_{{3}}x_{{4}}\\
\mbox{}+{x_{{5}}}^{3}+{x_{{7}}}^{3},{x_{{3}}}^{3}{x_{{8}}}^{3}+x_{{5}}x_{{6}}x_{{7}}x_{{8}}+{x_{{2}}}^{3}\\
\mbox{}+{x_{{4}}}^{3} \right\} \]}
\end{maplelatex}
\mapleresult
\begin{maplelatex}
\mapleinline{inert}{2d}{IrrIdeal = (x[5]*x[6]*x[8], x[6]*x[7]*x[8], x[1]*x[6]*x[8], x[1]*x[2]*x[6], x[3]*x[5]*x[8], x[3]*x[7]*x[8], x[1]*x[3]*x[8], x[2]*x[5], x[2]*x[7], x[1]*x[2]*x[3], x[1]*x[4]*x[6], x[4]*x[5], x[4]*x[7], x[1]*x[3]*x[4])}{\[\displaystyle {\it IrrIdeal}={x_{{5}}x_{{6}}x_{{8}},x_{{6}}x_{{7}}x_{{8}}\\
\mbox{},x_{{1}}x_{{6}}x_{{8}},x_{{1}}x_{{2}}x_{{6}},x_{{3}}x_{{5}}x_{{8}},x_{{3}}\\
\mbox{}x_{{7}}x_{{8}},x_{{1}}x_{{3}}x_{{8}},x_{{2}}x_{{5}},x_{{2}}x_{{7}},x_{{1}}\\
\mbox{}x_{{2}}x_{{3}},x_{{1}}x_{{4}}x_{{6}},x_{{4}}x_{{5}},x_{{4}}x_{{7}},x_{{1}}\\
\mbox{}x_{{3}}x_{{4}}}\]}
\end{maplelatex}
\mapleresult
\begin{maplelatex}
\mapleinline{inert}{2d}{A = {1, 3, 11, 12}, ftv = [Matrix(
\mbox{},1,1]]\]}
\end{maplelatex}
\mapleresult
\begin{maplelatex}
\mapleinline{inert}{2d}{Polynomials = {x[1]^3*x[6]^3+x[1]*x[2]*x[3]*x[4]+x[5]^3+x[7]^3, x[2]^3*x[8]^3+x[5]*x[6]*x[7]*x[8]+x[3]^3+x[4]^3}}{\[\displaystyle {\it Polynomials}= \left\{ {x_{{1}}}^{3}{x_{{6}}}^{3}+x_{{1}}x_{{2}}x_{{3}}x_{{4}}\\
\mbox{}+{x_{{5}}}^{3}+{x_{{7}}}^{3},{x_{{2}}}^{3}{x_{{8}}}^{3}+x_{{5}}x_{{6}}x_{{7}}x_{{8}}+{x_{{3}}}^{3}\\
\mbox{}+{x_{{4}}}^{3} \right\} \]}
\end{maplelatex}
\mapleresult
\begin{maplelatex}
\mapleinline{inert}{2d}{IrrIdeal = (x[5]*x[6]*x[8], x[6]*x[7]*x[8], x[1]*x[6]*x[8], x[2]*x[5]*x[8], x[2]*x[7]*x[8], x[1]*x[2]*x[8], x[1]*x[3]*x[6], x[3]*x[5], x[3]*x[7], x[1]*x[2]*x[3], x[1]*x[4]*x[6], x[4]*x[5], x[4]*x[7], x[1]*x[2]*x[4])}{\[\displaystyle {\it IrrIdeal}={x_{{5}}x_{{6}}x_{{8}},x_{{6}}x_{{7}}x_{{8}}\\
\mbox{},x_{{1}}x_{{6}}x_{{8}},x_{{2}}x_{{5}}x_{{8}},x_{{2}}x_{{7}}x_{{8}},x_{{1}}\\
\mbox{}x_{{2}}x_{{8}},x_{{1}}x_{{3}}x_{{6}},x_{{3}}x_{{5}},x_{{3}}x_{{7}},x_{{1}}\\
\mbox{}x_{{2}}x_{{3}},x_{{1}}x_{{4}}x_{{6}},x_{{4}}x_{{5}},x_{{4}}x_{{7}},x_{{1}}\\
\mbox{}x_{{2}}x_{{4}}}\]}
\end{maplelatex}
\mapleresult
\begin{maplelatex}
\mapleinline{inert}{2d}{A = {1, 10, 11, 12}, ftv = [Matrix(
\mbox{},0,1,1,1]]\]}
\end{maplelatex}
\mapleresult
\begin{maplelatex}
\mapleinline{inert}{2d}{Polynomials = {x[3]^3+x[4]^3+x[5]^3+x[6]*x[7]*x[8], x[1]^3*x[7]^3+x[2]^3*x[8]^3+x[1]*x[2]*x[3]*x[4]*x[5]+x[6]^3}}{\[\displaystyle {\it Polynomials}= \left\{ {x_{{3}}}^{3}+{x_{{4}}}^{3}+{x_{{5}}}^{3}+x_{{6}}x_{{7}}x_{{8}}\\
\mbox{},{x_{{1}}}^{3}{x_{{7}}}^{3}+{x_{{2}}}^{3}{x_{{8}}}^{3}+x_{{1}}x_{{2}}x_{{3}}x_{{4}}x_{{5}}\\
\mbox{}+{x_{{6}}}^{3} \right\} \]}
\end{maplelatex}
\mapleresult
\begin{maplelatex}
\mapleinline{inert}{2d}{IrrIdeal = (x[6]*x[7]*x[8], x[1]*x[7]*x[8], x[2]*x[7]*x[8], x[1]*x[3]*x[7], x[3]*x[6], x[2]*x[3]*x[8], x[1]*x[2]*x[3], x[1]*x[4]*x[7], x[4]*x[6], x[2]*x[4]*x[8], x[1]*x[2]*x[4], x[1]*x[5]*x[7], x[5]*x[6], x[2]*x[5]*x[8], x[1]*x[2]*x[5])}{\[\displaystyle {\it IrrIdeal}={x_{{6}}x_{{7}}x_{{8}},x_{{1}}x_{{7}}x_{{8}}\\
\mbox{},x_{{2}}x_{{7}}x_{{8}},x_{{1}}x_{{3}}x_{{7}},x_{{3}}x_{{6}},x_{{2}}\\
\mbox{}x_{{3}}x_{{8}},x_{{1}}x_{{2}}x_{{3}},x_{{1}}x_{{4}}x_{{7}},x_{{4}}x_{{6}}\\
\mbox{},x_{{2}}x_{{4}}x_{{8}},x_{{1}}x_{{2}}x_{{4}},x_{{1}}x_{{5}}x_{{7}},x_{{5}}\\
\mbox{}x_{{6}},x_{{2}}x_{{5}}x_{{8}},x_{{1}}x_{{2}}x_{{5}}}\]}
\end{maplelatex}
\mapleresult
\begin{maplelatex}
\mapleinline{inert}{2d}{A = {2, 3, 10, 11}, ftv = [Matrix(
\mbox{},1,1]]\]}
\end{maplelatex}
\mapleresult
\begin{maplelatex}
\mapleinline{inert}{2d}{Polynomials = {x[1]^3*x[5]^3+x[1]*x[2]*x[3]*x[4]+x[6]^3+x[7]^3, x[4]^3*x[8]^3+x[5]*x[6]*x[7]*x[8]+x[2]^3+x[3]^3}}{\[\displaystyle {\it Polynomials}= \left\{ {x_{{1}}}^{3}{x_{{5}}}^{3}+x_{{1}}x_{{2}}x_{{3}}x_{{4}}\\
\mbox{}+{x_{{6}}}^{3}+{x_{{7}}}^{3},{x_{{4}}}^{3}{x_{{8}}}^{3}+x_{{5}}x_{{6}}x_{{7}}x_{{8}}+{x_{{2}}}^{3}\\
\mbox{}+{x_{{3}}}^{3} \right\} \]}
\end{maplelatex}
\mapleresult
\begin{maplelatex}
\mapleinline{inert}{2d}{IrrIdeal = (x[1]*x[5]*x[8], x[5]*x[6]*x[8], x[5]*x[7]*x[8], x[1]*x[2]*x[5], x[1]*x[3]*x[5], x[1]*x[4]*x[8], x[4]*x[6]*x[8], x[4]*x[7]*x[8], x[1]*x[2]*x[4], x[2]*x[6], x[2]*x[7], x[1]*x[3]*x[4], x[3]*x[6], x[3]*x[7])}{\[\displaystyle {\it IrrIdeal}={x_{{1}}x_{{5}}x_{{8}},x_{{5}}x_{{6}}x_{{8}}\\
\mbox{},x_{{5}}x_{{7}}x_{{8}},x_{{1}}x_{{2}}x_{{5}},x_{{1}}x_{{3}}x_{{5}},x_{{1}}\\
\mbox{}x_{{4}}x_{{8}},x_{{4}}x_{{6}}x_{{8}},x_{{4}}x_{{7}}x_{{8}},x_{{1}}x_{{2}}\\
\mbox{}x_{{4}},x_{{2}}x_{{6}},x_{{2}}x_{{7}},x_{{1}}x_{{3}}x_{{4}},x_{{3}}x_{{6}}\\
\mbox{},x_{{3}}x_{{7}}}\]}
\end{maplelatex}
\mapleresult
\begin{maplelatex}
\mapleinline{inert}{2d}{A = {2, 3, 10, 12}, ftv = [Matrix(
\mbox{},1,1]]\]}
\end{maplelatex}
\mapleresult
\begin{maplelatex}
\mapleinline{inert}{2d}{Polynomials = {x[1]^3*x[5]^3+x[1]*x[2]*x[3]*x[4]+x[6]^3+x[7]^3, x[3]^3*x[8]^3+x[5]*x[6]*x[7]*x[8]+x[2]^3+x[4]^3}}{\[\displaystyle {\it Polynomials}= \left\{ {x_{{1}}}^{3}{x_{{5}}}^{3}+x_{{1}}x_{{2}}x_{{3}}x_{{4}}\\
\mbox{}+{x_{{6}}}^{3}+{x_{{7}}}^{3},{x_{{3}}}^{3}{x_{{8}}}^{3}+x_{{5}}x_{{6}}x_{{7}}x_{{8}}+{x_{{2}}}^{3}\\
\mbox{}+{x_{{4}}}^{3} \right\} \]}
\end{maplelatex}
\mapleresult
\begin{maplelatex}
\mapleinline{inert}{2d}{IrrIdeal = (x[1]*x[5]*x[8], x[5]*x[6]*x[8], x[5]*x[7]*x[8], x[1]*x[2]*x[5], x[1]*x[3]*x[8], x[3]*x[6]*x[8], x[3]*x[7]*x[8], x[1]*x[2]*x[3], x[2]*x[6], x[2]*x[7], x[1]*x[4]*x[5], x[1]*x[3]*x[4], x[4]*x[6], x[4]*x[7])}{\[\displaystyle {\it IrrIdeal}={x_{{1}}x_{{5}}x_{{8}},x_{{5}}x_{{6}}x_{{8}}\\
\mbox{},x_{{5}}x_{{7}}x_{{8}},x_{{1}}x_{{2}}x_{{5}},x_{{1}}x_{{3}}x_{{8}},x_{{3}}\\
\mbox{}x_{{6}}x_{{8}},x_{{3}}x_{{7}}x_{{8}},x_{{1}}x_{{2}}x_{{3}},x_{{2}}x_{{6}}\\
\mbox{},x_{{2}}x_{{7}},x_{{1}}x_{{4}}x_{{5}},x_{{1}}x_{{3}}x_{{4}},x_{{4}}\\
\mbox{}x_{{6}},x_{{4}}x_{{7}}}\]}
\end{maplelatex}
\mapleresult
\begin{maplelatex}
\mapleinline{inert}{2d}{A = {2, 3, 11, 12}, ftv = [Matrix(
\mbox{},1,1]]\]}
\end{maplelatex}
\mapleresult
\begin{maplelatex}
\mapleinline{inert}{2d}{Polynomials = {x[1]^3*x[5]^3+x[1]*x[2]*x[3]*x[4]+x[6]^3+x[7]^3, x[2]^3*x[8]^3+x[5]*x[6]*x[7]*x[8]+x[3]^3+x[4]^3}}{\[\displaystyle {\it Polynomials}= \left\{ {x_{{1}}}^{3}{x_{{5}}}^{3}+x_{{1}}x_{{2}}x_{{3}}x_{{4}}\\
\mbox{}+{x_{{6}}}^{3}+{x_{{7}}}^{3},{x_{{2}}}^{3}{x_{{8}}}^{3}+x_{{5}}x_{{6}}x_{{7}}x_{{8}}+{x_{{3}}}^{3}\\
\mbox{}+{x_{{4}}}^{3} \right\} \]}
\end{maplelatex}
\mapleresult
\begin{maplelatex}
\mapleinline{inert}{2d}{IrrIdeal = (x[1]*x[5]*x[8], x[5]*x[6]*x[8], x[5]*x[7]*x[8], x[1]*x[2]*x[8], x[2]*x[6]*x[8], x[2]*x[7]*x[8], x[1]*x[3]*x[5], x[1]*x[2]*x[3], x[3]*x[6], x[3]*x[7], x[1]*x[4]*x[5], x[1]*x[2]*x[4], x[4]*x[6], x[4]*x[7])}{\[\displaystyle {\it IrrIdeal}={x_{{1}}x_{{5}}x_{{8}},x_{{5}}x_{{6}}x_{{8}}\\
\mbox{},x_{{5}}x_{{7}}x_{{8}},x_{{1}}x_{{2}}x_{{8}},x_{{2}}x_{{6}}x_{{8}},x_{{2}}\\
\mbox{}x_{{7}}x_{{8}},x_{{1}}x_{{3}}x_{{5}},x_{{1}}x_{{2}}x_{{3}},x_{{3}}x_{{6}}\\
\mbox{},x_{{3}}x_{{7}},x_{{1}}x_{{4}}x_{{5}},x_{{1}}x_{{2}}x_{{4}},x_{{4}}\\
\mbox{}x_{{6}},x_{{4}}x_{{7}}}\]}
\end{maplelatex}
\mapleresult
\begin{maplelatex}
\mapleinline{inert}{2d}{A = {2, 10, 11, 12}, ftv = [Matrix(
\mbox{},0,1,1,1]]\]}
\end{maplelatex}
\mapleresult
\begin{maplelatex}
\mapleinline{inert}{2d}{Polynomials = {x[3]^3+x[4]^3+x[5]^3+x[6]*x[7]*x[8], x[1]^3*x[6]^3+x[2]^3*x[8]^3+x[1]*x[2]*x[3]*x[4]*x[5]+x[7]^3}}{\[\displaystyle {\it Polynomials}= \left\{ {x_{{3}}}^{3}+{x_{{4}}}^{3}+{x_{{5}}}^{3}+x_{{6}}x_{{7}}x_{{8}}\\
\mbox{},{x_{{1}}}^{3}{x_{{6}}}^{3}+{x_{{2}}}^{3}{x_{{8}}}^{3}+x_{{1}}x_{{2}}x_{{3}}x_{{4}}x_{{5}}\\
\mbox{}+{x_{{7}}}^{3} \right\} \]}
\end{maplelatex}
\mapleresult
\begin{maplelatex}
\mapleinline{inert}{2d}{IrrIdeal = (x[1]*x[6]*x[8], x[6]*x[7]*x[8], x[2]*x[6]*x[8], x[1]*x[3]*x[6], x[1]*x[2]*x[3], x[3]*x[7], x[2]*x[3]*x[8], x[1]*x[4]*x[6], x[1]*x[2]*x[4], x[4]*x[7], x[2]*x[4]*x[8], x[1]*x[5]*x[6], x[1]*x[2]*x[5], x[5]*x[7], x[2]*x[5]*x[8])}{\[\displaystyle {\it IrrIdeal}={x_{{1}}x_{{6}}x_{{8}},x_{{6}}x_{{7}}x_{{8}}\\
\mbox{},x_{{2}}x_{{6}}x_{{8}},x_{{1}}x_{{3}}x_{{6}},x_{{1}}x_{{2}}x_{{3}},x_{{3}}\\
\mbox{}x_{{7}},x_{{2}}x_{{3}}x_{{8}},x_{{1}}x_{{4}}x_{{6}},x_{{1}}x_{{2}}x_{{4}}\\
\mbox{},x_{{4}}x_{{7}},x_{{2}}x_{{4}}x_{{8}},x_{{1}}x_{{5}}x_{{6}},x_{{1}}\\
\mbox{}x_{{2}}x_{{5}},x_{{5}}x_{{7}},x_{{2}}x_{{5}}x_{{8}}}\]}
\end{maplelatex}
\mapleresult
\begin{maplelatex}
\mapleinline{inert}{2d}{A = {3, 10, 11, 12}, ftv = [Matrix(
\mbox{},0,1,1,1]]\]}
\end{maplelatex}
\mapleresult
\begin{maplelatex}
\mapleinline{inert}{2d}{Polynomials = {x[3]^3+x[4]^3+x[5]^3+x[6]*x[7]*x[8], x[1]^3*x[6]^3+x[2]^3*x[7]^3+x[1]*x[2]*x[3]*x[4]*x[5]+x[8]^3}}{\[\displaystyle {\it Polynomials}= \left\{ {x_{{3}}}^{3}+{x_{{4}}}^{3}+{x_{{5}}}^{3}+x_{{6}}x_{{7}}x_{{8}}\\
\mbox{},{x_{{1}}}^{3}{x_{{6}}}^{3}+{x_{{2}}}^{3}{x_{{7}}}^{3}+x_{{1}}x_{{2}}x_{{3}}x_{{4}}x_{{5}}\\
\mbox{}+{x_{{8}}}^{3} \right\} \]}
\end{maplelatex}
\mapleresult
\begin{maplelatex}
\mapleinline{inert}{2d}{IrrIdeal = (x[1]*x[6]*x[7], x[6]*x[7]*x[8], x[2]*x[6]*x[7], x[1]*x[3]*x[6], x[1]*x[2]*x[3], x[2]*x[3]*x[7], x[3]*x[8], x[1]*x[4]*x[6], x[1]*x[2]*x[4], x[2]*x[4]*x[7], x[4]*x[8], x[1]*x[5]*x[6], x[1]*x[2]*x[5], x[2]*x[5]*x[7], x[5]*x[8])}{\[\displaystyle {\it IrrIdeal}={x_{{1}}x_{{6}}x_{{7}},x_{{6}}x_{{7}}x_{{8}}\\
\mbox{},x_{{2}}x_{{6}}x_{{7}},x_{{1}}x_{{3}}x_{{6}},x_{{1}}x_{{2}}x_{{3}},x_{{2}}\\
\mbox{}x_{{3}}x_{{7}},x_{{3}}x_{{8}},x_{{1}}x_{{4}}x_{{6}},x_{{1}}x_{{2}}x_{{4}}\\
\mbox{},x_{{2}}x_{{4}}x_{{7}},x_{{4}}x_{{8}},x_{{1}}x_{{5}}x_{{6}},x_{{1}}\\
\mbox{}x_{{2}}x_{{5}},x_{{2}}x_{{5}}x_{{7}},x_{{5}}x_{{8}}}\]}
\end{maplelatex}
\mapleresult
\begin{maplelatex}
\mapleinline{inert}{2d}{A = {1, 2, 3, 10, 11}, ftv = [Matrix(
\mbox{}]]\]}
\end{maplelatex}
\mapleresult
\begin{maplelatex}
\mapleinline{inert}{2d}{Polynomials = {x[1]*x[2]*x[3]+x[4]^3+x[5]^3+x[6]^3, x[3]^3*x[7]^3+x[4]*x[5]*x[6]*x[7]+x[1]^3+x[2]^3}}{\[\displaystyle {\it Polynomials}= \left\{ x_{{1}}x_{{2}}x_{{3}}+{x_{{4}}}^{3}+{x_{{5}}}^{3}+{x_{{6}}}^{3}\\
\mbox{},{x_{{3}}}^{3}{x_{{7}}}^{3}+x_{{4}}x_{{5}}x_{{6}}x_{{7}}+{x_{{1}}}^{3}+{x_{{2}}}^{3} \right\} \]}
\end{maplelatex}
\mapleresult
\begin{maplelatex}
\mapleinline{inert}{2d}{IrrIdeal = (x[4]*x[7], x[5]*x[7], x[6]*x[7], x[1]*x[4], x[1]*x[5], x[1]*x[6], x[2]*x[4], x[2]*x[5], x[2]*x[6], x[3]*x[7], x[1]*x[3], x[2]*x[3])}{\[\displaystyle {\it IrrIdeal}={x_{{4}}x_{{7}},x_{{5}}x_{{7}},x_{{6}}x_{{7}},x_{{1}}\\
\mbox{}x_{{4}},x_{{1}}x_{{5}},x_{{1}}x_{{6}},x_{{2}}x_{{4}},x_{{2}}x_{{5}},x_{{2}}\\
\mbox{}x_{{6}},x_{{3}}x_{{7}},x_{{1}}x_{{3}},x_{{2}}x_{{3}}}\]}
\end{maplelatex}
\mapleresult
\begin{maplelatex}
\mapleinline{inert}{2d}{A = {1, 2, 3, 10, 12}, ftv = [Matrix(
\mbox{}]]\]}
\end{maplelatex}
\mapleresult
\begin{maplelatex}
\mapleinline{inert}{2d}{Polynomials = {x[1]*x[2]*x[3]+x[4]^3+x[5]^3+x[6]^3, x[2]^3*x[7]^3+x[4]*x[5]*x[6]*x[7]+x[1]^3+x[3]^3}}{\[\displaystyle {\it Polynomials}= \left\{ x_{{1}}x_{{2}}x_{{3}}+{x_{{4}}}^{3}+{x_{{5}}}^{3}+{x_{{6}}}^{3}\\
\mbox{},{x_{{2}}}^{3}{x_{{7}}}^{3}+x_{{4}}x_{{5}}x_{{6}}x_{{7}}+{x_{{1}}}^{3}+{x_{{3}}}^{3} \right\} \]}
\end{maplelatex}
\mapleresult
\begin{maplelatex}
\mapleinline{inert}{2d}{IrrIdeal = (x[4]*x[7], x[5]*x[7], x[6]*x[7], x[1]*x[4], x[1]*x[5], x[1]*x[6], x[2]*x[7], x[1]*x[2], x[3]*x[4], x[3]*x[5], x[3]*x[6], x[2]*x[3])}{\[\displaystyle {\it IrrIdeal}={x_{{4}}x_{{7}},x_{{5}}x_{{7}},x_{{6}}x_{{7}},x_{{1}}\\
\mbox{}x_{{4}},x_{{1}}x_{{5}},x_{{1}}x_{{6}},x_{{2}}x_{{7}},x_{{1}}x_{{2}},x_{{3}}\\
\mbox{}x_{{4}},x_{{3}}x_{{5}},x_{{3}}x_{{6}},x_{{2}}x_{{3}}}\]}
\end{maplelatex}
\mapleresult
\begin{maplelatex}
\mapleinline{inert}{2d}{A = {1, 2, 3, 11, 12}, ftv = [Matrix(
\mbox{}]]\]}
\end{maplelatex}
\mapleresult
\begin{maplelatex}
\mapleinline{inert}{2d}{Polynomials = {x[1]*x[2]*x[3]+x[4]^3+x[5]^3+x[6]^3, x[1]^3*x[7]^3+x[4]*x[5]*x[6]*x[7]+x[2]^3+x[3]^3}}{\[\displaystyle {\it Polynomials}= \left\{ x_{{1}}x_{{2}}x_{{3}}+{x_{{4}}}^{3}+{x_{{5}}}^{3}+{x_{{6}}}^{3}\\
\mbox{},{x_{{1}}}^{3}{x_{{7}}}^{3}+x_{{4}}x_{{5}}x_{{6}}x_{{7}}+{x_{{2}}}^{3}+{x_{{3}}}^{3} \right\} \]}
\end{maplelatex}
\mapleresult
\begin{maplelatex}
\mapleinline{inert}{2d}{IrrIdeal = (x[4]*x[7], x[5]*x[7], x[6]*x[7], x[1]*x[7], x[2]*x[4], x[2]*x[5], x[2]*x[6], x[1]*x[2], x[3]*x[4], x[3]*x[5], x[3]*x[6], x[1]*x[3])}{\[\displaystyle {\it IrrIdeal}={x_{{4}}x_{{7}},x_{{5}}x_{{7}},x_{{6}}x_{{7}},x_{{1}}\\
\mbox{}x_{{7}},x_{{2}}x_{{4}},x_{{2}}x_{{5}},x_{{2}}x_{{6}},x_{{1}}x_{{2}},x_{{3}}\\
\mbox{}x_{{4}},x_{{3}}x_{{5}},x_{{3}}x_{{6}},x_{{1}}x_{{3}}}\]}
\end{maplelatex}
\mapleresult
\begin{maplelatex}
\mapleinline{inert}{2d}{A = {1, 2, 10, 11, 12}, ftv = [Matrix(
\mbox{}]]\]}
\end{maplelatex}
\mapleresult
\begin{maplelatex}
\mapleinline{inert}{2d}{Polynomials = {x[2]^3+x[3]^3+x[4]^3+x[5]*x[6]*x[7], x[1]^3*x[7]^3+x[1]*x[2]*x[3]*x[4]+x[5]^3+x[6]^3}}{\[\displaystyle {\it Polynomials}= \left\{ {x_{{2}}}^{3}+{x_{{3}}}^{3}+{x_{{4}}}^{3}+x_{{5}}x_{{6}}x_{{7}}\\
\mbox{},{x_{{1}}}^{3}{x_{{7}}}^{3}+x_{{1}}x_{{2}}x_{{3}}x_{{4}}+{x_{{5}}}^{3}+{x_{{6}}}^{3} \right\} \]}
\end{maplelatex}
\mapleresult
\begin{maplelatex}
\mapleinline{inert}{2d}{IrrIdeal = (x[5]*x[7], x[6]*x[7], x[1]*x[7], x[2]*x[5], x[2]*x[6], x[1]*x[2], x[3]*x[5], x[3]*x[6], x[1]*x[3], x[4]*x[5], x[4]*x[6], x[1]*x[4])}{\[\displaystyle {\it IrrIdeal}={x_{{5}}x_{{7}},x_{{6}}x_{{7}},x_{{1}}x_{{7}},x_{{2}}\\
\mbox{}x_{{5}},x_{{2}}x_{{6}},x_{{1}}x_{{2}},x_{{3}}x_{{5}},x_{{3}}x_{{6}},x_{{1}}\\
\mbox{}x_{{3}},x_{{4}}x_{{5}},x_{{4}}x_{{6}},x_{{1}}x_{{4}}}\]}
\end{maplelatex}
\mapleresult
\begin{maplelatex}
\mapleinline{inert}{2d}{A = {1, 3, 10, 11, 12}, ftv = [Matrix(
\mbox{}]]\]}
\end{maplelatex}
\mapleresult
\begin{maplelatex}
\mapleinline{inert}{2d}{Polynomials = {x[2]^3+x[3]^3+x[4]^3+x[5]*x[6]*x[7], x[1]^3*x[6]^3+x[1]*x[2]*x[3]*x[4]+x[5]^3+x[7]^3}}{\[\displaystyle {\it Polynomials}= \left\{ {x_{{2}}}^{3}+{x_{{3}}}^{3}+{x_{{4}}}^{3}+x_{{5}}x_{{6}}x_{{7}}\\
\mbox{},{x_{{1}}}^{3}{x_{{6}}}^{3}+x_{{1}}x_{{2}}x_{{3}}x_{{4}}+{x_{{5}}}^{3}+{x_{{7}}}^{3} \right\} \]}
\end{maplelatex}
\mapleresult
\begin{maplelatex}
\mapleinline{inert}{2d}{IrrIdeal = (x[5]*x[6], x[6]*x[7], x[1]*x[6], x[2]*x[5], x[2]*x[7], x[1]*x[2], x[3]*x[5], x[3]*x[7], x[1]*x[3], x[4]*x[5], x[4]*x[7], x[1]*x[4])}{\[\displaystyle {\it IrrIdeal}={x_{{5}}x_{{6}},x_{{6}}x_{{7}},x_{{1}}x_{{6}},x_{{2}}\\
\mbox{}x_{{5}},x_{{2}}x_{{7}},x_{{1}}x_{{2}},x_{{3}}x_{{5}},x_{{3}}x_{{7}},x_{{1}}\\
\mbox{}x_{{3}},x_{{4}}x_{{5}},x_{{4}}x_{{7}},x_{{1}}x_{{4}}}\]}
\end{maplelatex}
\mapleresult
\begin{maplelatex}
\mapleinline{inert}{2d}{A = {2, 3, 10, 11, 12}, ftv = [Matrix(
\mbox{}]]\]}
\end{maplelatex}
\mapleresult
\begin{maplelatex}
\mapleinline{inert}{2d}{Polynomials = {x[2]^3+x[3]^3+x[4]^3+x[5]*x[6]*x[7], x[1]^3*x[5]^3+x[1]*x[2]*x[3]*x[4]+x[6]^3+x[7]^3}}{\[\displaystyle {\it Polynomials}= \left\{ {x_{{2}}}^{3}+{x_{{3}}}^{3}+{x_{{4}}}^{3}+x_{{5}}x_{{6}}x_{{7}}\\
\mbox{},{x_{{1}}}^{3}{x_{{5}}}^{3}+x_{{1}}x_{{2}}x_{{3}}x_{{4}}+{x_{{6}}}^{3}+{x_{{7}}}^{3} \right\} \]}
\end{maplelatex}
\mapleresult
\begin{maplelatex}
\mapleinline{inert}{2d}{IrrIdeal = (x[1]*x[5], x[5]*x[6], x[5]*x[7], x[1]*x[2], x[2]*x[6], x[2]*x[7], x[1]*x[3], x[3]*x[6], x[3]*x[7], x[1]*x[4], x[4]*x[6], x[4]*x[7])}{\[\displaystyle {\it IrrIdeal}={x_{{1}}x_{{5}},x_{{5}}x_{{6}},x_{{5}}x_{{7}},x_{{1}}\\
\mbox{}x_{{2}},x_{{2}}x_{{6}},x_{{2}}x_{{7}},x_{{1}}x_{{3}},x_{{3}}x_{{6}},x_{{3}}\\
\mbox{}x_{{7}},x_{{1}}x_{{4}},x_{{4}}x_{{6}},x_{{4}}x_{{7}}}\]}
\end{maplelatex}
\mapleresult
\begin{maplelatex}
\mapleinline{inert}{2d}{A = {1, 2, 3, 10, 11, 12}, ftv = [Matrix(
\end{maplelatex}
\mapleresult
\begin{maplelatex}
\mapleinline{inert}{2d}{Polynomials = {x[1]*x[2]*x[3]+x[4]^3+x[5]^3+x[6]^3, x[1]^3+x[2]^3+x[3]^3+x[4]*x[5]*x[6]}}{\[\displaystyle {\it Polynomials}= \left\{ x_{{1}}x_{{2}}x_{{3}}+{x_{{4}}}^{3}+{x_{{5}}}^{3}+{x_{{6}}}^{3}\\
\mbox{},{x_{{1}}}^{3}+{x_{{2}}}^{3}+{x_{{3}}}^{3}+x_{{4}}x_{{5}}x_{{6}} \right\} \]}
\end{maplelatex}
\mapleresult
\begin{maplelatex}
\mapleinline{inert}{2d}{IrrIdeal = (x[4], x[5], x[6], x[1], x[2], x[3])}{\[\displaystyle {\it IrrIdeal}={x_{{4}},x_{{5}},x_{{6}},x_{{1}},x_{{2}},x_{{3}}}\]}
\end{maplelatex}
\end{maplegroup}
}

\end{document}